\documentclass[11pt]{amsart}

\pdfoutput=1

\usepackage{graphicx}
\usepackage{amsmath,amsxtra,amssymb,latexsym, amscd,amsthm}

\theoremstyle{plain}
\newtheorem{thm}{Theorem}[section]

\newtheorem{lem}[thm]{Lemma}

\theoremstyle{definition}

\newtheorem{op}[thm]{Open Problem}
\newtheorem{rmk}[thm]{Remark}

\usepackage{geometry}
\numberwithin{equation}{section}
\numberwithin{figure}{section}
\numberwithin{table}{section}
\newtheorem{conj}{Conjecture}
\usepackage{amsthm,amsmath,amssymb}

\usepackage[colorlinks=true,citecolor=blue,linkcolor=black,urlcolor=blue]{hyperref}

\usepackage{graphicx}

\newcommand{\M}{\operatorname{M}}
\newcommand{\Hf}{\operatorname{H}}
\newcommand{\wt}{\operatorname{wt}}

\newcommand{\Vol}{\operatorname{Vol}}

\begin{document}

\title{Tiling Enumeration of Hexagons with Off-central Holes}

\author{Mihai Ciucu}
\address{Department of Mathematics, Indiana University, Bloomington, IN 47405. USA}
\email{mciucu@indiana.edu}
\thanks{M.C. was supported in part by NSF grants DMS-1101670 and DMS-1501052.}

\author{Tri Lai}
\address{Department of Mathematics, University of Nebraska -- Lincoln, Lincoln, NE 68588, USA}
\email{tlai3@unl.edu}
\thanks{T.L. was supported in part by Simons Foundation Collaboration Grant (\# 585923) and by the Institute for Mathematics and its Applications with funds provided by NSF grant DMS-0931945.}



\date{\today}

\dedicatory{}

\begin{abstract}
Motivated in part by Propp's intruded Aztec diamond regions, we consider hexagonal regions out of which two horizontal chains of triangular holes (called ferns) are removed, so that the chains are at the same height, and are attached to the boundary. By contrast with the intruded Aztec diamonds (whose number of domino tilings contain some large prime factors in their factorization), the number of lozenge tilings of our doubly-intruded hexagons turns out to be given by simple product formulas in which all factors are linear in the parameters. We present in fact $q$-versions of these formulas, which enumerate the corresponding plane-partitions-like structures by their volume. We also pose some natural statistical mechanics questions suggested by our set-up, which should be possible to tackle using our formulas.
\end{abstract}

\maketitle

\section{Introduction}

A large part of the intensity with which tilings of lattice regions have been studied in the last few decades has its origin in MacMahon's century-old theorem \cite{Mac} on the enumeration of plane partitions that fit in a given rectangular box (see \cite{And}\cite{Sta}\cite{Kup}\cite{Ste}\cite{KKZ} and the survey \cite{Bres} for more recent developments). In an equivalent formulation (see \cite{DT}), this states that the number of lozenge tilings of a semi-regular hexagon\footnote{ We consider the triangular lattice drawn so that one family of lattice lines is horizontal. A lozenge is the union of any two unit triangles that share an edge, and a lozenge tiling of a lattice region is a covering of it by lozenges, with no gaps or overlaps.}
of side-lengths $a$, $b$, $c$, $a$, $b$, $c$ (in cyclic order)
on the triangular lattice is given by the compelling formula
\begin{equation}\label{MacMahoneq}
P(a,b,c):=\frac{\Hf(a)\Hf(b)\Hf(c)\Hf(a+b+c)}{\Hf(a+b)\Hf(b+c)\Hf(c+a)},
\end{equation}
where $\Hf(n)$ (the hyperfactorial) is given by
\begin{equation*}
\Hf(n):=0!\,1!\,2!\cdots(n-1)!
\end{equation*}

The elegance of this formula invites one to search for extensions of it, and this led to our previous work \cite{pp1}, in which we considered triangular gaps along the vertical symmetry axis of the hexagon (this in turn led to the connections to electrostatics worked out in \cite{sc} and \cite{ov}), and more recently to \cite{ff} and \cite{fv}, in which we removed from the center of the hexagon four-lobed or multi-lobed structures called shamrocks and ferns, respectively. The number of lozenge tilings of each of these regions was seen to be given by a simple product formula. The papers \cite{Tri1} and \cite{Tri2} provide $q$-enumerations for lozenge tilings of hexagons from which three two-lobed structures, called bowties, respectively a shamrock, are removed next to the boundary.

In this paper we consider lattice hexagons in which ``intrusions'' (reminiscent of Propp's intruded Aztec diamonds described in \cite{Propp}) are made by two ferns lined up along a common horizontal lattice line. Again, the resulting regions turn out to have their lozenge tilings enumerated by simple product formulas (this does not seem to be the case for the intruded Aztec diamonds, whose number of domino tilings contain some large prime factors in their factorization). We now describe in detail our regions and state the formulas. It will help the exposition to first consider the unweighted case.

\section{Statement of unweighted tiling enumeration results}

Let $s(b_1,b_2,\dotsc,b_l)$ denote the number of lozenge tilings of the semihexagonal region $S(b_1,b_2,\dotsc,b_l)$ with the leftmost $b_1$ up-pointing unit triangles on its base removed, the next segment of length $b_2$ intact, the following $b_3$ removed, and so on (see an illustration in Figure\ref{semihexmultiple}; note that the $b_i$'s, together with the requirement that the region contains an equal number of unit triangles of the two orientations --- a necessary condition for the existence of tilings --- determine the lengths of all four sides of the semihexagon). By the Cohn-Larsen-Propp \cite{CLP} interpretation of the Gelfand-Tsetlin result \cite{GT} we have that\footnote{ The first equality in (2.1) holds due to forced lozenges in the tilings of $S(b_1,b_2,\dotsc,b_{2l})$, after whose removal one is left precisely with the region $S(b_1,b_2,\dotsc,b_{2l-1})$.}${}^{,}$\footnote{ We include here the original formula for convenience. Let $T_{m,n}(x_1,\dotsc,x_n)$ be the region obtained from the trapezoid of side lengths $m$, $n$, $m+n$, $n$ (clockwise from top) by removing the up-pointing unit triangles from along its bottom that are in positions $x_1,x_2,\dotsc,x_n$ as counted from left to right. Then the number of lozenge tilings of  $T_{m,n}(x_1,\dotsc,x_n)$ is equal to
$\prod_{1\leq i<j\leq n}\frac{x_j-x_i}{j-i}$.
}

\begin{align}\label{semieq}
s(b_1,b_2,\dots,b_{2l-1})&=s(b_1,b_2,\dots,b_{2l})\notag\\
&=\dfrac{1}{\Hf(b_1+b_{3}+b_{5}+\dotsc+b_{2l-1})}\notag\\
&\,\,\,\times\dfrac{\prod_{\substack{1\leq i\leq j\leq 2l-1,\,\text{$j-i+1$ odd}}}\Hf(b_i+b_{i+1}+\dotsc+b_{j})}{\prod_{\substack{1\leq i\leq j\leq 2l-1,\,\text{$j-i+1$ even}}}\Hf(b_i+b_{i+1}+\dotsc+b_{j})}.
\end{align}

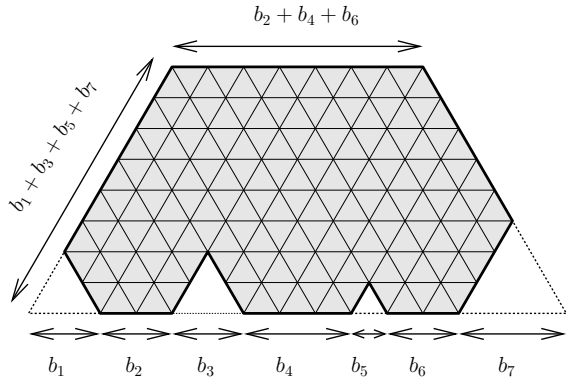
\begin{figure}\centering
\setlength{\unitlength}{3947sp}%
\begingroup\makeatletter\ifx\SetFigFont\undefined%
\gdef\SetFigFont#1#2#3#4#5{%
  \reset@font\fontsize{#1}{#2pt}%
  \fontfamily{#3}\fontseries{#4}\fontshape{#5}%
  \selectfont}%
\fi\endgroup%
\resizebox{8cm}{!}{
\begin{picture}(0,0)%
\includegraphics{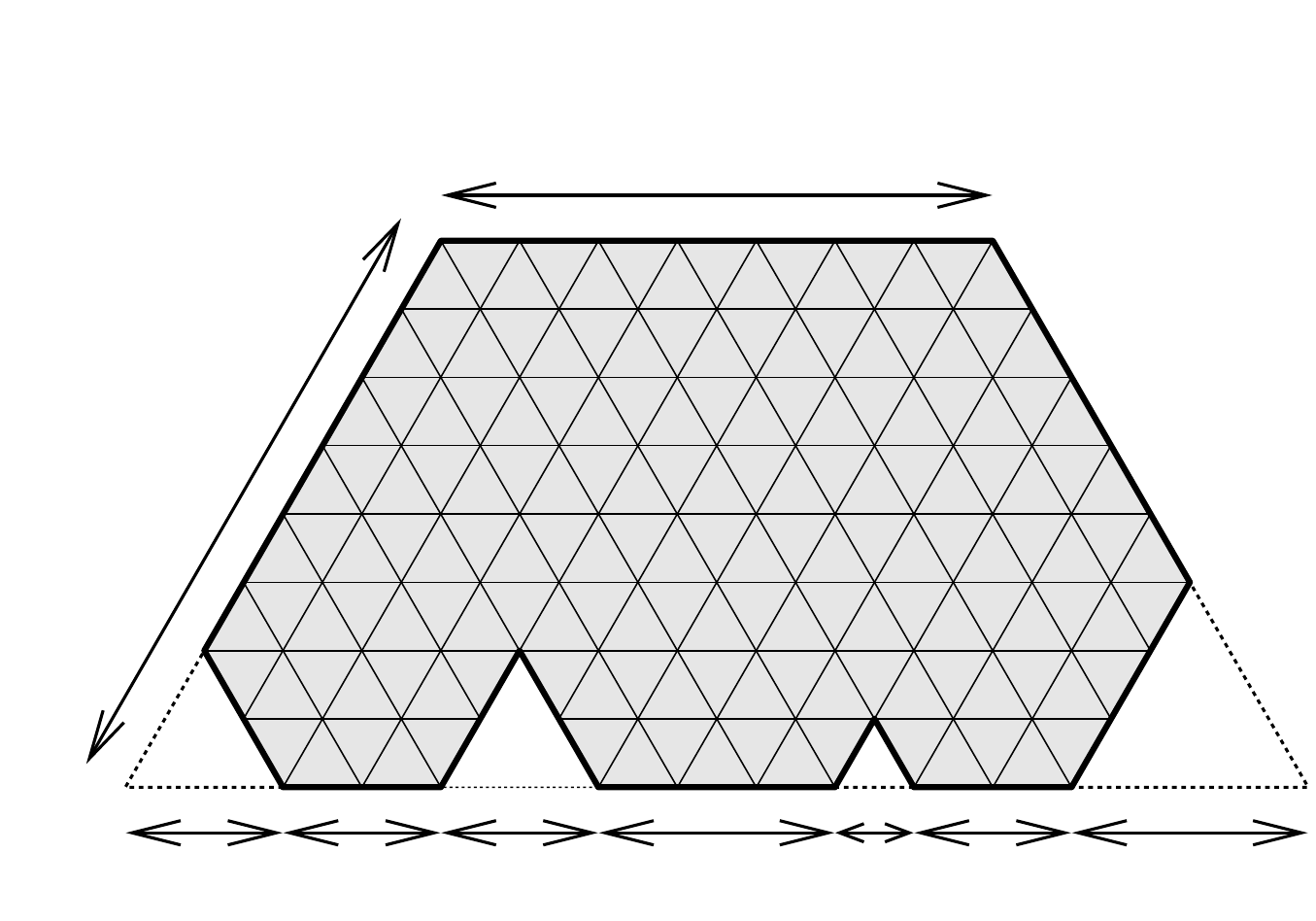}%
\end{picture}

\begin{picture}(6488,4544)(1524,-5559)
\put(4531,-1700){\makebox(0,0)[lb]{\smash{{\SetFigFont{14}{16.8}{\rmdefault}{\mddefault}{\itdefault}{$b_2+b_4+b_6$}%
}}}}
\put(2000,-3800){\rotatebox{60.0}{\makebox(0,0)[lb]{\smash{{\SetFigFont{14}{16.8}{\rmdefault}{\mddefault}{\itdefault}{$b_1+b_3+b_5+b_7$}%
}}}}}
\put(2296,-5544){\makebox(0,0)[lb]{\smash{{\SetFigFont{14}{16.8}{\rmdefault}{\mddefault}{\itdefault}{$b_1$}%
}}}}
\put(3136,-5544){\makebox(0,0)[lb]{\smash{{\SetFigFont{14}{16.8}{\rmdefault}{\mddefault}{\itdefault}{$b_2$}%
}}}}
\put(3916,-5544){\makebox(0,0)[lb]{\smash{{\SetFigFont{14}{16.8}{\rmdefault}{\mddefault}{\itdefault}{$b_3$}%
}}}}
\put(4771,-5544){\makebox(0,0)[lb]{\smash{{\SetFigFont{14}{16.8}{\rmdefault}{\mddefault}{\itdefault}{$b_4$}%
}}}}
\put(5589,-5544){\makebox(0,0)[lb]{\smash{{\SetFigFont{14}{16.8}{\rmdefault}{\mddefault}{\itdefault}{$b_5$}%
}}}}
\put(6219,-5544){\makebox(0,0)[lb]{\smash{{\SetFigFont{14}{16.8}{\rmdefault}{\mddefault}{\itdefault}{$b_6$}%
}}}}
\put(7186,-5544){\makebox(0,0)[lb]{\smash{{\SetFigFont{14}{16.8}{\rmdefault}{\mddefault}{\itdefault}{$b_7$}%
}}}}
\end{picture}}
\caption{The semihexagon with multiple holes $S(2,2,2,3,1,2,3)$.}
\label{semihexmultiple}
\end{figure}

We recall from \cite{fv} that a {\it fern} is a sequence of contiguous equilateral triangles of alternating orientations on the triangular lattice. In \cite{fv} the first author showed that the region obtained from a lattice hexagon by removing an arbitrary fern from its center has tilings enumerated by a simple product formula. The present paper considers lattice hexagons with {\it two} ferns removed, lined up along a common horizontal lattice line and touching the boundary of the hexagon. We present simple product formulas for the number of their lozenge tilings, thus obtaining a new generalization of MacMahon's formula (\ref{MacMahoneq}).

There are two cases to distinguish, depending on the height of the common axis of the removed ferns (see Figure \ref{arrayQ}). Both depend on non-negative integer parameters $x$, $y$, $z$ and $t$, which determine the side-lengths of the hexagon from which the ferns are removed and the location of the horizontal from along which they are removed, and on two (possibly empty) lists $a_1,\dotsc,a_m$ and $b_1,\dotsc,b_n$ of non-negative integers specifying the sizes of the lobes of the two ferns.

\begin{figure}\centering
\setlength{\unitlength}{3947sp}%
\begingroup\makeatletter\ifx\SetFigFont\undefined%
\gdef\SetFigFont#1#2#3#4#5{%
  \reset@font\fontsize{#1}{#2pt}%
  \fontfamily{#3}\fontseries{#4}\fontshape{#5}%
  \selectfont}%
\fi\endgroup%
\resizebox{15cm}{!}{
\begin{picture}(0,0)%
\includegraphics{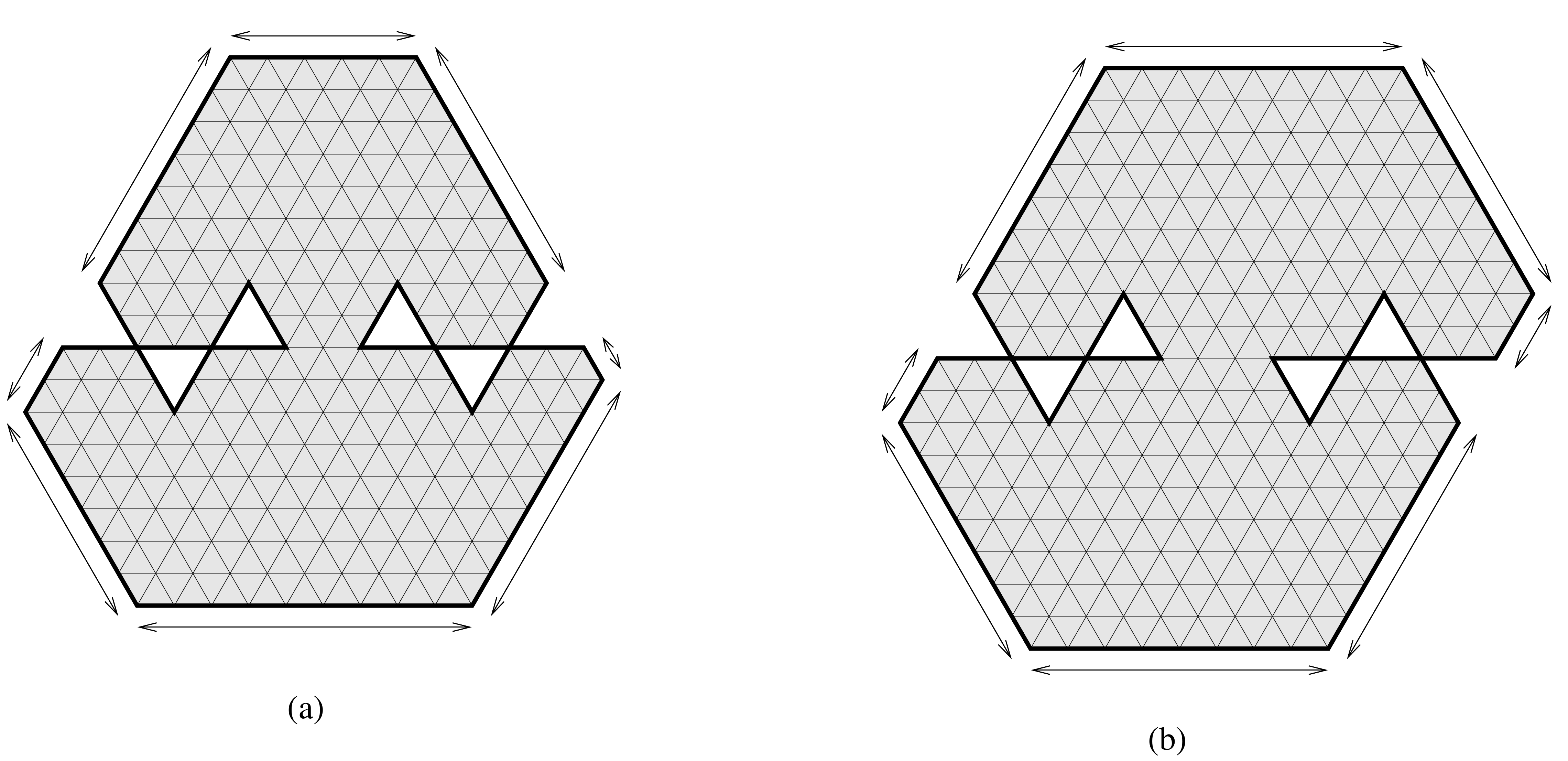}%
\end{picture}%
\begin{picture}(21881,10616)(21730,-12577)
\put(23086,-4831){\rotatebox{60.0}{\makebox(0,0)[lb]{\smash{{\SetFigFont{20}{24.0}{\rmdefault}{\mddefault}{\itdefault}{$y+a_3+b_1+b_3$}%
}}}}}
\put(25606,-2341){\makebox(0,0)[lb]{\smash{{\SetFigFont{20}{24.0}{\rmdefault}{\mddefault}{\itdefault}{$z+a_2+b_2$}%
}}}}
\put(28336,-3121){\rotatebox{300.0}{\makebox(0,0)[lb]{\smash{{\SetFigFont{20}{24.0}{\rmdefault}{\mddefault}{\itdefault}{$y+a_1+a_3+b_3$}%
}}}}}
\put(30481,-6871){\makebox(0,0)[lb]{\smash{{\SetFigFont{20}{24.0}{\rmdefault}{\mddefault}{\itdefault}{$t$}%
}}}}
\put(29596,-9811){\rotatebox{60.0}{\makebox(0,0)[lb]{\smash{{\SetFigFont{20}{24.0}{\rmdefault}{\mddefault}{\itdefault}{$y+x+a_2+b_2$}%
}}}}}
\put(21840,-8821){\rotatebox{300.0}{\makebox(0,0)[lb]{\smash{{\SetFigFont{20}{24.0}{\rmdefault}{\mddefault}{\itdefault}{$y+t+a_2+b_2$}%
}}}}}
\put(24481,-11191){\makebox(0,0)[lb]{\smash{{\SetFigFont{20}{24.0}{\rmdefault}{\mddefault}{\itdefault}{$z+a_1+a_3+b_1+b_3$}%
}}}}
\put(21796,-7021){\makebox(0,0)[lb]{\smash{{\SetFigFont{20}{24.0}{\rmdefault}{\mddefault}{\itdefault}{$x$}%
}}}}
\put(22861,-6661){\makebox(0,0)[lb]{\smash{{\SetFigFont{20}{24.0}{\rmdefault}{\mddefault}{\itdefault}{$a_1$}%
}}}}
\put(23986,-7246){\makebox(0,0)[lb]{\smash{{\SetFigFont{20}{24.0}{\rmdefault}{\mddefault}{\itdefault}{$a_2$}%
}}}}
\put(25006,-6631){\makebox(0,0)[lb]{\smash{{\SetFigFont{20}{24.0}{\rmdefault}{\mddefault}{\itdefault}{$a_3$}%
}}}}
\put(29236,-6601){\makebox(0,0)[lb]{\smash{{\SetFigFont{20}{24.0}{\rmdefault}{\mddefault}{\itdefault}{$b_1$}%
}}}}
\put(28156,-7216){\makebox(0,0)[lb]{\smash{{\SetFigFont{20}{24.0}{\rmdefault}{\mddefault}{\itdefault}{$b_2$}%
}}}}
\put(27046,-6631){\makebox(0,0)[lb]{\smash{{\SetFigFont{20}{24.0}{\rmdefault}{\mddefault}{\itdefault}{$b_3$}%
}}}}
\put(35282,-5116){\rotatebox{60.0}{\makebox(0,0)[lb]{\smash{{\SetFigFont{20}{24.0}{\rmdefault}{\mddefault}{\itdefault}{$y+t+a_3+b_2$}%
}}}}}
\put(38176,-2446){\makebox(0,0)[lb]{\smash{{\SetFigFont{20}{24.0}{\rmdefault}{\mddefault}{\itdefault}{$z+a_2+b_1+b_3$}%
}}}}
\put(42090,-3196){\rotatebox{300.0}{\makebox(0,0)[lb]{\smash{{\SetFigFont{20}{24.0}{\rmdefault}{\mddefault}{\itdefault}{$y+a_1+a_3+b_2$}%
}}}}}
\put(43246,-6946){\makebox(0,0)[lb]{\smash{{\SetFigFont{20}{24.0}{\rmdefault}{\mddefault}{\itdefault}{$t$}%
}}}}
\put(41296,-10741){\rotatebox{60.0}{\makebox(0,0)[lb]{\smash{{\SetFigFont{20}{24.0}{\rmdefault}{\mddefault}{\itdefault}{$x+y+a_2+b_3$}%
}}}}}
\put(36916,-11791){\makebox(0,0)[lb]{\smash{{\SetFigFont{20}{24.0}{\rmdefault}{\mddefault}{\itdefault}{$z+a_1+a_3+b_2$}%
}}}}
\put(34020,-9001){\rotatebox{300.0}{\makebox(0,0)[lb]{\smash{{\SetFigFont{20}{24.0}{\rmdefault}{\mddefault}{\itdefault}{$y+a_2+b_1+b_3$}%
}}}}}
\put(33886,-7126){\makebox(0,0)[lb]{\smash{{\SetFigFont{20}{24.0}{\rmdefault}{\mddefault}{\itdefault}{$x$}%
}}}}
\put(35041,-6751){\makebox(0,0)[lb]{\smash{{\SetFigFont{20}{24.0}{\rmdefault}{\mddefault}{\itdefault}{$a_1$}%
}}}}
\put(36181,-7336){\makebox(0,0)[lb]{\smash{{\SetFigFont{20}{24.0}{\rmdefault}{\mddefault}{\itdefault}{$a_2$}%
}}}}
\put(37186,-6781){\makebox(0,0)[lb]{\smash{{\SetFigFont{20}{24.0}{\rmdefault}{\mddefault}{\itdefault}{$a_3$}%
}}}}
\put(39796,-7366){\makebox(0,0)[lb]{\smash{{\SetFigFont{20}{24.0}{\rmdefault}{\mddefault}{\itdefault}{$b_3$}%
}}}}
\put(40816,-6796){\makebox(0,0)[lb]{\smash{{\SetFigFont{20}{24.0}{\rmdefault}{\mddefault}{\itdefault}{$b_2$}%
}}}}
\put(41956,-7456){\makebox(0,0)[lb]{\smash{{\SetFigFont{20}{24.0}{\rmdefault}{\mddefault}{\itdefault}{$b_1$}%
}}}}
\end{picture}}
\caption{The regions (a) $P_{2,1,1,1}(2,2,2;2,2,2)$ and (b) $Q_{2,1,2,2}(2,2,2;2,2,2)$.}
\label{arrayQ}
\end{figure}
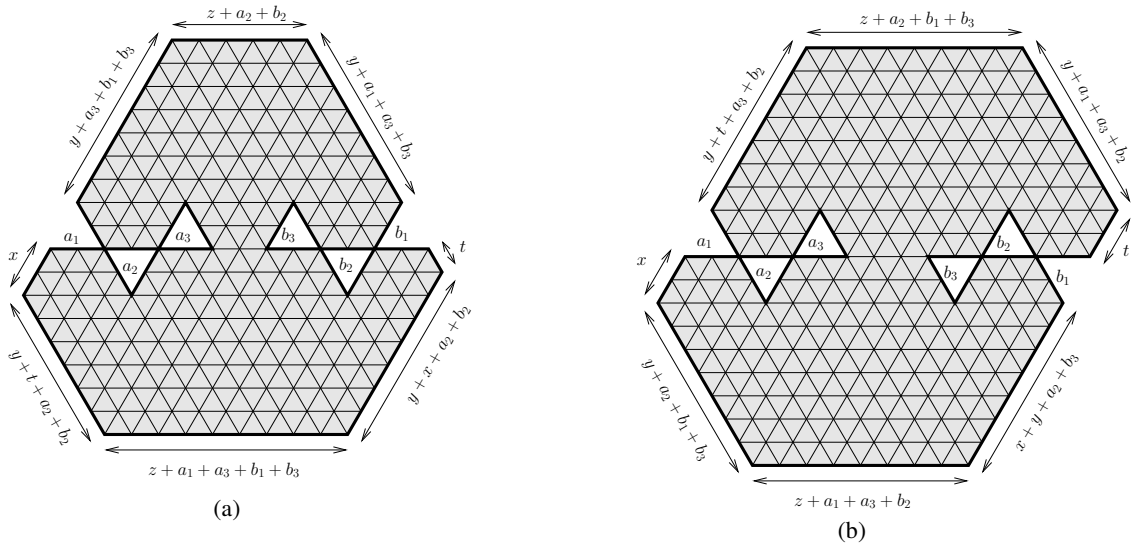

If the horizontal along which the ferns are lined up leaves the western and eastern vertices of the hexagon on the same side of it\footnote{ Without loss of generality we may assume that both vertices are below this horizontal.}, we obtain the $P$-regions, defined as follows. Set
\begin{align}
o_a&:=a_1+a_3+a_5+\cdots\\
e_a&:=a_2+a_4+a_6+\cdots\\
o_b&:=b_1+b_3+b_5+\cdots\\
e_b&:=b_2+b_4+b_6+\cdots\\
a&:=a_1+a_2+a_3+a_4+\cdots\\
b&:=b_1+b_2+b_3+b_4+\cdots
\end{align}
From the hexagon
of side-lengths\footnote{From now on, we always list the side-lengths of a hexagon on the triangular lattice in clockwise order, starting from the northwest side.} $x+y+o_a+o_b,z+e_a+e_b,y+t+o_a+o_b, x+y+e_a+e_b, z+o_a+o_b,y+t+e_a+e_b$, remove a fern of lobe-sizes (from left to right) $a_1,\dotsc,a_m$ and a fern of lobe-sizes (from right to left) $b_1,\dotsc,b_n$ as indicated in Figure \ref{arrayQ}(a) (in the figure, $x=2$, $y=z=t=1$, $m=n=3$, and $a_1=a_2=a_3=b_1=b_2=b_3=2$). The resulting region is defined to be our $P$-region $P_{x,y,z,t}(a_1,\dotsc,a_m;b_1,\dotsc,b_n)$. One readily sees that the distance separating the two ferns (i.e. the distance between the rightmost point of the left fern and the leftmost point of the right fern) is equal to $y+z$.

The second family of regions, which we call $Q$-regions, corresponds to the case when the horizontal through the ferns leaves the western and eastern vertices of the hexagon on opposite sides. Without loss of generality, we may assume that the western vertex is below and the eastern above this horizontal. Start with the hexagon of side-lengths $x+y+t+o_a+e_b,z+e_a+o_b, y+o_a+e_b, x+y+t+e_a+o_b, z+o_a+e_b, y+e_a+o_b$, and remove from it the same two ferns as before, but positioned as indicated in Figure \ref{arrayQ}(b) (in the figure, $x=2$, $y=1$, $z=t=2$, $m=n=3$ and $a_1=a_2=a_3=b_1=b_2=b_3=2$). We denote the resulting region by $Q_{x,y,z,t}(a_1,\dotsc,a_m;b_1,\dotsc,b_n)$. Just as in the case of the $P$-regions, the separation between the two ferns is $y+z$.

\medskip
Our results are stated in terms of the $\Phi$ and $\Psi$ functions defined as follows. For non-negative integers $x,y,z,t$ and $a_1,\dotsc,a_m,b_1,\dotsc,b_n$, define $\Phi_{x,y,z,t}(a_1,\dotsc,a_m;b_1,\dotsc,b_n)$ by
\small{\begin{align}\label{function1}
\Phi_{x,y,z,t}&(a_1,\dotsc,a_{2l};b_1,b_2,\dotsc,b_{2k}):=  s(a_1,\dotsc,a_{2l-1},a_{2l}+y+z+b_{2k},b_{2k-1},\dotsc,b_{1})\notag\\
  &\times s(x,a_1,\dotsc,a_{2l},y+z,b_{2k},\dotsc,b_{1},t)\notag\\
 &\times \frac{\Hf(y)\Hf(z)\Hf\left(a+b+x+2y+z+t\right)}
 {\Hf(y+z)\Hf\left(a+b+x+2y+t\right)}
\frac{\Hf\left(a+b+x+y+t\right)}{\Hf\left(a+b+x+y+z+t\right)}\notag\\
&\times \frac{\Hf\left(a+x+y\right)\Hf\left(b+y+t\right)}{\Hf\left(a+x\right)\Hf\left(b+t\right)}
 \frac{\Hf\left(e_a+e_b+x+t\right)\Hf\left(o_a+o_b\right)}{\Hf\left(e_a+e_b+x+y+t\right)\Hf\left(o_a+o_b+y\right)}
\end{align}}
\normalsize
and
\[\Phi_{x,y,z,t}(a_1,\dotsc,a_{2l-1};b_1,b_2,\dotsc,b_{2k}):=\Phi_{x,y,z,t}(a_1,\dotsc,a_{2l-1},0;b_1,b_2,\dotsc,b_{2k}),\]
\[\Phi_{x,y,z,t}(a_1,\dotsc,a_{2l};b_1,b_2,\dotsc,b_{2k-1}):=\Phi_{x,y,z,t}(a_1,\dotsc,a_{2l};b_1,b_2,\dotsc,b_{2k-1},0),\]
\[\Phi_{x,y,z,t}(a_1,\dotsc,a_{2l-1},0;b_1,b_2,\dotsc,b_{2k-1}):=\Phi_{x,y,z,t}(a_1,\dotsc,a_{2l-1},0;b_1,b_2,\dotsc,b_{2k-1},0).\]

The counterpart functions $\Psi_{x,y,z,t}(a_1,\dotsc,a_m;b_1,\dotsc,b_n)$ are defined by
\small{\begin{align}\label{function2}
\Psi_{x,y,z,t}&(a_1,\dotsc,a_{2l};\ b_1,b_2,\dotsc,b_{2k}):=  s(a_1,\dotsc,a_{2l-1},a_{2l}+y+z,b_{2k},\dotsc,b_{1},t)\notag\\
  &\times s(x,a_1,\dotsc,a_{2l},y+z+b_{2k},b_{2k-1},\dotsc,b_{1})\notag\\
&\times  \frac{\Hf(y)\Hf(z)\Hf\left(a+b+x+2y+z+t\right)}
 {\Hf(y+z)\Hf\left(a+b+x+2y+t\right)}\frac{\Hf\left(a+b+x+y+t\right)}{\Hf\left(a+b+x+y+z+t\right)}\notag\\
&\times \frac{\Hf\left(a+x+y\right)\Hf\left(b+y+t\right)}{\Hf\left(a+x\right)\Hf\left(b+t\right)}
  \frac{\Hf\left(e_a+o_b+x\right)\Hf\left(o_a+e_b+t\right)}{\Hf\left(e_a+o_b+x+y\right)\Hf\left(o_a+e_b+y+t\right)}
\end{align}}
\normalsize  and
\[\Psi_{x,y,z,t}(a_1,\dotsc,a_{2l-1};b_1,b_2,\dotsc,b_{2k}):=\Psi_{x,y,z,t}(a_1,\dotsc,a_{2l-1},0;b_1,b_2,\dotsc,b_{2k}),\]
\[\Psi_{x,y,z,t}(a_1,\dotsc,a_{2l};b_1,b_2,\dotsc,b_{2k-1}):=\Psi_{x,y,z,t}(a_1,\dotsc,a_{2l};b_1,b_2,\dotsc,b_{2k-1},0),\]

\begin{figure}[h]
\centerline{
\hfill
{\includegraphics[width=0.363\textwidth]{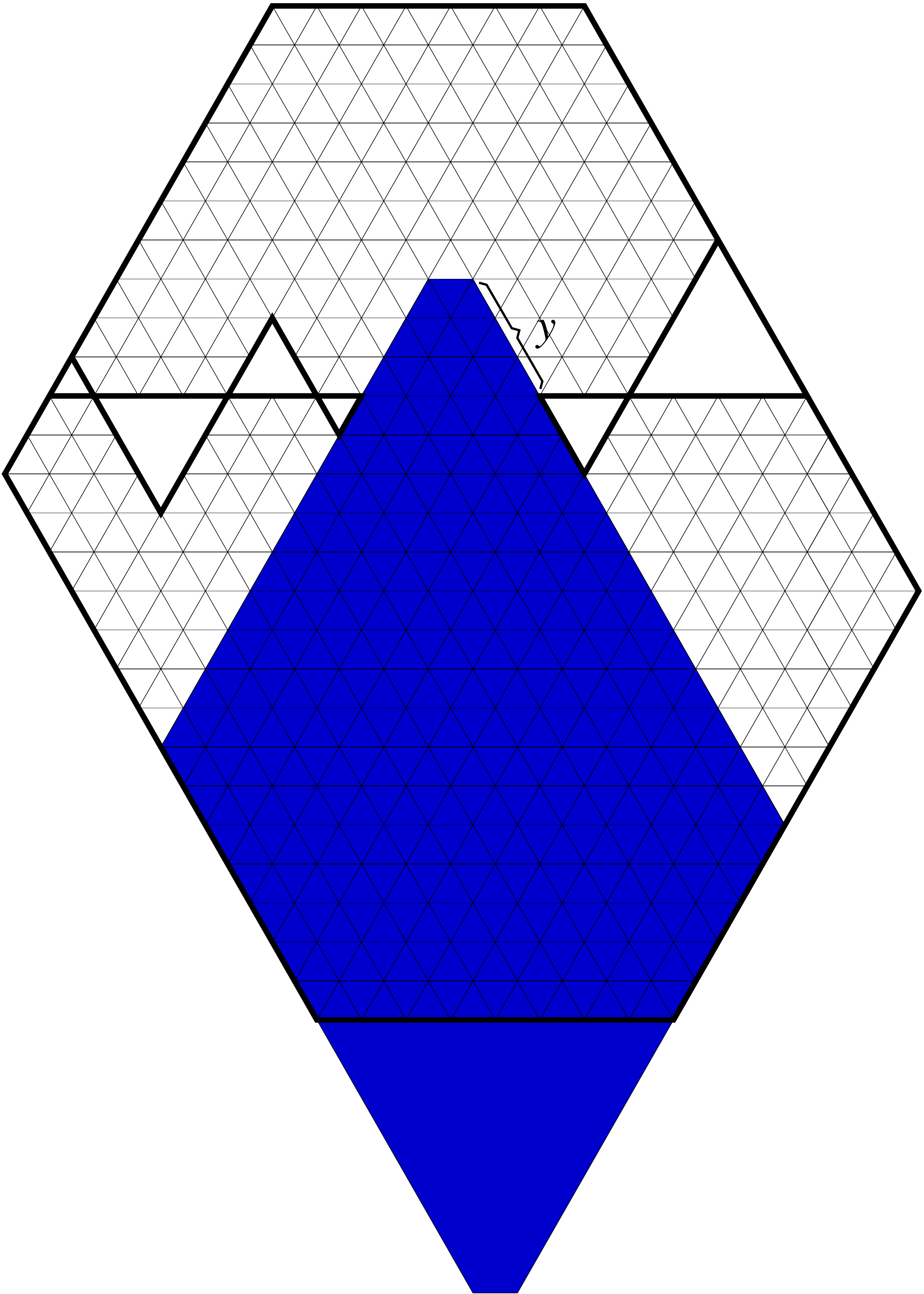}}
\hfill
{\includegraphics[width=0.428\textwidth]{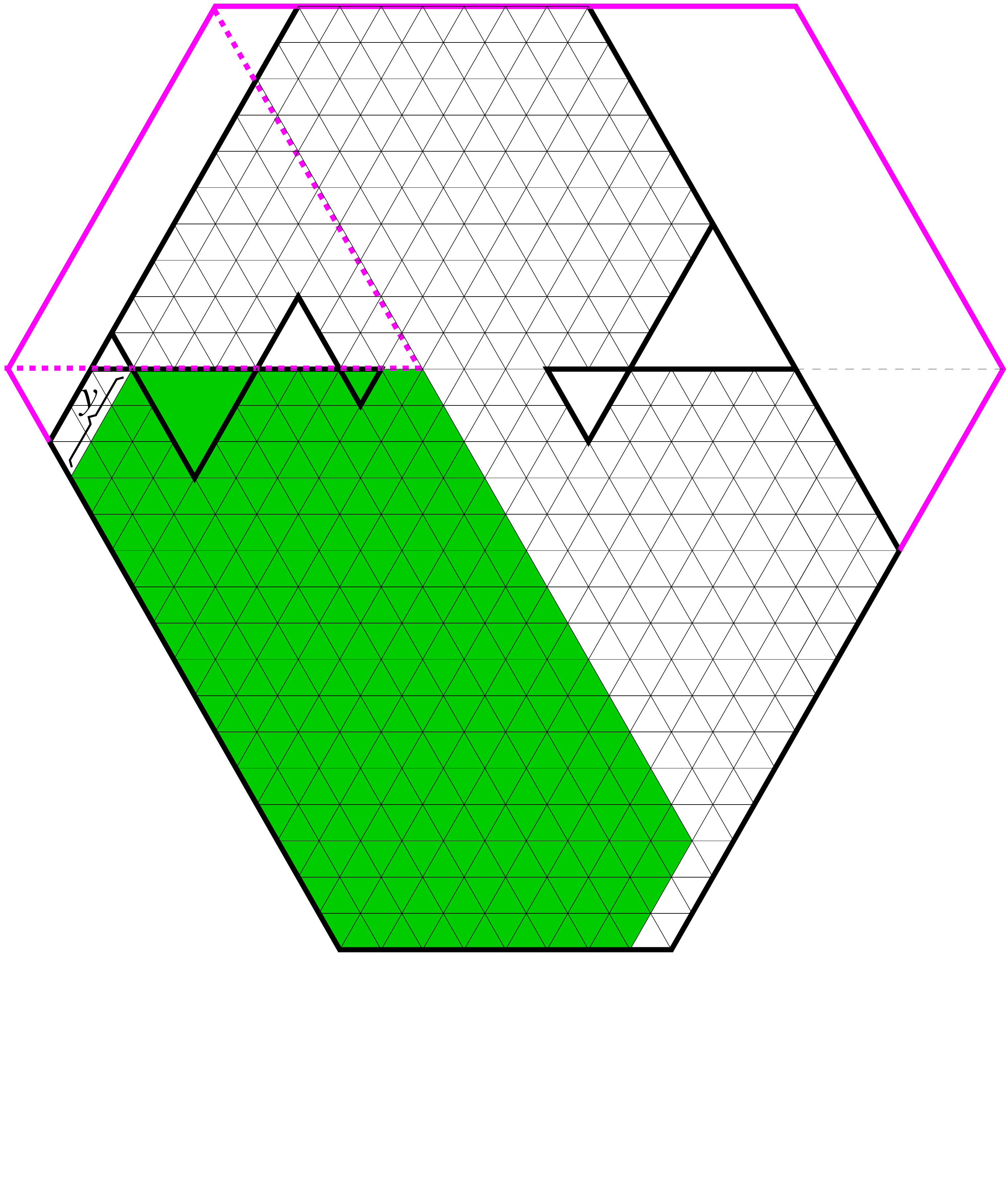}}
\hfill
}
\vskip-0.1in
\caption{The two hexagons in the numerator.}
\vskip-0.1in
\label{numeratorhex}
\end{figure}

\begin{figure}[h]
\centerline{
\hfill
{\includegraphics[width=0.363\textwidth]{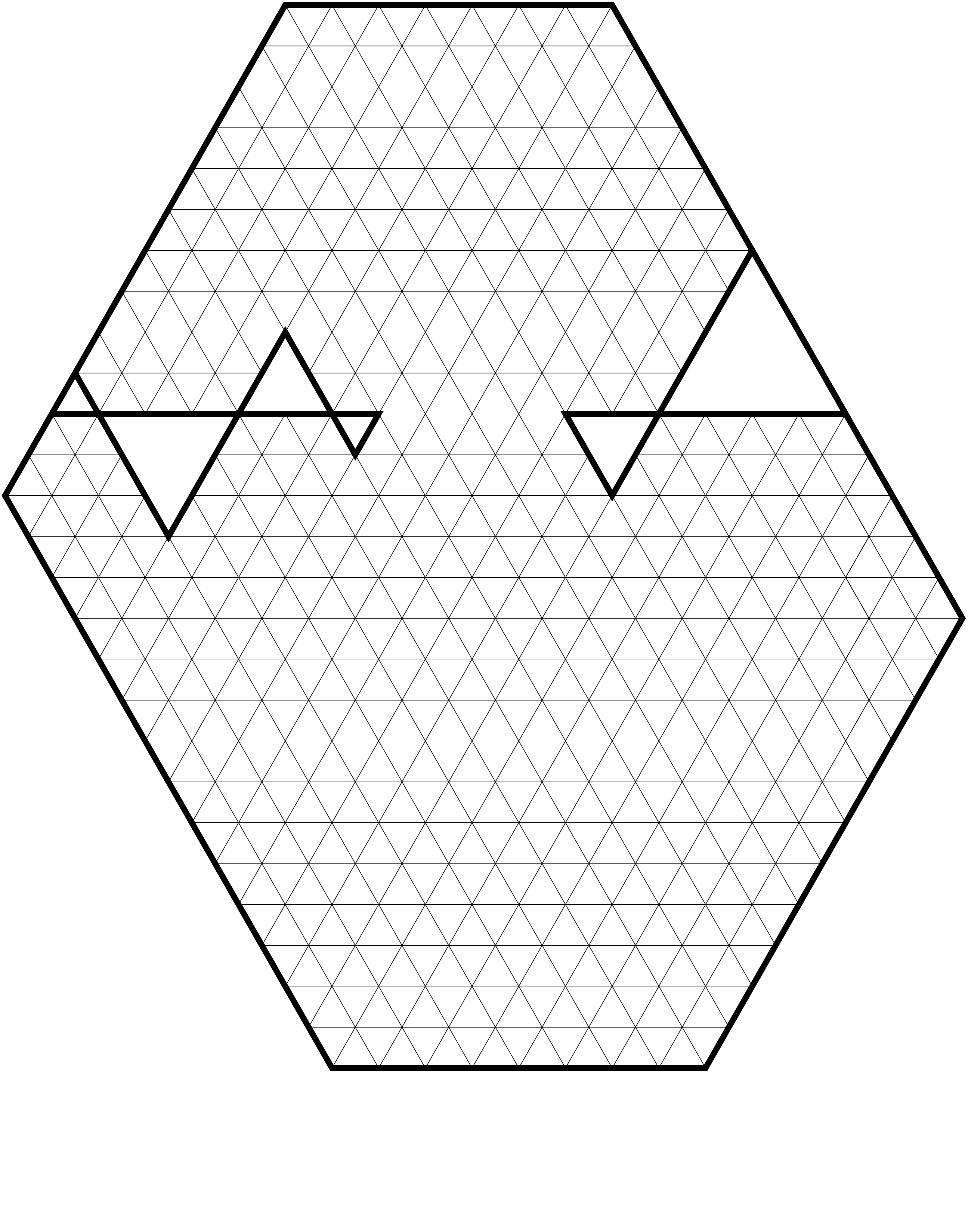}}
\hfill
{\includegraphics[width=0.363\textwidth]{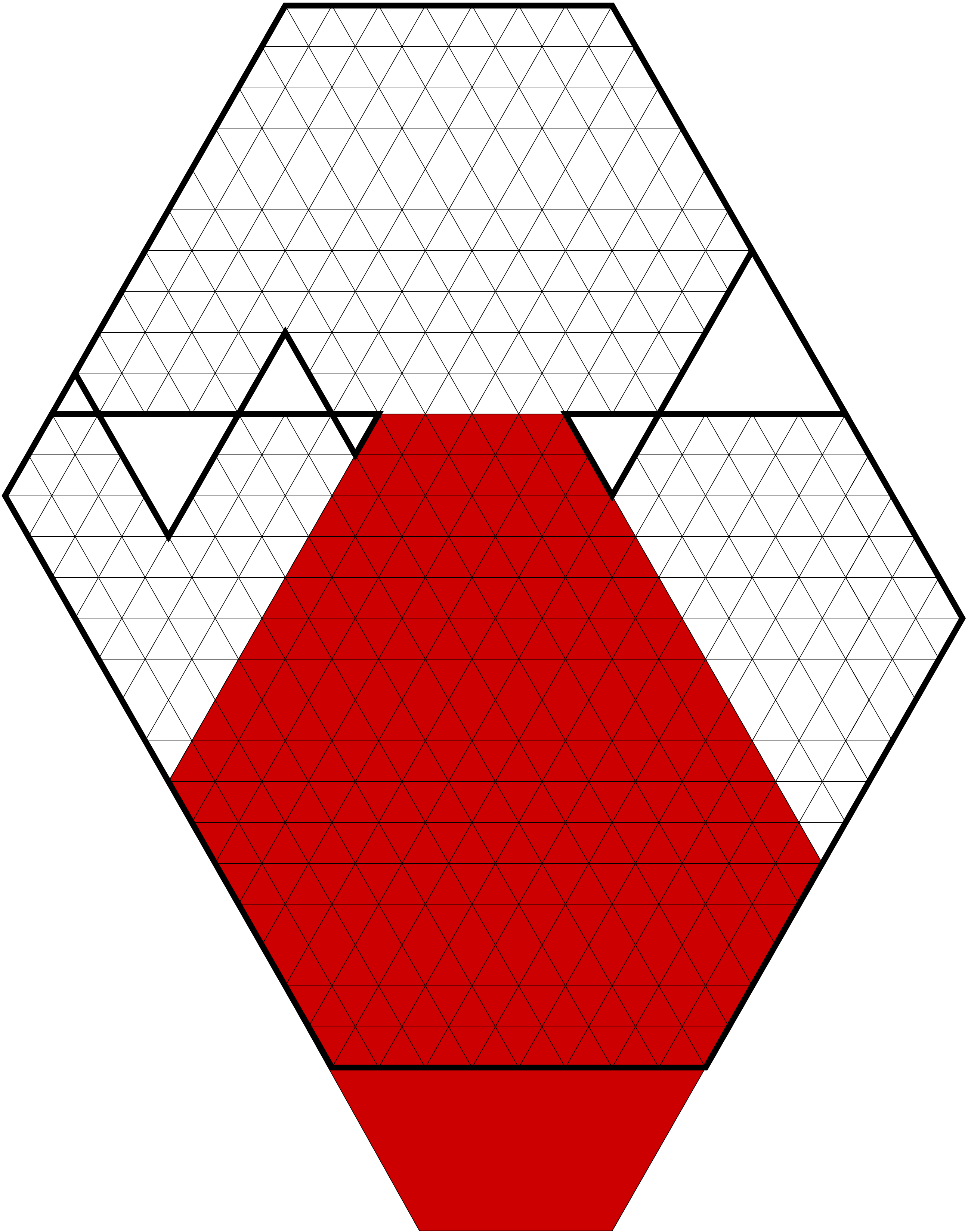}}
\hfill
}
\vskip-0.1in
\caption{The $P$-region (left) and the hexagon in the denominator (right).}
\vskip-0.1in
\label{denominatorhex}
\end{figure}

\[\Psi_{x,y,z,t}(a_1,\dotsc,a_{2l-1};b_1,b_2,\dotsc,b_{2k-1}):=\Psi_{x,y,z,t}(a_1,\dotsc,a_{2l-1},0;b_1,b_2,\dotsc,b_{2k-1},0).\]

\medskip
The main result of this paper concerning unweighted tiling enumeration (which then led, via its three dimensional interpretation presented in the next section, to the $q$-generalization we give in Section 4) is the following. Throughout the paper, $\M(R)$ denotes the number of lozenge tilings of region $R$.

\begin{thm}\label{main1}
Let $x,y,z,t$, $a_1,\dotsc,a_m$ and $b_1,\dotsc,b_n$ be non-negative integers. Then
\begin{equation}\label{main1eq1}
 \M\Big(P_{x,y,z,t}(a_1,\dotsc,a_m;b_1,\dotsc,b_n)\Big)=\Phi_{x,y,z,t}(a_1,\dotsc,a_m;b_1,\dotsc,b_n)
\end{equation}
and
\begin{equation}\label{main1eq2}
 \M\Big(Q_{x,y,z,t}(a_1,\dotsc,a_m;b_1,\dotsc,b_n)\Big)=\Psi_{x,y,z,t}(a_1,\dotsc,a_m;b_1,\dotsc,b_n).
 \end{equation}
\end{thm}

\begin{rmk} (\textsc{Geometrical interpretation})
\label{geominterpr}
It turns out that we can re-write the $\Phi$- and $\Psi$-functions  in terms of MacMahon's product formula $P(a,b,c)$ given by \eqref{MacMahoneq}. Indeed, we have

\small{\begin{align}\label{function1b}
\Phi_{x,y,z,t}&(a_1,\dotsc,a_{2l};b_1,b_2,\dotsc,b_{2k}):=
 \frac{P(z, b+y+t, a+x+y)P(e_a+e_b+x+t,o_a+o_b,y)}{P(a+x,b+t,y+z)}\notag\\
 &\times s(a_1,\dotsc,a_{2l-1},a_{2l}+y+z+b_{2k},b_{2k-1},\dotsc,b_{1})\notag\\
  &\times s(x,a_1,\dotsc,a_{2l},y+z,b_{2k},\dotsc,b_{1},t),
\end{align}}
\normalsize and
\small{\begin{align}\label{function2b}
\Psi_{x,y,z,t}&(a_1,\dotsc,a_{2l};b_1,b_2,\dotsc,b_{2k}):=
  \frac{P(z, b+y+t, a+x+y)P(e_a+o_b+x,o_a+e_b+t,y)}{P(a+x,b+t,y+z)}\notag\\
 &\times s(a_1,\dotsc,a_{2l-1},a_{2l}+y+z,b_{2k},\dotsc,b_{1},t)\notag\\
  &\times s(x,a_1,\dotsc,a_{2l},y+z+b_{2k},b_{2k-1},\dotsc,b_{1}).
\end{align}}
\end{rmk}

The three hexagons corresponding to the $P$-functions on the right hand side of (\ref{function1b}) are indicated in Figures \ref{numeratorhex} and \ref{denominatorhex}. The picture on the left in Figure \ref{denominatorhex} shows the $P$-region itself. The shaded hexagon in the picture on the right in  Figure \ref{denominatorhex} has side-lengths equal to the arguments of the $P$-function in the denominator on the right hand side of (\ref{function1b}). The lengths of its top three sides are clear from the picture; the shaded hexagon is the semi-regular hexagon determined by these three consecutive sides.

The hexagons corresponding to the two $P$-functions in the numerator on the right hand side of (\ref{function1b}) are described in Figure \ref{numeratorhex}. The one on the left is the semi-regular hexagon obtained from the shaded hexagon in Figure \ref{denominatorhex} by ``stretching'' it $y$ units as indicated in the figure. The one on the right in Figure \ref{numeratorhex} is obtained by (1) enlarging the boundary of the $P$-region to the indicated outer symmetric hexagon (shown in magenta), (2) considering the point inside the resulting contour that determines an equilateral triangle with its northwestern side, and (3) ``trimming'' the angle determined by the point considered in (2) and the western and southwestern corners of the enlarged contour (with the western corner at the apex) by cutting off an equilateral triangle of side $y$, as indicated in the figure; the resulting three consecutive side-lengths determine the shaded semi-regular hexagon on the right in Figure \ref{numeratorhex}.

We note that the above geometrical interpretation follows directly from the definition of the functions $\Phi$ and $\Psi$. We do not yet have a combinatorial explanation for that. This is discussed some more in Section 8.

\section{A unified formulation of the two parts of Theorem 2.1}

We explain in this section how the lozenge tiling enumerations of our $P$- and $Q$-regions can be stated in a unified way. This will be a useful point to keep in mind later, when we present our proofs.

As we will see below, our $P$- and $Q$-regions can be viewed as sub-regions of vertically symmetric hexagons with a fern removed from the western corner, and another fern removed from the eastern corner. Denote the lobe sizes in the former, from left to right, by $a_0,a_1,\dotsc,a_m$ (we will refer to this as an ${a}$-fern), and the lobe sizes in the latter, from {\it right to left}, by $b_0,b_1,\dotsc,b_n$ (we call this a ${b}$-fern). The left fern will always be positioned so that the lobe of size $a_0$ points downward, but we will allow the right fern to have the lobe of size $b_0$ point either down or up (this captures the full generality). Set
\begin{align}
u_a&:=a_1+a_3+a_5+\cdots\\\
d_a&:=a_0+a_2+a_4+\cdots\\
u_b&:=\left.
  \begin{cases}
    b_1+b_3+b_5+\cdots, & \text{if the lobe of size $b_0$ points down }  \\
    b_0+b_2+b_4+\cdots, & \text{if the lobe of size $b_0$ points up }
   \end{cases}
  \right. \\
d_b&:=\left.
  \begin{cases}
    b_0+b_2+b_4+\cdots, & \text{if the lobe of size $b_0$ points down }  \\
    b_1+b_3+b_5+\cdots, & \text{if the lobe of size $b_0$ points up }
   \end{cases}
  \right.
\end{align}
(so that e.g. $u_b$ is the sum of the {\it u}p-pointing lobe sizes in the $b$-fern).

\begin{figure}\centering
\setlength{\unitlength}{3947sp}%
\begingroup\makeatletter\ifx\SetFigFont\undefined%
\gdef\SetFigFont#1#2#3#4#5{%
  \reset@font\fontsize{#1}{#2pt}%
  \fontfamily{#3}\fontseries{#4}\fontshape{#5}%
  \selectfont}%
\fi\endgroup%
\resizebox{13cm}{!}{
\begin{picture}(0,0)%
\includegraphics{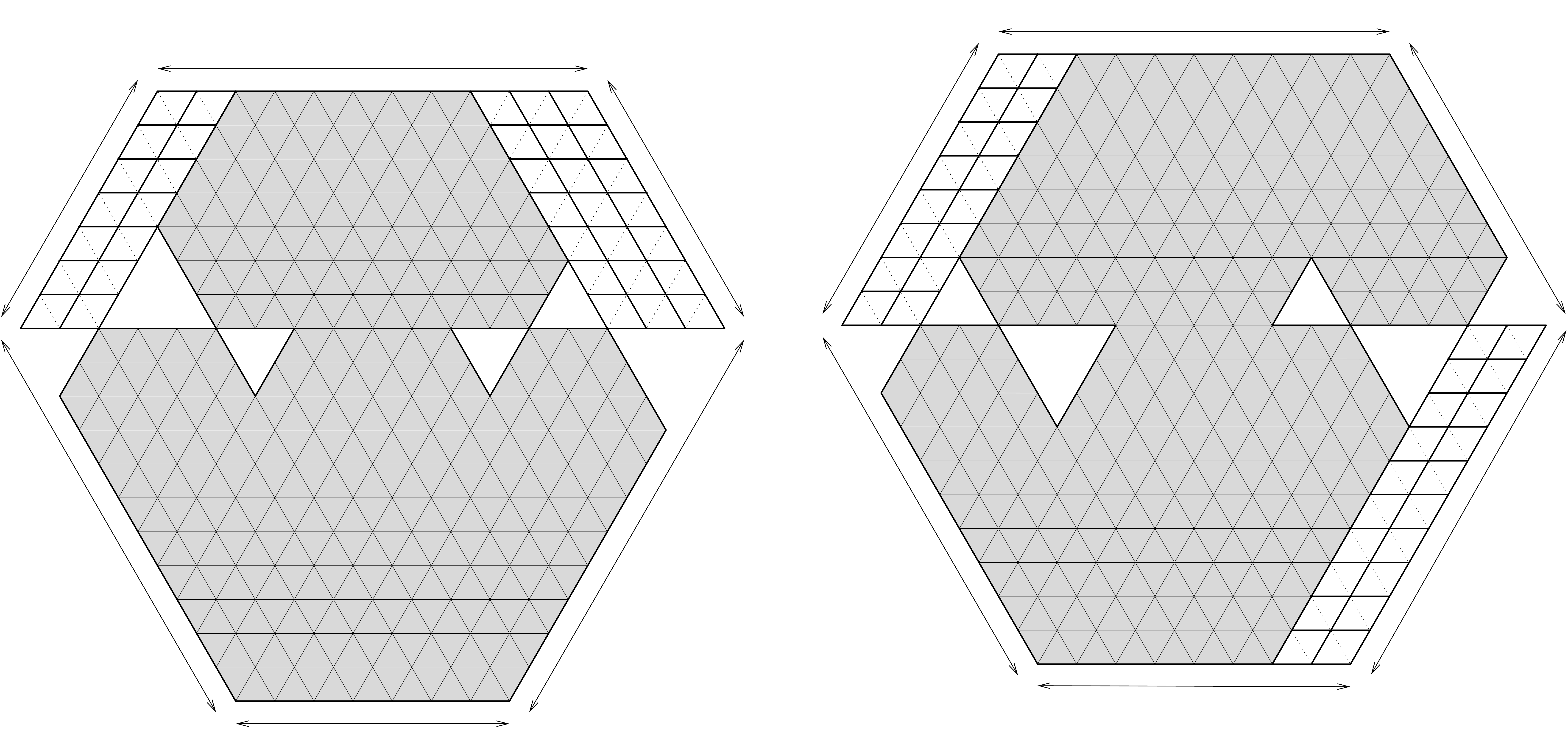}%
\end{picture}%
%
%

\begin{picture}(32794,15679)(2886,-17382)
\put(4078,-11991){\rotatebox{300.0}{\makebox(0,0)[lb]{\smash{{\SetFigFont{20}{24.0}{\rmdefault}{\mddefault}{\itdefault}{$y+d_a+d_b$}%
}}}}}
\put(20867,-6341){\rotatebox{60.0}{\makebox(0,0)[lb]{\smash{{\SetFigFont{20}{24.0}{\rmdefault}{\mddefault}{\itdefault}{$y+u_a+u_b$}%
}}}}}
\put(26757,-2111){\makebox(0,0)[lb]{\smash{{\SetFigFont{20}{24.0}{\rmdefault}{\mddefault}{\itdefault}{$z+d_a+d_b$}%
}}}}
\put(33957,-4761){\rotatebox{300.0}{\makebox(0,0)[lb]{\smash{{\SetFigFont{20}{24.0}{\rmdefault}{\mddefault}{\itdefault}{$y+u_a+u_b$}%
}}}}}
\put(33377,-13471){\rotatebox{60.0}{\makebox(0,0)[lb]{\smash{{\SetFigFont{20}{24.0}{\rmdefault}{\mddefault}{\itdefault}{$y+d_a+d_b$}%
}}}}}
\put(21397,-12041){\rotatebox{300.0}{\makebox(0,0)[lb]{\smash{{\SetFigFont{20}{24.0}{\rmdefault}{\mddefault}{\itdefault}{$y+d_a+d_b$}%
}}}}}
\put(26677,-16491){\makebox(0,0)[lb]{\smash{{\SetFigFont{20}{24.0}{\rmdefault}{\mddefault}{\itdefault}{$z+u_a+u_b$}%
}}}}
\put(21127,-9051){\makebox(0,0)[lb]{\smash{{\SetFigFont{20}{24.0}{\rmdefault}{\mddefault}{\itdefault}{$a_0$}%
}}}}
\put(22817,-8191){\makebox(0,0)[lb]{\smash{{\SetFigFont{20}{24.0}{\rmdefault}{\mddefault}{\itdefault}{$a_1$}%
}}}}
\put(24887,-9241){\makebox(0,0)[lb]{\smash{{\SetFigFont{20}{24.0}{\rmdefault}{\mddefault}{\itdefault}{$a_2$}%
}}}}
\put(34357,-8101){\makebox(0,0)[lb]{\smash{{\SetFigFont{20}{24.0}{\rmdefault}{\mddefault}{\itdefault}{$b_0$}%
}}}}
\put(32217,-9371){\makebox(0,0)[lb]{\smash{{\SetFigFont{20}{24.0}{\rmdefault}{\mddefault}{\itdefault}{$b_1$}%
}}}}
\put(30147,-8141){\makebox(0,0)[lb]{\smash{{\SetFigFont{20}{24.0}{\rmdefault}{\mddefault}{\itdefault}{$b_2$}%
}}}}
\put(3848,-9167){\makebox(0,0)[lb]{\smash{{\SetFigFont{20}{24.0}{\rmdefault}{\mddefault}{\itdefault}{$a_0$}%
}}}}
\put(5858,-7987){\makebox(0,0)[lb]{\smash{{\SetFigFont{20}{24.0}{\rmdefault}{\mddefault}{\itdefault}{$a_1$}%
}}}}
\put(8068,-9157){\makebox(0,0)[lb]{\smash{{\SetFigFont{20}{24.0}{\rmdefault}{\mddefault}{\itdefault}{$a_2$}%
}}}}
\put(16588,-9397){\makebox(0,0)[lb]{\smash{{\SetFigFont{20}{24.0}{\rmdefault}{\mddefault}{\itdefault}{$b_0$}%
}}}}
\put(14588,-8237){\makebox(0,0)[lb]{\smash{{\SetFigFont{20}{24.0}{\rmdefault}{\mddefault}{\itdefault}{$b_1$}%
}}}}
\put(13008,-9197){\makebox(0,0)[lb]{\smash{{\SetFigFont{20}{24.0}{\rmdefault}{\mddefault}{\itdefault}{$b_2$}%
}}}}
\put(9448,-2787){\makebox(0,0)[lb]{\smash{{\SetFigFont{20}{24.0}{\rmdefault}{\mddefault}{\itdefault}{$z+d_a+d_b$}%
}}}}
\put(3329,-6934){\rotatebox{60.0}{\makebox(0,0)[lb]{\smash{{\SetFigFont{20}{24.0}{\rmdefault}{\mddefault}{\itdefault}{$y+u_a+u_b$}%
}}}}}
\put(16958,-4851){\rotatebox{300.0}{\makebox(0,0)[lb]{\smash{{\SetFigFont{20}{24.0}{\rmdefault}{\mddefault}{\itdefault}{$y+u_a+u_b$}%
}}}}}
\put(15849,-14254){\rotatebox{60.0}{\makebox(0,0)[lb]{\smash{{\SetFigFont{20}{24.0}{\rmdefault}{\mddefault}{\itdefault}{$y+d_a+d_b$}%
}}}}}
\put(9428,-17367){\makebox(0,0)[lb]{\smash{{\SetFigFont{20}{24.0}{\rmdefault}{\mddefault}{\itdefault}{$z+u_a+u_b$}%
}}}}
\end{picture}%
}
\caption{Obtaining a $P$-region from a $D^{(1)}$-region (left); obtaining a $Q$-region from a $D^{(2)}$-region (right).}\label{Dregion}
\end{figure}
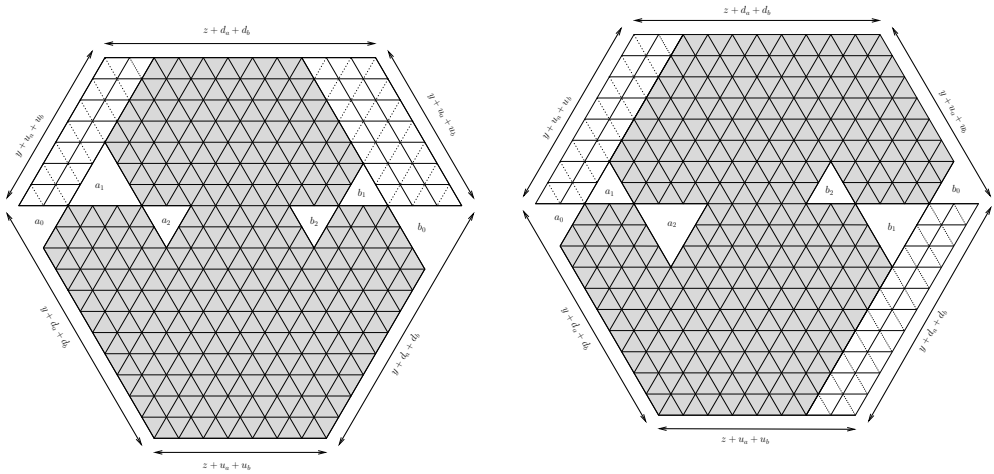

Let $y,z\geq0$ be integers, and consider on the triangular lattice the vertically symmetric hexagon of side-lengths $y+u_a+u_b$, $z+d_a+d_b$, $y+u_a+u_b$, $y+d_a+d_b$, $z+u_a+u_b$, $y+d_a+d_b$ (clockwise starting with the northwestern side). Define $D^{(1)}_{y,z}(a_0,\dotsc,a_m;b_0,\dotsc,b_n)$ to be the region obtained from this hexagon by removing an ${a}$-fern with down-pointing $a_0$-lobe from its western corner, and a ${b}$-fern with down-pointing $b_0$-lobe from its eastern corner. Define the region $D^{(2)}_{y,z}(a_0,\dotsc,a_m;b_0,\dotsc,b_n)$ in the same way, with the only difference that the $b_0$-lobe in the removed ${b}$-fern points upward (these regions are illustrated in Figure \ref{Dregion}; the letter of their name recalls the fact that the axes of the removed ferns are along the horizontal diagonal of the hexagon).

One readily sees that, after removing the forced lozenges in $D^{(1)}_{y,z}(a_0,\dotsc,a_m;b_0,\dotsc,b_n)$, the leftover region is our $P$-region $P_{a_0,y,z,b_0}(a_1,\dotsc,a_m;b_1,\dotsc,b_n)$. Similarly, it is readily checked that the region obtained from $D^{(2)}_{y,z}(a_0,\dotsc,a_m;b_0,\dotsc,b_n)$ by removing all the forced lozenges is our $Q$-region $Q_{a_0,y,z,b_0}(a_1,\dotsc,a_m;b_1,\dotsc,b_n)$.

Then Theorem 2.1 can be stated more elegantly as follows.

\begin{thm}
Let $y,z$ and $a_0,\dotsc,a_m,b_0,\dotsc,b_n$ be non-negative integers. Set $A=a_0+\cdots+a_m$, $B=b_0+\cdots+b_n$, and let $u_a$, $d_a$, $u_b$ and $d_b$ be given by the equations (3.1)--(3.4) above. Then for $i=1,2$ we have
\begin{align}\label{unified}
&\M\Big(D^{(i)}_{y,z}(a_0,\dotsc,a_m;b_0,\dotsc,b_n)\Big)=\M(S^+)\M(S^-)\notag\\
&\times
\frac{\Hf(y)\Hf(z)\Hf\left(A+B+2y+z\right)}{\Hf(y+z)\Hf\left(A+B+2y\right)}\frac{\Hf\left(A+B+y\right)}{\Hf\left(A+B+y+z\right)}\notag\\
 &\times \frac{\Hf\left(A+y\right)\Hf\left(B+y\right)}{\Hf\left(A\right)\Hf\left(B\right)}
 \frac{\Hf\left(d_a+d_b\right)\Hf\left(u_a+u_b\right)}{\Hf\left(d_a+d_b+y\right)\Hf\left(u_a+u_b+y\right)},\notag\\
\end{align}
where $S^+$ and $S^-$ are the dented semihexagons determined by the sequences of dents occurring above and below the axis of the ferns in $D^{(i)}_{y,z}(a_0,\dotsc,a_m;b_0,\dotsc,b_n)$, respectively $($in particular, their number of lozenge tilings is given by formula $(2.1)$$)$.

\end{thm}

The explicit statement of this result given in Theorem 2.1 will however be more convenient to work with when we present its proof.

\section{Three dimensional interpretation}

A plane partition is a rectangular array of non-negative integers with weakly decreasing rows (from left to right) and columns (from top to bottom). 

\begin{figure}\centering
\includegraphics[width=10cm]{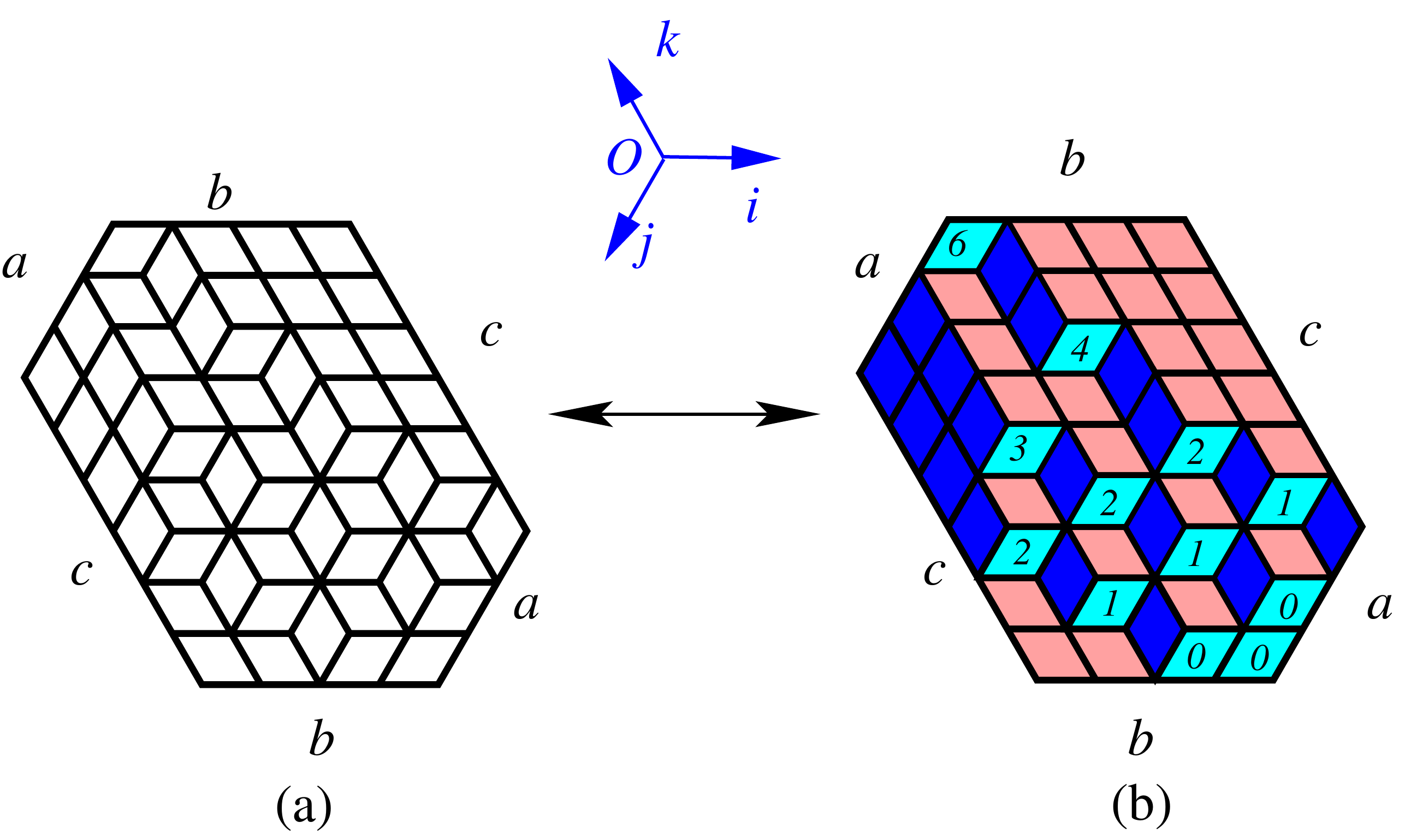}
\caption{The correspondence between plane partitions fitting a given box and lozenge tilings of a hexagon.}
\label{hextiled1}
\end{figure}

Plane partitions with $a$ rows, $b$ columns, and entries at most $c$ can conveniently be identified with their three dimensional diagrams --- stacks of unit cubes with certain monotonicity requirements fitting in an $a\times b\times c$ box (see Figure \ref{hextiled1}(b)). Namely, the heights of the columns of such a stack are required to be weakly decreasing from northeast to southwest and from northwest to south east. We call such stacks \emph{monotone stacks}. The latter are in turn in bijection with lozenge tilings of the semi-regular hexagon $H_{a,b,c}$ of side-lengths $a$, $b$, $c$, $a$, $b$, $c$ (in cyclic order).
 For example, the plane partition
\begin{equation} \pi=\begin{tabular}{rccccccccc}
6   & 4           & 2         & 1  \\\noalign{\smallskip\smallskip}
3   &  2          & 1           & 0         \\\noalign{\smallskip\smallskip}
2   &  1          & 0          &  0                    \\\noalign{\smallskip\smallskip}
\end{tabular}
\end{equation}
corresponds to the monotone stack of unit cubes  fitting in the $3\times 4\times 6$ box in Figure \ref{hextiled1}(b), which in turn corresponds to the lozenge tiling of $H_{3,4,6}$ pictured in Figure \ref{hextiled1}(a).

\begin{figure}[h]
\centerline{
\hfill
{\includegraphics[width=0.43\textwidth]{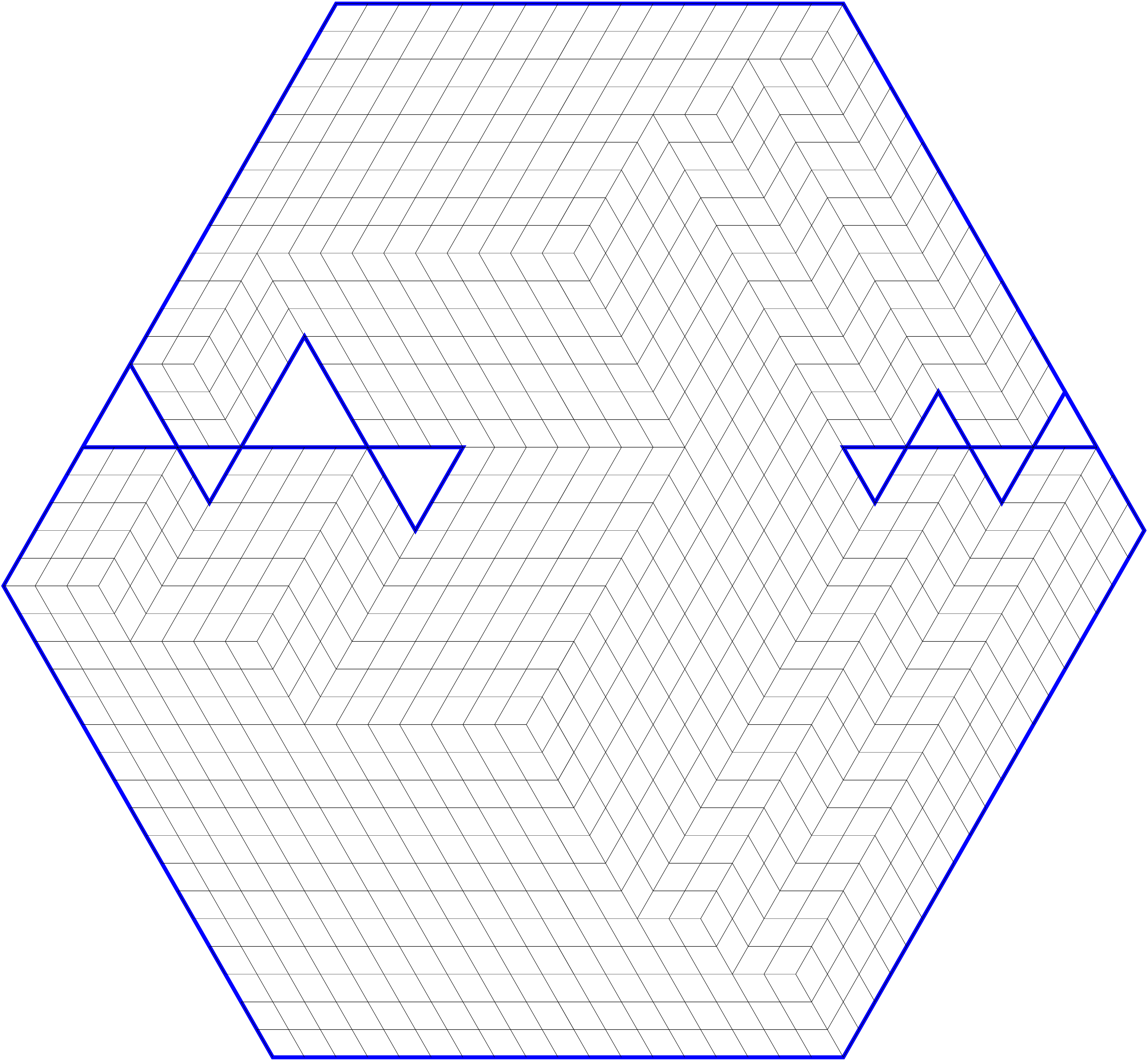}}
\hfill
{\includegraphics[width=0.43\textwidth]{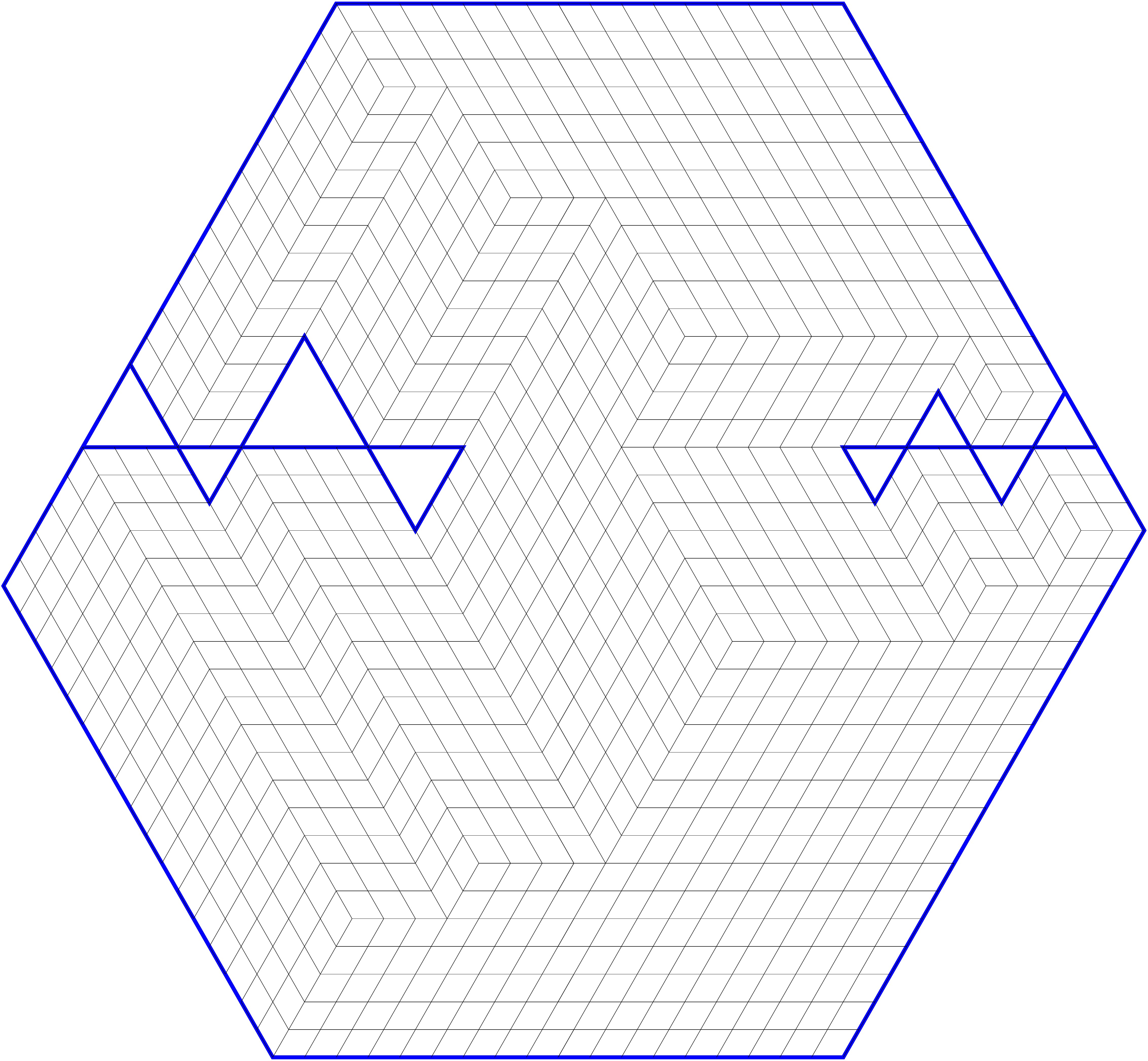}}
\hfill
}
\vskip-0.1in
\caption{The highest tiling (left) and the lowest tiling (right).}
\vskip-0.1in
\label{highlowtilings}
\end{figure}

\begin{figure}[h]
\vskip0.1in
\centerline{
\hfill
{\includegraphics[width=0.43\textwidth]{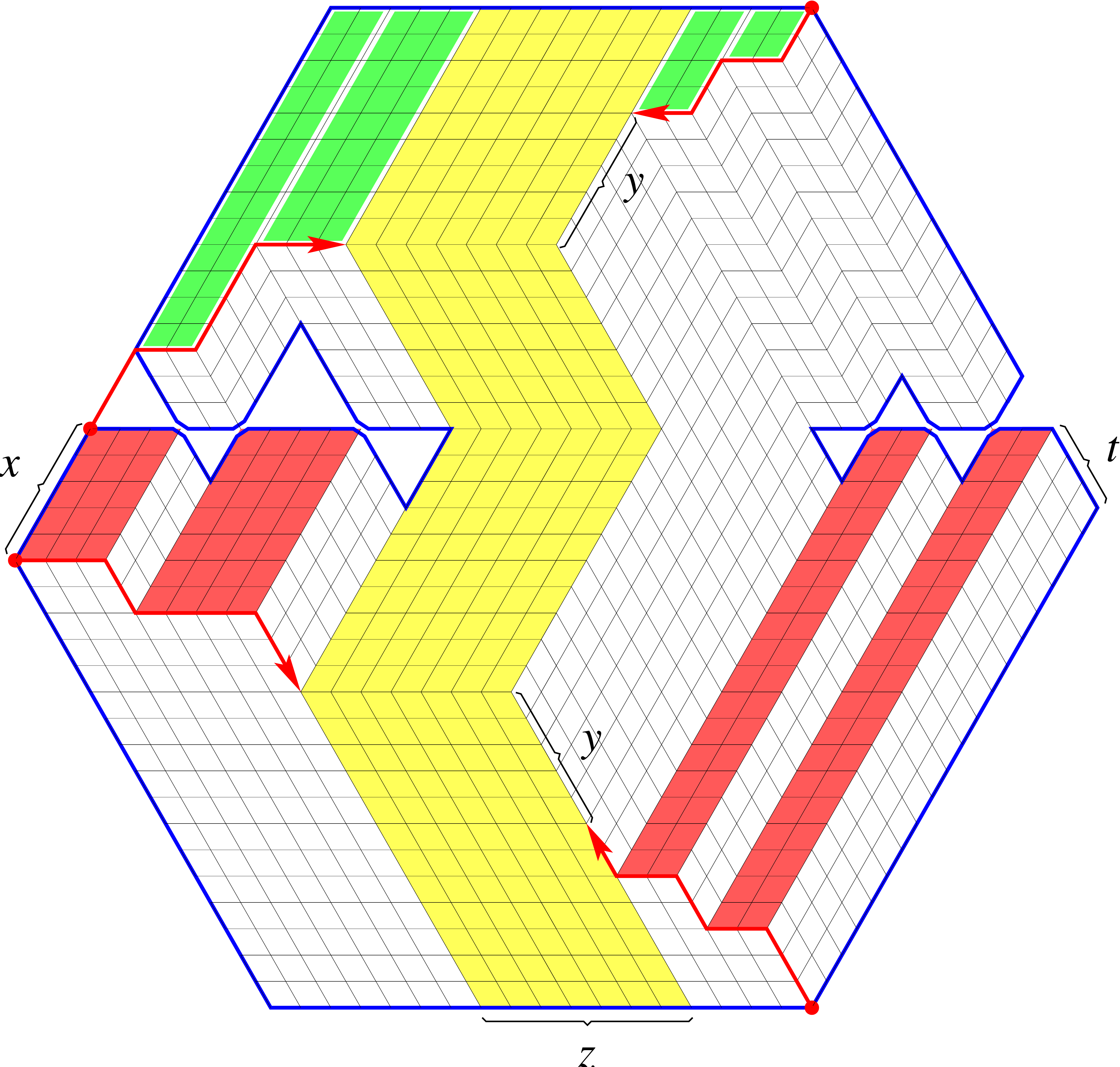}}
\hfill
{\includegraphics[width=0.43\textwidth]{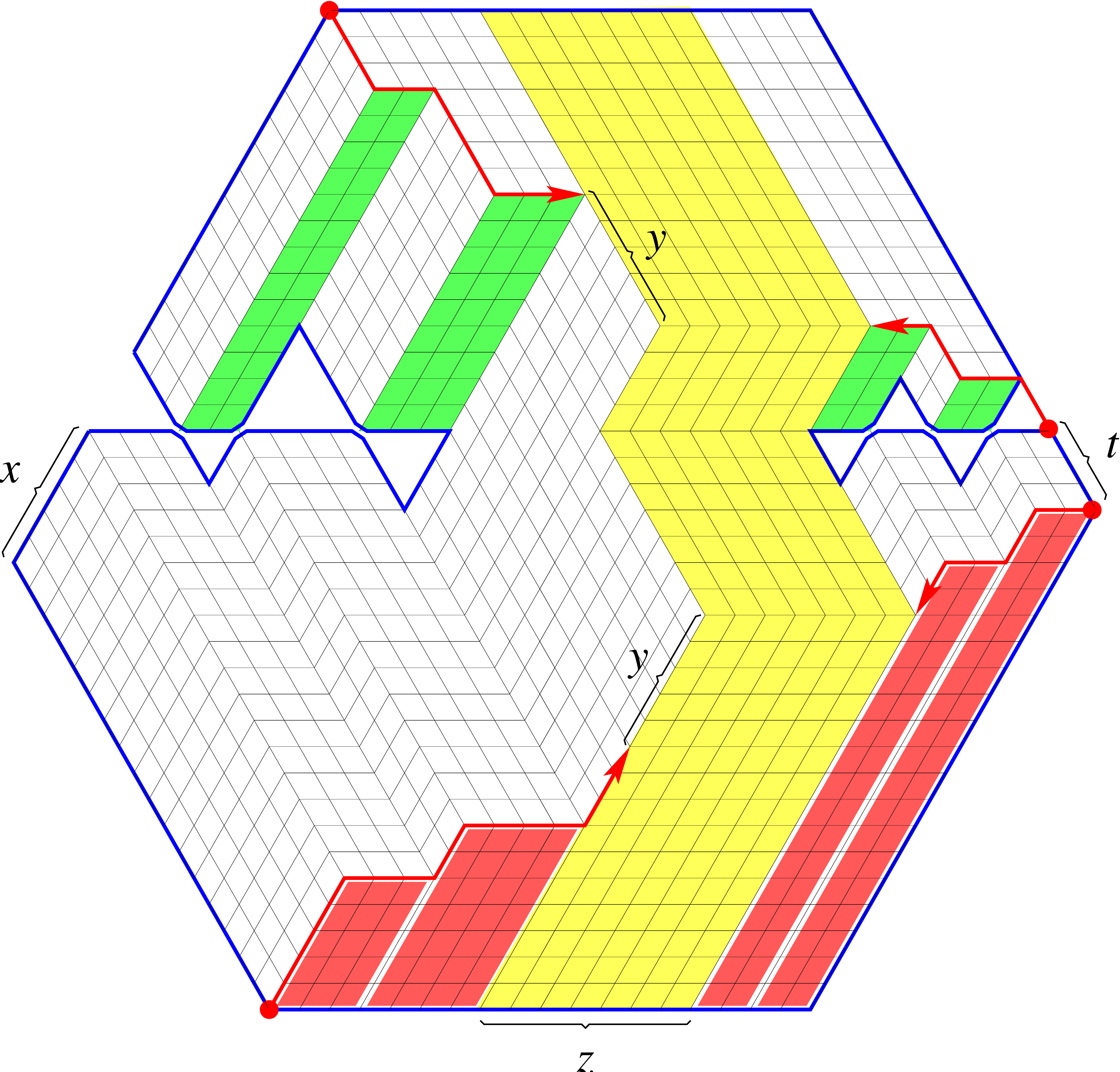}}
\hfill
}
\vskip-0.1in
\caption{The upper envelope (left) and the lower envelope (right).}
\vskip-0.1in
\label{highlowenvelope}
\end{figure}

\bigskip
\begin{figure}[h]
\vskip0.1in
\centerline{
\hfill
{\includegraphics[width=0.43\textwidth]{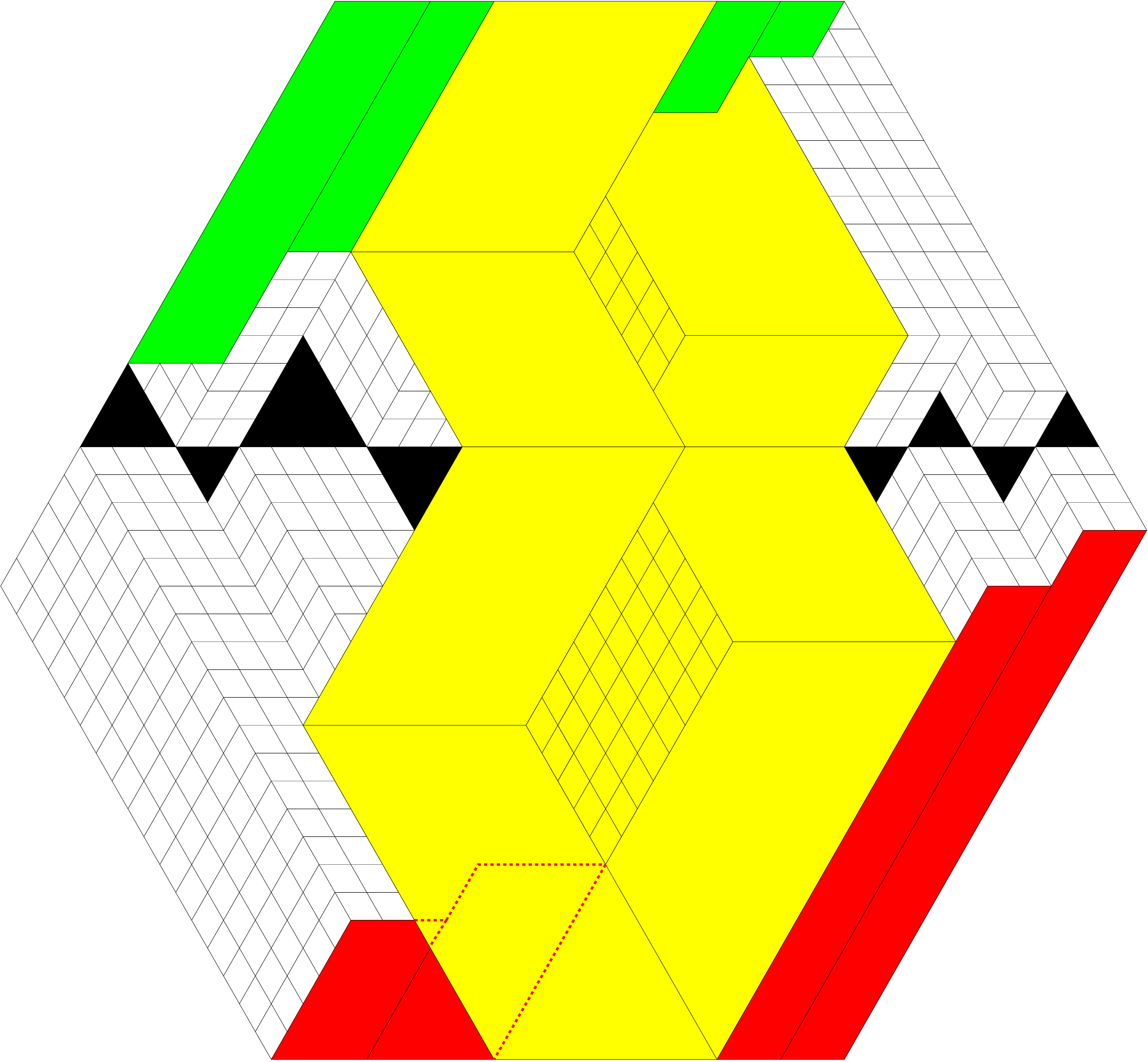}}
\hfill
{\includegraphics[width=0.43\textwidth]{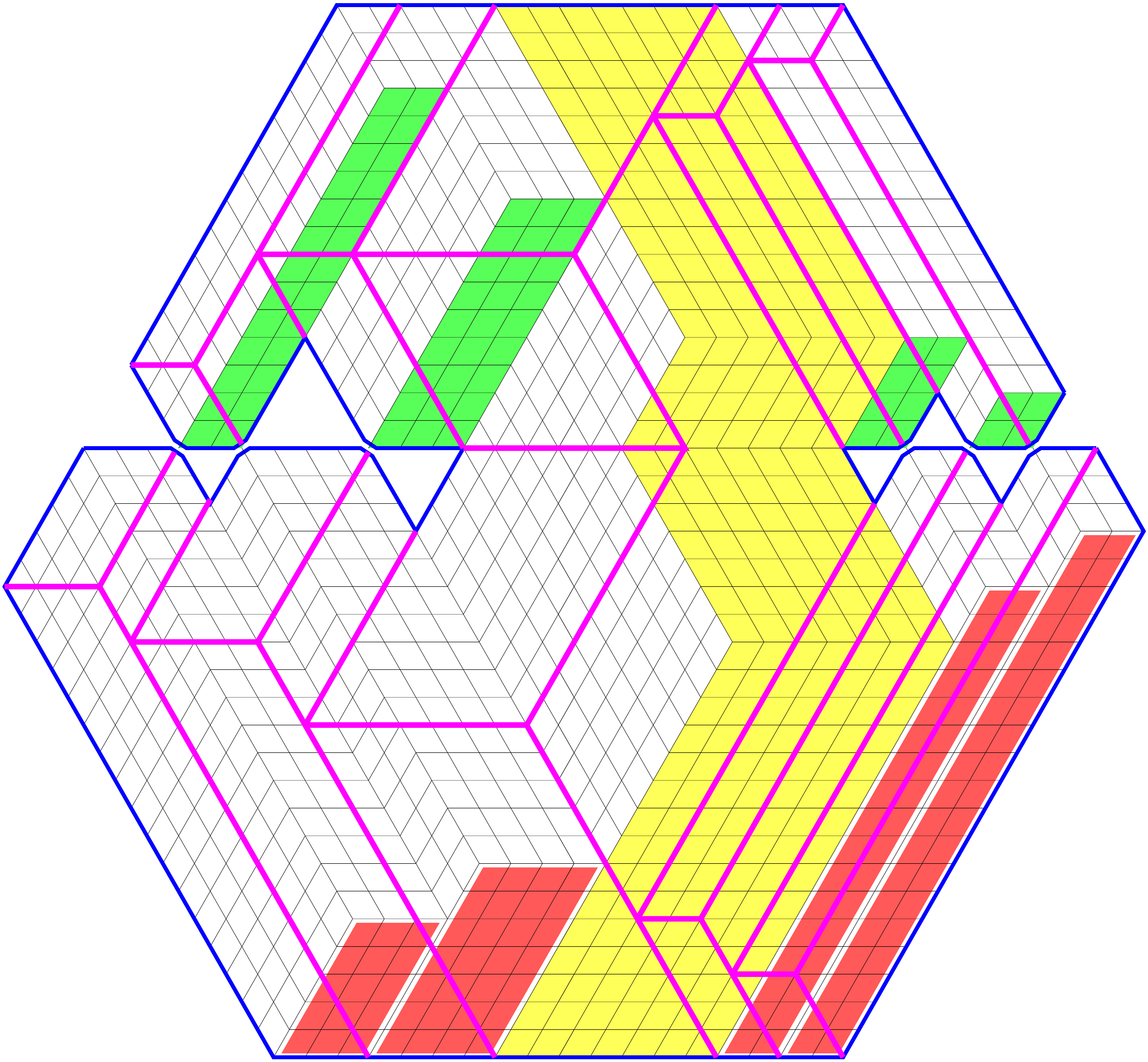}}
\hfill
}
\vskip-0.1in
\caption{The house with slab and roof (left) and the house with stairs and cupboards (right).}
\vskip-0.1in
\label{3Dinterpretation}
\end{figure}

It is less obvious how to interpret lozenge tilings of our $P$- and $Q$-regions as monotone stacks fitting in certain solids, but a quite satisfying such interpretation can nevertheless be given. This is the object of this section.

Viewing our triangular lattice as the orthogonal projection of a cubic lattice, so that traveling in the polar directions $1$, $2\pi/3$ and $-2\pi/3$ in the triangular lattice corresponds to moving up in the cubic lattice (and hence traveling in the directions $\pi/3$, $-1$ and $-\pi/3$ in the triangular lattice corresponds to moving {\it down} in the cubic lattice), each tiling of a doubly-intruded hexagon is the projection of a stepped surface in the cubic lattice. For consistency, choose the stepped surfaces corresponding to the tilings so that say the top left corner of the hexagon lifts to a common vertex of the cubic lattice. Then one readily sees that among all the surfaces obtained in this way, there is a ``highest'' one, in the sense that all the others are (weakly) below it, and that there is also a lowest one. These two tilings are shown in Figure \ref{highlowtilings} for a representative instance of a doubly intruded $P$-hexagon (there is a slight difference in the details according to the parities of the number of lobes of the two ferns; for brevity, we discuss only the case when both ferns have an even number of lobes, the other cases being perfectly analogous); the case of the $Q$-regions is similar.

The 3D surfaces corresponding to these two tilings are shown in Figure \ref{highlowenvelope} (in the figures, the portion of the boundary that projects to the ferns is shown with some ``rounded'' corners, in order to emphasize that these are really skew 3D contours, and that the pairs of rounded corners consist really of distinct points, even though they happen to project onto the same vertex of the triangular lattice). The 3D region inside which plane-partition-style stacks of unit cubes are in bijection with lozenge tilings of our doubly-intruded hexagon consists of the unit cubes of the cubic lattice that are contained in between these two surfaces (since the boundaries of the highest and lowest surfaces are the same, these two surfaces fit together perfectly).

In order to imagine this solid, we find the following description useful. Start with a ``house'' consisting of the union of two partially overlapping parallelepipeds, as indicated in the central portion of the picture on the left in Figure \ref{3Dinterpretation} (shown there in yellow). The values of the lengths of edges that determine it can be read off from Figure \ref{highlowenvelope}, once we note that the lengths of the steps of the red paths in those figures are as follows: in the picture on the left, $a_1,a_2,\dotsc,a_{2l-1},a_{2l}$ for the two on the left,  $b_1,b_2,\dotsc,b_{2k-1},b_{2k}$ for the top right, and $t,b_1,b_2,\dotsc,b_{2k-1},b_{2k}$ for the bottom right; and in the picture on the right, $a_1,a_2,\dotsc,a_{2l-1},a_{2l}$ for the top left, $x,a_1,a_2,\dotsc,a_{2l-1},a_{2l}$ for the bottom left, and $b_1,b_2,\dotsc,b_{2k-1},b_{2k}$ for the two on the right.

The 3D region we want is a certain enlargement of this house, with what we can imagine to be (storage) ``stairs'' built on an extension of the slab (base) of the house (this slab extension is shown in red in the picture on the left in Figure~\ref{3Dinterpretation}), and ``cupboards'' built under an extension of the flat roof of the house (this roof extension is shown in green in the same picture). The red rectangles in that picture are the footprints of the stairs; the green ones are the ceiling-prints of the cupboards. The sizes of these footprints and ceiling-prints can be read off from Figures~\ref{highlowenvelope}, using the description of the red paths in the previous paragraph; the heights of the stairs and cupboards are read off from the same figures.

The picture on the right in Figure \ref{3Dinterpretation} sketches a view of the obtained solid, by showing just the ``frame'' of the highest surface (in magenta) over the lowest surface. One can see how the red and green rectangles pair with magenta rectangular frames to form the stairs and cupboards, respectively.

We conclude this section by showing a fancier visualization of our augmented boxes, using the \texttt{Mathematica} code provided by one of the anonymous reviewers. Figure \ref{house1} shows more clearly the three dimensional solid we described above using Figure  \ref{3Dinterpretation}. Figure \ref{house2} shows the same solid from a slightly different point of view, revealing even more clearly the structure of the augmented box.

\begin{figure}\centering
\includegraphics[width=10cm]{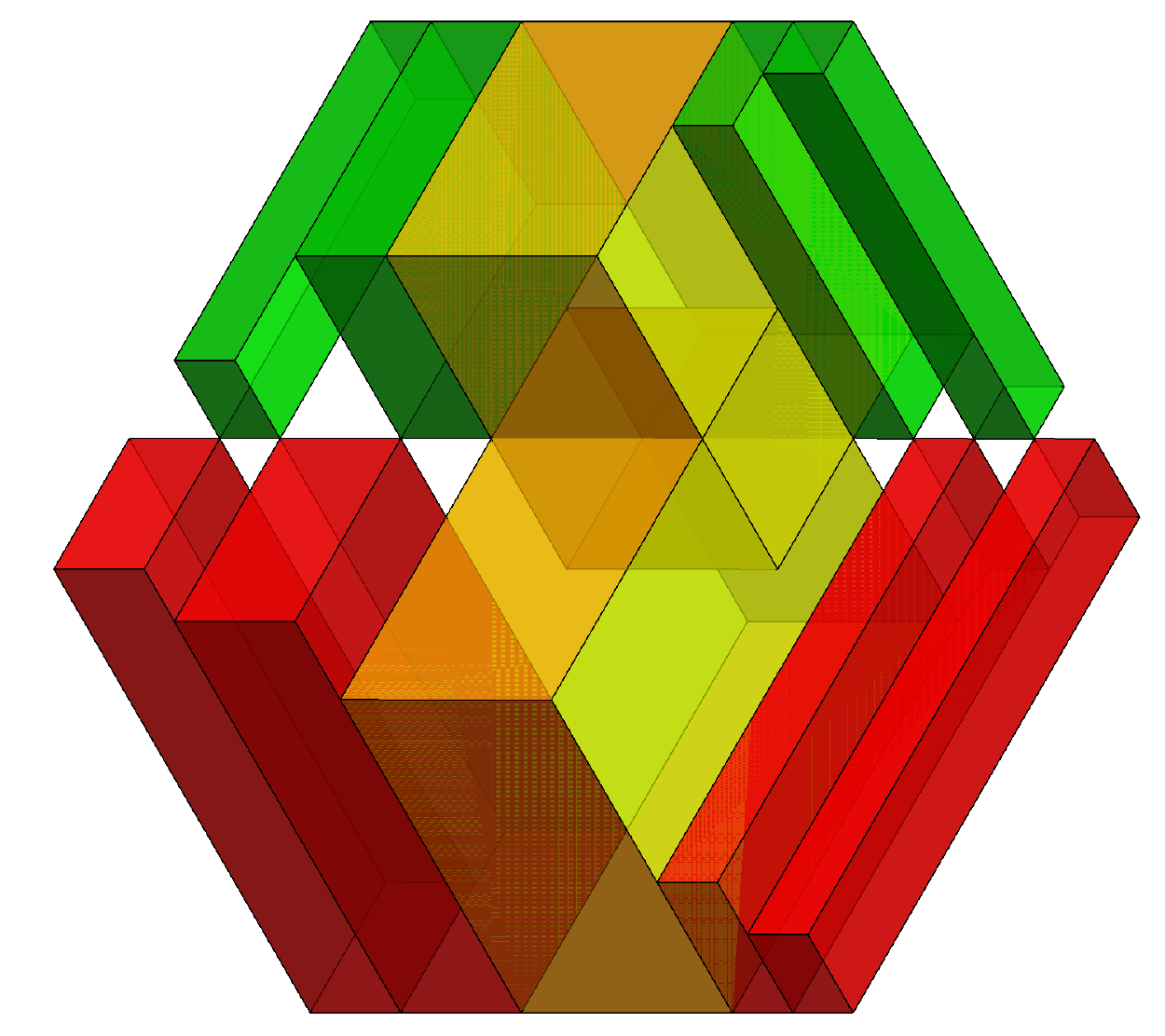}
\caption{The three dimensional visualization (made using \texttt{Mathematica 11}) of the augmented box corresponding to the region $P_{5,5,7,3}(3,2,4,3;\ 2,2,2,2)$.}\label{house1}
\end{figure}

\begin{figure}\centering
\includegraphics[width=10cm]{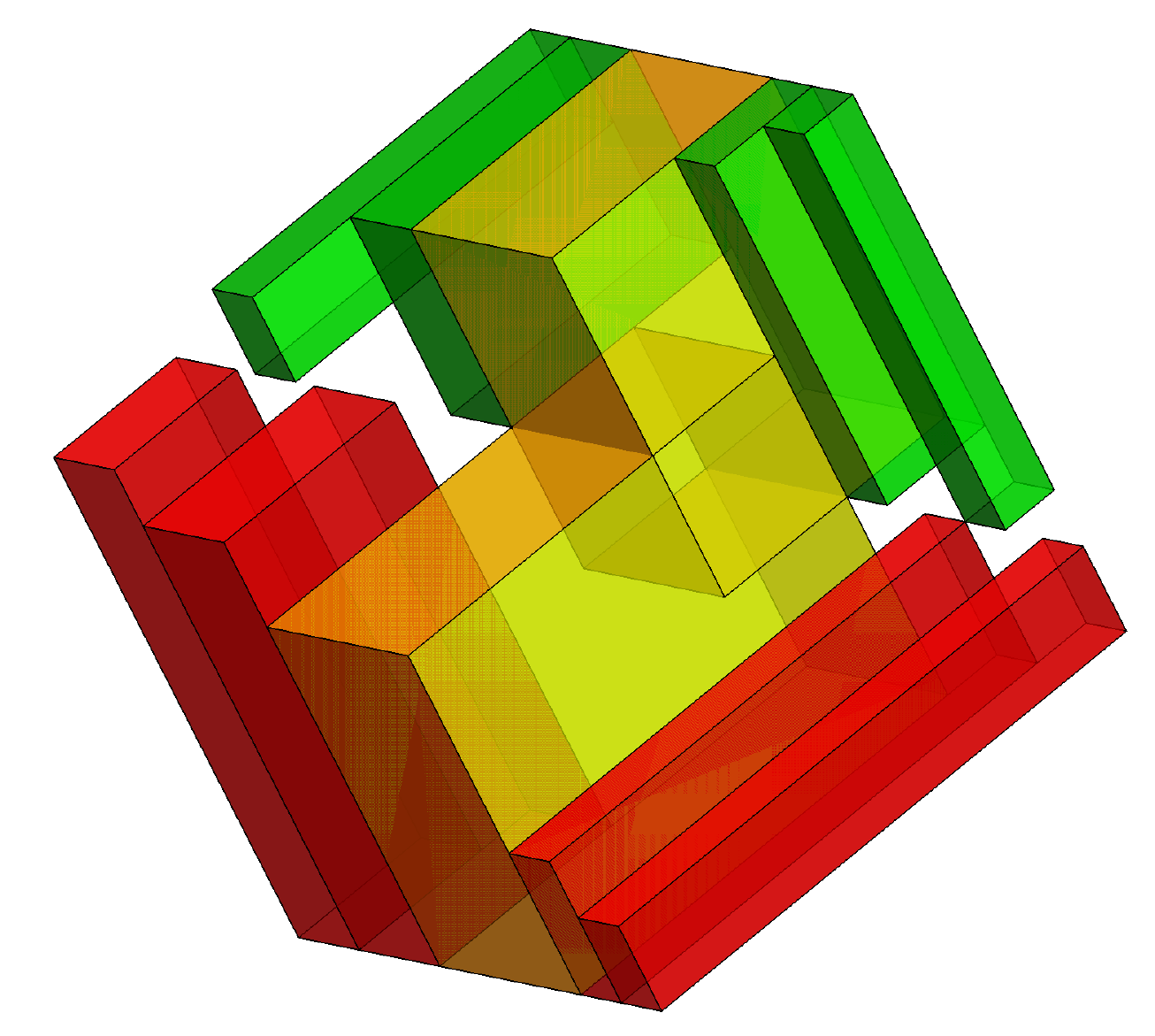}
\caption{The solid in Figure \ref{house1} viewed from a slightly different angle.}\label{house2}
\end{figure}

\section{Statement of weighted enumeration results}

By Figure \ref{hextiled1}, each lozenge tiling of a semi-regular hexagon corresponds to a plane partition (which we consider identified with its three dimensional diagram), which naturally aquires a volume --- the number of unit cubes in the corresponding monotone stack. MacMahon \cite{Mac} also gave an elegant refinement of formula (\ref{MacMahoneq}), which counts plane partitions according to their volume. This serves as our model for the weighted enumerations of the tilings of our $P$- and $Q$-regions which we present in this section.

We recall that the \emph{$q$-integer} $[n]_q$ is defined by $[n]_q:=1+q+q^2+\cdots+q^{n-1}$, the \emph{$q$-factorial} as $[n]_q!:=[1]_q[2]_q\cdots[n]_q$, and \emph{$q$-hyperfactorial} as $\Hf_q(n):=[1]_q![2]_q!\cdots[n-1]_q!$. MacMahon's \cite{Mac} refined enumeration of plane partitions according to their volume is the following.

\begin{thm}[MacMahon's $q$-formula]\label{MacMahoneqq}For non-negative integers $a,b,c$, we have
\begin{equation}
\sum_{\pi}q^{|\pi|}=\frac{\Hf_q(a)\Hf_q(b)\Hf_q(c)\Hf_q(a+b+c)}{\Hf_q(a+b)\Hf_q(b+c)\Hf_q(c+a)},
\end{equation}
where $|\pi|$ denotes the volume of the monotone stack $\pi$, and the sum is taken over all monotone stacks $\pi$ fitting in an $a\times b \times c$ box.
\end{thm}

Via the bijection depicted in Figure \ref{hextiled1}, our weighted tiling enumeration results can be stated as refined countings of monotone stacks fitting in certain solids according to their volume.

We call the three dimensional solids constructed in Section 4 {\it augmented boxes}. Each augmented box is determined by the $P$- or $Q$-region it corresponds to. Our formulas can be regarded as new generalizations of Theorem \ref{MacMahoneqq} corresponding to counting monotone stacks that fit inside a given augmented box according to their volume. Figures \ref{tilingarray}(a) and (b) show illustrative examples of monotone stacks inside augmented boxes corresponding to a $P$- and a $Q$-region, respectively.

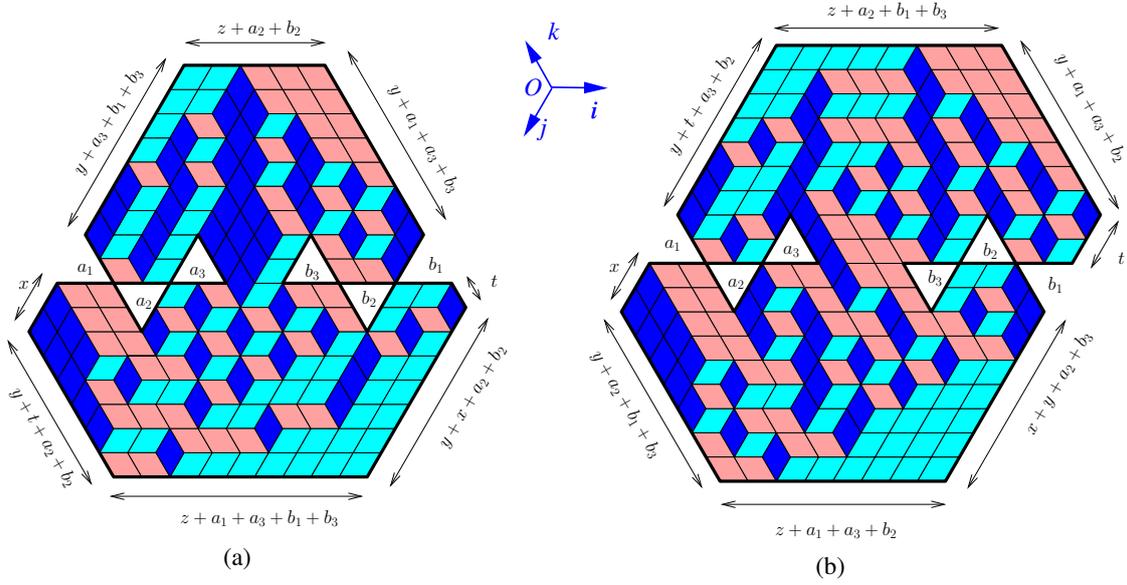
\begin{figure}\centering
\setlength{\unitlength}{3947sp}%
\begingroup\makeatletter\ifx\SetFigFont\undefined%
\gdef\SetFigFont#1#2#3#4#5{%
  \reset@font\fontsize{#1}{#2pt}%
  \fontfamily{#3}\fontseries{#4}\fontshape{#5}%
  \selectfont}%
\fi\endgroup%
\resizebox{15cm}{!}{
\begin{picture}(0,0)%
\includegraphics{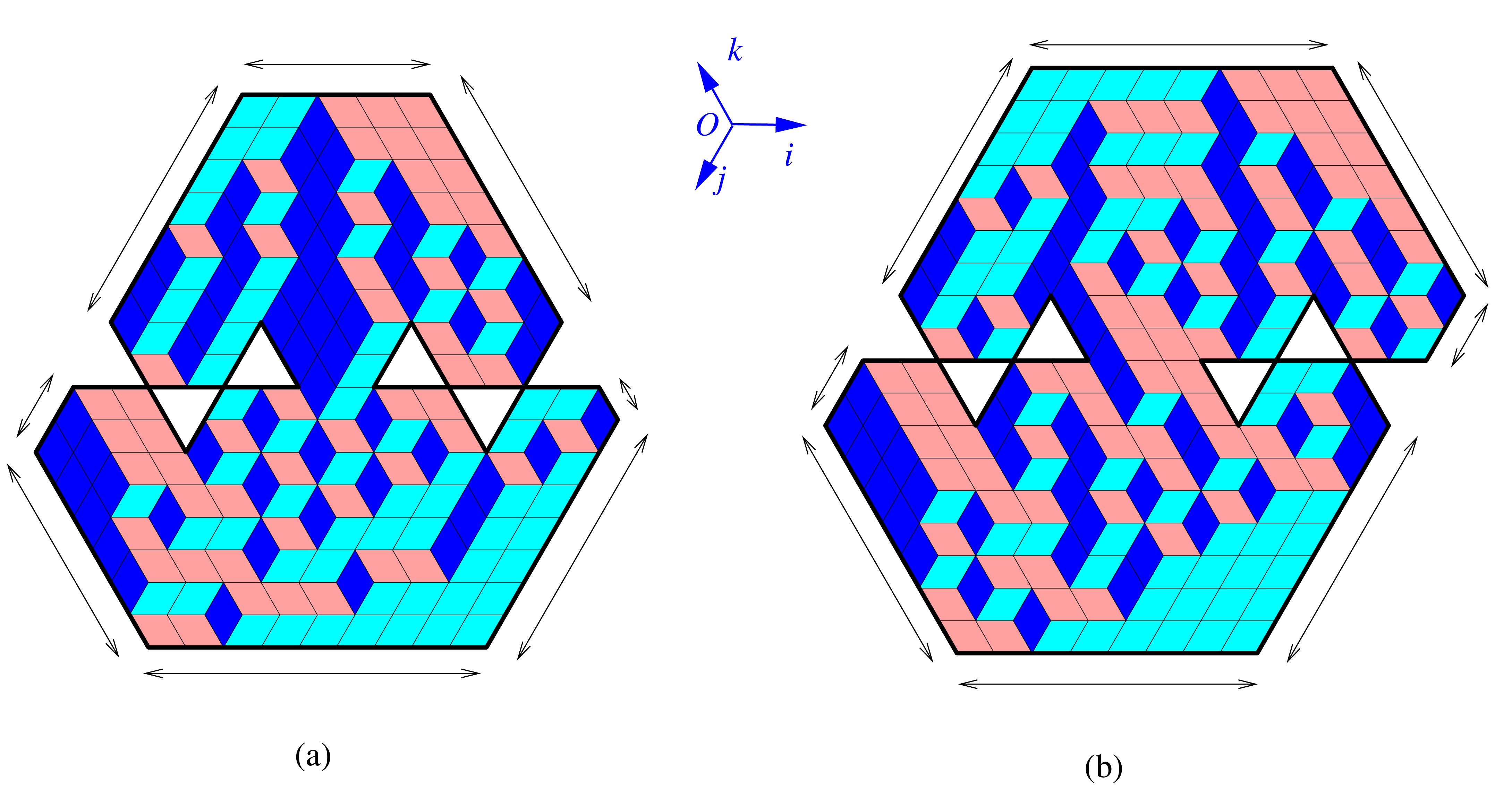}%
\end{picture}

\begin{picture}(20904,10916)(2668,-10897)
\put(21901,-5461){\makebox(0,0)[lb]{\smash{{\SetFigFont{20}{24.0}{\rmdefault}{\mddefault}{\itdefault}{$b_1$}%
}}}}
\put(4081,-3331){\rotatebox{60.0}{\makebox(0,0)[lb]{\smash{{\SetFigFont{20}{24.0}{\rmdefault}{\mddefault}{\itdefault}{$y+a_3+b_1+b_3$}%
}}}}}
\put(6571,-691){\makebox(0,0)[lb]{\smash{{\SetFigFont{20}{24.0}{\rmdefault}{\mddefault}{\itdefault}{$z+a_2+b_2$}%
}}}}
\put(9751,-1723){\rotatebox{300.0}{\makebox(0,0)[lb]{\smash{{\SetFigFont{20}{24.0}{\rmdefault}{\mddefault}{\itdefault}{$y+a_1+a_3+b_3$}%
}}}}}
\put(11623,-5407){\makebox(0,0)[lb]{\smash{{\SetFigFont{20}{24.0}{\rmdefault}{\mddefault}{\itdefault}{$t$}%
}}}}
\put(10816,-8416){\rotatebox{60.0}{\makebox(0,0)[lb]{\smash{{\SetFigFont{20}{24.0}{\rmdefault}{\mddefault}{\itdefault}{$y+x+a_2+b_2$}%
}}}}}
\put(2778,-7352){\rotatebox{300.0}{\makebox(0,0)[lb]{\smash{{\SetFigFont{20}{24.0}{\rmdefault}{\mddefault}{\itdefault}{$y+t+a_2+b_2$}%
}}}}}
\put(5889,-9796){\makebox(0,0)[lb]{\smash{{\SetFigFont{20}{24.0}{\rmdefault}{\mddefault}{\itdefault}{$z+a_1+a_3+b_1+b_3$}%
}}}}
\put(2910,-5445){\makebox(0,0)[lb]{\smash{{\SetFigFont{20}{24.0}{\rmdefault}{\mddefault}{\itdefault}{$x$}%
}}}}
\put(3991,-5131){\makebox(0,0)[lb]{\smash{{\SetFigFont{20}{24.0}{\rmdefault}{\mddefault}{\itdefault}{$a_1$}%
}}}}
\put(5071,-5761){\makebox(0,0)[lb]{\smash{{\SetFigFont{20}{24.0}{\rmdefault}{\mddefault}{\itdefault}{$a_2$}%
}}}}
\put(6061,-5191){\makebox(0,0)[lb]{\smash{{\SetFigFont{20}{24.0}{\rmdefault}{\mddefault}{\itdefault}{$a_3$}%
}}}}
\put(10471,-5101){\makebox(0,0)[lb]{\smash{{\SetFigFont{20}{24.0}{\rmdefault}{\mddefault}{\itdefault}{$b_1$}%
}}}}
\put(9241,-5731){\makebox(0,0)[lb]{\smash{{\SetFigFont{20}{24.0}{\rmdefault}{\mddefault}{\itdefault}{$b_2$}%
}}}}
\put(8191,-5221){\makebox(0,0)[lb]{\smash{{\SetFigFont{20}{24.0}{\rmdefault}{\mddefault}{\itdefault}{$b_3$}%
}}}}
\put(15047,-2956){\rotatebox{60.0}{\makebox(0,0)[lb]{\smash{{\SetFigFont{20}{24.0}{\rmdefault}{\mddefault}{\itdefault}{$y+t+a_3+b_2$}%
}}}}}
\put(17821,-361){\makebox(0,0)[lb]{\smash{{\SetFigFont{20}{24.0}{\rmdefault}{\mddefault}{\itdefault}{$z+a_2+b_1+b_3$}%
}}}}
\put(22051,-1141){\rotatebox{300.0}{\makebox(0,0)[lb]{\smash{{\SetFigFont{20}{24.0}{\rmdefault}{\mddefault}{\itdefault}{$y+a_1+a_3+b_2$}%
}}}}}
\put(23191,-4981){\makebox(0,0)[lb]{\smash{{\SetFigFont{20}{24.0}{\rmdefault}{\mddefault}{\itdefault}{$t$}%
}}}}
\put(21631,-8161){\rotatebox{60.0}{\makebox(0,0)[lb]{\smash{{\SetFigFont{20}{24.0}{\rmdefault}{\mddefault}{\itdefault}{$x+y+a_2+b_3$}%
}}}}}
\put(16861,-10021){\makebox(0,0)[lb]{\smash{{\SetFigFont{20}{24.0}{\rmdefault}{\mddefault}{\itdefault}{$z+a_1+a_3+b_2$}%
}}}}
\put(13485,-6676){\rotatebox{300.0}{\makebox(0,0)[lb]{\smash{{\SetFigFont{20}{24.0}{\rmdefault}{\mddefault}{\itdefault}{$y+a_2+b_1+b_3$}%
}}}}}
\put(13801,-5131){\makebox(0,0)[lb]{\smash{{\SetFigFont{20}{24.0}{\rmdefault}{\mddefault}{\itdefault}{$x$}%
}}}}
\put(14821,-4681){\makebox(0,0)[lb]{\smash{{\SetFigFont{20}{24.0}{\rmdefault}{\mddefault}{\itdefault}{$a_1$}%
}}}}
\put(15991,-5401){\makebox(0,0)[lb]{\smash{{\SetFigFont{20}{24.0}{\rmdefault}{\mddefault}{\itdefault}{$a_2$}%
}}}}
\put(17011,-4801){\makebox(0,0)[lb]{\smash{{\SetFigFont{20}{24.0}{\rmdefault}{\mddefault}{\itdefault}{$a_3$}%
}}}}
\put(19681,-5341){\makebox(0,0)[lb]{\smash{{\SetFigFont{20}{24.0}{\rmdefault}{\mddefault}{\itdefault}{$b_3$}%
}}}}
\put(20701,-4831){\makebox(0,0)[lb]{\smash{{\SetFigFont{20}{24.0}{\rmdefault}{\mddefault}{\itdefault}{$b_2$}%
}}}}
\end{picture}}
\caption{Viewing lozenge tilings of $P$- and $Q$-regions as monotone stacks of unit cubes fitting in an augmented box.}
\label{tilingarray}
\end{figure}

We define $s_q(\textbf{b})$ to be the expression obtained from formula \eqref{semieq} for $s(\textbf{b})$ by replacing each hyperfactorial by its $q$-analog:

\begin{align}\label{semieqq}
s_q(b_1,b_2,\dots,b_{2l-1})&=s_q(b_1,b_2,\dots,b_{2l})\notag\\
&=\dfrac{1}{\Hf_q(b_1+b_{3}+b_{5}+\dotsc+b_{2l-1})}\notag\\
&\times \dfrac{\prod_{\substack{1\leq i\leq j\leq 2l-1,\,\text{$j-i+1$ odd}}}\Hf_q(b_i+b_{i+1}+\dotsc+b_{j})}{\prod_{\substack{1\leq i\leq j\leq 2l-1,\,\text{$j-i+1$ even}}}\Hf_q(b_i+b_{i+1}+\dotsc+b_{j})}.
\end{align}

We define the $q$-analog
$\Phi^{q}_{x,y,z,t}(a_1,\dotsc,a_m;b_1,\dotsc,b_n)$ of
$\Phi_{x,y,z,t}(a_1,\dotsc,a_m;b_1,\dotsc,b_n)$ and the $q$-analog
$\Psi^{q}_{x,y,z,t}(a_1,\dotsc,a_m;b_1,\dotsc,b_n)$ of
$\Psi_{x,y,z,t}(a_1,\dotsc,a_m;b_1,\dotsc,b_n)$
by replacing each hyperfactorial and each $s$-function by the corresponding $q$-hyperfactorial and $s_q$-function, respectively.

The main result of this paper is the following.

\begin{thm}\label{main2}
Let $x,y,z,t$ and $a_1,\dotsc,a_m,b_1,\dotsc,b_n$ be non-negative integers.

$(${\rm a}$)$ We have
\begin{equation}\label{main2eq1}
\sum_{\pi\in P}q^{|\pi|}=\Phi^{q}_{x,y,z,t}(a_1,\dotsc,a_m;b_1,\dotsc,b_n),
\end{equation}
where the sum is taken over all monotone stacks $\pi$ that fit in the augmented box corresponding to the region $P_{x,y,z,t}(a_1,\dotsc,a_m;b_1,\dotsc,b_n)$.

$(${\rm b}$)$ We have
\begin{equation}\label{main2eq2}
\sum_{\pi\in Q}q^{|\pi|}=\Psi^{q}_{x,y,z,t}(a_1,\dotsc,a_{n};b_1,\dotsc,b_n),
 \end{equation}
 where the sum is taken over all monotone stacks $\pi$ that fit in the augmented box corresponding to the region $Q_{x,y,z,t}(a_1,\dotsc,a_m;b_1,\dotsc,b_n)$.
\end{thm}

The $q=1$ specialization of the above result recovers Theorem \ref{main1}. However, we chose to state Theorem \ref{main1} first in order to show how that unweighted version, through its three dimensional interpretation described in Section 3, led us, in the spirit of MacMahon's $q$-generalization of the plane partition enumeration formula (\ref{MacMahoneq}), to the $q$-version stated above.


Such `\emph{simple}' product $q$-formulas turn out to be fairly rare in the field of tiling enumeration (here `simple' means that the size of the largest factor in the product is linear in the parameters). We refer the reader for instance to \cite{Stanley}, \cite{Tri1} and \cite{Tri2} for some other examples and for further discussion of $q$-enumerations of lozenge tilings.

The rest of the paper is organized as follows. In Section \ref{Pre}, we go over some preliminary results that we will employ in our proofs. The section after that is devoted to the $q$-enumeration of lozenge tilings of a hexagon with one fern removed. We will use the latter in our proof of Theorem~\ref{main2} in Section~8.
In Section~9 we present an open problem concerning a Schur function identity suggested by our results. Section~10 considers a statistical physics problem that can be tackled using our enumeration formulas.


\section{Preliminaries}\label{Pre}
In general, we allow lozenges in a region to carry some weights, which are determined by the positions of the lozenges (equivalently, we allow the edges of the planar dual to have weights on them). Define the \emph{weight} of a lozenge tiling to be the product of weights of its lozenges. In the weighted case, we use the notation $\M(R)$ for the sum of weights of all tilings of the region $R$. When the region is unweighted (i.e. all lozenges have weight $1$), then $\M(R)$ is exactly the number of lozenge tilings of the region. This is consistent with the use of the notation $\M(R)$ in the previous sections.

A \emph{forced lozenge} in a region $R$ on the triangular lattice is a lozenge contained in any tiling of $R$. Assume that we remove several forced lozenges $l_1,l_2\dotsc,l_n$ from the region $R$ and get a new region $R'$. Then clearly
\begin{equation}\label{forcedeq}
\M(R)=\M(R')\prod_{i=1}^{n}wt(l_i),
\end{equation}
where $wt(l_i)$ is the weight of the lozenge $l_i$.

\begin{lem}[Region-splitting Lemma]\label{GS}
Let $R$ be a balanced region\footnote{ A balanced region on the triangular lattice is a finite lattice region which contains the same number of unit triangles of each orientation. Since each lozenge covers one unit triangle of each orientation, this is a necessary condition for a region to admit lozenge tilings.} on the triangular lattice. Assume that a sub-region $Q$ of $R$ satisfies the following two conditions:
\begin{enumerate}
\item[(i)] \text{\rm{(Separating Condition)}} All unit triangles in $Q$ that are adjacent to some unit triangle of $R-Q$ have the same orientation.


\item[(ii)] \text{\rm{(Balancing Condition)}} $Q$ is balanced.
\end{enumerate}
Then
\begin{equation}
\M(R)=\M(Q)\, \M(R-Q).
\end{equation}
\end{lem}
\begin{proof}
Let $G$ be the dual graph of $R$, and $H$ the dual graph of $Q$. Then $H$ satisfies the conditions in the Graph-splitting Lemma 3.6(a) in \cite{Tri}, and the lemma follows.
\end{proof}

Let $G$ be a finite simple graph without loops. A \emph{perfect matching} of $G$ is a collection of disjoint edges covering all vertices of $G$. The \emph{(planar) dual graph} of a region $R$ on the triangular lattice is the graph whose vertices are the unit triangles in $R$, and whose edges connect precisely those pairs of unit triangles that share an edge. In the weighted case, the edge of the dual graph carries the same weight as its corresponding lozenge in the region. The tilings of a region can be identified with the perfect matchings of its dual graph. The sum of weights of all perfect matchings of the graph $G$ is denoted by $\M(G)$, where the \emph{weight} of a perfect matching is the product of weights of its constituent edges  (writing $\M(R)$ for the weighted sum of lozenge tilings of the region $R$ is justified by the above mentioned identification between tilings and perfect matchings).

In our proofs we use Kuo's graphical condensation method \cite{Kuo}. We include below the forms of Kuo's results that we will need.

\begin{thm}[Theorem 5.1 \cite{Kuo}]\label{kuothm}
Let $G=(V_1,V_2,E)$ be a (weighted) bipartite planar graph in which $|V_1|=|V_2|$. Assume that  $u, v, w, s$ are four vertices appearing in this cyclic order on a face of $G$ so that $u,w \in V_1$ and $v,s \in V_2$. Then
\begin{equation}\label{kuoeq}
\M(G)\M(G-\{u, v, w, s\})=\M(G-\{u, v\})\M(G-\{ w, s\})+\M(G-\{u, s\})\M(G-\{v, w\}).
\end{equation}
\end{thm}

\begin{thm}[Theorem 5.2 \cite{Kuo}]\label{kuothm2}
Let $G=(V_1,V_2,E)$ be a (weighted) bipartite planar graph in which $|V_1|=|V_2|$. Assume that  $u, v, w, s$ are four vertices appearing in this cyclic order on a face of $G$ so that $u,v \in V_1$ and $w,s \in V_2$. Then
\begin{equation}\label{kuoeq2}
\M(G-\{u, s\})\M(G-\{v, w\})=\M(G)\M(G-\{u, v, w, s\})+\M(G-\{u,w\})\M(G-\{v, s\}).
\end{equation}
\end{thm}

\begin{thm}[Theorem 5.3 \cite{Kuo}]\label{kuothm3}
Let $G=(V_1,V_2,E)$ be a (weighted) bipartite planar graph in which $|V_1|=|V_2|+1$. Assume that  $u, v, w, s$ are four vertices appearing in this cyclic order on a face of $G$ so that $u,v,w \in V_1$ and $s \in V_2$. Then
\begin{equation}\label{kuoeq3}
\M(G-\{v\})\M(G-\{u, w,s\})=\M(G-\{u\})\M(G-\{ v, w, s\})+\M(G-\{w\})\M(G-\{u,v, s\}).
\end{equation}
\end{thm}

There are three orientations of the lozenges on the triangular lattice: vertical, left-tilting, and right-tilting (see Figure \ref{rhumbustype}). As each tiling of a semi-regular hexagon can be identified with a monotone stack, each monotone stack yields a weight assignment to the corresponding tiling as follows. Given a monotone stack, the top of each column in the stack corresponds to a right-tilting lozenge (see Figure~\ref{hextiled1}). We assign to each right-tilting lozenge the weight $q^{h}$, where $h$ is the height of the stack of unit cubes under it; assign  weight 1 to all left-tilting and all vertical lozenges. Define the \emph{natural $q$-weight} of the tiling to be the product of the weights of all its constituent lozenges; clearly, this is precisely $q$ to the volume of the corresponding monotone stack. We call the described weight assignment the \emph{natural $q$-weight assignment}, and we denote it by $\wt_0$.

Similarly, using the bijection described in Section 3 between lozenge tilings of our $P$- and $Q$-regions and monotone stacks that fit inside the corresponding augmented boxes, we obtain the natural $q$-weight assignment for tilings of $P$- and $Q$-regions. However, in contrast to the case of a simple box, it is not hard to see that for an augmented box this weight assignment to the right-tilting lozenges depends, in general, on the tiling containing it (so the same lozenge may carry different weights in different tilings). See Figure \ref{Naturalweight} for an example this fact for a hexagon, viewed as a special case of $P$-and $Q$-regions. This prevents us from applying the method of Kuo condensation directly for proving our results. We surmount this obstacle by introducing the following two additional $q$-weight assignments that are independent of the choice of the tiling.

\begin{figure}\centering
\includegraphics[width=6cm]{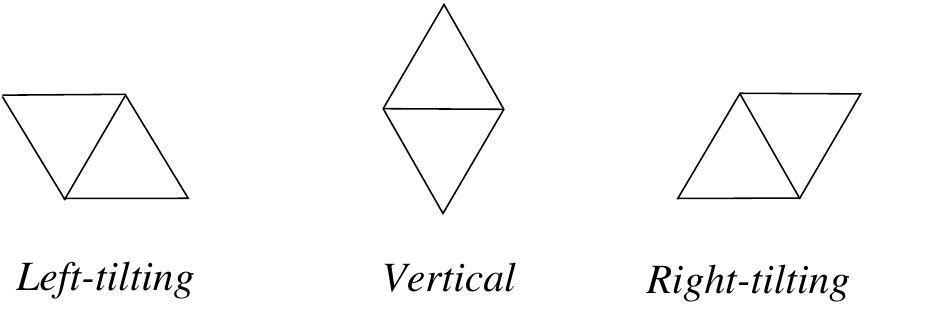}
\caption{Three orientations of lozenges.}
\label{rhumbustype}
\end{figure}

\begin{figure}\centering
\includegraphics[width=8cm]{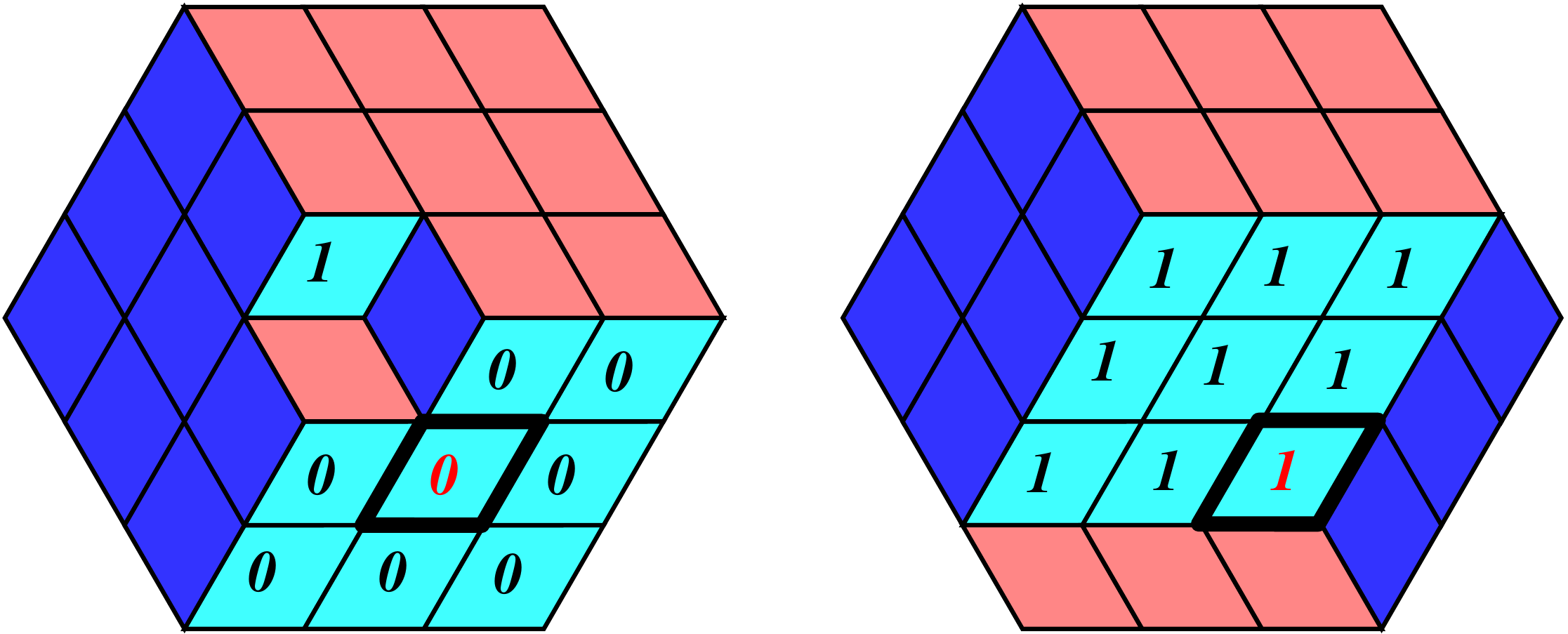}
\caption{Natural $q$-weighting depends on the tilings; the lozenge with label $i$ is weighted by $q^{i}$. The same right-tilting lozenge, that is restricted by the bold contour, has different weights in two different tilings.}\label{Naturalweight}
\end{figure}

In the first of these additional $q$-weight assignments, denoted $\wt_1$, each right-tilting lozenge has weight $q^{x}$, where~$x$ is the distance (measured in altitudes of unit triangles) from its top to the bottom of the hexagon; all other lozenges are weighted by $1$ (see Figure \ref{tilingarrayweight}(a)). In the second one, denoted $\wt_2$, each {\it left}-tilting lozenge whose top side is $y$ units above the bottom edge of the hexagon is weighted by $q^{y}$, and all other lozenges have weight $1$ (see Figure \ref{tilingarrayweight}(b)). The weight of a tiling is defined as the product of the weights of its lozenges, as usual.

One readily sees that the above $q$-weight assignments can be applied to any bounded simply-connected region on the triangular lattice, in particular to dented semihexagons and to our $P$- and $Q$-regions. Denote the corresponding weighted tiling counts of a simply connected region $R$ by
\[\M_1(R)=\sum_{T\in \mathcal{T}(R)}\wt_1(T) \text{  and } \M_2(R)=\sum_{T\in \mathcal{T}(R)}\wt_2(T),\]
where $\mathcal{T}(R)$ denotes the set of all tilings of $R$. We call $\M_1(R)$ and $\M_2(R)$ \emph{tiling generating functions} of $R$.

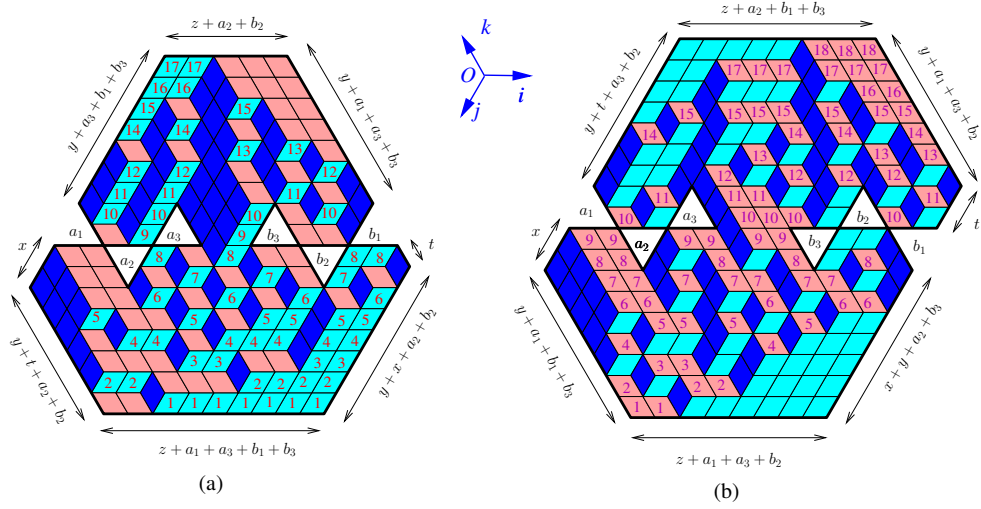
\begin{figure}\centering
\setlength{\unitlength}{3947sp}%
\begingroup\makeatletter\ifx\SetFigFont\undefined%
\gdef\SetFigFont#1#2#3#4#5{%
  \reset@font\fontsize{#1}{#2pt}%
  \fontfamily{#3}\fontseries{#4}\fontshape{#5}%
  \selectfont}%
\fi\endgroup%
\resizebox{13cm}{!}{
\begin{picture}(0,0)%
\includegraphics{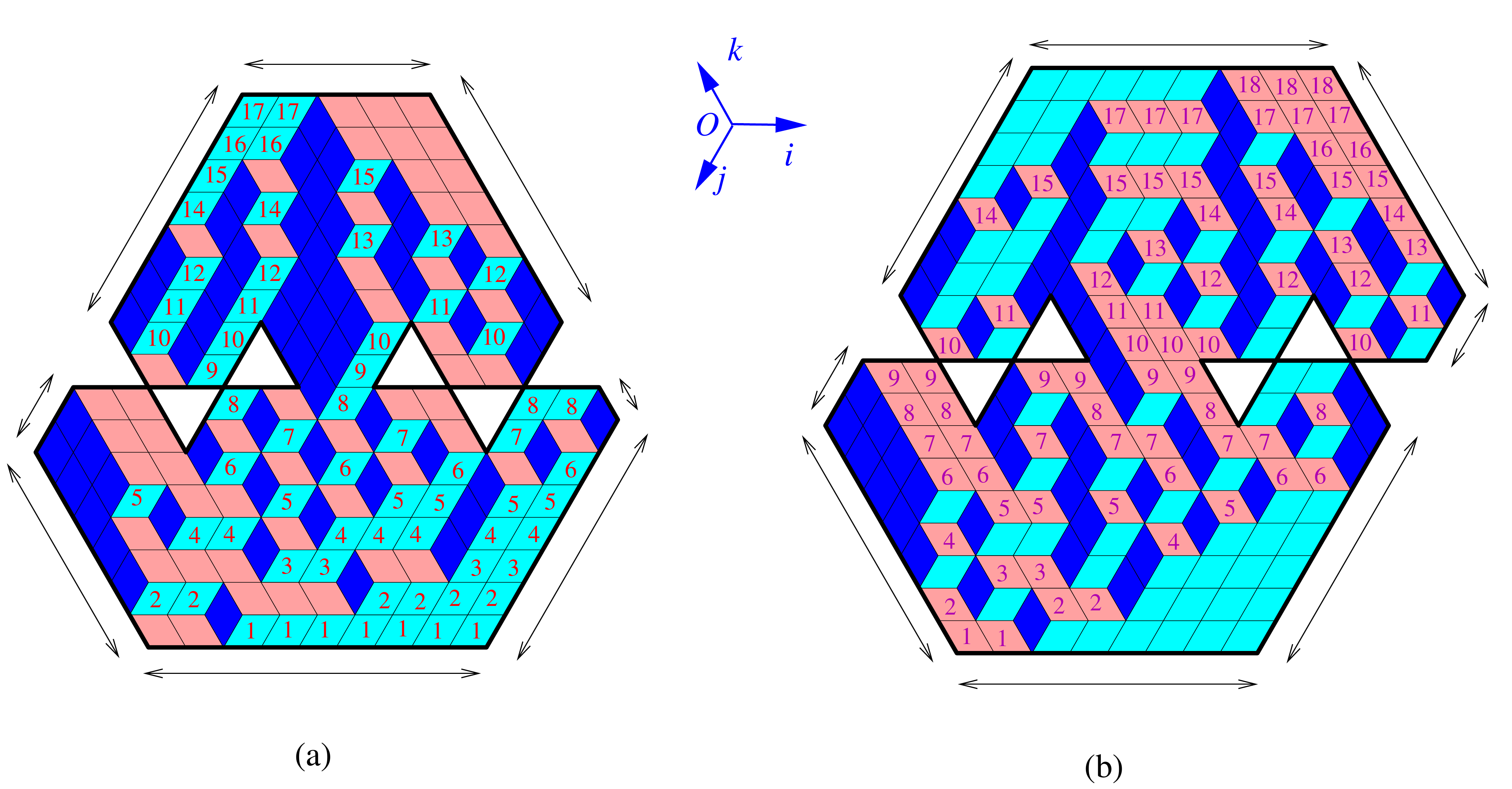}%
\end{picture}%
\begin{picture}(20904,10916)(2668,-10897)
\put(15991,-5401){\makebox(0,0)[lb]{\smash{{\SetFigFont{20}{24.0}{\rmdefault}{\mddefault}{\itdefault}{$a_2$}%
}}}}
\put(15991,-5401){\makebox(0,0)[lb]{\smash{{\SetFigFont{20}{24.0}{\rmdefault}{\mddefault}{\itdefault}{$a_2$}%
}}}}
\put(4081,-3331){\rotatebox{60.0}{\makebox(0,0)[lb]{\smash{{\SetFigFont{20}{24.0}{\rmdefault}{\mddefault}{\itdefault}{$y+a_3+b_1+b_3$}%
}}}}}
\put(6571,-691){\makebox(0,0)[lb]{\smash{{\SetFigFont{20}{24.0}{\rmdefault}{\mddefault}{\itdefault}{$z+a_2+b_2$}%
}}}}
\put(9751,-1723){\rotatebox{300.0}{\makebox(0,0)[lb]{\smash{{\SetFigFont{20}{24.0}{\rmdefault}{\mddefault}{\itdefault}{$y+a_1+a_3+b_3$}%
}}}}}
\put(11623,-5407){\makebox(0,0)[lb]{\smash{{\SetFigFont{20}{24.0}{\rmdefault}{\mddefault}{\itdefault}{$t$}%
}}}}
\put(10681,-8581){\rotatebox{60.0}{\makebox(0,0)[lb]{\smash{{\SetFigFont{20}{24.0}{\rmdefault}{\mddefault}{\itdefault}{$y+x+a_2+b_2$}%
}}}}}
\put(2778,-7352){\rotatebox{300.0}{\makebox(0,0)[lb]{\smash{{\SetFigFont{20}{24.0}{\rmdefault}{\mddefault}{\itdefault}{$y+t+a_2+b_2$}%
}}}}}
\put(5889,-9796){\makebox(0,0)[lb]{\smash{{\SetFigFont{20}{24.0}{\rmdefault}{\mddefault}{\itdefault}{$z+a_1+a_3+b_1+b_3$}%
}}}}
\put(2910,-5445){\makebox(0,0)[lb]{\smash{{\SetFigFont{20}{24.0}{\rmdefault}{\mddefault}{\itdefault}{$x$}%
}}}}
\put(3991,-5131){\makebox(0,0)[lb]{\smash{{\SetFigFont{20}{24.0}{\rmdefault}{\mddefault}{\itdefault}{$a_1$}%
}}}}
\put(5071,-5761){\makebox(0,0)[lb]{\smash{{\SetFigFont{20}{24.0}{\rmdefault}{\mddefault}{\itdefault}{$a_2$}%
}}}}
\put(6061,-5191){\makebox(0,0)[lb]{\smash{{\SetFigFont{20}{24.0}{\rmdefault}{\mddefault}{\itdefault}{$a_3$}%
}}}}
\put(10351,-5191){\makebox(0,0)[lb]{\smash{{\SetFigFont{20}{24.0}{\rmdefault}{\mddefault}{\itdefault}{$b_1$}%
}}}}
\put(9241,-5731){\makebox(0,0)[lb]{\smash{{\SetFigFont{20}{24.0}{\rmdefault}{\mddefault}{\itdefault}{$b_2$}%
}}}}
\put(8191,-5221){\makebox(0,0)[lb]{\smash{{\SetFigFont{20}{24.0}{\rmdefault}{\mddefault}{\itdefault}{$b_3$}%
}}}}
\put(15047,-2956){\rotatebox{60.0}{\makebox(0,0)[lb]{\smash{{\SetFigFont{20}{24.0}{\rmdefault}{\mddefault}{\itdefault}{$y+t+a_3+b_2$}%
}}}}}
\put(17821,-361){\makebox(0,0)[lb]{\smash{{\SetFigFont{20}{24.0}{\rmdefault}{\mddefault}{\itdefault}{$z+a_2+b_1+b_3$}%
}}}}
\put(22051,-1141){\rotatebox{300.0}{\makebox(0,0)[lb]{\smash{{\SetFigFont{20}{24.0}{\rmdefault}{\mddefault}{\itdefault}{$y+a_1+a_3+b_2$}%
}}}}}
\put(23191,-4981){\makebox(0,0)[lb]{\smash{{\SetFigFont{20}{24.0}{\rmdefault}{\mddefault}{\itdefault}{$t$}%
}}}}
\put(21406,-8476){\rotatebox{60.0}{\makebox(0,0)[lb]{\smash{{\SetFigFont{20}{24.0}{\rmdefault}{\mddefault}{\itdefault}{$x+y+a_2+b_3$}%
}}}}}
\put(16861,-10021){\makebox(0,0)[lb]{\smash{{\SetFigFont{20}{24.0}{\rmdefault}{\mddefault}{\itdefault}{$z+a_1+a_3+b_2$}%
}}}}
\put(13485,-6676){\rotatebox{300.0}{\makebox(0,0)[lb]{\smash{{\SetFigFont{20}{24.0}{\rmdefault}{\mddefault}{\itdefault}{$y+a_1+b_1+b_3$}%
}}}}}
\put(13801,-5131){\makebox(0,0)[lb]{\smash{{\SetFigFont{20}{24.0}{\rmdefault}{\mddefault}{\itdefault}{$x$}%
}}}}
\put(14821,-4681){\makebox(0,0)[lb]{\smash{{\SetFigFont{20}{24.0}{\rmdefault}{\mddefault}{\itdefault}{$a_1$}%
}}}}
\put(15991,-5401){\makebox(0,0)[lb]{\smash{{\SetFigFont{20}{24.0}{\rmdefault}{\mddefault}{\itdefault}{$a_2$}%
}}}}
\put(17011,-4801){\makebox(0,0)[lb]{\smash{{\SetFigFont{20}{24.0}{\rmdefault}{\mddefault}{\itdefault}{$a_3$}%
}}}}
\put(19681,-5341){\makebox(0,0)[lb]{\smash{{\SetFigFont{20}{24.0}{\rmdefault}{\mddefault}{\itdefault}{$b_3$}%
}}}}
\put(20701,-4831){\makebox(0,0)[lb]{\smash{{\SetFigFont{20}{24.0}{\rmdefault}{\mddefault}{\itdefault}{$b_2$}%
}}}}
\put(21901,-5461){\makebox(0,0)[lb]{\smash{{\SetFigFont{20}{24.0}{\rmdefault}{\mddefault}{\itdefault}{$b_1$}%
}}}}
\put(15991,-5401){\makebox(0,0)[lb]{\smash{{\SetFigFont{20}{24.0}{\rmdefault}{\mddefault}{\itdefault}{$a_2$}%
}}}}
\end{picture}}
\caption{(a) The weight assignment $\wt_1$ on a sample tiling of a $P$-region. (b) The weight assignment $\wt_2$ on a sample tiling of a $Q$-region.}
\label{tilingarrayweight}
\end{figure}

The following is a consequence of MacMahon's Theorem \ref{MacMahoneqq}.

\begin{lem} For any non-negative integers $a,b,c$, we have
\label{lemsixfive}
  \begin{equation}
\M_1(H_{a,b,c})=q^{b\binom{a+1}{2}}\frac{\Hf_q(a)\Hf_q(b)\Hf_q(c)\Hf_q(a+b+c)}{\Hf_q(a+b)\Hf_q(b+c)\Hf_q(c+a)}
\end{equation}
and
\begin{equation}
\M_2(H_{a,b,c})=q^{b\binom{c+1}{2}}\frac{\Hf_q(a)\Hf_q(b)\Hf_q(c)\Hf_q(a+b+c)}{\Hf_q(a+b)\Hf_q(b+c)\Hf_q(c+a)}.
\end{equation}
\end{lem}
\begin{proof}
The first equality was proved in Corollary 3.2 in \cite{Tri1}; it holds because if we think of the plane partition as consisting of $a$ slices parallel to the sides of length $b$ (see Figure \ref{hextiled1}(a)), then each of the $b$ right-tilting lozenges $L$ in the $i$th slice from the front has weight $q^{i+h}$, where $h$ is the number of unit cubes under $L$, for $i=1,\dotsc,a$. Reflecting the hexagon $H_{a,b,c}$ weighted by $\wt_2$ over a vertical line, we get the hexagon $H_{c,b,a}$ weighted by $\wt_1$. Then the second equality follows from the first one.
\end{proof}

The next lemma gives the $q$-enumerations of lozenge tilings of dented semihexagons.

\begin{lem}\label{semi}For any
non-negative integers $b_1,b_2,b_3,\dotsc,b_{2l}$ we have
\begin{align}
\M_1(S(b_1,b_2,\dotsc,b_{2l}))&=\M_1(S(b_1,b_2,\dotsc,b_{2l-1}))\notag\\
&=q^{\sum_{i=1}^{l-1}b_{2i}\binom{b_{2i+1}+\dotsc+b_{2l-1}+1}{2}}s_{q}(b_1,b_2,\dotsc,b_{2l-1}).
\end{align}
and
\begin{equation}
\M_2(S(b_1,b_2,\dotsc,b_{n}))=\M_1(S(b_n,b_{n-1},\dotsc,b_{1})).
\end{equation}
\end{lem}
\begin{proof}
The first equation follows from the $q$-enumeration of column-strict plane partitions that fit in a given Young diagram and have bounded part size (see e.g. \cite[pp. 374--375]{Stanley}) and a bijection between tilings of a dented semihexagon and column-strict plane partitions (in a picture analogous to Figure~\ref{tilingarrayweight}(a) for a dented hexagon, the corresponding column-strict plane partition is given by the numbers written in the right-tilting lozenges). The second equality follows from the first one by reflecting the region over a vertical line.
\end{proof}

We now present the connection between the weighted enumerations (under $\wt_1$ and $\wt_2$) of the lozenge tilings of our regions and the $q$-enumeration of the corresponding monotone stacks according to volume. Both weighted tiling enumerations turn out to differ from the latter just by a power of $q$. However, the exponents of $q$ are somewhat complicated expressions in $x$, $y$, $z$, $t$, the $a_i$'s and the $b_i$'s, and for convenient reference we define them below.

We adopt the following notation for $0\leq a<b$:
\begin{equation}
\label{sq}
\left\langle\begin{matrix}
b\\ a\end{matrix}\right\rangle:=(a+1)+(a+2)+\dotsc+b.
\end{equation}


For any non-negative integers $x,y,z,t$ and sequences of non-negative integers $\textbf{a}=(a_1,\dotsc,a_m)$ and $\textbf{ b}=(b_1,\dotsc,b_n)$ define two functions as follows:

\begin{align}
\label{A1}
A^{(1)}_{x,y,z,t}(\textbf{a};\textbf{b})=&\sum_{i=1}^{\lfloor \frac{m+1}{2}\rfloor}a_{2i-1}\left\langle\begin{matrix}x+\sum_{j=1}^{i-1}a_{2j}\\0\end{matrix}\right\rangle+z\left\langle\begin{matrix}x+y+e_a\\0\end{matrix}\right\rangle \notag\\
&+\sum_{i=1}^{\lfloor\frac{n+1}{2}\rfloor}b_{2i-1}\left\langle\begin{matrix}x+y+e_a+e_b-\sum_{j=1}^{i-1}b_{2j}\\0\end{matrix}\right\rangle\notag\\
&
+\sum_{i=1}^{\lfloor\frac{m}{2}\rfloor}a_{2i}\left\langle\begin{matrix}\Delta+y+o_a+o_b-\sum_{j=1}^{i}a_{2j-1}\\ \Delta\end{matrix}\right\rangle+z\left\langle\begin{matrix}\Delta+o_b\\\Delta\end{matrix}\right\rangle\notag\\
&
+\sum_{i=1}^{\lfloor\frac{n}{2}\rfloor}b_{2i}\left\langle\begin{matrix}\Delta+\sum_{j=0}^{i-1}b_{2j+1} \\\Delta \end{matrix}\right\rangle,
\end{align}
\begin{align}
\label{A2}
A^{(2)}_{x,y,z,t}(\textbf{a};\textbf{b})=
&\sum_{i=1}^{\lfloor\frac{m+1}{2}\rfloor}a_{2i-1}\left\langle\begin{matrix}
                                                       y+t+e_a+e_b-\sum_{j=1}^{i-1}a_{2j} \\
                                                       0
                                                     \end{matrix}\right\rangle+z\left\langle\begin{matrix}y+t+e_b\\0\end{matrix}\right\rangle \notag\\
&+\sum_{i=1}^{\lfloor\frac{n+1}{2}\rfloor}b_{2i-1}\left\langle\begin{matrix}t+\sum_{j=1}^{i-1}b_{2j}\\0\end{matrix}\right\rangle\notag\\
&
+\sum_{i=1}^{\lfloor\frac{m}{2}\rfloor}a_{2i}\left\langle\begin{matrix}\Delta+\sum_{j=1}^{i}a_{2j-1}\\\Delta\end{matrix}\right\rangle+z\left\langle\begin{matrix}\Delta+o_a\\\Delta\end{matrix}\right\rangle\notag\\
&
+\sum_{i=1}^{\lfloor\frac{n}{2}\rfloor}b_{2i}\left\langle\begin{matrix}\Delta+y+o_a+o_b-\sum_{j=1}^{i}b_{2i-1}\\\Delta\end{matrix}\right\rangle,
\end{align}
where $\Delta=x+y+t+e_a+e_b$. We note that
\begin{equation}
\label{A1A2symm}
A^{(2)}_{x,y,z,t}(\textbf{a};\textbf{b})=A^{(1)}_{t,y,z,x}(\textbf{b};\textbf{a}).
\end{equation}
Similarly, we define two more functions:
\begin{align}
\label{B1}
B^{(1)}_{x,y,z,t}(\textbf{a};\textbf{b})=&
\sum_{i=1}^{\lfloor\frac{m+1}{2}\rfloor}a_{2i-1}\left\langle\begin{matrix}
                                                                                            x+\sum_{j=1}^{i-1}a_{2j} \\
                                                                                            0
                                                                                          \end{matrix}\right\rangle+ z\left\langle\begin{matrix}
                                                                                                            x+y+e_a \\
                                                                                                            0
                                                                                                          \end{matrix}\right\rangle \notag\\
&+\sum_{i=1}^{\lfloor\frac{n}{2}\rfloor}b_{2i}\left\langle\begin{matrix}
                                                                                                                                                                x+y+e_a+o_b-\sum_{j=1}^{i}b_{2j-1} \\
                                                                                                                                                                0
                                                                                                                                                              \end{matrix}\right\rangle\notag\\
&
+\sum_{i=1}^{\lfloor\frac{m}{2}\rfloor}a_{2i}\left\langle\begin{matrix}
                                                \square+y+t+o_a+e_b-\sum_{j=1}^{i}a_{2j-1} \\
                                                 \square
                                              \end{matrix}\right\rangle+z\left\langle\begin{matrix}
                                                               \square+t+e_b \\
                                                               \square
                                                             \end{matrix}\right\rangle\notag\\
&+\sum_{i=1}^{\lfloor\frac{n+1}{2}\rfloor}b_{2i-1}\left\langle\begin{matrix}
                                                                                                                            \square+t+\sum_{j=1}^{i-1}b_{2j} \\
                                                                                                                            \square
                                                                                                                          \end{matrix}\right\rangle
\end{align}
and
\begin{align}
\label{B2}
B^{(2)}_{x,y,z,t}(\textbf{a};\textbf{b})=&
\sum_{i=1}^{\lfloor\frac{m+1}{2}\rfloor}a_{2i-1}\left\langle\begin{matrix}
                                                                                            y+e_a+o_b-\sum_{j=1}^{i-1}a_{2j} \\
                                                                                            0
                                                                                          \end{matrix}\right\rangle+ z\left\langle\begin{matrix}
                                                                                                           o_b \\
                                                                                                            0
                                                                                                          \end{matrix}\right\rangle
+\sum_{i=1}^{\lfloor\frac{n}{2}\rfloor}b_{2i}\left\langle\begin{matrix}
                                                                                                                                                               \sum_{j=1}^{i}b_{2j-1} \\
                                                                                                                                                                0
                                                                                                                                                              \end{matrix}\right\rangle\notag\\
&+\sum_{i=1}^{\lfloor\frac{m}{2}\rfloor}a_{2i}\left\langle\begin{matrix}
                                                \square+\sum_{j=1}^{i}a_{2j-1} \\
                                                \square
                                              \end{matrix}\right\rangle+z\left\langle\begin{matrix}
                                                               \square+y+o_a \\
                                                               \square
                                                             \end{matrix}\right\rangle\notag\\
&+\sum_{i=1}^{\lfloor\frac{n+1}{2}\rfloor}b_{2i-1}\left\langle\begin{matrix}
                                                                                                                            \square+y+o_a+e_b-\sum_{j=1}^{i-1}b_{2j} \\
                                                                                                                            \square
                                                                                                                          \end{matrix}\right\rangle,
\end{align}
where $\square= x+y+e_a+o_b$.

There is a connection between the expressions in \eqref{B1} and \eqref{B2} which is similar to \eqref{A1A2symm}, but more complicated to state. Note that the sums in \eqref{B1} and \eqref{B2} can be rewritten as sums over the even- (or odd-) indexed $a_i$'s (resp., $b_i$'s), multiplied by symbols \eqref{sq}
in the same fashion as in
\eqref{B1} and \eqref {B2}, involving sums of the odd (or even-) $a_i$'s (resp., $b_i$'s) with indices that are strictly less than~$i$. 
Then the expression in \eqref{B2} is obtained from the expression in \eqref{B1} by swapping simultaneously~(1)~the $a$- and $b$-lists, (2) summation over odd-indexed $a_i$'s and summation over even-indexed $b_i$'s, and~(3)~summation over even-indexed $a_i$'s and summation over odd-indexed~$b_i$'s.

We are now ready to state the result connecting $\wt_1$- and $\wt_2$-tiling enumeration with $q$-enumeration of monotone stacks by volume.

\begin{lem}\label{ratio} Let $x,y,z,t$ be non-negative integers, and $\textbf{a}=\{a_i\}_{i=1}^{m}$ and $\textbf{b}=\{b_j\}_{j=1}^{n}$ sequences of non-negative integers. Then we have

\begin{equation}\label{ratio1}
\M_1(P_{x,y,z,t}(\textbf{a};\textbf{b}))=q^{A^{(1)}_{x,y,z,t}(\textbf{a};\textbf{b})}\sum_{\pi}q^{|\pi|},
\end{equation}
\begin{equation}\label{ratio2}
\M_2(P_{x,y,z,t}(\textbf{a};\textbf{b}))=q^{A^{(2)}_{x,y,z,t}(\textbf{a};\textbf{b})}\sum_{\pi}q^{|\pi|},
\end{equation}
\begin{equation}\label{ratio3}
\M_1(Q_{x,y,z,t}(\textbf{a};\textbf{b}))=q^{B^{(1)}_{x,y,z,t}(\textbf{a};\textbf{b})}\sum_{\pi}q^{|\pi|},
\end{equation}
and
\begin{equation}\label{ratio4}
\M_2(Q_{x,y,z,t}(\textbf{a};\textbf{b}))=q^{B^{(2)}_{x,y,z,t}(\textbf{a};\textbf{b})}\sum_{\pi}q^{|\pi|},
\end{equation}
where in each equation
the sum on the right-hand side is taken over all monotone stacks of unit cubes $\pi$ fitting in the augmented box
corresponding to the region on the left hand side.
\end{lem}
\begin{proof}
The proof of this lemma follows along the lines of the proofs of  Proposition 3.1 in \cite{Tri1} and Proposition 2.4  in \cite{Tri2}. We present the proofs of \eqref{ratio1} and \eqref{ratio2}; equations \eqref{ratio3} and \eqref{ratio4}  can be proved in an analogous manner.

Let $T$ be any lozenge tiling of $P:=P_{x,y,z,t}(\textbf{a};\textbf{b})$.
We will show that
\begin{equation}\label{ratio5}
\frac{\wt_1(T)}{\wt_0(T)}=q^{A_1},
\end{equation}
where $A_1=A^{(1)}_{x,y,z,t}(\textbf{a};\textbf{b})$.
Then (\ref{ratio1}) will follow by adding these equations over all lozenge tilings $T$ of $P$.

Recall that the lozenge tilings of $P$ can be viewed as monotone stacks of unit cubes fitting in an augmented box. The augmented box can be divided into cupboard-rooms, stair-rooms, and the house. There are $\lfloor\frac{m}{2}\rfloor +\lfloor\frac{n}{2}\rfloor$ cupboard-rooms and $\lceil\frac{m}{2}\rceil +\lceil\frac{n}{2}\rceil$ stair-rooms, and the house in the middle can be divided further into three rooms as in Figure \ref{subarraytiling}(b). This means that the augmented box corresponding to $P_{x,y,z,t}(\textbf{a};\textbf{b})$ can be partitioned into $m+n+3$ rooms, each being a parallelepiped.
Let $B_i$ be the $i$th parallelepiped in the partition, and let its dimensions be $x_i\times y_i\times z_i$, for $i=1,2,\dotsc,m+n+3$.
Under the projection in Figure \ref{3Dinterpretation}, $B_i$ projects onto a hexagon $H_{x_i,y_i,z_i}$ of side-lengths $x_i$, $y_i$, $z_i$, $x_i$, $y_i$, $z_i$ (in clockwise order, starting with the northwestern side). Denote by $t_i$ the distance between the bottom side of $H_{x_i,y_i,z_i}$ and the bottom side of $P_{x,y,z,t}(\textbf{a};\textbf{b})$.
One readily sees that $t_i=0$ if $B_i$ is a stair-room or a ground room of the house, $t_i=x+e_a$ for the central room of the house, and $t_i=x+y+t+e_a+e_b$ if $B_i$ is a cupboard-room or the upper room of the house.

Let $T$ be a lozenge tiling of $P$. Let $\pi=\pi_T$ be the corresponding monotone stack of unit cubes that fits in the augmented box corresponding to $P$. Partition $\pi$ into the $m+n+3$ sub-stacks $\pi_i$ fitting into the rooms $B_i$. Each sub-stack $\pi_i$ gives (by projection)  a lozenge tiling $T_i$ of the hexagon $H_{x_i,y_i,z_i}$ (see Figure \ref{subarraytiling}(a) for an example). We note that $T_i$ can generally \emph{not} be obtained by restricting $T$ to the hexagon $H_{x_i,y_i,z_i}$ (which may not even be possible, due to lozenges crossing the border of $H_{x_i,y_i,z_i}$).

\begin{figure}\centering
\setlength{\unitlength}{3947sp}%
\begingroup\makeatletter\ifx\SetFigFont\undefined%
\gdef\SetFigFont#1#2#3#4#5{%
  \reset@font\fontsize{#1}{#2pt}%
  \fontfamily{#3}\fontseries{#4}\fontshape{#5}%
  \selectfont}%
\fi\endgroup%
\resizebox{15cm}{!}{
\begin{picture}(0,0)%
\includegraphics{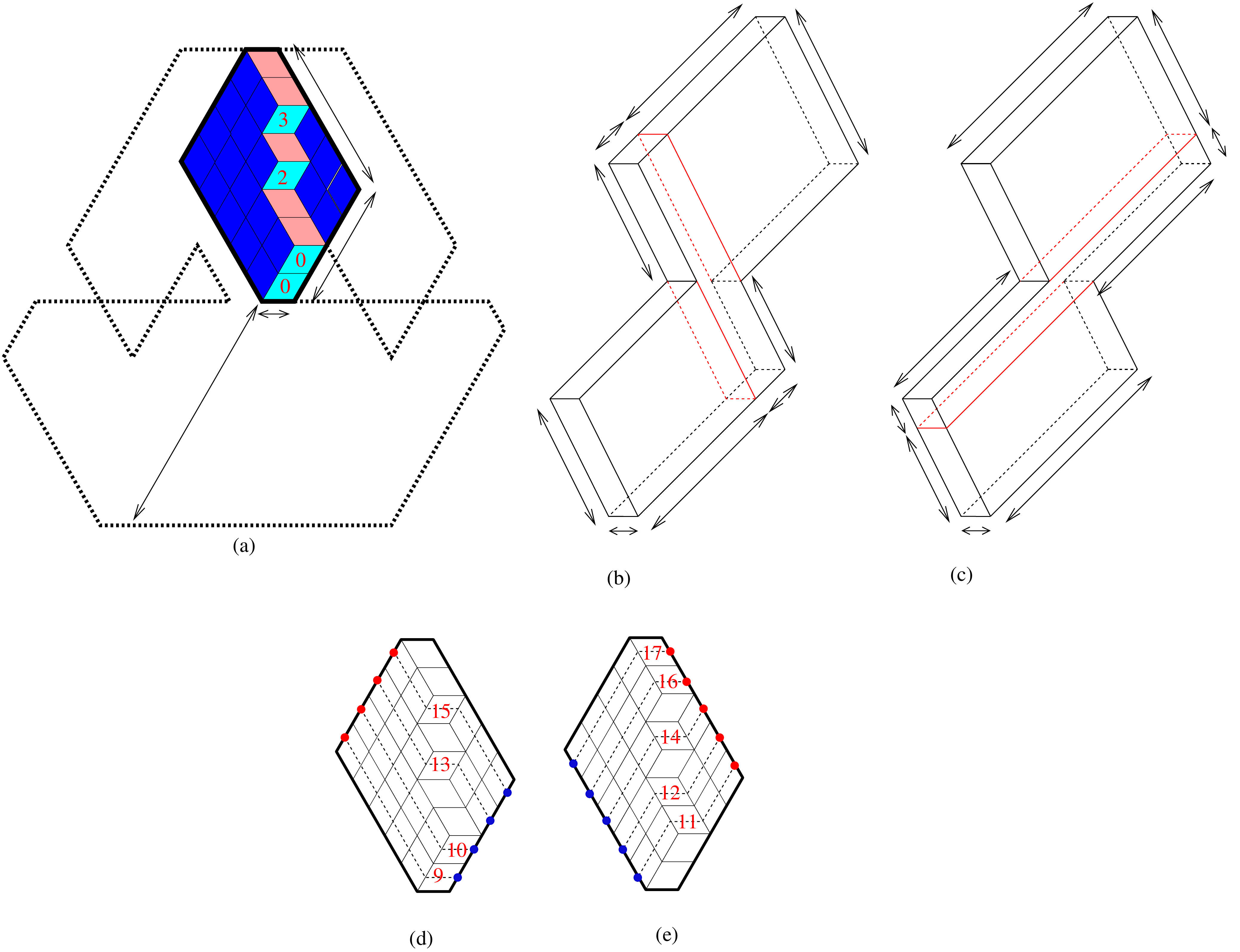}%
\end{picture}%
%
%

\begin{picture}(19922,15286)(2500,-15609)
\put(7981,-2056){\makebox(0,0)[lb]{\smash{{\SetFigFont{14}{16.8}{\rmdefault}{\mddefault}{\updefault}{$z_8=5$}%
}}}}
\put(5776,-7366){\makebox(0,0)[lb]{\smash{{\SetFigFont{14}{16.8}{\rmdefault}{\mddefault}{\updefault}{$t_8=\Delta=8$}%
}}}}
\put(8266,-4381){\makebox(0,0)[lb]{\smash{{\SetFigFont{14}{16.8}{\rmdefault}{\mddefault}{\updefault}{$x_8=4$}%
}}}}
\put(6706,-5731){\makebox(0,0)[lb]{\smash{{\SetFigFont{14}{16.8}{\rmdefault}{\mddefault}{\updefault}{$y_8=1$}%
}}}}
\put(12391,-9121){\makebox(0,0)[lb]{\smash{{\SetFigFont{14}{16.8}{\rmdefault}{\mddefault}{\updefault}{$z$}%
}}}}
\put(18031,-9211){\makebox(0,0)[lb]{\smash{{\SetFigFont{14}{16.8}{\rmdefault}{\mddefault}{\updefault}{$z$}%
}}}}
\put(15271,-6961){\makebox(0,0)[lb]{\smash{{\SetFigFont{14}{16.8}{\rmdefault}{\mddefault}{\updefault}{$y$}%
}}}}
\put(12084,-2371){\makebox(0,0)[lb]{\smash{{\SetFigFont{14}{16.8}{\rmdefault}{\mddefault}{\updefault}{$y$}%
}}}}
\put(13816,-8431){\rotatebox{45.0}{\makebox(0,0)[lb]{\smash{{\SetFigFont{14}{16.8}{\rmdefault}{\mddefault}{\updefault}{$x+e_a$}%
}}}}}
\put(19546,-8146){\rotatebox{45.0}{\makebox(0,0)[lb]{\smash{{\SetFigFont{14}{16.8}{\rmdefault}{\mddefault}{\updefault}{$x+y+e_a$}%
}}}}}
\put(11071,-7501){\rotatebox{295.0}{\makebox(0,0)[lb]{\smash{{\SetFigFont{14}{16.8}{\rmdefault}{\mddefault}{\updefault}{$y+t+e_b$}%
}}}}}
\put(15961,-1156){\rotatebox{295.0}{\makebox(0,0)[lb]{\smash{{\SetFigFont{14}{16.8}{\rmdefault}{\mddefault}{\updefault}{$y+o_a$}%
}}}}}
\put(13111,-1216){\rotatebox{45.0}{\makebox(0,0)[lb]{\smash{{\SetFigFont{14}{16.8}{\rmdefault}{\mddefault}{\updefault}{$o_b$}%
}}}}}
\put(14926,-5011){\rotatebox{295.0}{\makebox(0,0)[lb]{\smash{{\SetFigFont{14}{16.8}{\rmdefault}{\mddefault}{\updefault}{$t+e_b$}%
}}}}}
\put(12211,-3931){\rotatebox{295.0}{\makebox(0,0)[lb]{\smash{{\SetFigFont{14}{16.8}{\rmdefault}{\mddefault}{\updefault}{$o_a$}%
}}}}}
\put(21524,-1123){\rotatebox{295.0}{\makebox(0,0)[lb]{\smash{{\SetFigFont{14}{16.8}{\rmdefault}{\mddefault}{\updefault}{$o_a$}%
}}}}}
\put(16674,-7231){\makebox(0,0)[lb]{\smash{{\SetFigFont{14}{16.8}{\rmdefault}{\mddefault}{\updefault}{$y$}%
}}}}
\put(17372,-5701){\rotatebox{45.0}{\makebox(0,0)[lb]{\smash{{\SetFigFont{14}{16.8}{\rmdefault}{\mddefault}{\updefault}{$x+e_a$}%
}}}}}
\put(16955,-7893){\rotatebox{295.0}{\makebox(0,0)[lb]{\smash{{\SetFigFont{14}{16.8}{\rmdefault}{\mddefault}{\updefault}{$t+e_a$}%
}}}}}
\put(21128,-4661){\rotatebox{45.0}{\makebox(0,0)[lb]{\smash{{\SetFigFont{14}{16.8}{\rmdefault}{\mddefault}{\updefault}{$o_b$}%
}}}}}
\put(18196,-1741){\rotatebox{45.0}{\makebox(0,0)[lb]{\smash{{\SetFigFont{14}{16.8}{\rmdefault}{\mddefault}{\updefault}{$y+o_b$}%
}}}}}
\put(22254,-2446){\makebox(0,0)[lb]{\smash{{\SetFigFont{14}{16.8}{\rmdefault}{\mddefault}{\updefault}{$y$}%
}}}}
\end{picture}}
  \caption{(a) The sub-stack $\pi_8$ corresponding to the tiling in Figure \ref{tilingarray}. (b) and (c) Dividing the house into three rooms. (d) Encoding the tiling $T_8$ as a family of $x_8$ disjoint lozenge paths. (e) Encoding the tiling $T_8$ as a family of $z_8$ disjoint lozenge paths.  }\label{subarraytiling}
\end{figure}
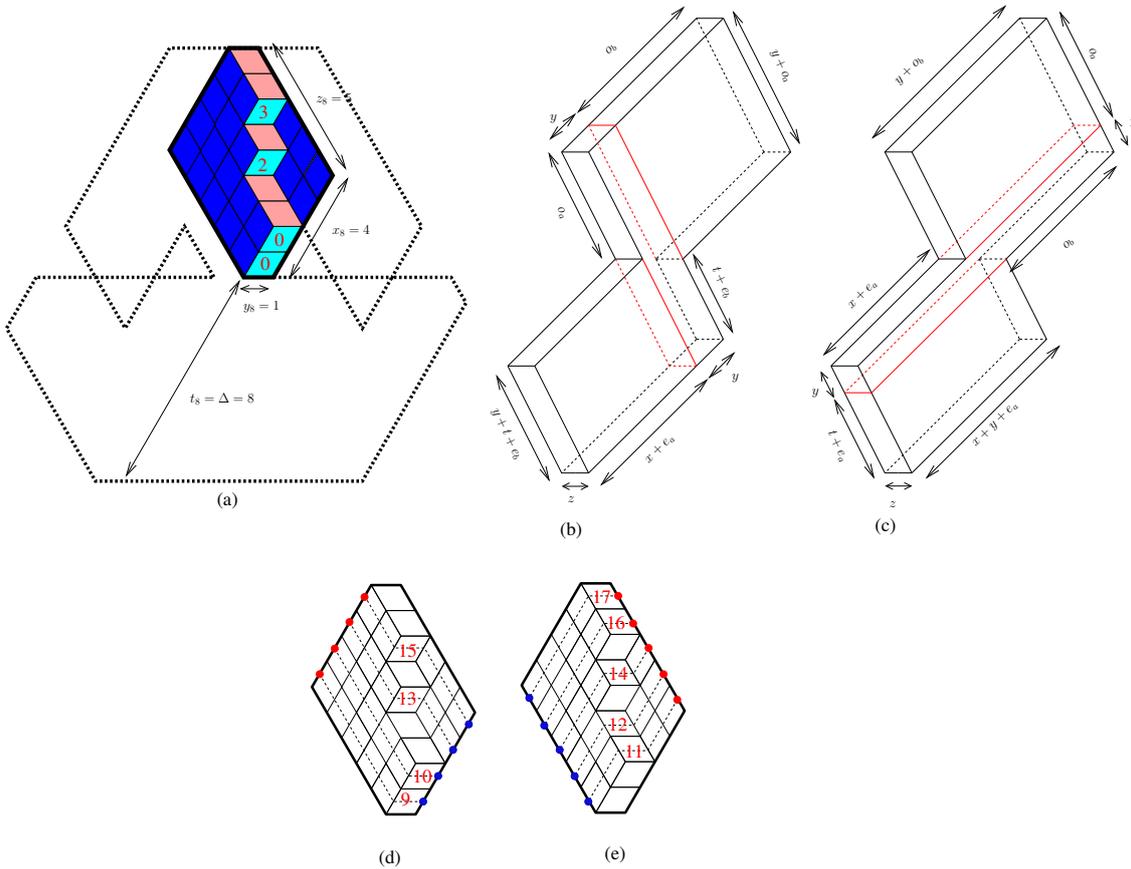

The weight $\wt_1$ on the lozenges of $T$ yields, by restriction, a weight on the lozenges of $T_i$. Under it, each right-lozenge is weighted by $q^{t_i+l}$, where $l$ is the distance between the top of the lozenge and the bottom of the hexagon $H_{x_i,y_i,z_i}$; all vertical and left-lozenges are weighted by $1$.
Encode the tiling $T_i$ as a family of $x_i$ non-intersecting paths of lozenges starting from the northwest side and ending at the southeast side of $H_{x_i,y_i,z_i}$ (see Figure \ref{subarraytiling} (d)). Dividing the weight of each right-lozenge on the $j$-th lozenge path (ordered from bottom to top) by $q^{j+t_i}$, for all $i=1,2,\dots,m+n+3$ and $j=1,2,\dots, x_i$, we get back the natural $q$-weight assignment $\wt_0$ on the whole tiling $T$ of $P$. It is easy to see that  each of our $x_i$ lozenge paths contains exactly $y_i$ right-lozenges (indeed, each such path intersects each of the $y_i$ analogous non-intersecting paths of lozenges connecting the top and bottom sides of $H_{x_i,y_i,z_i}$ in $T_i$  in precisely one --- necessarily right-leaning --- lozenge). Therefore,
\begin{equation}
\frac{\wt_1(T)}{\wt_0(T)}=\frac{\wt_1(T)}{q^{|\pi|}}=q^{\sum_{i=1}^{m+n+3}(x_iy_it_i+y_ix_i(x_i+1)/2)}.
\end{equation}
A straightforward calculation shows that, once the values of the $x_i$'s, $y_i$'s and $z_i$'s (which can be read from Figures \ref{3Dinterpretation} and \ref{highlowenvelope}) are plugged in, the sum above becomes precisely the expression given in \eqref{A1}, thus
proving \eqref{ratio5}.

We now turn to the enumeration under our second weight. Let the tiling $T$ be weighted by $\wt_2$. Consider again the monotone stack $\pi=\pi_T$ corresponding to $T$, but now split this stack into $m+n+3$ sub-stacks $\pi_i$ using a slightly different partitioning of $\pi$ into sub-stacks $\pi_i$: Namely, use the same $\lceil\frac{m}{2}\rceil +\lceil\frac{n}{2}\rceil$ stair-rooms and $\lfloor\frac{m}{2}\rfloor +\lfloor\frac{n}{2}\rfloor$ cupboard-rooms as before, but split the house into three rooms by extending the inner {\it horizontal} surfaces as in Figure \ref{subarraytiling}(c). Each sub-stack $\pi_i$ yields a tiling $T_i$ as before.

Encode this time each such tiling $T_i$ as a $z_i$-tuple of disjoint lozenge-paths starting from the northeast side and ending at the southwest side (see Figure \ref{subarraytiling}(e)). The weight assignment $\wt_2$ yields, by restriction, a weight assignment on each $T_i$. Dividing the weight of each left lozenge on the $j$-th lozenge path (ordered from bottom to top) by $q^{j+t_i}$ (with precisely the same definition of $t_i$ as before), for all $i=1,2,\dots,m+n+3$ and $j=1,2,\dots, z_i$, we get a new weight assignment $\wt_{*}$ on the whole tiling $T$ of $P$. It is easy to see that $\wt_{*}(T)=q^{|\pi^c|}=q^{\sum_{i=1}^{m+n+3}x_iy_iz_i-|\pi|}$, where $\pi^c$ is the complement of the stack $\pi$ in the augmented box ${B}$ corresponding to $P_{x,y,z,t}(\textbf{a};\textbf{b})$. Thus, we get
\begin{equation}
\frac{\wt_2(T)}{\wt_*(T)}=\frac{\wt_2(T)}{q^{\sum_{i=1}^{m+n+3}x_iy_iz_i-|\pi|}}=q^{\sum_{i=1}^{m+n+3}(z_iy_it_i+y_iz_i(z_i+1)/2))}.
\end{equation}
This implies that
\begin{equation}
\frac{\M_2(P_{x,y,z,t}(\textbf{a};\textbf{b}))}{\sum_{\pi}\left(\frac{1}{q}\right)^{|\pi|}}=q^{\Vol(B)+\sum_{i=1}^{m+n+3}(z_iy_it_i+y_iz_i(z_i+1)/2)}.
\end{equation}
%
However, we have
\begin{align}
q^{\Vol(B)}\sum_{\pi}\left(\frac{1}{q}\right)^{|\pi|}
&=\sum_{\pi} q^{\Vol(B)-|\pi|}\\
&=\sum_{\pi} q^{|\pi^c|}\\
&=\sum_{\pi} q^{|\pi|}.
\end{align}

It is not hard to verify that
\begin{equation}
\sum_{i=1}^{m+n+3}(z_iy_it_i+y_iz_i(z_i+1)/2)
=
A^{(2)}_{x,y,z,t}(\textbf{a};\textbf{b}),
\end{equation}
and then \eqref{ratio2} follows.
\end{proof}

\section{Hexagon with one intrusion}\label{onefern}

This section is devoted to the $q$-enumerations of lozenge tilings of a hexagon with a single, boundary-touching fern removed from it. We will use this result in our proof of Theorem \ref{main2} in the next section.

Let $x,y,z,t,a_1,\dotsc,a_n$ be non-negative integers, and set $\textbf{a}=\{a_i\}_{i=1}^{n}$. Consider a hexagon of side-lengths $x+y+o_a$, $z+e_a,t+o_a,x+y+e_a,z+o_a,t+e_a$ (clockwise, starting with the northwestern side). At $x$ units above the leftmost vertex of the hexagon, remove a fern of lobe sizes $a_1,a_2,\dots,a_n$ (from left to right), so that the triangular lobes of side-length $a_{2i-1}$ are pointing up, and the lobes of side-length $a_{2i}$ are pointing down, for all $i\geq1$ (Figure \ref{arrayhole} illustrates the case $x=3$, $y=2$, $z=3$, $t=2$, $a_1=a_2=\dotsc=a_6=2$). Denote the resulting region by $R_{x,y,z,t}(\mathbf{a})$.

We point out that, in order to avoid splitting the family of $R$-regions into two subfamilies, the parameter $t$ has a different significance than in the definition of the $P$- and $Q$-regions. Namely, the parameter $t$ in the $R$-regions corresponds to $t-y$ or $y-t$ in the definition of the $P$- and $Q$-regions, respectively. Thus, $R_{x,y,z,t}(\textbf{a})$ is the region $P_{x,y,z,t-y}(\textbf{a}; \emptyset)$ if $t\geq y$, and is $Q_{x,y,z,y-t}(\textbf{a}; \emptyset)$ if $t\leq y$.


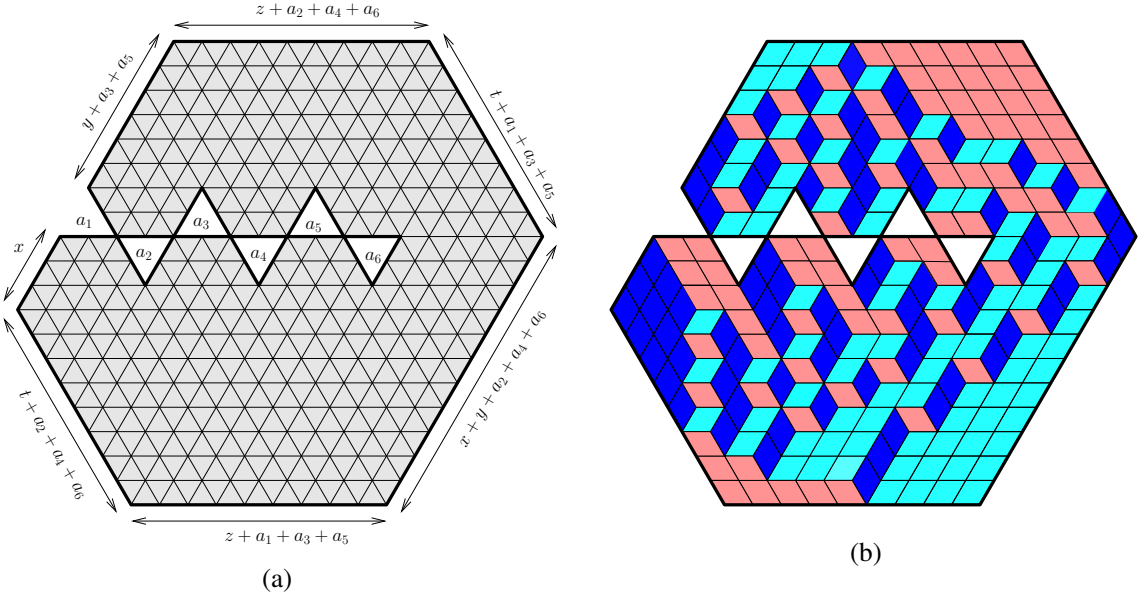
\begin{figure}\centering
\setlength{\unitlength}{3947sp}%
\begingroup\makeatletter\ifx\SetFigFont\undefined%
\gdef\SetFigFont#1#2#3#4#5{%
  \reset@font\fontsize{#1}{#2pt}%
  \fontfamily{#3}\fontseries{#4}\fontshape{#5}%
  \selectfont}%
\fi\endgroup%
\resizebox{15cm}{!}{
\begin{picture}(0,0)%
\includegraphics{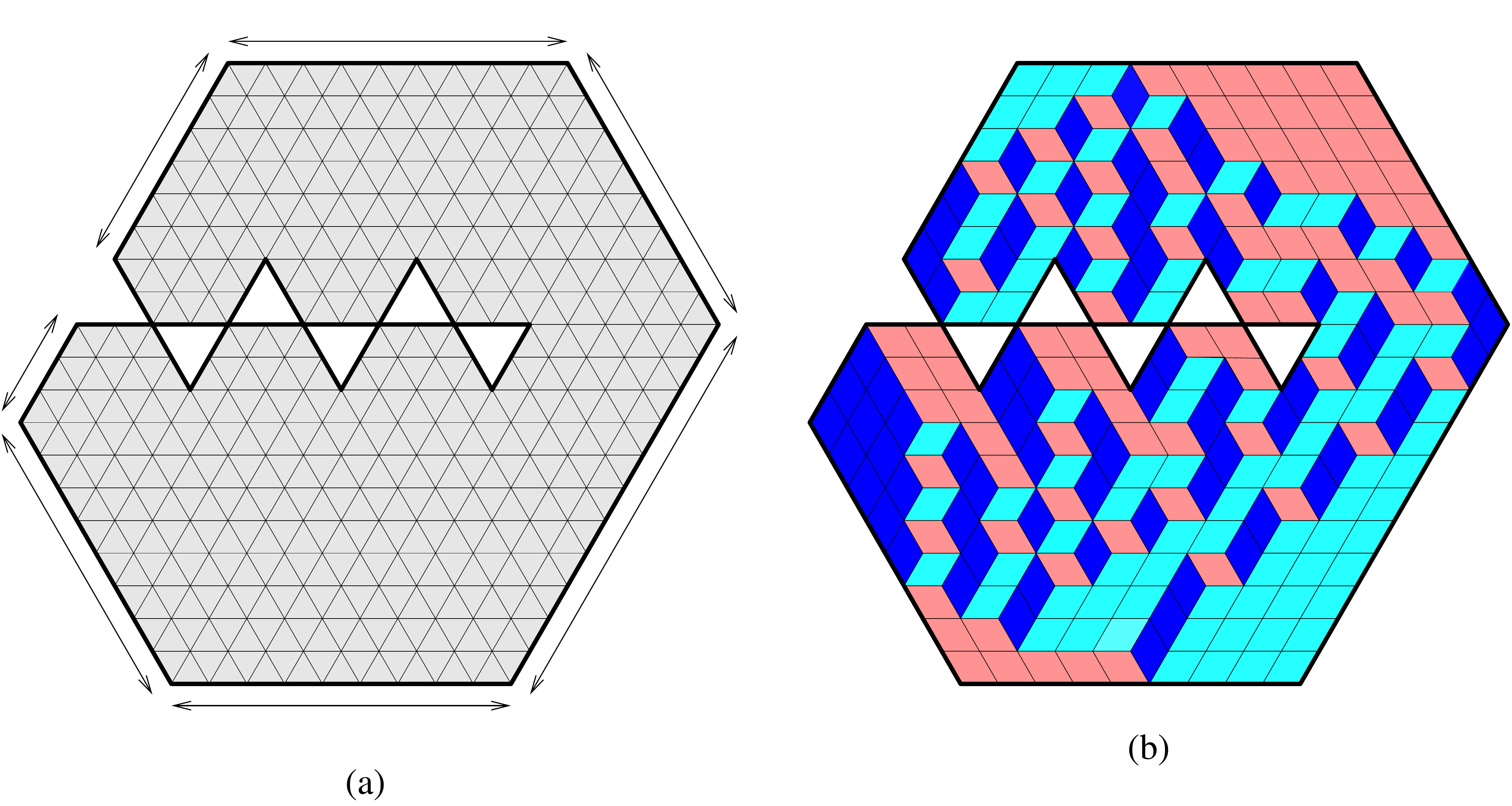}%
\end{picture}%

\begin{picture}(20815,11090)(284,-10441)
\put(8776,-7711){\rotatebox{60.0}{\makebox(0,0)[lb]{\smash{{\SetFigFont{20}{24.0}{\rmdefault}{\mddefault}{\itdefault}{$x+y+a_2+a_4+a_6$}%
}}}}}
\put(1861,-1846){\rotatebox{60.0}{\makebox(0,0)[lb]{\smash{{\SetFigFont{20}{24.0}{\rmdefault}{\mddefault}{\itdefault}{$y+a_3+a_5$}%
}}}}}
\put(1636,-3601){\makebox(0,0)[lb]{\smash{{\SetFigFont{20}{24.0}{\rmdefault}{\mddefault}{\itdefault}{$a_1$}%
}}}}
\put(2731,-4186){\makebox(0,0)[lb]{\smash{{\SetFigFont{20}{24.0}{\rmdefault}{\mddefault}{\itdefault}{$a_2$}%
}}}}
\put(3781,-3616){\makebox(0,0)[lb]{\smash{{\SetFigFont{20}{24.0}{\rmdefault}{\mddefault}{\itdefault}{$a_3$}%
}}}}
\put(4831,-4216){\makebox(0,0)[lb]{\smash{{\SetFigFont{20}{24.0}{\rmdefault}{\mddefault}{\itdefault}{$a_4$}%
}}}}
\put(5836,-3676){\makebox(0,0)[lb]{\smash{{\SetFigFont{20}{24.0}{\rmdefault}{\mddefault}{\itdefault}{$a_5$}%
}}}}
\put(6931,-4246){\makebox(0,0)[lb]{\smash{{\SetFigFont{20}{24.0}{\rmdefault}{\mddefault}{\itdefault}{$a_6$}%
}}}}
\put(496,-4096){\makebox(0,0)[lb]{\smash{{\SetFigFont{20}{24.0}{\rmdefault}{\mddefault}{\itdefault}{$x$}%
}}}}
\put(586,-6721){\rotatebox{300.0}{\makebox(0,0)[lb]{\smash{{\SetFigFont{20}{24.0}{\rmdefault}{\mddefault}{\itdefault}{$t+a_2+a_4+a_6$}%
}}}}}
\put(9206,-1186){\rotatebox{300.0}{\makebox(0,0)[lb]{\smash{{\SetFigFont{20}{24.0}{\rmdefault}{\mddefault}{\itdefault}{$t+a_1+a_3+a_5$}%
}}}}}
\put(4921,299){\makebox(0,0)[lb]{\smash{{\SetFigFont{20}{24.0}{\rmdefault}{\mddefault}{\itdefault}{$z+a_2+a_4+a_6$}%
}}}}
\put(4351,-9436){\makebox(0,0)[lb]{\smash{{\SetFigFont{20}{24.0}{\rmdefault}{\mddefault}{\itdefault}{$z+a_1+a_3+a_5$}%
}}}}
\end{picture}}
\caption{(a) A hexagon with a fern removed. (b) Viewing a tiling of an $R$-region as a stack of unit cubes.}
\label{arrayhole}
\end{figure}

\begin{thm}\label{qmain1} 
Let $x,y,z,t,a_1,\dotsc,a_n$ be non-negative integers, and set $\textbf{a}=(a_i)_{i=1}^{n}$.
Then if $y\leq t$, we have
\begin{align}\label{qmain1eq1a}
\M_1(R_{x,y,z,t}(\textbf{a}))&=q^{A^{(1)}_{x,y,z,t-y}(\textbf{a};\emptyset)-A^{(1)}_{t-y,y,z,x}(\emptyset;\textbf{a})}M_2(R_{x,y,z,t}(\textbf{a}))\notag\\
&=q^{A^{(1)}_{x,y,z,t-y}(\textbf{a};\emptyset)}\Phi^{q}_{x,y,z,t-y}(\textbf{a};\emptyset),
\end{align}
while if $y\geq t$, we have
\begin{align}\label{qmain1eq2a}
\M_1(R_{x,y,z,t}(\textbf{a}))&=q^{B^{(1)}_{x,t,z,y-t}(\textbf{a};\emptyset)-B^{(1)}_{y-t,t,z,x}(\emptyset;\textbf{a})}M_2(R_{x,y,z,t}(\textbf{a}))\notag\\
&=q^{B^{(1)}_{x,t,z,y-t}(\textbf{a};\emptyset)}\Psi^{q}_{x,t,z,y-t}(\textbf{a};\emptyset).
\end{align}

\end{thm}

\begin{proof}
The first equalities in \eqref{qmain1eq1a} and \eqref{qmain1eq2a} follow from \eqref{ratio1}--\eqref{ratio4}, together with \eqref{A1A2symm} and an equivalent formula for $B^{(1)}$ (see the paragraph after \eqref{B2}). It is therefore enough to consider the weight $\wt_1$, and show that if $y\leq t$ we have
\begin{align}\label{qmain1eq1}
\M_1(R_{x,y,z,t}(\textbf{a}))&=q^{A^{(1)}_{x,y,z,t-y}(\textbf{a};\emptyset)}\Phi^{q}_{x,y,z,t-y}(\textbf{a};\emptyset)\notag\\
&=q^{A^{(1)}_{x,y,z,t-y}(\textbf{a};\emptyset)}\frac{\Hf_q(y)\Hf_q(z)\Hf_q\left(a+x+y+z+t\right)} {\Hf_q(y+z)\Hf_q\left(a+x+y+t\right)}\frac{\Hf_q\left(a+x+t\right)}{\Hf_q\left(a+x+z+t\right)}
\notag\\
 &\times
 \frac{\Hf_q\left(a+x+y\right)\Hf_q\left(t\right)}{\Hf_q\left(a+x\right)\Hf_q\left(t-y\right)}
 \frac{\Hf_q\left(e_a+x+t-y\right)\Hf_q\left(o_a\right)}{\Hf_q\left(e_a+x+t\right)\Hf_q\left(o_a+y\right)}\notag\\
 &\times s_q(a_1,\dotsc,a_{2l-1},a_{2l}+y+z) s_q(x,a_1,\dotsc,a_{2l},y+z,t-y),
\end{align}
while if $y\geq t$ we have
\begin{align}\label{qmain1eq2}
\M_1(R_{x,y,z,t}(\textbf{a}))&=q^{B^{(1)}_{x,t,z,y-t}(\textbf{a};\emptyset)}\Psi^{q}_{x,t,z,y-t}(\textbf{a};\emptyset)\notag\\
&= q^{B^{(1)}_{x,t,z,y-t}(\textbf{a};\emptyset)} \frac{\Hf_q(t)\Hf_q(z)\Hf_q\left(a+x+y+z+t\right)}
 {\Hf_q(t+z)\Hf_q\left(a+x+y+t\right)}\frac{\Hf_q\left(a+x+y\right)}{\Hf_q\left(a+x+y+z\right)}
\notag\\
 &\times
 \frac{\Hf_q\left(a+x+t\right)\Hf_q\left(y\right)}{\Hf_q\left(a+x\right)\Hf_q\left(y-t\right)}
  \frac{\Hf_q\left(e_a+x\right)\Hf_q\left(o_a)+y-t\right)}{\Hf_q\left(e_a+x+t\right)\Hf_q\left(o_a+y\right)}\notag\\
 &\times s_q(a_1,\dotsc,a_{2l-1},a_{2l}+t+z,y-t)s_q(x,a_1,\dotsc,a_{2l},t+z),
\end{align}
where $s_q$ is given by equation \eqref{semieqq}.

If $n=0$, we have no removed triangle, and the theorem follows from Lemma \ref{lemsixfive}.

Suppose therefore that $n\geq1$. We prove (\ref{qmain1eq1}) and (\ref{qmain1eq2}) by induction on $y+z+t$. Our base cases are the situations when at least one of $y$, $z$, and $t$ is equal to $0$.

If $y=0$ (so $y\leq t$), the portion of $R_{x,y,z,t}(\mathbf{a})$ that is above the horizontal line containing the fern axis is readily seen to be balanced (i.e., to have the same number of up- and down-pointing unit triangles). Therefore, by the region-splitting Lemma \ref{GS}, the weighted count of the tilings of $R_{x,y,z,t}(\mathbf{a})$ is equal to the product of the weighted counts of the resulting two semihexagons with dents illustrated in Figure~\ref{arraybasecase2}(a) (note that these have several forced left lozenges, each of which have weight $1$).

There are two cases to distinguish. If $n=2l$, then the upper subregion is the dented semihexagon $S(a_1,a_2,\dotsc,a_{2l-1})$, and the lower region is $S(x,a_1,a_2,\dotsc,a_{2l},z,t)$ reflected across its base (see Figure~\ref{arraybasecase2}(a)). However, these dented semihexagons are \emph{not} weighted by $\wt_1$; so we need to modify the weights of lozenges to bring back the weight assignment $\wt_1$. For the upper subregion, dividing the weight of each right lozenge by $q^{x+t+e_a}$ one obtains the weight assignment $\wt_1$. For the lower dented semihexagon, divide the weight of each right lozenge by $q^{x+t+e_a+1}$, reflect across its top side, and then reflect the resulting region across a vertical line to get the region $S(t,z,a_{2l},a_{2l-1},\dotsc,a_{1},x)$  weighted by $\wt_1$ with $q$ replaced by $q^{-1}$ (the first reflection turns the right-leaning lozenge into a left-leaning one, but the second turns it back into a right-leaning one). Taking into account all such contributions to the exponent of $q$, we obtain
\begin{equation}
\M_1(R_{x,0,z,t}(\text{a}))=q^{D}\M_1(S_q(a_1,a_2,\dotsc,a_{2l-1}))\M_1(S_{q^{-1}}(t,z,a_{2l},a_{2l-1},\dotsc,a_{1},x)),
\end{equation}
where

\begin{align}
D=(x+t+e_a)&\sum_{i=1}^{l-1}a_{2l-2i}\sum_{j=1}^{i}a_{2l-2j+1}\notag\\
&+(x+t+e_a+1)\left(z\left(x+e_a\right)+\sum_{i=1}^{l}a_{2i-1}\left(x+\sum_{j=1}^{i-1}a_{2j}\right)\right)
\end{align}
and where we use the subscripts $q$ and $q^{-1}$ in the semihexagons to emphasize their weight assignments.
Express $s_{q^{-1}}(t,z,a_{2l},a_{2l-1},\dotsc,a_{1},x)$ in terms of  $s_{q}(t,z,a_{2l},a_{2l-1},\dotsc,a_{1},x)$ by using the simple fact $[n]_{q^{-1}}=[n]_q/q^{n-1}$, and use Lemma \ref{semi} to obtain (\ref{qmain1eq1}). For the remaining case when  $n=2l-1$, the upper region is $S(a_1,a_2,\dotsc,a_{2l-1})$ and the lower region is $S(x,a_1,a_2,\dotsc,a_{2l-2},a_{2l-1}+z,t)$ reflected across the base (this is illustrated in Figure \ref{arraybasecase2}(b)). Equation (\ref{qmain1eq1}) follows then from Lemma \ref{semi} in the same fashion.

If $t=0$ (so $y\geq t$), the region-splitting Lemma \ref{GS} allows us again to separate our region into two weighted dented semihexagons. The one on the bottom contains now some forced right lozenges, and we have to keep track of their weights, as they contribute non-trivial multiplicative factors to the weighted tiling count.
If $n=2l$, the upper region is $S(a_1,a_2,\dots,a_{2l}+z,y)$, the lower one is $S(x,a_1,a_2,\dotsc,a_{2l})$ reflected across the base (see Figure \ref{arraybasecase2}(e) for an illustration), and the product of the weights of the forced lozenges is $q^{z\binom{x+e_a+1}{2}}$. If $n=2l-1$, the upper part is $S(a_1,a_2,\dots,a_{2l-1},z,y)$, the lower part is still a reflected version of $S(x,a_1,a_2,\dotsc,a_{2l-1})$, and the product of the weights of the forced lozenges is $q^{(z+a_{2l-1})\binom{x+e_a+1}{2}}$  (see Figure \ref{arraybasecase2}(f)). Similarly to the case when $y=0$, the weightings on these dented semihexagons can be turned into $\wt_1$ at the expense of a multiplicative factor equal to a power of $q$, and (\ref{qmain1eq2}) follows from Lemma~\ref{semi}.

If $z=0$, we can apply the region-splitting Lemma \ref{GS} one more time, obtaining two weighted dented semihexagons and several forced vertical lozenges (whose weights are all $1$, as none of them is right-leaning). There are again two subcases, depending on the parity of $n$. If $n=2l$, then the upper part is $S(a_1,a_2,\dots,a_{2l},y)$ and the lower part is a reflected version of $S(x,a_1,a_2,\dotsc,a_{2l-1},a_{2l}+t)$  (see Figure \ref{arraybasecase2}(c)). If $n=2l-1$, the upper part is $S(a_1,a_2,\dots,a_{2l-2},a_{2l-1}+t)$ and the lower part is a reflected $S(x,a_1,a_2,\dotsc,a_{2l-1},y)$ (see Figure \ref{arraybasecase2}(d)). As in the above two base cases, we can change the weights of the lozenges in these dented semihexagons to make them to be weighted by $\wt_1$,  and  then (\ref{qmain1eq1})  and (\ref{qmain1eq2}) follow again from Lemma \ref{semi}.

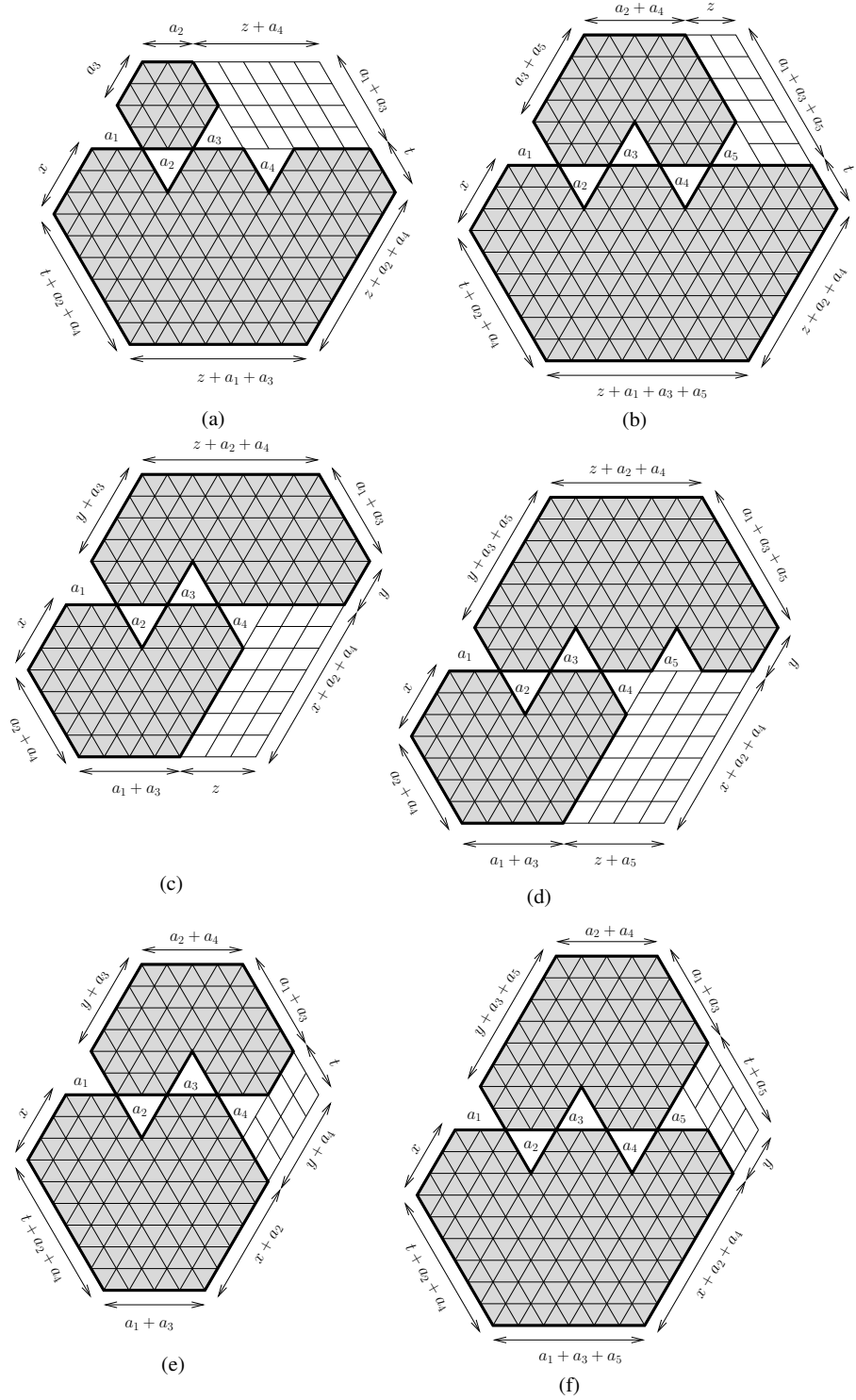
\begin{figure}\centering
\setlength{\unitlength}{3947sp}%
\begingroup\makeatletter\ifx\SetFigFont\undefined%
\gdef\SetFigFont#1#2#3#4#5{%
  \reset@font\fontsize{#1}{#2pt}%
  \fontfamily{#3}\fontseries{#4}\fontshape{#5}%
  \selectfont}%
\fi\endgroup%
\resizebox{12cm}{!}{
\begin{picture}(0,0)%
\includegraphics{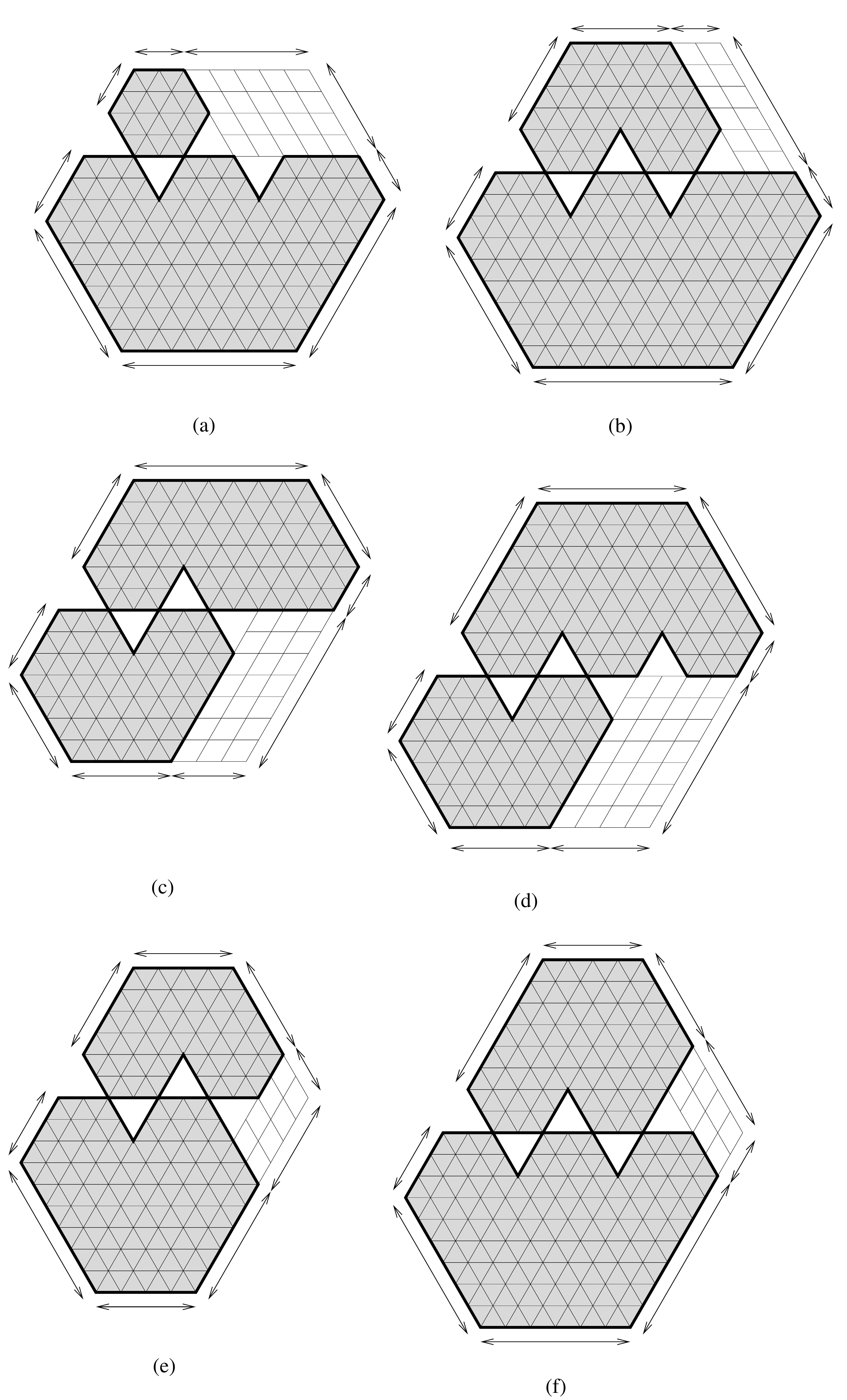}%
\end{picture}%
%
%

\begin{picture}(17611,29120)(289,-28462)
\put(1050,-3108){\rotatebox{60.0}{\makebox(0,0)[lb]{\smash{{\SetFigFont{20}{24.0}{\rmdefault}{\mddefault}{\itdefault}{$x$}%
}}}}}
\put(2040,-1068){\rotatebox{60.0}{\makebox(0,0)[lb]{\smash{{\SetFigFont{20}{24.0}{\rmdefault}{\mddefault}{\itdefault}{$a_3$}%
}}}}}
\put(3628,-170){\makebox(0,0)[lb]{\smash{{\SetFigFont{20}{24.0}{\rmdefault}{\mddefault}{\itdefault}{$a_2$}%
}}}}
\put(5101,-143){\makebox(0,0)[lb]{\smash{{\SetFigFont{20}{24.0}{\rmdefault}{\mddefault}{\itdefault}{$z+a_4$}%
}}}}
\put(970,-5135){\rotatebox{300.0}{\makebox(0,0)[lb]{\smash{{\SetFigFont{20}{24.0}{\rmdefault}{\mddefault}{\itdefault}{$t+a_2+a_4$}%
}}}}}
\put(4200,-7413){\makebox(0,0)[lb]{\smash{{\SetFigFont{20}{24.0}{\rmdefault}{\mddefault}{\itdefault}{$z+a_1+a_3$}%
}}}}
\put(7755,-5703){\rotatebox{60.0}{\makebox(0,0)[lb]{\smash{{\SetFigFont{20}{24.0}{\rmdefault}{\mddefault}{\itdefault}{$z+a_2+a_4$}%
}}}}}
\put(654,-22687){\rotatebox{60.0}{\makebox(0,0)[lb]{\smash{{\SetFigFont{20}{24.0}{\rmdefault}{\mddefault}{\itdefault}{$x$}%
}}}}}
\put(1931,-20411){\rotatebox{60.0}{\makebox(0,0)[lb]{\smash{{\SetFigFont{20}{24.0}{\rmdefault}{\mddefault}{\itdefault}{$y+a_3$}%
}}}}}
\put(3641,-18991){\makebox(0,0)[lb]{\smash{{\SetFigFont{20}{24.0}{\rmdefault}{\mddefault}{\itdefault}{$a_2+a_4$}%
}}}}
\put(5891,-19797){\rotatebox{300.0}{\makebox(0,0)[lb]{\smash{{\SetFigFont{20}{24.0}{\rmdefault}{\mddefault}{\itdefault}{$a_1+a_3$}%
}}}}}
\put(6911,-21506){\rotatebox{300.0}{\makebox(0,0)[lb]{\smash{{\SetFigFont{20}{24.0}{\rmdefault}{\mddefault}{\itdefault}{$t$}%
}}}}}
\put(573,-24733){\rotatebox{300.0}{\makebox(0,0)[lb]{\smash{{\SetFigFont{20}{24.0}{\rmdefault}{\mddefault}{\itdefault}{$t+a_2+a_4$}%
}}}}}
\put(6566,-23711){\rotatebox{60.0}{\makebox(0,0)[lb]{\smash{{\SetFigFont{20}{24.0}{\rmdefault}{\mddefault}{\itdefault}{$y+a_4$}%
}}}}}
\put(5501,-25691){\rotatebox{60.0}{\makebox(0,0)[lb]{\smash{{\SetFigFont{20}{24.0}{\rmdefault}{\mddefault}{\itdefault}{$x+a_2$}%
}}}}}
\put(2716,-26997){\makebox(0,0)[lb]{\smash{{\SetFigFont{20}{24.0}{\rmdefault}{\mddefault}{\itdefault}{$a_1+a_3$}%
}}}}
\put(619,-12529){\rotatebox{60.0}{\makebox(0,0)[lb]{\smash{{\SetFigFont{20}{24.0}{\rmdefault}{\mddefault}{\itdefault}{$x$}%
}}}}}
\put(1834,-10309){\rotatebox{60.0}{\makebox(0,0)[lb]{\smash{{\SetFigFont{20}{24.0}{\rmdefault}{\mddefault}{\itdefault}{$y+a_3$}%
}}}}}
\put(4111,-8796){\makebox(0,0)[lb]{\smash{{\SetFigFont{20}{24.0}{\rmdefault}{\mddefault}{\itdefault}{$z+a_2+a_4$}%
}}}}
\put(7456,-9516){\rotatebox{300.0}{\makebox(0,0)[lb]{\smash{{\SetFigFont{20}{24.0}{\rmdefault}{\mddefault}{\itdefault}{$a_1+a_3$}%
}}}}}
\put(348,-14419){\rotatebox{300.0}{\makebox(0,0)[lb]{\smash{{\SetFigFont{20}{24.0}{\rmdefault}{\mddefault}{\itdefault}{$a_2+a_4$}%
}}}}}
\put(2451,-15928){\makebox(0,0)[lb]{\smash{{\SetFigFont{20}{24.0}{\rmdefault}{\mddefault}{\itdefault}{$a_1+a_3$}%
}}}}
\put(4501,-15921){\makebox(0,0)[lb]{\smash{{\SetFigFont{20}{24.0}{\rmdefault}{\mddefault}{\itdefault}{$z$}%
}}}}
\put(6665,-14278){\rotatebox{60.0}{\makebox(0,0)[lb]{\smash{{\SetFigFont{20}{24.0}{\rmdefault}{\mddefault}{\itdefault}{$x+a_2+a_4$}%
}}}}}
\put(8056,-11905){\rotatebox{60.0}{\makebox(0,0)[lb]{\smash{{\SetFigFont{20}{24.0}{\rmdefault}{\mddefault}{\itdefault}{$y$}%
}}}}}
\put(2296,-2386){\makebox(0,0)[lb]{\smash{{\SetFigFont{20}{24.0}{\rmdefault}{\mddefault}{\itdefault}{$a_1$}%
}}}}
\put(3436,-2941){\makebox(0,0)[lb]{\smash{{\SetFigFont{20}{24.0}{\rmdefault}{\mddefault}{\itdefault}{$a_2$}%
}}}}
\put(4411,-2446){\makebox(0,0)[lb]{\smash{{\SetFigFont{20}{24.0}{\rmdefault}{\mddefault}{\itdefault}{$a_3$}%
}}}}
\put(5505,-3003){\makebox(0,0)[lb]{\smash{{\SetFigFont{20}{24.0}{\rmdefault}{\mddefault}{\itdefault}{$a_4$}%
}}}}
\put(1666,-21972){\makebox(0,0)[lb]{\smash{{\SetFigFont{20}{24.0}{\rmdefault}{\mddefault}{\itdefault}{$a_1$}%
}}}}
\put(2876,-22571){\makebox(0,0)[lb]{\smash{{\SetFigFont{20}{24.0}{\rmdefault}{\mddefault}{\itdefault}{$a_2$}%
}}}}
\put(3926,-22031){\makebox(0,0)[lb]{\smash{{\SetFigFont{20}{24.0}{\rmdefault}{\mddefault}{\itdefault}{$a_3$}%
}}}}
\put(4946,-22601){\makebox(0,0)[lb]{\smash{{\SetFigFont{20}{24.0}{\rmdefault}{\mddefault}{\itdefault}{$a_4$}%
}}}}
\put(1609,-11764){\makebox(0,0)[lb]{\smash{{\SetFigFont{20}{24.0}{\rmdefault}{\mddefault}{\itdefault}{$a_1$}%
}}}}
\put(2854,-12454){\makebox(0,0)[lb]{\smash{{\SetFigFont{20}{24.0}{\rmdefault}{\mddefault}{\itdefault}{$a_2$}%
}}}}
\put(3874,-11854){\makebox(0,0)[lb]{\smash{{\SetFigFont{20}{24.0}{\rmdefault}{\mddefault}{\itdefault}{$a_3$}%
}}}}
\put(4939,-12454){\makebox(0,0)[lb]{\smash{{\SetFigFont{20}{24.0}{\rmdefault}{\mddefault}{\itdefault}{$a_4$}%
}}}}
\put(9714,-3457){\rotatebox{60.0}{\makebox(0,0)[lb]{\smash{{\SetFigFont{20}{24.0}{\rmdefault}{\mddefault}{\itdefault}{$x$}%
}}}}}
\put(10846,-1291){\rotatebox{60.0}{\makebox(0,0)[lb]{\smash{{\SetFigFont{20}{24.0}{\rmdefault}{\mddefault}{\itdefault}{$a_3+a_5$}%
}}}}}
\put(12792,238){\makebox(0,0)[lb]{\smash{{\SetFigFont{20}{24.0}{\rmdefault}{\mddefault}{\itdefault}{$a_2+a_4$}%
}}}}
\put(14674,280){\makebox(0,0)[lb]{\smash{{\SetFigFont{20}{24.0}{\rmdefault}{\mddefault}{\itdefault}{$z$}%
}}}}
\put(16146,-743){\rotatebox{300.0}{\makebox(0,0)[lb]{\smash{{\SetFigFont{20}{24.0}{\rmdefault}{\mddefault}{\itdefault}{$a_1+a_3+a_5$}%
}}}}}
\put(17479,-3001){\rotatebox{300.0}{\makebox(0,0)[lb]{\smash{{\SetFigFont{20}{24.0}{\rmdefault}{\mddefault}{\itdefault}{$t$}%
}}}}}
\put(16740,-6462){\rotatebox{60.0}{\makebox(0,0)[lb]{\smash{{\SetFigFont{20}{24.0}{\rmdefault}{\mddefault}{\itdefault}{$z+a_2+a_4$}%
}}}}}
\put(9569,-5390){\rotatebox{300.0}{\makebox(0,0)[lb]{\smash{{\SetFigFont{20}{24.0}{\rmdefault}{\mddefault}{\itdefault}{$t+a_2+a_4$}%
}}}}}
\put(12410,-7697){\makebox(0,0)[lb]{\smash{{\SetFigFont{20}{24.0}{\rmdefault}{\mddefault}{\itdefault}{$z+a_1+a_3+a_5$}%
}}}}
\put(8721,-23423){\rotatebox{60.0}{\makebox(0,0)[lb]{\smash{{\SetFigFont{20}{24.0}{\rmdefault}{\mddefault}{\itdefault}{$x$}%
}}}}}
\put(10027,-20968){\rotatebox{60.0}{\makebox(0,0)[lb]{\smash{{\SetFigFont{20}{24.0}{\rmdefault}{\mddefault}{\itdefault}{$y+a_3+a_5$}%
}}}}}
\put(12181,-18868){\makebox(0,0)[lb]{\smash{{\SetFigFont{20}{24.0}{\rmdefault}{\mddefault}{\itdefault}{$a_2+a_4$}%
}}}}
\put(14363,-19613){\rotatebox{300.0}{\makebox(0,0)[lb]{\smash{{\SetFigFont{20}{24.0}{\rmdefault}{\mddefault}{\itdefault}{$a_1+a_3$}%
}}}}}
\put(15427,-21472){\rotatebox{300.0}{\makebox(0,0)[lb]{\smash{{\SetFigFont{20}{24.0}{\rmdefault}{\mddefault}{\itdefault}{$t+a_5$}%
}}}}}
\put(7528,-956){\rotatebox{300.0}{\makebox(0,0)[lb]{\smash{{\SetFigFont{20}{24.0}{\rmdefault}{\mddefault}{\itdefault}{$a_1+a_3$}%
}}}}}
\put(8452,-2633){\rotatebox{300.0}{\makebox(0,0)[lb]{\smash{{\SetFigFont{20}{24.0}{\rmdefault}{\mddefault}{\itdefault}{$t$}%
}}}}}
\put(15945,-23708){\rotatebox{60.0}{\makebox(0,0)[lb]{\smash{{\SetFigFont{20}{24.0}{\rmdefault}{\mddefault}{\itdefault}{$y$}%
}}}}}
\put(14581,-26408){\rotatebox{60.0}{\makebox(0,0)[lb]{\smash{{\SetFigFont{20}{24.0}{\rmdefault}{\mddefault}{\itdefault}{$x+a_2+a_4$}%
}}}}}
\put(8481,-25342){\rotatebox{300.0}{\makebox(0,0)[lb]{\smash{{\SetFigFont{20}{24.0}{\rmdefault}{\mddefault}{\itdefault}{$t+a_2+a_4$}%
}}}}}
\put(11227,-27677){\makebox(0,0)[lb]{\smash{{\SetFigFont{20}{24.0}{\rmdefault}{\mddefault}{\itdefault}{$a_1+a_3+a_5$}%
}}}}
\put(16473,-13351){\rotatebox{60.0}{\makebox(0,0)[lb]{\smash{{\SetFigFont{20}{24.0}{\rmdefault}{\mddefault}{\itdefault}{$y$}%
}}}}}
\put(15061,-15921){\rotatebox{60.0}{\makebox(0,0)[lb]{\smash{{\SetFigFont{20}{24.0}{\rmdefault}{\mddefault}{\itdefault}{$x+a_2+a_4$}%
}}}}}
\put(12257,-9301){\makebox(0,0)[lb]{\smash{{\SetFigFont{20}{24.0}{\rmdefault}{\mddefault}{\itdefault}{$z+a_2+a_4$}%
}}}}
\put(8551,-13869){\rotatebox{60.0}{\makebox(0,0)[lb]{\smash{{\SetFigFont{20}{24.0}{\rmdefault}{\mddefault}{\itdefault}{$x$}%
}}}}}
\put(8165,-15514){\rotatebox{300.0}{\makebox(0,0)[lb]{\smash{{\SetFigFont{20}{24.0}{\rmdefault}{\mddefault}{\itdefault}{$a_2+a_4$}%
}}}}}
\put(10224,-17408){\makebox(0,0)[lb]{\smash{{\SetFigFont{20}{24.0}{\rmdefault}{\mddefault}{\itdefault}{$a_1+a_3$}%
}}}}
\put(12359,-17401){\makebox(0,0)[lb]{\smash{{\SetFigFont{20}{24.0}{\rmdefault}{\mddefault}{\itdefault}{$z+a_5$}%
}}}}
\put(15405,-10071){\rotatebox{300.0}{\makebox(0,0)[lb]{\smash{{\SetFigFont{20}{24.0}{\rmdefault}{\mddefault}{\itdefault}{$a_1+a_3+a_5$}%
}}}}}
\put(9872,-11556){\rotatebox{60.0}{\makebox(0,0)[lb]{\smash{{\SetFigFont{20}{24.0}{\rmdefault}{\mddefault}{\itdefault}{$y+a_3+a_5$}%
}}}}}
\put(9557,-13133){\makebox(0,0)[lb]{\smash{{\SetFigFont{20}{24.0}{\rmdefault}{\mddefault}{\itdefault}{$a_1$}%
}}}}
\put(10743,-13828){\makebox(0,0)[lb]{\smash{{\SetFigFont{20}{24.0}{\rmdefault}{\mddefault}{\itdefault}{$a_2$}%
}}}}
\put(11738,-13201){\makebox(0,0)[lb]{\smash{{\SetFigFont{20}{24.0}{\rmdefault}{\mddefault}{\itdefault}{$a_3$}%
}}}}
\put(12857,-13828){\makebox(0,0)[lb]{\smash{{\SetFigFont{20}{24.0}{\rmdefault}{\mddefault}{\itdefault}{$a_4$}%
}}}}
\put(13801,-13251){\makebox(0,0)[lb]{\smash{{\SetFigFont{20}{24.0}{\rmdefault}{\mddefault}{\itdefault}{$a_5$}%
}}}}
\put(13955,-22708){\makebox(0,0)[lb]{\smash{{\SetFigFont{20}{24.0}{\rmdefault}{\mddefault}{\itdefault}{$a_5$}%
}}}}
\put(15027,-2763){\makebox(0,0)[lb]{\smash{{\SetFigFont{20}{24.0}{\rmdefault}{\mddefault}{\itdefault}{$a_5$}%
}}}}
\put(10799,-2725){\makebox(0,0)[lb]{\smash{{\SetFigFont{20}{24.0}{\rmdefault}{\mddefault}{\itdefault}{$a_1$}%
}}}}
\put(9752,-22671){\makebox(0,0)[lb]{\smash{{\SetFigFont{20}{24.0}{\rmdefault}{\mddefault}{\itdefault}{$a_1$}%
}}}}
\put(10917,-23291){\makebox(0,0)[lb]{\smash{{\SetFigFont{20}{24.0}{\rmdefault}{\mddefault}{\itdefault}{$a_2$}%
}}}}
\put(11927,-3382){\makebox(0,0)[lb]{\smash{{\SetFigFont{20}{24.0}{\rmdefault}{\mddefault}{\itdefault}{$a_2$}%
}}}}
\put(12969,-2738){\makebox(0,0)[lb]{\smash{{\SetFigFont{20}{24.0}{\rmdefault}{\mddefault}{\itdefault}{$a_3$}%
}}}}
\put(11847,-22721){\makebox(0,0)[lb]{\smash{{\SetFigFont{20}{24.0}{\rmdefault}{\mddefault}{\itdefault}{$a_3$}%
}}}}
\put(12951,-23316){\makebox(0,0)[lb]{\smash{{\SetFigFont{20}{24.0}{\rmdefault}{\mddefault}{\itdefault}{$a_4$}%
}}}}
\put(14022,-3321){\makebox(0,0)[lb]{\smash{{\SetFigFont{20}{24.0}{\rmdefault}{\mddefault}{\itdefault}{$a_4$}%
}}}}
\end{picture}%
}
\caption{Three base cases in the proof of Theorem \ref{qmain1}: (a) and (b) $y=0$, (c) and (d) $t=0$, and (e) and (f) $z=0$.}
\label{arraybasecase2}
\end{figure}

\medskip

For the induction step, assume that $y,z,t\geq 1$ and that the theorem holds for any region $R_{x,y,z,t}(\textbf{a})$ in which the sum of the  $y$-, $z$- and $t$-parameters is strictly less than $y+z+t$. We now apply Kuo's Theorem \ref{kuothm} to the {dual graph} $G$ of $R_{x,y,z,t}(\textbf{a})$ weighted by $\wt_1$ (i.e. $G$ is the graph whose vertices are unit triangles in $R$ and whose edges connect precisely those pairs of unit triangles which form a lozenge, inheriting the weight of that lozenge).

Each vertex of $G$ corresponds to a unit triangle of $R$. We pick the four vertices $u,v,w,s$ as indicated in Figure \ref{Kuoarray}(b). Note that the lowest shaded unit triangle corresponds to $u$, and $v,w,s$ correspond to the next shaded unit triangles as we move counter-clockwise from the lowest one.

\begin{figure}\centering
\setlength{\unitlength}{3947sp}%
\begingroup\makeatletter\ifx\SetFigFont\undefined%
\gdef\SetFigFont#1#2#3#4#5{%
  \reset@font\fontsize{#1}{#2pt}%
  \fontfamily{#3}\fontseries{#4}\fontshape{#5}%
  \selectfont}%
\fi\endgroup%
\resizebox{12cm}{!}{
\begin{picture}(0,0)%
\includegraphics{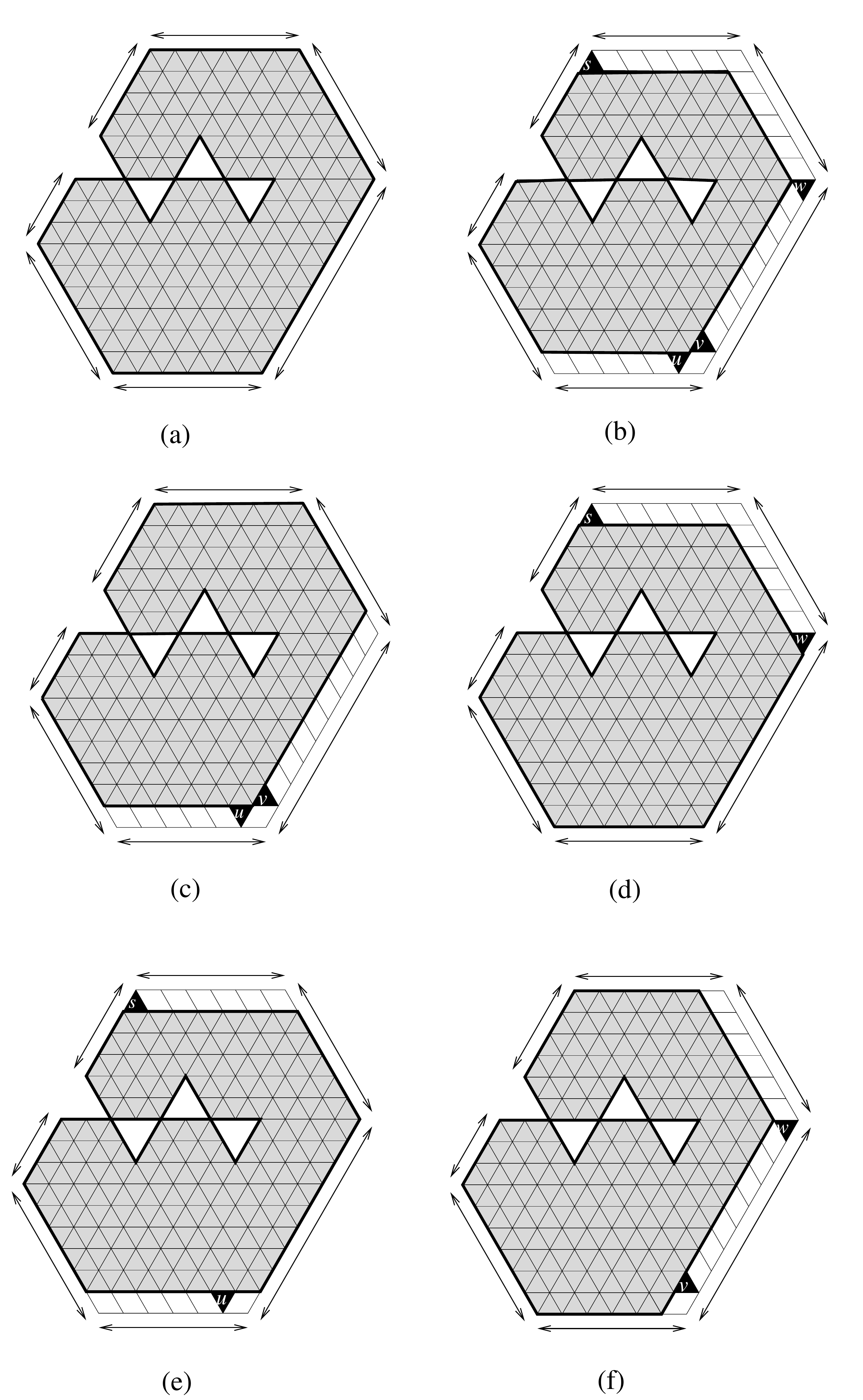}%
\end{picture}%
\begin{picture}(13470,21896)(331,-21067)
\put(886,-2356){\rotatebox{60.0}{\makebox(0,0)[lb]{\smash{{\SetFigFont{20}{24.0}{\rmdefault}{\mddefault}{\itdefault}{$x$}%
}}}}}
\put(1801,-691){\rotatebox{60.0}{\makebox(0,0)[lb]{\smash{{\SetFigFont{20}{24.0}{\rmdefault}{\mddefault}{\itdefault}{$y+a_3$}%
}}}}}
\put(3316,449){\makebox(0,0)[lb]{\smash{{\SetFigFont{20}{24.0}{\rmdefault}{\mddefault}{\itdefault}{$z+a_2+a_4$}%
}}}}
\put(5716,-406){\rotatebox{300.0}{\makebox(0,0)[lb]{\smash{{\SetFigFont{20}{24.0}{\rmdefault}{\mddefault}{\itdefault}{$t+a_1+a_3$}%
}}}}}
\put(5431,-4456){\rotatebox{60.0}{\makebox(0,0)[lb]{\smash{{\SetFigFont{20}{24.0}{\rmdefault}{\mddefault}{\itdefault}{$x+y+a_2+a_4$}%
}}}}}
\put(2686,-5566){\makebox(0,0)[lb]{\smash{{\SetFigFont{20}{24.0}{\rmdefault}{\mddefault}{\itdefault}{$z+a_1+a_3$}%
}}}}
\put(766,-3901){\rotatebox{300.0}{\makebox(0,0)[lb]{\smash{{\SetFigFont{20}{24.0}{\rmdefault}{\mddefault}{\itdefault}{$t+a_2+a_4$}%
}}}}}
\put(1621,-1816){\makebox(0,0)[lb]{\smash{{\SetFigFont{20}{24.0}{\rmdefault}{\mddefault}{\itdefault}{$a_1$}%
}}}}
\put(2521,-2236){\makebox(0,0)[lb]{\smash{{\SetFigFont{20}{24.0}{\rmdefault}{\mddefault}{\itdefault}{$a_2$}%
}}}}
\put(3286,-1861){\makebox(0,0)[lb]{\smash{{\SetFigFont{20}{24.0}{\rmdefault}{\mddefault}{\itdefault}{$a_3$}%
}}}}
\put(4066,-2266){\makebox(0,0)[lb]{\smash{{\SetFigFont{20}{24.0}{\rmdefault}{\mddefault}{\itdefault}{$a_4$}%
}}}}
\put(7802,-2364){\rotatebox{60.0}{\makebox(0,0)[lb]{\smash{{\SetFigFont{20}{24.0}{\rmdefault}{\mddefault}{\itdefault}{$x$}%
}}}}}
\put(8717,-699){\rotatebox{60.0}{\makebox(0,0)[lb]{\smash{{\SetFigFont{20}{24.0}{\rmdefault}{\mddefault}{\itdefault}{$y+a_3$}%
}}}}}
\put(10232,441){\makebox(0,0)[lb]{\smash{{\SetFigFont{20}{24.0}{\rmdefault}{\mddefault}{\itdefault}{$z+a_2+a_4$}%
}}}}
\put(12632,-414){\rotatebox{300.0}{\makebox(0,0)[lb]{\smash{{\SetFigFont{20}{24.0}{\rmdefault}{\mddefault}{\itdefault}{$t+a_1+a_3$}%
}}}}}
\put(12347,-4464){\rotatebox{60.0}{\makebox(0,0)[lb]{\smash{{\SetFigFont{20}{24.0}{\rmdefault}{\mddefault}{\itdefault}{$x+y+a_2+a_4$}%
}}}}}
\put(9602,-5574){\makebox(0,0)[lb]{\smash{{\SetFigFont{20}{24.0}{\rmdefault}{\mddefault}{\itdefault}{$z+a_1+a_3$}%
}}}}
\put(7682,-3909){\rotatebox{300.0}{\makebox(0,0)[lb]{\smash{{\SetFigFont{20}{24.0}{\rmdefault}{\mddefault}{\itdefault}{$t+a_2+a_4$}%
}}}}}
\put(8537,-1824){\makebox(0,0)[lb]{\smash{{\SetFigFont{20}{24.0}{\rmdefault}{\mddefault}{\itdefault}{$a_1$}%
}}}}
\put(9437,-2244){\makebox(0,0)[lb]{\smash{{\SetFigFont{20}{24.0}{\rmdefault}{\mddefault}{\itdefault}{$a_2$}%
}}}}
\put(10202,-1869){\makebox(0,0)[lb]{\smash{{\SetFigFont{20}{24.0}{\rmdefault}{\mddefault}{\itdefault}{$a_3$}%
}}}}
\put(10982,-2274){\makebox(0,0)[lb]{\smash{{\SetFigFont{20}{24.0}{\rmdefault}{\mddefault}{\itdefault}{$a_4$}%
}}}}

\put(947,-9468){\rotatebox{60.0}{\makebox(0,0)[lb]{\smash{{\SetFigFont{20}{24.0}{\rmdefault}{\mddefault}{\itdefault}{$x$}%
}}}}}
\put(1862,-7803){\rotatebox{60.0}{\makebox(0,0)[lb]{\smash{{\SetFigFont{20}{24.0}{\rmdefault}{\mddefault}{\itdefault}{$y+a_3$}%
}}}}}
\put(3377,-6663){\makebox(0,0)[lb]{\smash{{\SetFigFont{20}{24.0}{\rmdefault}{\mddefault}{\itdefault}{$z+a_2+a_4$}%
}}}}
\put(5777,-7518){\rotatebox{300.0}{\makebox(0,0)[lb]{\smash{{\SetFigFont{20}{24.0}{\rmdefault}{\mddefault}{\itdefault}{$t+a_1+a_3$}%
}}}}}
\put(5492,-11568){\rotatebox{60.0}{\makebox(0,0)[lb]{\smash{{\SetFigFont{20}{24.0}{\rmdefault}{\mddefault}{\itdefault}{$x+y+a_2+a_4$}%
}}}}}
\put(2747,-12678){\makebox(0,0)[lb]{\smash{{\SetFigFont{20}{24.0}{\rmdefault}{\mddefault}{\itdefault}{$z+a_1+a_3$}%
}}}}
\put(827,-11013){\rotatebox{300.0}{\makebox(0,0)[lb]{\smash{{\SetFigFont{20}{24.0}{\rmdefault}{\mddefault}{\itdefault}{$t+a_2+a_4$}%
}}}}}
\put(1682,-8928){\makebox(0,0)[lb]{\smash{{\SetFigFont{20}{24.0}{\rmdefault}{\mddefault}{\itdefault}{$a_1$}%
}}}}
\put(2582,-9348){\makebox(0,0)[lb]{\smash{{\SetFigFont{20}{24.0}{\rmdefault}{\mddefault}{\itdefault}{$a_2$}%
}}}}
\put(3347,-8973){\makebox(0,0)[lb]{\smash{{\SetFigFont{20}{24.0}{\rmdefault}{\mddefault}{\itdefault}{$a_3$}%
}}}}
\put(4127,-9378){\makebox(0,0)[lb]{\smash{{\SetFigFont{20}{24.0}{\rmdefault}{\mddefault}{\itdefault}{$a_4$}%
}}}}

\put(7802,-9459){\rotatebox{60.0}{\makebox(0,0)[lb]{\smash{{\SetFigFont{20}{24.0}{\rmdefault}{\mddefault}{\itdefault}{$x$}%
}}}}}
\put(8717,-7794){\rotatebox{60.0}{\makebox(0,0)[lb]{\smash{{\SetFigFont{20}{24.0}{\rmdefault}{\mddefault}{\itdefault}{$y+a_3$}%
}}}}}
\put(10232,-6654){\makebox(0,0)[lb]{\smash{{\SetFigFont{20}{24.0}{\rmdefault}{\mddefault}{\itdefault}{$z+a_2+a_4$}%
}}}}
\put(12632,-7509){\rotatebox{300.0}{\makebox(0,0)[lb]{\smash{{\SetFigFont{20}{24.0}{\rmdefault}{\mddefault}{\itdefault}{$t+a_1+a_3$}%
}}}}}
\put(12347,-11559){\rotatebox{60.0}{\makebox(0,0)[lb]{\smash{{\SetFigFont{20}{24.0}{\rmdefault}{\mddefault}{\itdefault}{$x+y+a_2+a_4$}%
}}}}}
\put(9602,-12669){\makebox(0,0)[lb]{\smash{{\SetFigFont{20}{24.0}{\rmdefault}{\mddefault}{\itdefault}{$z+a_1+a_3$}%
}}}}
\put(7682,-11004){\rotatebox{300.0}{\makebox(0,0)[lb]{\smash{{\SetFigFont{20}{24.0}{\rmdefault}{\mddefault}{\itdefault}{$t+a_2+a_4$}%
}}}}}
\put(8537,-8919){\makebox(0,0)[lb]{\smash{{\SetFigFont{20}{24.0}{\rmdefault}{\mddefault}{\itdefault}{$a_1$}%
}}}}
\put(9437,-9339){\makebox(0,0)[lb]{\smash{{\SetFigFont{20}{24.0}{\rmdefault}{\mddefault}{\itdefault}{$a_2$}%
}}}}
\put(10202,-8964){\makebox(0,0)[lb]{\smash{{\SetFigFont{20}{24.0}{\rmdefault}{\mddefault}{\itdefault}{$a_3$}%
}}}}
\put(10982,-9369){\makebox(0,0)[lb]{\smash{{\SetFigFont{20}{24.0}{\rmdefault}{\mddefault}{\itdefault}{$a_4$}%
}}}}

\put(663,-17073){\rotatebox{60.0}{\makebox(0,0)[lb]{\smash{{\SetFigFont{20}{24.0}{\rmdefault}{\mddefault}{\itdefault}{$x$}%
}}}}}
\put(1578,-15408){\rotatebox{60.0}{\makebox(0,0)[lb]{\smash{{\SetFigFont{20}{24.0}{\rmdefault}{\mddefault}{\itdefault}{$y+a_3$}%
}}}}}
\put(3093,-14268){\makebox(0,0)[lb]{\smash{{\SetFigFont{20}{24.0}{\rmdefault}{\mddefault}{\itdefault}{$z+a_2+a_4$}%
}}}}
\put(5493,-15123){\rotatebox{300.0}{\makebox(0,0)[lb]{\smash{{\SetFigFont{20}{24.0}{\rmdefault}{\mddefault}{\itdefault}{$t+a_1+a_3$}%
}}}}}
\put(5208,-19173){\rotatebox{60.0}{\makebox(0,0)[lb]{\smash{{\SetFigFont{20}{24.0}{\rmdefault}{\mddefault}{\itdefault}{$x+y+a_2+a_4$}%
}}}}}
\put(2463,-20283){\makebox(0,0)[lb]{\smash{{\SetFigFont{20}{24.0}{\rmdefault}{\mddefault}{\itdefault}{$z+a_1+a_3$}%
}}}}
\put(543,-18618){\rotatebox{300.0}{\makebox(0,0)[lb]{\smash{{\SetFigFont{20}{24.0}{\rmdefault}{\mddefault}{\itdefault}{$t+a_2+a_4$}%
}}}}}
\put(1398,-16533){\makebox(0,0)[lb]{\smash{{\SetFigFont{20}{24.0}{\rmdefault}{\mddefault}{\itdefault}{$a_1$}%
}}}}
\put(2298,-16953){\makebox(0,0)[lb]{\smash{{\SetFigFont{20}{24.0}{\rmdefault}{\mddefault}{\itdefault}{$a_2$}%
}}}}
\put(3063,-16578){\makebox(0,0)[lb]{\smash{{\SetFigFont{20}{24.0}{\rmdefault}{\mddefault}{\itdefault}{$a_3$}%
}}}}
\put(3843,-16983){\makebox(0,0)[lb]{\smash{{\SetFigFont{20}{24.0}{\rmdefault}{\mddefault}{\itdefault}{$a_4$}%
}}}}

\put(7533,-17088){\rotatebox{60.0}{\makebox(0,0)[lb]{\smash{{\SetFigFont{20}{24.0}{\rmdefault}{\mddefault}{\itdefault}{$x$}%
}}}}}
\put(8448,-15423){\rotatebox{60.0}{\makebox(0,0)[lb]{\smash{{\SetFigFont{20}{24.0}{\rmdefault}{\mddefault}{\itdefault}{$y+a_3$}%
}}}}}
\put(9963,-14283){\makebox(0,0)[lb]{\smash{{\SetFigFont{20}{24.0}{\rmdefault}{\mddefault}{\itdefault}{$z+a_2+a_4$}%
}}}}
\put(12363,-15138){\rotatebox{300.0}{\makebox(0,0)[lb]{\smash{{\SetFigFont{20}{24.0}{\rmdefault}{\mddefault}{\itdefault}{$t+a_1+a_3$}%
}}}}}
\put(12078,-19188){\rotatebox{60.0}{\makebox(0,0)[lb]{\smash{{\SetFigFont{20}{24.0}{\rmdefault}{\mddefault}{\itdefault}{$x+y+a_2+a_4$}%
}}}}}
\put(9333,-20298){\makebox(0,0)[lb]{\smash{{\SetFigFont{20}{24.0}{\rmdefault}{\mddefault}{\itdefault}{$z+a_1+a_3$}%
}}}}
\put(7413,-18633){\rotatebox{300.0}{\makebox(0,0)[lb]{\smash{{\SetFigFont{20}{24.0}{\rmdefault}{\mddefault}{\itdefault}{$t+a_2+a_4$}%
}}}}}
\put(8268,-16548){\makebox(0,0)[lb]{\smash{{\SetFigFont{20}{24.0}{\rmdefault}{\mddefault}{\itdefault}{$a_1$}%
}}}}
\put(9168,-16968){\makebox(0,0)[lb]{\smash{{\SetFigFont{20}{24.0}{\rmdefault}{\mddefault}{\itdefault}{$a_2$}%
}}}}
\put(9933,-16593){\makebox(0,0)[lb]{\smash{{\SetFigFont{20}{24.0}{\rmdefault}{\mddefault}{\itdefault}{$a_3$}%
}}}}
\put(10713,-16998){\makebox(0,0)[lb]{\smash{{\SetFigFont{20}{24.0}{\rmdefault}{\mddefault}{\itdefault}{$a_4$}%
}}}}
\end{picture}}
\caption{Obtaining the recurrence for the numbers of tilings of $R$-regions.}
\label{Kuoarray}
\end{figure}

First, we consider the region corresponding to the graph $G-\{u,v,w,s\}$ (see Figure \ref{Kuoarray}(b)). The removal of the four vertices $u,v,w,s$ yields several lozenges that are forced to be in any tilings of the leftover region. By removing these forced lozenges, whose weight product is $q$ (as there is only one right-leaning forced lozenge), we get a weighted version of the region $R_{x,y-1,z,t-1}(\textbf{a})$ (see the region restricted by the bold contour in Figure \ref{Kuoarray}(c)). Upon dividing the weight of each right-leaning lozenge by $q$, this weight becomes the weight assignment $\wt_1$.
This contributes a multiplicative factor of $q^{N_{x,y-1,z,t-1}}(\textbf{a})$, where $N_{x,y,z,t}(\textbf{a})$ is the total number of right lozenges in a tiling of the region $R_{x,y-1,z,t-1}(\textbf{a})$ (which is the same for all tilings). A straightforward calculation gives
\begin{equation}
N_{x,y,z,t}(\textbf{a})=xo_a+ye_a+z(x+y+e_a)+\sum_{i=2}^{\lfloor\frac{n+1}{2}\rfloor}a_{2i-1}\sum_{j=1}^{i-1}a_{2j}+\sum_{i=1}^{\lfloor\frac{n}{2}\rfloor}a_{2(k-i)}\sum_{j=1}^{i}a_{2k-2j+1}.
\end{equation}
Therefore, we have
\begin{equation}
\M(G-\{u,v,w,s\})=q\cdot q^{N_{x,y-1,z,t-1}(\textbf{a})} \M_1\Big(R_{x,y-1,z,t-1}(\textbf{a})\Big).
\end{equation}

Similarly, by accounting for the weights of the forced lozenges in Figures \ref{Kuoarray}(c)--(f) and suitably normalizing the weight assignments if needed, we obtain
\begin{equation}
\M(G-\{u,v\})=q\cdot q^{N_{x,y,z,t-1}(\textbf{a})}\M_1\Big(R_{x,y,z,t-1}(\textbf{a})\Big).
\end{equation}
\begin{equation}
\M(G-\{w,s\})=\M_1\Big(R_{x,y-1,z,t}(\textbf{a})\Big),
\end{equation}
\begin{equation}
\M(G-\{u,s\})=q\cdot q^{N_{x,y-1,z+1,t-1}(\textbf{a})}\M_1\Big(R_{x,y-1,z+1,t-1}(\textbf{a})\Big),
\end{equation}
and
\begin{equation}
\M(G-\{w,v\})=q\M_1\Big(R_{x,y,z-1,t}(\textbf{a})\Big).
\end{equation}

Substituting the above five expressions into the equation \eqref{kuoeq}, we get
\begin{align}\label{qrecurrence1}
q^{N_{x,y-1,z,t-1}(\textbf{a})}\M_1&\Big(R_{x,y,z,t}(\textbf{a})\Big)\M_1\Big(R_{x,y-1,z,t-1}(\textbf{a})\Big)\notag\\
&=q^{N_{x,y,z,t-1}(\textbf{a})}\M_1\Big(R_{x,y,z,t-1}(\textbf{a})\Big)\M_1\Big(R_{x,y-1,z,t}(\textbf{a})\Big)\notag\\
&+q^{N_{x,y-1,z+1,t-1}(\textbf{a})+1}\M_1\Big(R_{x,y-1,z+1,t-1}(\textbf{a})\Big)\M_1\Big(R_{x,y-1,z,t-1}(\textbf{a})\Big),
\end{align}
which yields the recurrence
\begin{align}\label{qrecurrence2}
q^{-z-e_a}\M_1\Big(R_{x,y,z,t}(\textbf{a})\Big)&\M_1\Big(R_{x,y-1,z,t-1}(\textbf{a})\Big)=\M_1\Big(R_{x,y,z,t-1}(\textbf{a})\Big)\M_1\Big(R_{x,y-1,z,t}(\textbf{a})\Big)\notag\\
&+q^{x+y-z}\M_1\Big(R_{x,y-1,z+1,t-1}(\textbf{a})\Big)\M_1\Big(R_{x,y-1,z,t-1}(\textbf{a})\Big).
\end{align}
Clearly, all regions in the above recurrence, except for the first one, have the sum of their $y$-, $z$-, $t$-parameters strictly less than $y+z+t$. Thus by the induction hypothesis, the weighted tiling counts of these regions are given by the formulas (\ref{qmain1eq2}) and (\ref{qmain1eq1}). This provides an explicit expression for $\M_1\Big(R_{x,y,z,t}(\textbf{a})\Big)$, and a straightforward calculation checks that this expression agrees with the formulas in the statement of the theorem.
\end{proof}

\section{Proof of Theorem \ref{main2}}

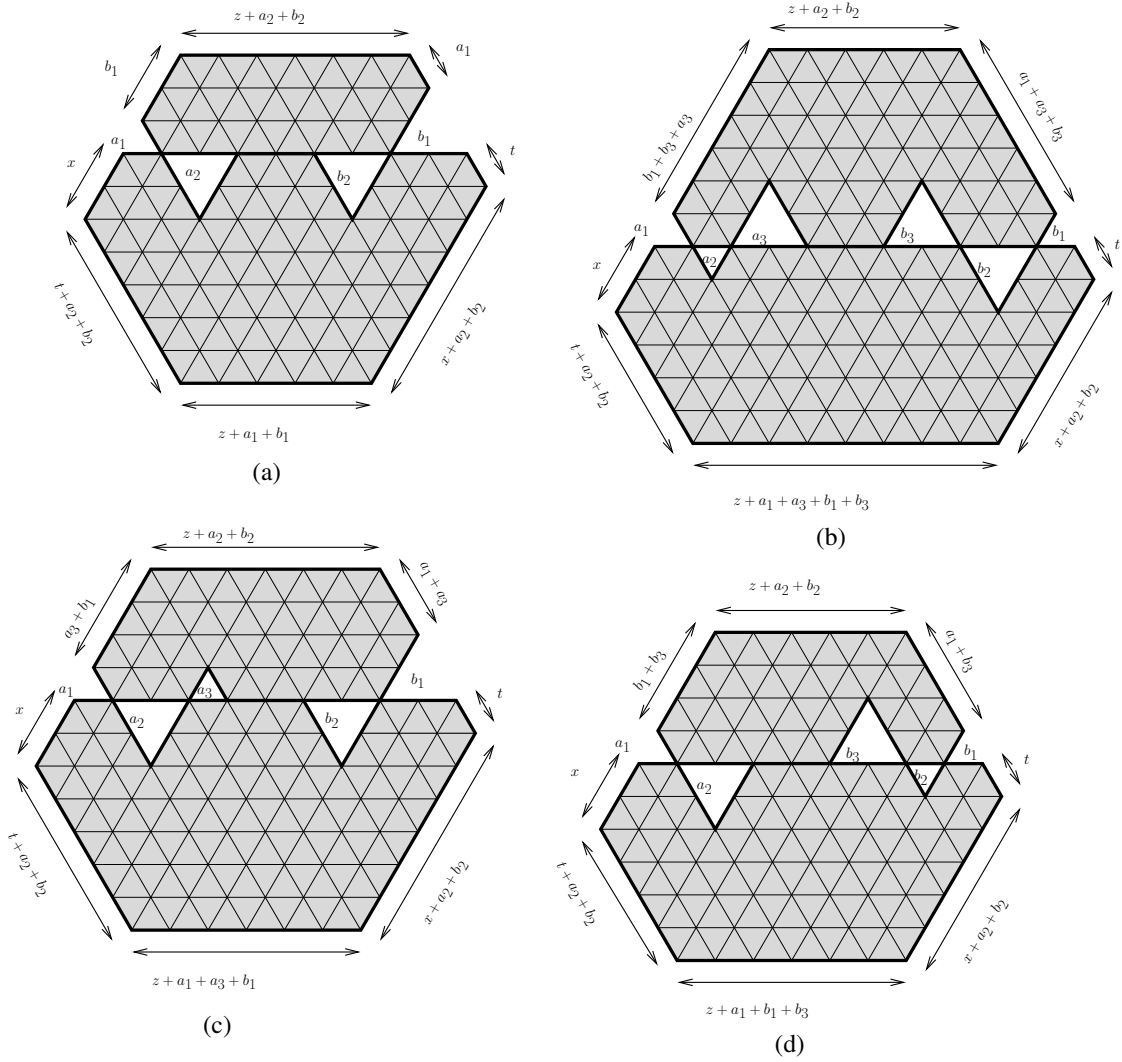
\begin{figure}\centering
\setlength{\unitlength}{3947sp}%
\begingroup\makeatletter\ifx\SetFigFont\undefined%
\gdef\SetFigFont#1#2#3#4#5{%
  \reset@font\fontsize{#1}{#2pt}%
  \fontfamily{#3}\fontseries{#4}\fontshape{#5}%
  \selectfont}%
\fi\endgroup%

\resizebox{15cm}{!}{
\begin{picture}(0,0)%
\includegraphics{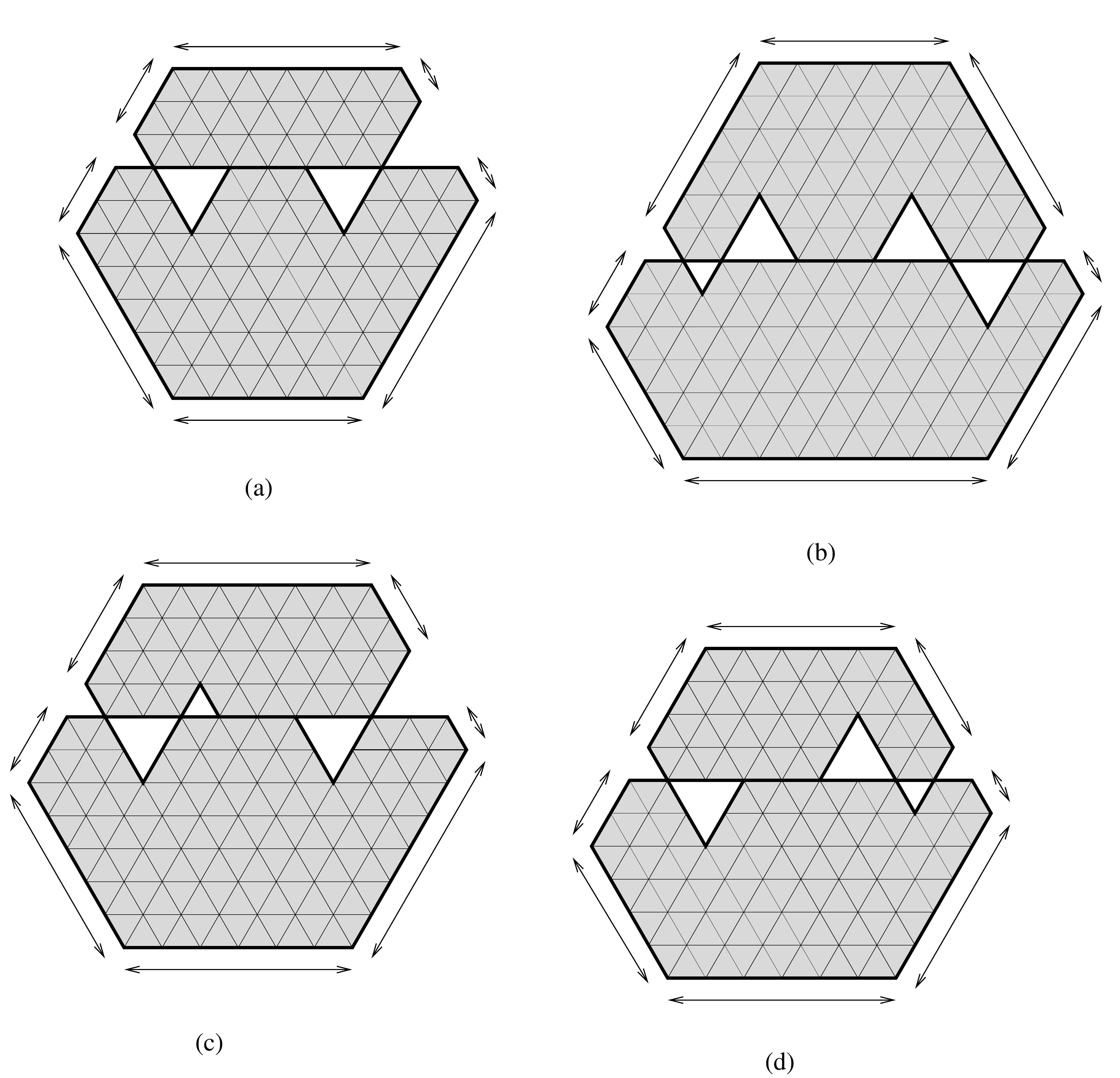}%
\end{picture}%
%
%

\begin{picture}(30968,29405)(335,-28781)
\put(6498,-251){\makebox(0,0)[lb]{\smash{{\SetFigFont{41}{49.2}{\rmdefault}{\mddefault}{\itdefault}{$z+a_2+b_2$}%
}}}}
\put(2988,-1661){\makebox(0,0)[lb]{\smash{{\SetFigFont{41}{49.2}{\rmdefault}{\mddefault}{\itdefault}{$b_1$}%
}}}}
\put(12558,-1121){\makebox(0,0)[lb]{\smash{{\SetFigFont{41}{49.2}{\rmdefault}{\mddefault}{\itdefault}{$a_1$}%
}}}}
\put(14028,-3971){\makebox(0,0)[lb]{\smash{{\SetFigFont{41}{49.2}{\rmdefault}{\mddefault}{\itdefault}{$t$}%
}}}}
\put(12288,-9761){\rotatebox{60.0}{\makebox(0,0)[lb]{\smash{{\SetFigFont{41}{49.2}{\rmdefault}{\mddefault}{\itdefault}{$x+a_2+b_2$}%
}}}}}
\put(6078,-11711){\makebox(0,0)[lb]{\smash{{\SetFigFont{41}{49.2}{\rmdefault}{\mddefault}{\itdefault}{$z+a_1+b_1$}%
}}}}
\put(1608,-7541){\rotatebox{300.0}{\makebox(0,0)[lb]{\smash{{\SetFigFont{41}{49.2}{\rmdefault}{\mddefault}{\itdefault}{$t+a_2+b_2$}%
}}}}}
\put(1968,-4301){\makebox(0,0)[lb]{\smash{{\SetFigFont{41}{49.2}{\rmdefault}{\mddefault}{\itdefault}{$x$}%
}}}}
\put(3168,-3671){\makebox(0,0)[lb]{\smash{{\SetFigFont{41}{49.2}{\rmdefault}{\mddefault}{\itdefault}{$a_1$}%
}}}}
\put(5191,-4531){\makebox(0,0)[lb]{\smash{{\SetFigFont{41}{49.2}{\rmdefault}{\mddefault}{\itdefault}{$a_2$}%
}}}}
\put(9318,-4661){\makebox(0,0)[lb]{\smash{{\SetFigFont{41}{49.2}{\rmdefault}{\mddefault}{\itdefault}{$b_2$}%
}}}}
\put(11508,-3581){\makebox(0,0)[lb]{\smash{{\SetFigFont{41}{49.2}{\rmdefault}{\mddefault}{\itdefault}{$b_1$}%
}}}}
\put(350,-22643){\rotatebox{300.0}{\makebox(0,0)[lb]{\smash{{\SetFigFont{41}{49.2}{\rmdefault}{\mddefault}{\itdefault}{$t+a_2+b_2$}%
}}}}}
\put(15285,-23437){\rotatebox{300.0}{\makebox(0,0)[lb]{\smash{{\SetFigFont{41}{49.2}{\rmdefault}{\mddefault}{\itdefault}{$t+a_2+b_2$}%
}}}}}
\put(15655,-9226){\rotatebox{300.0}{\makebox(0,0)[lb]{\smash{{\SetFigFont{41}{49.2}{\rmdefault}{\mddefault}{\itdefault}{$t+a_2+b_2$}%
}}}}}
\put(28981,-12001){\rotatebox{60.0}{\makebox(0,0)[lb]{\smash{{\SetFigFont{41}{49.2}{\rmdefault}{\mddefault}{\itdefault}{$x+a_2+b_2$}%
}}}}}
\put(26520,-26148){\rotatebox{60.0}{\makebox(0,0)[lb]{\smash{{\SetFigFont{41}{49.2}{\rmdefault}{\mddefault}{\itdefault}{$x+a_2+b_2$}%
}}}}}
\put(11825,-25084){\rotatebox{60.0}{\makebox(0,0)[lb]{\smash{{\SetFigFont{41}{49.2}{\rmdefault}{\mddefault}{\itdefault}{$x+a_2+b_2$}%
}}}}}
\put(30481,-6526){\makebox(0,0)[lb]{\smash{{\SetFigFont{41}{49.2}{\rmdefault}{\mddefault}{\itdefault}{$t$}%
}}}}
\put(28016,-20670){\makebox(0,0)[lb]{\smash{{\SetFigFont{41}{49.2}{\rmdefault}{\mddefault}{\itdefault}{$t$}%
}}}}
\put(13681,-18856){\makebox(0,0)[lb]{\smash{{\SetFigFont{41}{49.2}{\rmdefault}{\mddefault}{\itdefault}{$t$}%
}}}}
\put(21586,-76){\makebox(0,0)[lb]{\smash{{\SetFigFont{41}{49.2}{\rmdefault}{\mddefault}{\itdefault}{$z+a_2+b_2$}%
}}}}
\put(20521,-15901){\makebox(0,0)[lb]{\smash{{\SetFigFont{41}{49.2}{\rmdefault}{\mddefault}{\itdefault}{$z+a_2+b_2$}%
}}}}
\put(5116,-14476){\makebox(0,0)[lb]{\smash{{\SetFigFont{41}{49.2}{\rmdefault}{\mddefault}{\itdefault}{$z+a_2+b_2$}%
}}}}
\put(16276,-7096){\makebox(0,0)[lb]{\smash{{\SetFigFont{41}{49.2}{\rmdefault}{\mddefault}{\itdefault}{$x$}%
}}}}
\put(15671,-21000){\makebox(0,0)[lb]{\smash{{\SetFigFont{41}{49.2}{\rmdefault}{\mddefault}{\itdefault}{$x$}%
}}}}
\put(556,-19306){\makebox(0,0)[lb]{\smash{{\SetFigFont{41}{49.2}{\rmdefault}{\mddefault}{\itdefault}{$x$}%
}}}}
\put(16886,-20205){\makebox(0,0)[lb]{\smash{{\SetFigFont{41}{49.2}{\rmdefault}{\mddefault}{\itdefault}{$a_1$}%
}}}}
\put(1741,-18706){\makebox(0,0)[lb]{\smash{{\SetFigFont{41}{49.2}{\rmdefault}{\mddefault}{\itdefault}{$a_1$}%
}}}}
\put(17401,-6046){\makebox(0,0)[lb]{\smash{{\SetFigFont{41}{49.2}{\rmdefault}{\mddefault}{\itdefault}{$a_1$}%
}}}}
\put(19246,-6916){\makebox(0,0)[lb]{\smash{{\SetFigFont{41}{49.2}{\rmdefault}{\mddefault}{\itdefault}{$a_2$}%
}}}}
\put(3661,-19546){\makebox(0,0)[lb]{\smash{{\SetFigFont{41}{49.2}{\rmdefault}{\mddefault}{\itdefault}{$a_2$}%
}}}}
\put(19106,-21360){\makebox(0,0)[lb]{\smash{{\SetFigFont{41}{49.2}{\rmdefault}{\mddefault}{\itdefault}{$a_2$}%
}}}}
\put(28786,-6196){\makebox(0,0)[lb]{\smash{{\SetFigFont{41}{49.2}{\rmdefault}{\mddefault}{\itdefault}{$b_1$}%
}}}}
\put(26381,-20430){\makebox(0,0)[lb]{\smash{{\SetFigFont{41}{49.2}{\rmdefault}{\mddefault}{\itdefault}{$b_1$}%
}}}}
\put(11326,-18466){\makebox(0,0)[lb]{\smash{{\SetFigFont{41}{49.2}{\rmdefault}{\mddefault}{\itdefault}{$b_1$}%
}}}}
\put(26746,-7186){\makebox(0,0)[lb]{\smash{{\SetFigFont{41}{49.2}{\rmdefault}{\mddefault}{\itdefault}{$b_2$}%
}}}}
\put(25021,-21106){\makebox(0,0)[lb]{\smash{{\SetFigFont{41}{49.2}{\rmdefault}{\mddefault}{\itdefault}{$b_2$}%
}}}}
\put(8986,-19546){\makebox(0,0)[lb]{\smash{{\SetFigFont{41}{49.2}{\rmdefault}{\mddefault}{\itdefault}{$b_2$}%
}}}}
\put(20611,-6301){\makebox(0,0)[lb]{\smash{{\SetFigFont{41}{49.2}{\rmdefault}{\mddefault}{\itdefault}{$a_3$}%
}}}}
\put(5506,-18736){\makebox(0,0)[lb]{\smash{{\SetFigFont{41}{49.2}{\rmdefault}{\mddefault}{\itdefault}{$a_3$}%
}}}}
\put(24661,-6241){\makebox(0,0)[lb]{\smash{{\SetFigFont{41}{49.2}{\rmdefault}{\mddefault}{\itdefault}{$b_3$}%
}}}}
\put(23171,-20490){\makebox(0,0)[lb]{\smash{{\SetFigFont{41}{49.2}{\rmdefault}{\mddefault}{\itdefault}{$b_3$}%
}}}}
\put(20101,-13561){\makebox(0,0)[lb]{\smash{{\SetFigFont{41}{49.2}{\rmdefault}{\mddefault}{\itdefault}{$z+a_1+a_3+b_1+b_3$}%
}}}}
\put(19354,-27510){\makebox(0,0)[lb]{\smash{{\SetFigFont{41}{49.2}{\rmdefault}{\mddefault}{\itdefault}{$z+a_1+b_1+b_3$}%
}}}}
\put(4276,-26724){\makebox(0,0)[lb]{\smash{{\SetFigFont{41}{49.2}{\rmdefault}{\mddefault}{\itdefault}{$z+a_1+a_3+b_1$}%
}}}}
\put(17851,-4711){\rotatebox{60.0}{\makebox(0,0)[lb]{\smash{{\SetFigFont{41}{49.2}{\rmdefault}{\mddefault}{\itdefault}{$b_1+b_3+a_3$}%
}}}}}
\put(2026,-17236){\rotatebox{60.0}{\makebox(0,0)[lb]{\smash{{\SetFigFont{41}{49.2}{\rmdefault}{\mddefault}{\itdefault}{$a_3+b_1$}%
}}}}}
\put(17589,-18661){\rotatebox{60.0}{\makebox(0,0)[lb]{\smash{{\SetFigFont{41}{49.2}{\rmdefault}{\mddefault}{\itdefault}{$b_1+b_3$}%
}}}}}
\put(27880,-1726){\rotatebox{300.0}{\makebox(0,0)[lb]{\smash{{\SetFigFont{41}{49.2}{\rmdefault}{\mddefault}{\itdefault}{$a_1+a_3+b_3$}%
}}}}}
\put(11568,-15189){\rotatebox{300.0}{\makebox(0,0)[lb]{\smash{{\SetFigFont{41}{49.2}{\rmdefault}{\mddefault}{\itdefault}{$a_1+a_3$}%
}}}}}
\put(25876,-17049){\rotatebox{300.0}{\makebox(0,0)[lb]{\smash{{\SetFigFont{41}{49.2}{\rmdefault}{\mddefault}{\itdefault}{$a_1+b_3$}%
}}}}}
\end{picture}}
\caption{Base cases for the $P$-regions when $y=0$.}
\label{arraybasenew1}
\end{figure}

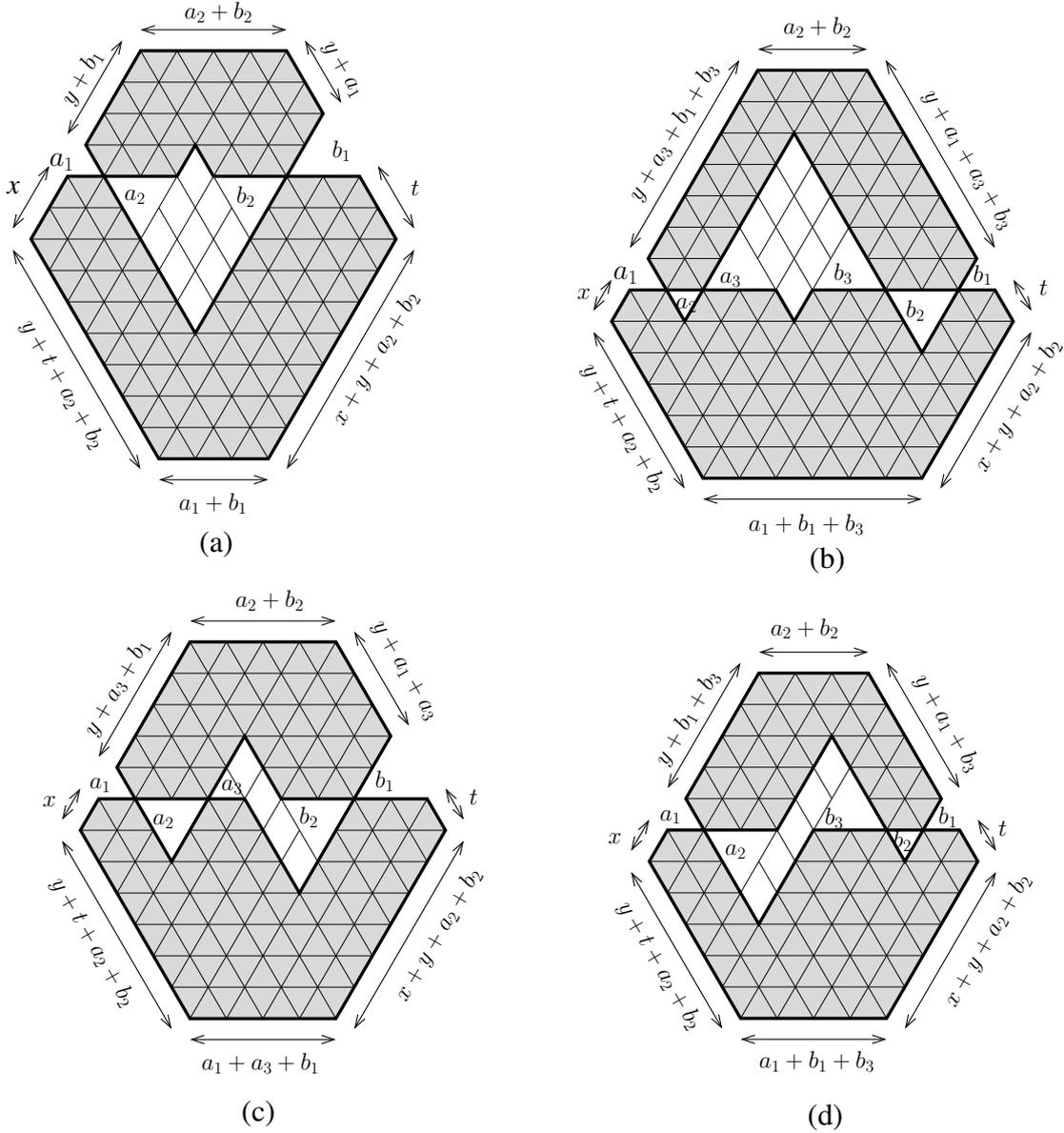
\begin{figure}\centering
\setlength{\unitlength}{3947sp}%
\begingroup\makeatletter\ifx\SetFigFont\undefined%
\gdef\SetFigFont#1#2#3#4#5{%
  \reset@font\fontsize{#1}{#2pt}%
  \fontfamily{#3}\fontseries{#4}\fontshape{#5}%
  \selectfont}%
\fi\endgroup%
\resizebox{15cm}{!}{
\begin{picture}(0,0)%
\includegraphics{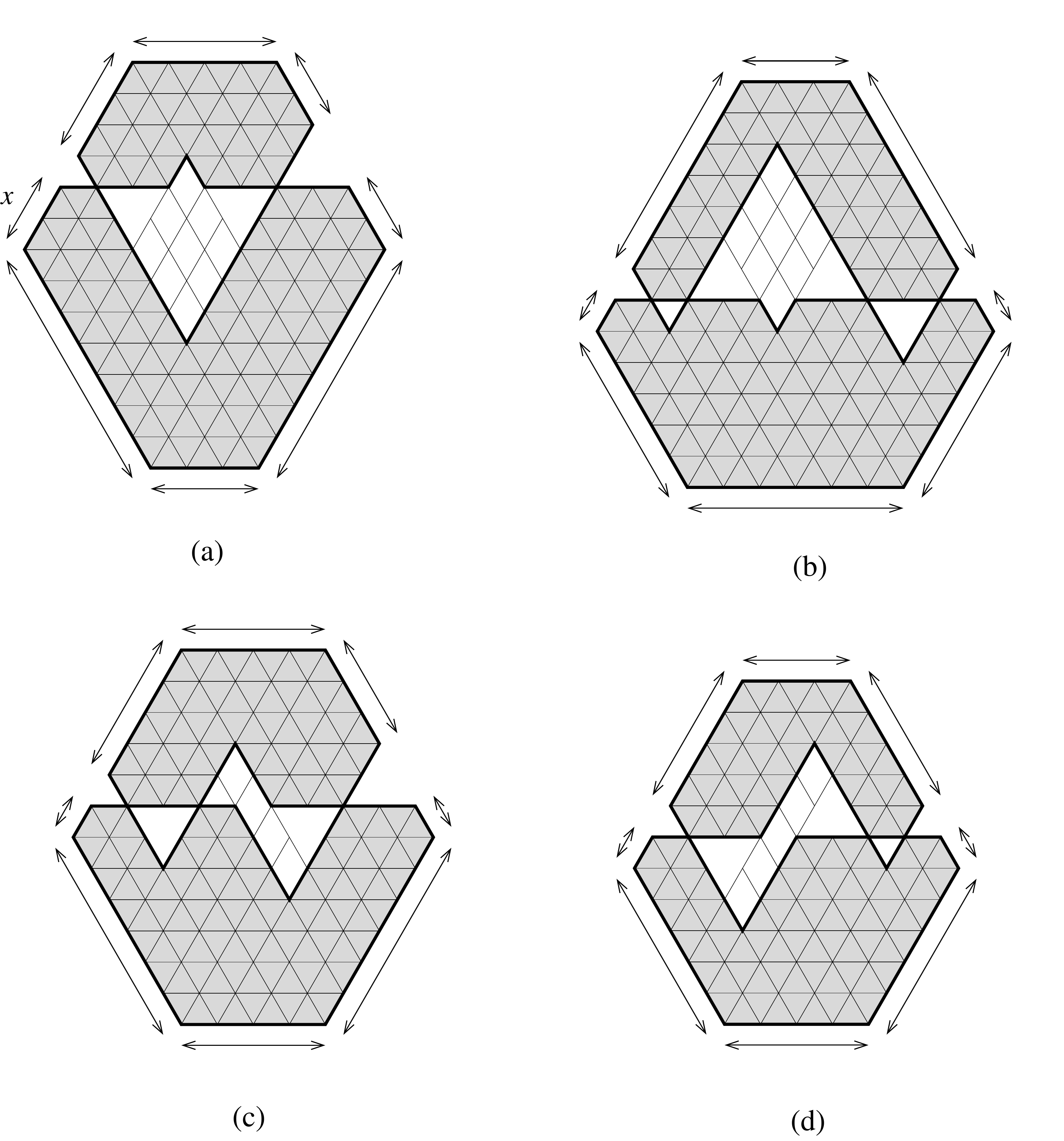}%
\end{picture}%
%
%
\begin{picture}(15334,16417)(984,-16373)
\put(2172,-11396){\makebox(0,0)[lb]{\smash{{\SetFigFont{20}{24.0}{\rmdefault}{\mddefault}{\itdefault}{$a_1$}%
}}}}
\put(3072,-11921){\makebox(0,0)[lb]{\smash{{\SetFigFont{20}{24.0}{\rmdefault}{\mddefault}{\itdefault}{$a_2$}%
}}}}
\put(5157,-11906){\makebox(0,0)[lb]{\smash{{\SetFigFont{20}{24.0}{\rmdefault}{\mddefault}{\itdefault}{$b_2$}%
}}}}
\put(6237,-11381){\makebox(0,0)[lb]{\smash{{\SetFigFont{20}{24.0}{\rmdefault}{\mddefault}{\itdefault}{$b_1$}%
}}}}
\put(10287,-11854){\makebox(0,0)[lb]{\smash{{\SetFigFont{20}{24.0}{\rmdefault}{\mddefault}{\itdefault}{$a_1$}%
}}}}
\put(11217,-12349){\makebox(0,0)[lb]{\smash{{\SetFigFont{20}{24.0}{\rmdefault}{\mddefault}{\itdefault}{$a_2$}%
}}}}
\put(13606,-12241){\makebox(0,0)[lb]{\smash{{\SetFigFont{20}{24.0}{\rmdefault}{\mddefault}{\itdefault}{$b_2$}%
}}}}
\put(14255,-11899){\makebox(0,0)[lb]{\smash{{\SetFigFont{20}{24.0}{\rmdefault}{\mddefault}{\itdefault}{$b_1$}%
}}}}
\put(4047,-11426){\makebox(0,0)[lb]{\smash{{\SetFigFont{20}{24.0}{\rmdefault}{\mddefault}{\itdefault}{$a_3$}%
}}}}
\put(9112,-4383){\makebox(0,0)[lb]{\smash{{\SetFigFont{20}{24.0}{\rmdefault}{\mddefault}{\itdefault}{$x$}%
}}}}
\put(15689,-4323){\makebox(0,0)[lb]{\smash{{\SetFigFont{20}{24.0}{\rmdefault}{\mddefault}{\itdefault}{$t$}%
}}}}
\put(10004,-2995){\rotatebox{60.0}{\makebox(0,0)[lb]{\smash{{\SetFigFont{20}{24.0}{\rmdefault}{\mddefault}{\itdefault}{$y+a_3+b_1+b_3$}%
}}}}}
\put(12044,-580){\makebox(0,0)[lb]{\smash{{\SetFigFont{20}{24.0}{\rmdefault}{\mddefault}{\itdefault}{$a_2+b_2$}%
}}}}
\put(14039,-1450){\rotatebox{300.0}{\makebox(0,0)[lb]{\smash{{\SetFigFont{20}{24.0}{\rmdefault}{\mddefault}{\itdefault}{$y+a_1+a_3+b_3$}%
}}}}}
\put(14924,-6880){\rotatebox{60.0}{\makebox(0,0)[lb]{\smash{{\SetFigFont{20}{24.0}{\rmdefault}{\mddefault}{\itdefault}{$x+y+a_2+b_2$}%
}}}}}
\put(11551,-7711){\makebox(0,0)[lb]{\smash{{\SetFigFont{20}{24.0}{\rmdefault}{\mddefault}{\itdefault}{$a_1+b_1+b_3$}%
}}}}
\put(9184,-5371){\rotatebox{300.0}{\makebox(0,0)[lb]{\smash{{\SetFigFont{20}{24.0}{\rmdefault}{\mddefault}{\itdefault}{$y+t+a_2+b_2$}%
}}}}}
\put(1194,-4747){\rotatebox{300.0}{\makebox(0,0)[lb]{\smash{{\SetFigFont{20}{24.0}{\rmdefault}{\mddefault}{\itdefault}{$y+t+a_2+b_2$}%
}}}}}
\put(3444,-7402){\makebox(0,0)[lb]{\smash{{\SetFigFont{20}{24.0}{\rmdefault}{\mddefault}{\itdefault}{$a_1+b_1$}%
}}}}
\put(3519,-367){\makebox(0,0)[lb]{\smash{{\SetFigFont{20}{24.0}{\rmdefault}{\mddefault}{\itdefault}{$a_2+b_2$}%
}}}}
\put(5499,-742){\rotatebox{300.0}{\makebox(0,0)[lb]{\smash{{\SetFigFont{20}{24.0}{\rmdefault}{\mddefault}{\itdefault}{$y+a_1$}%
}}}}}
\put(1929,-1612){\rotatebox{60.0}{\makebox(0,0)[lb]{\smash{{\SetFigFont{20}{24.0}{\rmdefault}{\mddefault}{\itdefault}{$y+b_1$}%
}}}}}
\put(5784,-6202){\rotatebox{60.0}{\makebox(0,0)[lb]{\smash{{\SetFigFont{20}{24.0}{\rmdefault}{\mddefault}{\itdefault}{$x+y+a_2+b_2$}%
}}}}}
\put(1504,-11651){\makebox(0,0)[lb]{\smash{{\SetFigFont{20}{24.0}{\rmdefault}{\mddefault}{\itdefault}{$x$}%
}}}}
\put(7579,-11629){\makebox(0,0)[lb]{\smash{{\SetFigFont{20}{24.0}{\rmdefault}{\mddefault}{\itdefault}{$t$}%
}}}}
\put(4249,-8801){\makebox(0,0)[lb]{\smash{{\SetFigFont{20}{24.0}{\rmdefault}{\mddefault}{\itdefault}{$a_2+b_2$}%
}}}}
\put(3766,-15436){\makebox(0,0)[lb]{\smash{{\SetFigFont{20}{24.0}{\rmdefault}{\mddefault}{\itdefault}{$a_1+a_3+b_1$}%
}}}}
\put(6678,-14468){\rotatebox{60.0}{\makebox(0,0)[lb]{\smash{{\SetFigFont{20}{24.0}{\rmdefault}{\mddefault}{\itdefault}{$x+y+a_2+b_2$}%
}}}}}
\put(1624,-12751){\rotatebox{300.0}{\makebox(0,0)[lb]{\smash{{\SetFigFont{20}{24.0}{\rmdefault}{\mddefault}{\itdefault}{$y+t+a_2+b_2$}%
}}}}}
\put(2251,-10666){\rotatebox{60.0}{\makebox(0,0)[lb]{\smash{{\SetFigFont{20}{24.0}{\rmdefault}{\mddefault}{\itdefault}{$y+a_3+b_1$}%
}}}}}
\put(6202,-9056){\rotatebox{300.0}{\makebox(0,0)[lb]{\smash{{\SetFigFont{20}{24.0}{\rmdefault}{\mddefault}{\itdefault}{$y+a_1+a_3$}%
}}}}}
\put(9552,-12184){\makebox(0,0)[lb]{\smash{{\SetFigFont{20}{24.0}{\rmdefault}{\mddefault}{\itdefault}{$x$}%
}}}}
\put(15109,-12079){\makebox(0,0)[lb]{\smash{{\SetFigFont{20}{24.0}{\rmdefault}{\mddefault}{\itdefault}{$t$}%
}}}}
\put(9691,-13081){\rotatebox{300.0}{\makebox(0,0)[lb]{\smash{{\SetFigFont{20}{24.0}{\rmdefault}{\mddefault}{\itdefault}{$y+t+a_2+b_2$}%
}}}}}
\put(14493,-14543){\rotatebox{60.0}{\makebox(0,0)[lb]{\smash{{\SetFigFont{20}{24.0}{\rmdefault}{\mddefault}{\itdefault}{$x+y+a_2+b_2$}%
}}}}}
\put(11862,-9199){\makebox(0,0)[lb]{\smash{{\SetFigFont{20}{24.0}{\rmdefault}{\mddefault}{\itdefault}{$a_2+b_2$}%
}}}}
\put(10381,-11086){\rotatebox{60.0}{\makebox(0,0)[lb]{\smash{{\SetFigFont{20}{24.0}{\rmdefault}{\mddefault}{\itdefault}{$y+b_1+b_3$}%
}}}}}
\put(13912,-9836){\rotatebox{300.0}{\makebox(0,0)[lb]{\smash{{\SetFigFont{20}{24.0}{\rmdefault}{\mddefault}{\itdefault}{$y+a_1+b_3$}%
}}}}}
\put(11701,-15406){\makebox(0,0)[lb]{\smash{{\SetFigFont{20}{24.0}{\rmdefault}{\mddefault}{\itdefault}{$a_1+b_1+b_3$}%
}}}}
\put(1598,-2477){\makebox(0,0)[lb]{\smash{{\SetFigFont{25}{30.0}{\rmdefault}{\mddefault}{\itdefault}{$a_1$}%
}}}}
\put(2667,-2966){\makebox(0,0)[lb]{\smash{{\SetFigFont{20}{24.0}{\rmdefault}{\mddefault}{\itdefault}{$a_2$}%
}}}}
\put(5622,-2351){\makebox(0,0)[lb]{\smash{{\SetFigFont{20}{24.0}{\rmdefault}{\mddefault}{\itdefault}{$b_1$}%
}}}}
\put(4261,-3001){\makebox(0,0)[lb]{\smash{{\SetFigFont{20}{24.0}{\rmdefault}{\mddefault}{\itdefault}{$b_2$}%
}}}}
\put(9627,-4099){\makebox(0,0)[lb]{\smash{{\SetFigFont{25}{30.0}{\rmdefault}{\mddefault}{\itdefault}{$a_1$}%
}}}}
\put(10516,-4516){\makebox(0,0)[lb]{\smash{{\SetFigFont{20}{24.0}{\rmdefault}{\mddefault}{\itdefault}{$a_2$}%
}}}}
\put(14750,-4129){\makebox(0,0)[lb]{\smash{{\SetFigFont{20}{24.0}{\rmdefault}{\mddefault}{\itdefault}{$b_1$}%
}}}}
\put(13790,-4639){\makebox(0,0)[lb]{\smash{{\SetFigFont{20}{24.0}{\rmdefault}{\mddefault}{\itdefault}{$b_2$}%
}}}}
\put(12762,-4106){\makebox(0,0)[lb]{\smash{{\SetFigFont{20}{24.0}{\rmdefault}{\mddefault}{\itdefault}{$b_3$}%
}}}}
\put(11142,-4151){\makebox(0,0)[lb]{\smash{{\SetFigFont{20}{24.0}{\rmdefault}{\mddefault}{\itdefault}{$a_3$}%
}}}}
\put(12635,-11914){\makebox(0,0)[lb]{\smash{{\SetFigFont{20}{24.0}{\rmdefault}{\mddefault}{\itdefault}{$b_3$}%
}}}}
\put(6744,-2902){\makebox(0,0)[lb]{\smash{{\SetFigFont{20}{24.0}{\rmdefault}{\mddefault}{\itdefault}{$t$}%
}}}}
\end{picture}
}
\caption{Base cases for $P$-region when $z=0$.}
\label{arrayQbase}
\end{figure}

\begin{figure}
  \centering
  \setlength{\unitlength}{3947sp}%
\begingroup\makeatletter\ifx\SetFigFont\undefined%
\gdef\SetFigFont#1#2#3#4#5{%
  \reset@font\fontsize{#1}{#2pt}%
  \fontfamily{#3}\fontseries{#4}\fontshape{#5}%
  \selectfont}%
\fi\endgroup%
\resizebox{12cm}{!}{
\begin{picture}(0,0)%
\includegraphics{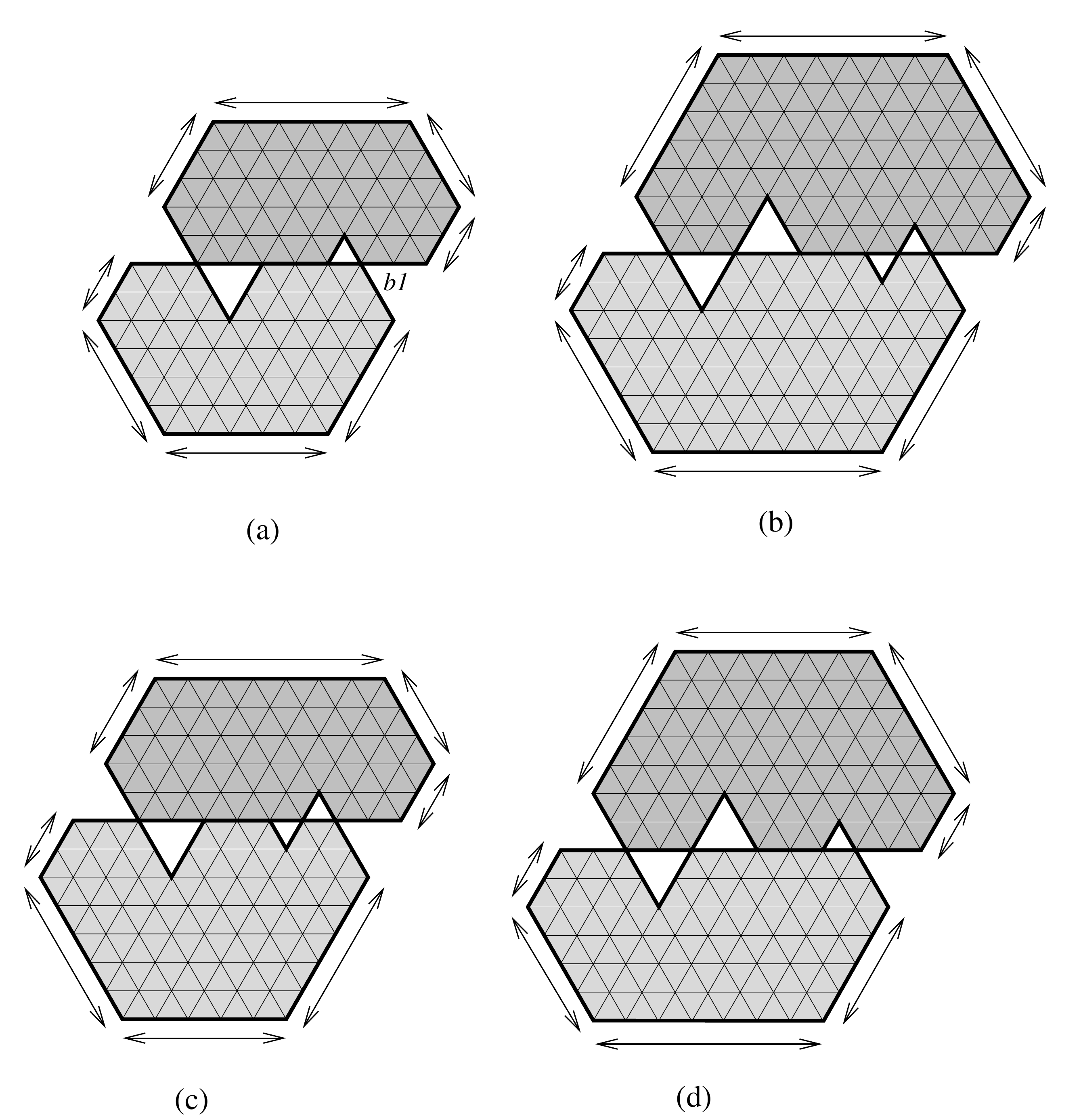}%
\end{picture}%

\begin{picture}(12846,13313)(481,-12980)
\put(1245,-3908){\rotatebox{300.0}{\makebox(0,0)[lb]{\smash{{\SetFigFont{17}{20.4}{\rmdefault}{\mddefault}{\itdefault}{$a_2+b_1$}%
}}}}}
\put(6443,-9822){\makebox(0,0)[lb]{\smash{{\SetFigFont{17}{20.4}{\rmdefault}{\mddefault}{\itdefault}{$x$}%
}}}}
\put(7471,-8532){\rotatebox{60.0}{\makebox(0,0)[lb]{\smash{{\SetFigFont{17}{20.4}{\rmdefault}{\mddefault}{\itdefault}{$t+a_3+b_2$}%
}}}}}
\put(9173,-7047){\makebox(0,0)[lb]{\smash{{\SetFigFont{17}{20.4}{\rmdefault}{\mddefault}{\itdefault}{$z+a_2+b_1$}%
}}}}
\put(11392,-7550){\rotatebox{300.0}{\makebox(0,0)[lb]{\smash{{\SetFigFont{17}{20.4}{\rmdefault}{\mddefault}{\itdefault}{$a1+a_3+b_2$}%
}}}}}
\put(11941,-9815){\makebox(0,0)[lb]{\smash{{\SetFigFont{17}{20.4}{\rmdefault}{\mddefault}{\itdefault}{$t$}%
}}}}
\put(10906,-11705){\rotatebox{60.0}{\makebox(0,0)[lb]{\smash{{\SetFigFont{17}{20.4}{\rmdefault}{\mddefault}{\itdefault}{$x+a_2$}%
}}}}}
\put(6472,-11075){\rotatebox{300.0}{\makebox(0,0)[lb]{\smash{{\SetFigFont{17}{20.4}{\rmdefault}{\mddefault}{\itdefault}{$a_2+b_1$}%
}}}}}
\put(8168,-12410){\makebox(0,0)[lb]{\smash{{\SetFigFont{17}{20.4}{\rmdefault}{\mddefault}{\itdefault}{$z+a_1+a_3+b_2$}%
}}}}
\put(7351,-9665){\makebox(0,0)[lb]{\smash{{\SetFigFont{17}{20.4}{\rmdefault}{\mddefault}{\itdefault}{$a_1$}%
}}}}
\put(8161,-10077){\makebox(0,0)[lb]{\smash{{\SetFigFont{17}{20.4}{\rmdefault}{\mddefault}{\itdefault}{$a_2$}%
}}}}
\put(8911,-9657){\makebox(0,0)[lb]{\smash{{\SetFigFont{17}{20.4}{\rmdefault}{\mddefault}{\itdefault}{$a_3$}%
}}}}
\put(10988,-10107){\makebox(0,0)[lb]{\smash{{\SetFigFont{17}{20.4}{\rmdefault}{\mddefault}{\itdefault}{$b_1$}%
}}}}
\put(10186,-10077){\makebox(0,0)[lb]{\smash{{\SetFigFont{17}{20.4}{\rmdefault}{\mddefault}{\itdefault}{$b_2$}%
}}}}
\put(1499,-9292){\makebox(0,0)[lb]{\smash{{\SetFigFont{17}{20.4}{\rmdefault}{\mddefault}{\itdefault}{$a_1$}%
}}}}
\put(2354,-9705){\makebox(0,0)[lb]{\smash{{\SetFigFont{17}{20.4}{\rmdefault}{\mddefault}{\itdefault}{$a_2$}%
}}}}
\put(4762,-9772){\makebox(0,0)[lb]{\smash{{\SetFigFont{17}{20.4}{\rmdefault}{\mddefault}{\itdefault}{$b_1$}%
}}}}
\put(4102,-9742){\makebox(0,0)[lb]{\smash{{\SetFigFont{17}{20.4}{\rmdefault}{\mddefault}{\itdefault}{$b_2$}%
}}}}
\put(3712,-9304){\makebox(0,0)[lb]{\smash{{\SetFigFont{17}{20.4}{\rmdefault}{\mddefault}{\itdefault}{$b_3$}%
}}}}
\put(1589,-8261){\rotatebox{60.0}{\makebox(0,0)[lb]{\smash{{\SetFigFont{17}{20.4}{\rmdefault}{\mddefault}{\itdefault}{$t+b_2$}%
}}}}}
\put(3037,-7346){\makebox(0,0)[lb]{\smash{{\SetFigFont{17}{20.4}{\rmdefault}{\mddefault}{\itdefault}{$z+a_2+b_1+b_3$}%
}}}}
\put(5482,-7706){\rotatebox{300.0}{\makebox(0,0)[lb]{\smash{{\SetFigFont{17}{20.4}{\rmdefault}{\mddefault}{\itdefault}{$a_1+b_2$}%
}}}}}
\put(5827,-9386){\makebox(0,0)[lb]{\smash{{\SetFigFont{17}{20.4}{\rmdefault}{\mddefault}{\itdefault}{$t$}%
}}}}
\put(4522,-11771){\rotatebox{60.0}{\makebox(0,0)[lb]{\smash{{\SetFigFont{17}{20.4}{\rmdefault}{\mddefault}{\itdefault}{$x+a_2+b_3$}%
}}}}}
\put(2377,-12289){\makebox(0,0)[lb]{\smash{{\SetFigFont{17}{20.4}{\rmdefault}{\mddefault}{\itdefault}{$z+a_1+b_2$}%
}}}}
\put(496,-10606){\rotatebox{300.0}{\makebox(0,0)[lb]{\smash{{\SetFigFont{17}{20.4}{\rmdefault}{\mddefault}{\itdefault}{$a_1+b_1+b_3$}%
}}}}}
\put(610,-9639){\makebox(0,0)[lb]{\smash{{\SetFigFont{17}{20.4}{\rmdefault}{\mddefault}{\itdefault}{$x$}%
}}}}
\put(7007,-2736){\makebox(0,0)[lb]{\smash{{\SetFigFont{17}{20.4}{\rmdefault}{\mddefault}{\itdefault}{$x$}%
}}}}
\put(7915,-1476){\rotatebox{60.0}{\makebox(0,0)[lb]{\smash{{\SetFigFont{17}{20.4}{\rmdefault}{\mddefault}{\itdefault}{$t+a_3+b_2$}%
}}}}}
\put(9436, 14){\makebox(0,0)[lb]{\smash{{\SetFigFont{17}{20.4}{\rmdefault}{\mddefault}{\itdefault}{$z+a_2+b_1+b_3$}%
}}}}
\put(12211,-381){\rotatebox{300.0}{\makebox(0,0)[lb]{\smash{{\SetFigFont{17}{20.4}{\rmdefault}{\mddefault}{\itdefault}{$a_1+a_3+b_2$}%
}}}}}
\put(12876,-2635){\makebox(0,0)[lb]{\smash{{\SetFigFont{17}{20.4}{\rmdefault}{\mddefault}{\itdefault}{$t$}%
}}}}
\put(11642,-4956){\rotatebox{60.0}{\makebox(0,0)[lb]{\smash{{\SetFigFont{17}{20.4}{\rmdefault}{\mddefault}{\itdefault}{$x+a_2+b_3$}%
}}}}}
\put(8852,-5571){\makebox(0,0)[lb]{\smash{{\SetFigFont{17}{20.4}{\rmdefault}{\mddefault}{\itdefault}{$z+a_1+a_3+b_2$}%
}}}}
\put(6970,-3996){\rotatebox{300.0}{\makebox(0,0)[lb]{\smash{{\SetFigFont{17}{20.4}{\rmdefault}{\mddefault}{\itdefault}{$a_2+b_1+b_3$}%
}}}}}
\put(7832,-2530){\makebox(0,0)[lb]{\smash{{\SetFigFont{17}{20.4}{\rmdefault}{\mddefault}{\itdefault}{$a_1$}%
}}}}
\put(8680,-2987){\makebox(0,0)[lb]{\smash{{\SetFigFont{17}{20.4}{\rmdefault}{\mddefault}{\itdefault}{$a_2$}%
}}}}
\put(9407,-2552){\makebox(0,0)[lb]{\smash{{\SetFigFont{17}{20.4}{\rmdefault}{\mddefault}{\itdefault}{$a_3$}%
}}}}
\put(11838,-3017){\makebox(0,0)[lb]{\smash{{\SetFigFont{17}{20.4}{\rmdefault}{\mddefault}{\itdefault}{$b_1$}%
}}}}
\put(11170,-2987){\makebox(0,0)[lb]{\smash{{\SetFigFont{17}{20.4}{\rmdefault}{\mddefault}{\itdefault}{$b_2$}%
}}}}
\put(10818,-2552){\makebox(0,0)[lb]{\smash{{\SetFigFont{17}{20.4}{\rmdefault}{\mddefault}{\itdefault}{$b_3$}%
}}}}
\put(1359,-3050){\makebox(0,0)[lb]{\smash{{\SetFigFont{17}{20.4}{\rmdefault}{\mddefault}{\itdefault}{$x$}%
}}}}
\put(2259,-1674){\rotatebox{60.0}{\makebox(0,0)[lb]{\smash{{\SetFigFont{17}{20.4}{\rmdefault}{\mddefault}{\itdefault}{$t+b_2$}%
}}}}}
\put(3676,-707){\makebox(0,0)[lb]{\smash{{\SetFigFont{17}{20.4}{\rmdefault}{\mddefault}{\itdefault}{$z+a_2+b_1$}%
}}}}
\put(5776,-1044){\rotatebox{300.0}{\makebox(0,0)[lb]{\smash{{\SetFigFont{17}{20.4}{\rmdefault}{\mddefault}{\itdefault}{$a_1+b_2$}%
}}}}}
\put(6121,-2754){\makebox(0,0)[lb]{\smash{{\SetFigFont{17}{20.4}{\rmdefault}{\mddefault}{\itdefault}{$t$}%
}}}}
\put(5049,-4667){\rotatebox{60.0}{\makebox(0,0)[lb]{\smash{{\SetFigFont{17}{20.4}{\rmdefault}{\mddefault}{\itdefault}{$x+a_2$}%
}}}}}
\put(3069,-5274){\makebox(0,0)[lb]{\smash{{\SetFigFont{17}{20.4}{\rmdefault}{\mddefault}{\itdefault}{$z+a_1+b_2$}%
}}}}
\put(2206,-2653){\makebox(0,0)[lb]{\smash{{\SetFigFont{17}{20.4}{\rmdefault}{\mddefault}{\itdefault}{$a_1$}%
}}}}
\put(3054,-3110){\makebox(0,0)[lb]{\smash{{\SetFigFont{17}{20.4}{\rmdefault}{\mddefault}{\itdefault}{$a_2$}%
}}}}
\put(4344,-3088){\makebox(0,0)[lb]{\smash{{\SetFigFont{17}{20.4}{\rmdefault}{\mddefault}{\itdefault}{$b_2$}%
}}}}
\end{picture}}
  \caption{Base cases for $Q$-regions when $y=0$.}\label{arrayPbase3}
\end{figure}
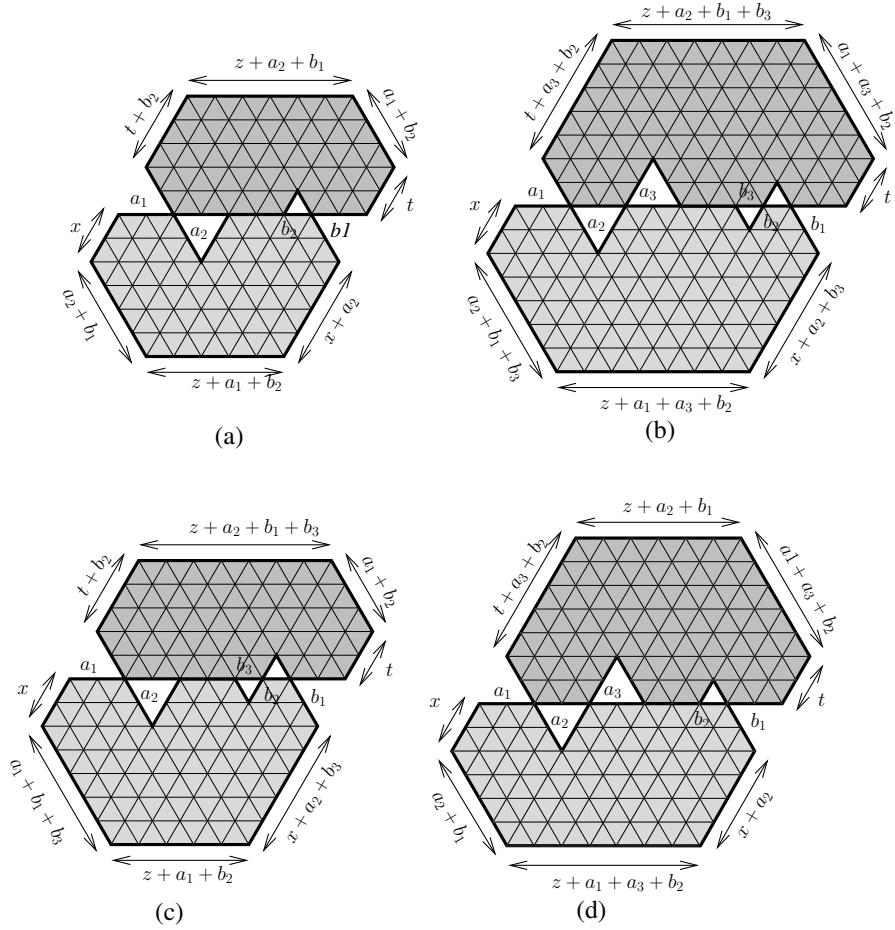

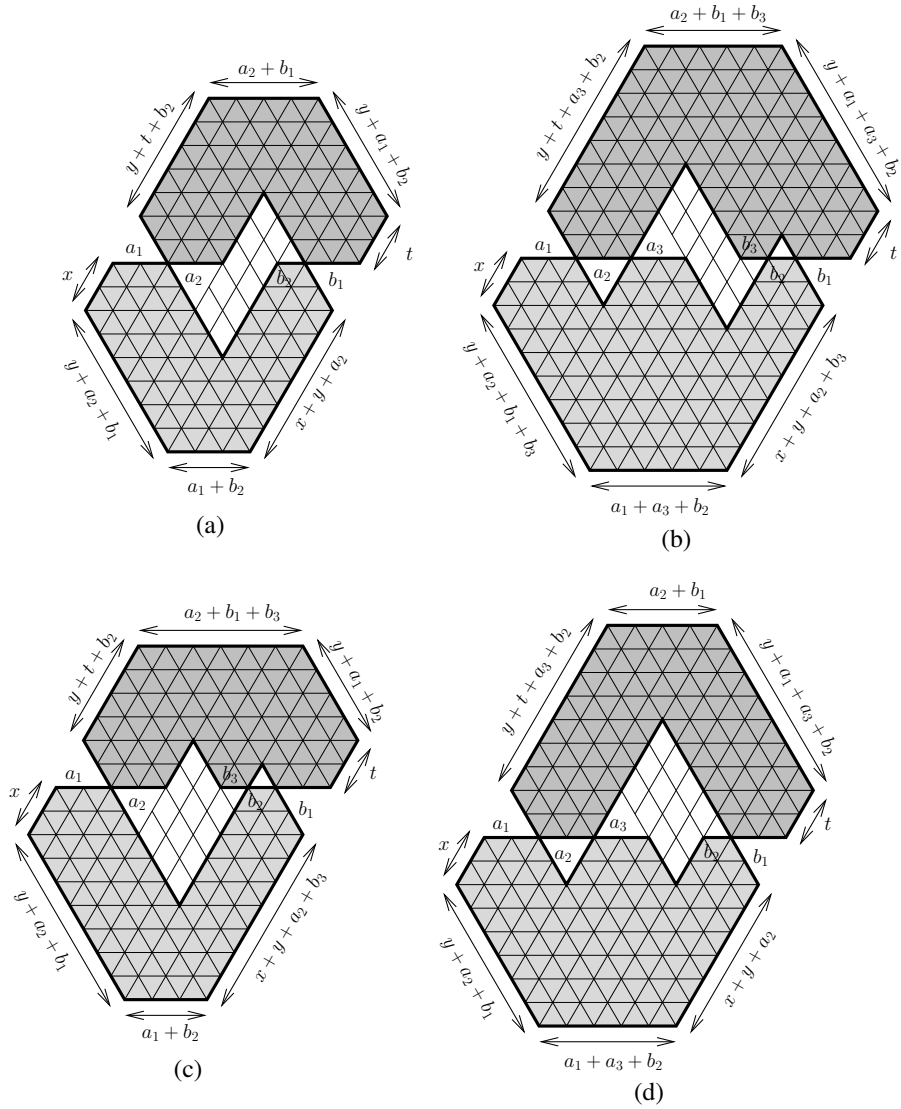
\begin{figure}
  \centering
  \setlength{\unitlength}{3947sp}%
\begingroup\makeatletter\ifx\SetFigFont\undefined%
\gdef\SetFigFont#1#2#3#4#5{%
  \reset@font\fontsize{#1}{#2pt}%
  \fontfamily{#3}\fontseries{#4}\fontshape{#5}%
  \selectfont}%
\fi\endgroup%
\resizebox{12cm}{!}{
\begin{picture}(0,0)%
\includegraphics{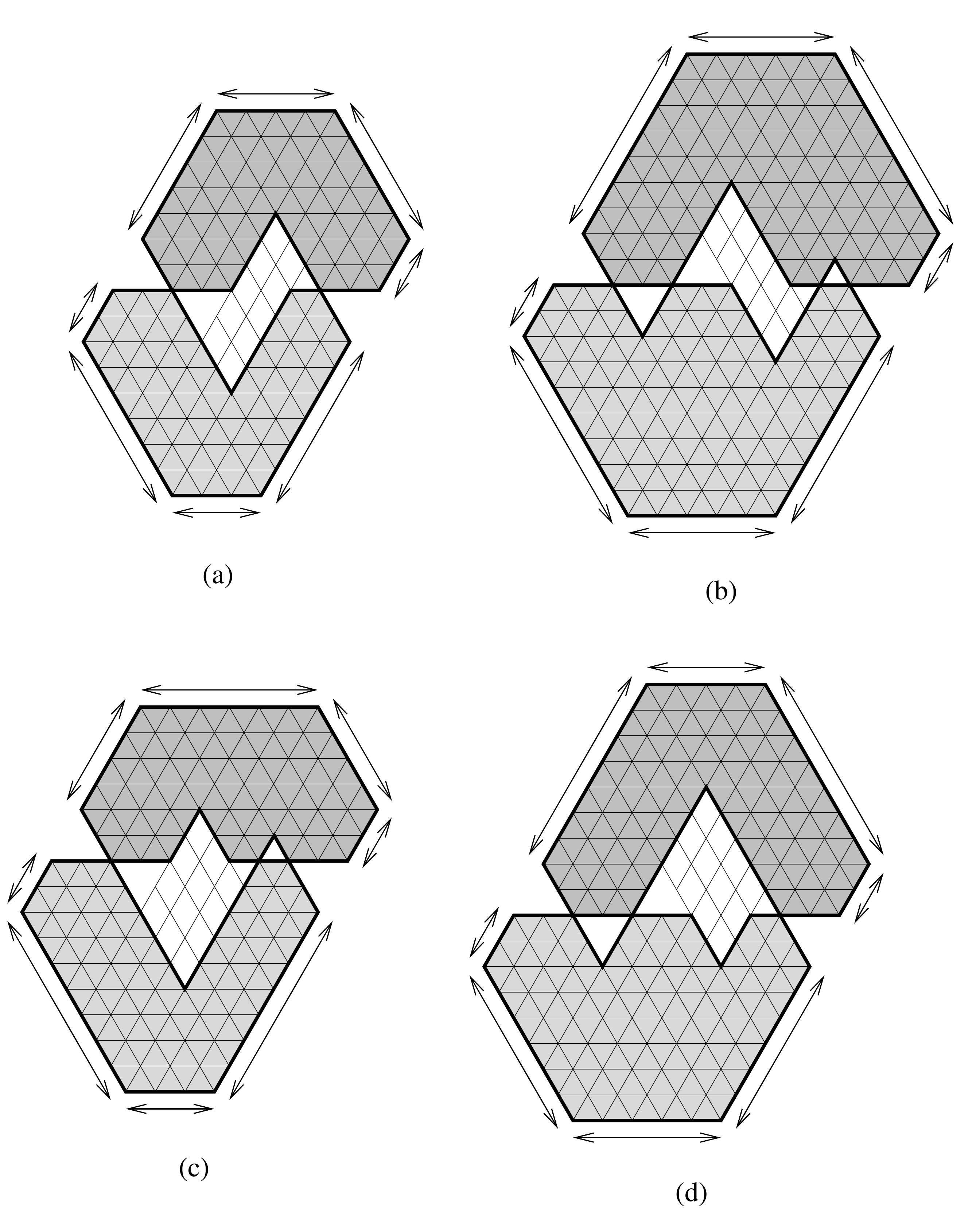}%
\end{picture}%
\begin{picture}(12900,15927)(1623,-15620)
\put(9228,-2032){\rotatebox{60.0}{\makebox(0,0)[lb]{\smash{{\SetFigFont{17}{20.4}{\rmdefault}{\mddefault}{\itdefault}{$y+t+a_3+b_2$}%
}}}}}
\put(11028,-12){\makebox(0,0)[lb]{\smash{{\SetFigFont{17}{20.4}{\rmdefault}{\mddefault}{\itdefault}{$a_2+b_1+b_3$}%
}}}}
\put(13244,-647){\rotatebox{300.0}{\makebox(0,0)[lb]{\smash{{\SetFigFont{17}{20.4}{\rmdefault}{\mddefault}{\itdefault}{$y+a_1+a_3+b_2$}%
}}}}}
\put(14158,-3442){\makebox(0,0)[lb]{\smash{{\SetFigFont{17}{20.4}{\rmdefault}{\mddefault}{\itdefault}{$t$}%
}}}}
\put(8238,-3622){\makebox(0,0)[lb]{\smash{{\SetFigFont{17}{20.4}{\rmdefault}{\mddefault}{\itdefault}{$x$}%
}}}}
\put(12618,-6372){\rotatebox{60.0}{\makebox(0,0)[lb]{\smash{{\SetFigFont{17}{20.4}{\rmdefault}{\mddefault}{\itdefault}{$x+y+a_2+b_3$}%
}}}}}
\put(10186,-7066){\makebox(0,0)[lb]{\smash{{\SetFigFont{17}{20.4}{\rmdefault}{\mddefault}{\itdefault}{$a_1+a_3+b_2$}%
}}}}
\put(8122,-4681){\rotatebox{300.0}{\makebox(0,0)[lb]{\smash{{\SetFigFont{17}{20.4}{\rmdefault}{\mddefault}{\itdefault}{$y+a_2+b_1+b_3$}%
}}}}}
\put(9108,-3312){\makebox(0,0)[lb]{\smash{{\SetFigFont{17}{20.4}{\rmdefault}{\mddefault}{\itdefault}{$a_1$}%
}}}}
\put(9938,-3722){\makebox(0,0)[lb]{\smash{{\SetFigFont{17}{20.4}{\rmdefault}{\mddefault}{\itdefault}{$a_2$}%
}}}}
\put(10688,-3332){\makebox(0,0)[lb]{\smash{{\SetFigFont{17}{20.4}{\rmdefault}{\mddefault}{\itdefault}{$a_3$}%
}}}}
\put(13118,-3772){\makebox(0,0)[lb]{\smash{{\SetFigFont{17}{20.4}{\rmdefault}{\mddefault}{\itdefault}{$b_1$}%
}}}}
\put(12448,-3752){\makebox(0,0)[lb]{\smash{{\SetFigFont{17}{20.4}{\rmdefault}{\mddefault}{\itdefault}{$b_2$}%
}}}}
\put(12048,-3322){\makebox(0,0)[lb]{\smash{{\SetFigFont{17}{20.4}{\rmdefault}{\mddefault}{\itdefault}{$b_3$}%
}}}}
\put(8735,-10258){\rotatebox{60.0}{\makebox(0,0)[lb]{\smash{{\SetFigFont{17}{20.4}{\rmdefault}{\mddefault}{\itdefault}{$y+t+a_3+b_2$}%
}}}}}
\put(10735,-8278){\makebox(0,0)[lb]{\smash{{\SetFigFont{17}{20.4}{\rmdefault}{\mddefault}{\itdefault}{$a_2+b_1$}%
}}}}
\put(12331,-8933){\rotatebox{300.0}{\makebox(0,0)[lb]{\smash{{\SetFigFont{17}{20.4}{\rmdefault}{\mddefault}{\itdefault}{$y+a_1+a_3+b_2$}%
}}}}}
\put(13225,-11628){\makebox(0,0)[lb]{\smash{{\SetFigFont{17}{20.4}{\rmdefault}{\mddefault}{\itdefault}{$t$}%
}}}}
\put(11895,-14168){\rotatebox{60.0}{\makebox(0,0)[lb]{\smash{{\SetFigFont{17}{20.4}{\rmdefault}{\mddefault}{\itdefault}{$x+y+a_2$}%
}}}}}
\put(9525,-15018){\makebox(0,0)[lb]{\smash{{\SetFigFont{17}{20.4}{\rmdefault}{\mddefault}{\itdefault}{$a_1+a_3+b_2$}%
}}}}
\put(7795,-13218){\rotatebox{300.0}{\makebox(0,0)[lb]{\smash{{\SetFigFont{17}{20.4}{\rmdefault}{\mddefault}{\itdefault}{$y+a_2+b_1$}%
}}}}}
\put(7735,-11908){\makebox(0,0)[lb]{\smash{{\SetFigFont{17}{20.4}{\rmdefault}{\mddefault}{\itdefault}{$x$}%
}}}}
\put(8505,-11603){\makebox(0,0)[lb]{\smash{{\SetFigFont{17}{20.4}{\rmdefault}{\mddefault}{\itdefault}{$a_1$}%
}}}}
\put(9385,-12033){\makebox(0,0)[lb]{\smash{{\SetFigFont{17}{20.4}{\rmdefault}{\mddefault}{\itdefault}{$a_2$}%
}}}}
\put(10165,-11593){\makebox(0,0)[lb]{\smash{{\SetFigFont{17}{20.4}{\rmdefault}{\mddefault}{\itdefault}{$a_3$}%
}}}}
\put(11510,-12023){\makebox(0,0)[lb]{\smash{{\SetFigFont{17}{20.4}{\rmdefault}{\mddefault}{\itdefault}{$b_2$}%
}}}}
\put(12210,-12103){\makebox(0,0)[lb]{\smash{{\SetFigFont{17}{20.4}{\rmdefault}{\mddefault}{\itdefault}{$b_1$}%
}}}}
\put(3391,-2235){\rotatebox{60.0}{\makebox(0,0)[lb]{\smash{{\SetFigFont{17}{20.4}{\rmdefault}{\mddefault}{\itdefault}{$y+t+b_2$}%
}}}}}
\put(4891,-805){\makebox(0,0)[lb]{\smash{{\SetFigFont{17}{20.4}{\rmdefault}{\mddefault}{\itdefault}{$a_2+b_1$}%
}}}}
\put(6591,-1275){\rotatebox{300.0}{\makebox(0,0)[lb]{\smash{{\SetFigFont{17}{20.4}{\rmdefault}{\mddefault}{\itdefault}{$y+a_1+b_2$}%
}}}}}
\put(5801,-5915){\rotatebox{60.0}{\makebox(0,0)[lb]{\smash{{\SetFigFont{17}{20.4}{\rmdefault}{\mddefault}{\itdefault}{$x+y+a_2$}%
}}}}}
\put(4171,-6785){\makebox(0,0)[lb]{\smash{{\SetFigFont{17}{20.4}{\rmdefault}{\mddefault}{\itdefault}{$a_1+b_2$}%
}}}}
\put(2457,-4930){\rotatebox{300.0}{\makebox(0,0)[lb]{\smash{{\SetFigFont{17}{20.4}{\rmdefault}{\mddefault}{\itdefault}{$y+a_2+b_1$}%
}}}}}
\put(2391,-3755){\makebox(0,0)[lb]{\smash{{\SetFigFont{17}{20.4}{\rmdefault}{\mddefault}{\itdefault}{$x$}%
}}}}
\put(7276,-3466){\makebox(0,0)[lb]{\smash{{\SetFigFont{17}{20.4}{\rmdefault}{\mddefault}{\itdefault}{$t$}%
}}}}
\put(5028,-11295){\makebox(0,0)[lb]{\smash{{\SetFigFont{17}{20.4}{\rmdefault}{\mddefault}{\itdefault}{$b_2$}%
}}}}
\put(4668,-10905){\makebox(0,0)[lb]{\smash{{\SetFigFont{17}{20.4}{\rmdefault}{\mddefault}{\itdefault}{$b_3$}%
}}}}
\put(2433,-10895){\makebox(0,0)[lb]{\smash{{\SetFigFont{17}{20.4}{\rmdefault}{\mddefault}{\itdefault}{$a_1$}%
}}}}
\put(3333,-11295){\makebox(0,0)[lb]{\smash{{\SetFigFont{17}{20.4}{\rmdefault}{\mddefault}{\itdefault}{$a_2$}%
}}}}
\put(1739,-12535){\rotatebox{300.0}{\makebox(0,0)[lb]{\smash{{\SetFigFont{17}{20.4}{\rmdefault}{\mddefault}{\itdefault}{$y+a_2+b_1$}%
}}}}}
\put(5748,-11405){\makebox(0,0)[lb]{\smash{{\SetFigFont{17}{20.4}{\rmdefault}{\mddefault}{\itdefault}{$b_1$}%
}}}}
\put(3523,-14600){\makebox(0,0)[lb]{\smash{{\SetFigFont{17}{20.4}{\rmdefault}{\mddefault}{\itdefault}{$a_1+b_2$}%
}}}}
\put(1638,-11155){\makebox(0,0)[lb]{\smash{{\SetFigFont{17}{20.4}{\rmdefault}{\mddefault}{\itdefault}{$x$}%
}}}}
\put(5233,-13820){\rotatebox{60.0}{\makebox(0,0)[lb]{\smash{{\SetFigFont{17}{20.4}{\rmdefault}{\mddefault}{\itdefault}{$x+y+a_2+b_3$}%
}}}}}
\put(6763,-10900){\makebox(0,0)[lb]{\smash{{\SetFigFont{17}{20.4}{\rmdefault}{\mddefault}{\itdefault}{$t$}%
}}}}
\put(6219,-8965){\rotatebox{300.0}{\makebox(0,0)[lb]{\smash{{\SetFigFont{17}{20.4}{\rmdefault}{\mddefault}{\itdefault}{$y+a_1+b_2$}%
}}}}}
\put(4123,-8590){\makebox(0,0)[lb]{\smash{{\SetFigFont{17}{20.4}{\rmdefault}{\mddefault}{\itdefault}{$a_2+b_1+b_3$}%
}}}}
\put(2573,-9820){\rotatebox{60.0}{\makebox(0,0)[lb]{\smash{{\SetFigFont{17}{20.4}{\rmdefault}{\mddefault}{\itdefault}{$y+t+b_2$}%
}}}}}
\put(3291,-3360){\makebox(0,0)[lb]{\smash{{\SetFigFont{17}{20.4}{\rmdefault}{\mddefault}{\itdefault}{$a_1$}%
}}}}
\put(4131,-3780){\makebox(0,0)[lb]{\smash{{\SetFigFont{17}{20.4}{\rmdefault}{\mddefault}{\itdefault}{$a_2$}%
}}}}
\put(5436,-3810){\makebox(0,0)[lb]{\smash{{\SetFigFont{17}{20.4}{\rmdefault}{\mddefault}{\itdefault}{$b_2$}%
}}}}
\put(6176,-3840){\makebox(0,0)[lb]{\smash{{\SetFigFont{17}{20.4}{\rmdefault}{\mddefault}{\itdefault}{$b_1$}%
}}}}
\end{picture}}
  \caption{The base cases for $Q$-regions when $z=0$.}\label{arrayPbase2}
\end{figure}

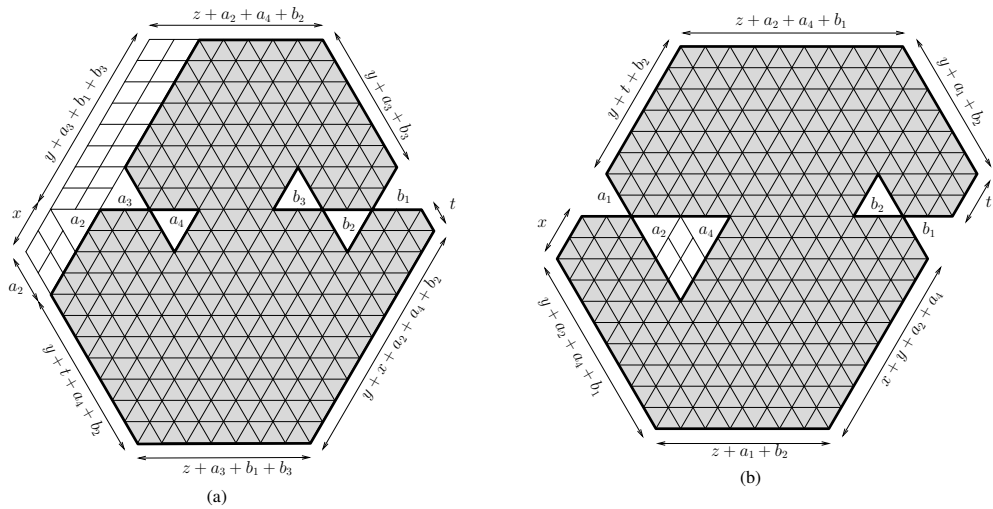
\begin{figure}\centering
\setlength{\unitlength}{3947sp}%
\begingroup\makeatletter\ifx\SetFigFont\undefined%
\gdef\SetFigFont#1#2#3#4#5{%
  \reset@font\fontsize{#1}{#2pt}%
  \fontfamily{#3}\fontseries{#4}\fontshape{#5}%
  \selectfont}%
\fi\endgroup%
\resizebox{13cm}{!}{
\begin{picture}(0,0)%
\includegraphics{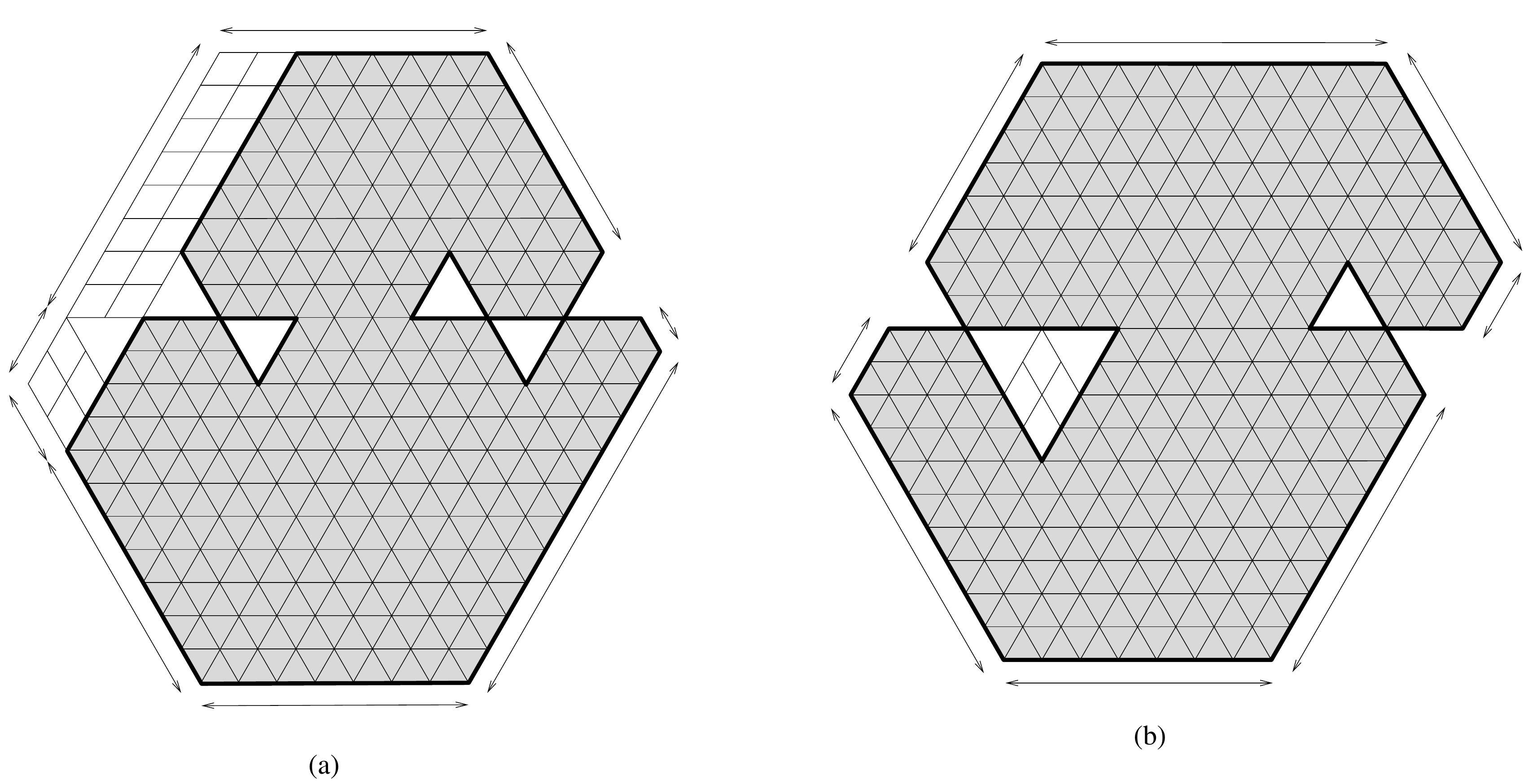}%
\end{picture}%
%
%

\begin{picture}(16300,8376)(3586,-12061)
\put(4629,-7321){\makebox(0,0)[lb]{\smash{{\SetFigFont{16}{14.4}{\rmdefault}{\mddefault}{\updefault}{$a_2$}%
}}}}
\put(5409,-6969){\makebox(0,0)[lb]{\smash{{\SetFigFont{16}{14.4}{\rmdefault}{\mddefault}{\updefault}{$a_3$}%
}}}}
\put(6264,-7321){\makebox(0,0)[lb]{\smash{{\SetFigFont{16}{14.4}{\rmdefault}{\mddefault}{\updefault}{$a_4$}%
}}}}
\put(8311,-6954){\makebox(0,0)[lb]{\smash{{\SetFigFont{16}{14.4}{\rmdefault}{\mddefault}{\updefault}{$b_3$}%
}}}}
\put(10036,-6931){\makebox(0,0)[lb]{\smash{{\SetFigFont{16}{14.4}{\rmdefault}{\mddefault}{\updefault}{$b_1$}%
}}}}
\put(9129,-7396){\makebox(0,0)[lb]{\smash{{\SetFigFont{16}{14.4}{\rmdefault}{\mddefault}{\updefault}{$b_2$}%
}}}}
\put(3661,-7224){\makebox(0,0)[lb]{\smash{{\SetFigFont{16}{14.4}{\rmdefault}{\mddefault}{\updefault}{$x$}%
}}}}
\put(10906,-7066){\makebox(0,0)[lb]{\smash{{\SetFigFont{16}{14.4}{\rmdefault}{\mddefault}{\updefault}{$t$}%
}}}}
\put(9534,-4801){\rotatebox{300.0}{\makebox(0,0)[lb]{\smash{{\SetFigFont{16}{14.4}{\rmdefault}{\mddefault}{\updefault}{$y+a_3+b_3$}%
}}}}}
\put(6646,-3901){\makebox(0,0)[lb]{\smash{{\SetFigFont{16}{14.4}{\rmdefault}{\mddefault}{\updefault}{$z+a_2+a_4+b_2$}%
}}}}
\put(4299,-6286){\rotatebox{60.0}{\makebox(0,0)[lb]{\smash{{\SetFigFont{16}{14.4}{\rmdefault}{\mddefault}{\updefault}{$y+a_3+b_1+b_3$}%
}}}}}
\put(3601,-8439){\makebox(0,0)[lb]{\smash{{\SetFigFont{16}{14.4}{\rmdefault}{\mddefault}{\updefault}{$a_2$}%
}}}}
\put(4203,-9369){\rotatebox{300.0}{\makebox(0,0)[lb]{\smash{{\SetFigFont{16}{14.4}{\rmdefault}{\mddefault}{\updefault}{$y+t+a_4+b_2$}%
}}}}}
\put(6474,-11469){\makebox(0,0)[lb]{\smash{{\SetFigFont{16}{14.4}{\rmdefault}{\mddefault}{\updefault}{$z+a_3+b_1+b_3$}%
}}}}
\put(9564,-10261){\rotatebox{60.0}{\makebox(0,0)[lb]{\smash{{\SetFigFont{16}{14.4}{\rmdefault}{\mddefault}{\updefault}{$y+x+a_2+a_4+b_2$}%
}}}}}
\put(13351,-6924){\makebox(0,0)[lb]{\smash{{\SetFigFont{16}{14.4}{\rmdefault}{\mddefault}{\updefault}{$a_1$}%
}}}}
\put(14251,-7486){\makebox(0,0)[lb]{\smash{{\SetFigFont{16}{14.4}{\rmdefault}{\mddefault}{\updefault}{$a_2$}%
}}}}
\put(15031,-7494){\makebox(0,0)[lb]{\smash{{\SetFigFont{16}{14.4}{\rmdefault}{\mddefault}{\updefault}{$a_4$}%
}}}}
\put(18736,-7486){\makebox(0,0)[lb]{\smash{{\SetFigFont{16}{14.4}{\rmdefault}{\mddefault}{\updefault}{$b_1$}%
}}}}
\put(17881,-7074){\makebox(0,0)[lb]{\smash{{\SetFigFont{16}{14.4}{\rmdefault}{\mddefault}{\updefault}{$b_2$}%
}}}}
\put(15631,-3999){\makebox(0,0)[lb]{\smash{{\SetFigFont{16}{14.4}{\rmdefault}{\mddefault}{\updefault}{$z+a_2+a_4+b_1$}%
}}}}
\put(19006,-4576){\rotatebox{300.0}{\makebox(0,0)[lb]{\smash{{\SetFigFont{16}{14.4}{\rmdefault}{\mddefault}{\updefault}{$y+a_1+b_2$}%
}}}}}
\put(13621,-5566){\rotatebox{60.0}{\makebox(0,0)[lb]{\smash{{\SetFigFont{16}{14.4}{\rmdefault}{\mddefault}{\updefault}{$y+t+b_2$}%
}}}}}
\put(19786,-7036){\makebox(0,0)[lb]{\smash{{\SetFigFont{16}{14.4}{\rmdefault}{\mddefault}{\updefault}{$t$}%
}}}}
\put(12361,-7374){\makebox(0,0)[lb]{\smash{{\SetFigFont{16}{14.4}{\rmdefault}{\mddefault}{\updefault}{$x$}%
}}}}
\put(12406,-8589){\rotatebox{300.0}{\makebox(0,0)[lb]{\smash{{\SetFigFont{16}{14.4}{\rmdefault}{\mddefault}{\updefault}{$y+a_2+a_4+b_1$}%
}}}}}
\put(15234,-11184){\makebox(0,0)[lb]{\smash{{\SetFigFont{16}{14.4}{\rmdefault}{\mddefault}{\updefault}{$z+a_1+b_2$}%
}}}}
\put(18151,-10036){\rotatebox{60.0}{\makebox(0,0)[lb]{\smash{{\SetFigFont{16}{14.4}{\rmdefault}{\mddefault}{\updefault}{$x+y+a_2+a_4$}%
}}}}}
\end{picture}}
\caption{Eliminate lobes of side length $0$ from the ferns.}\label{positivehole}
\end{figure}

\begin{figure}\centering
\setlength{\unitlength}{3947sp}%
\begingroup\makeatletter\ifx\SetFigFont\undefined%
\gdef\SetFigFont#1#2#3#4#5{%
  \reset@font\fontsize{#1}{#2pt}%
  \fontfamily{#3}\fontseries{#4}\fontshape{#5}%
  \selectfont}%
\fi\endgroup%
\resizebox{13cm}{!}{
\begin{picture}(0,0)%
\includegraphics{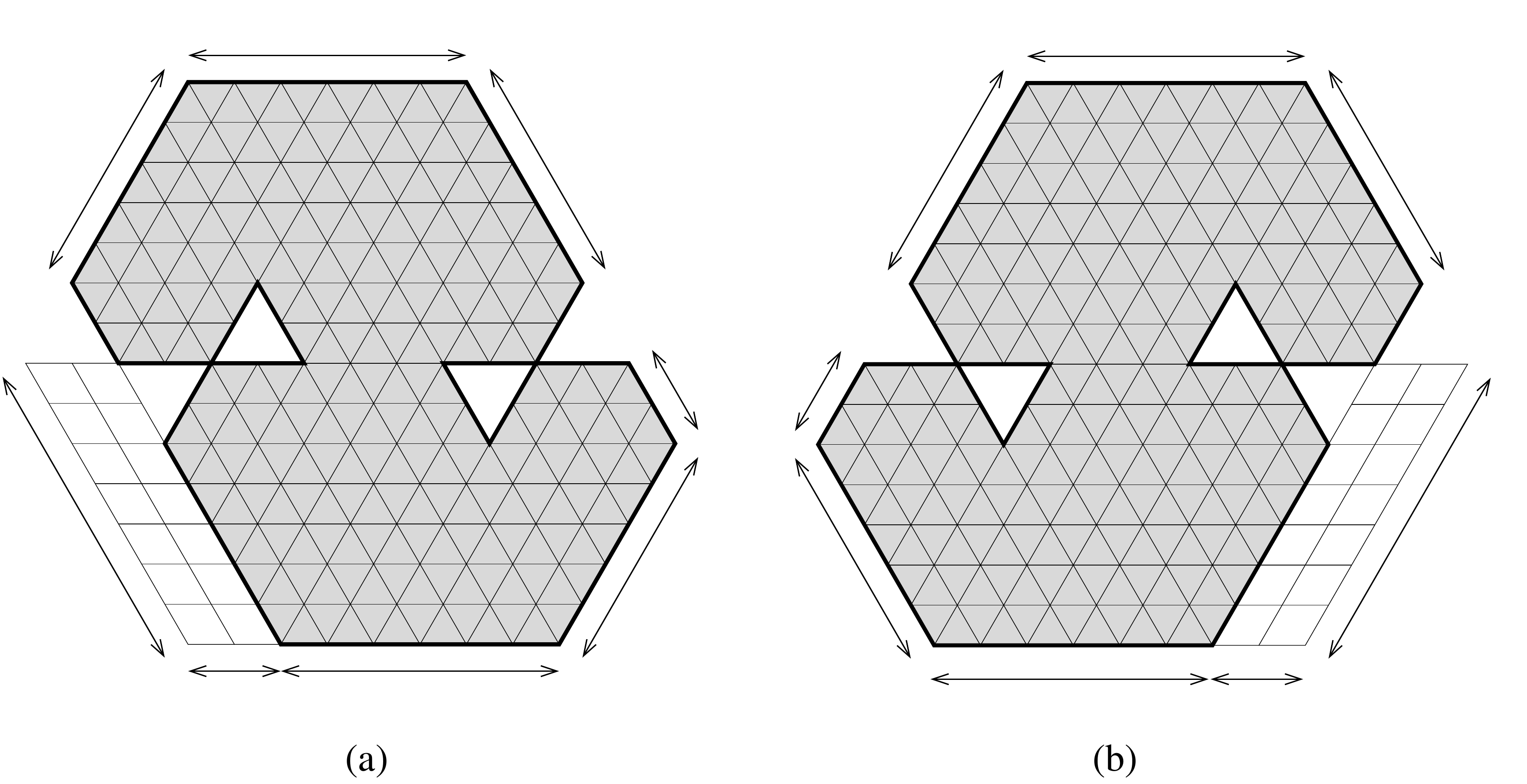}%
\end{picture}%
\begin{picture}(17087,8778)(2367,-8893)
\put(3116,-2371){\rotatebox{60.0}{\makebox(0,0)[lb]{\smash{{\SetFigFont{20}{24.0}{\rmdefault}{\mddefault}{\itdefault}{$y+a_3+b_1$}%
}}}}}
\put(5336,-541){\makebox(0,0)[lb]{\smash{{\SetFigFont{20}{24.0}{\rmdefault}{\mddefault}{\itdefault}{$z+a_2+b_2$}%
}}}}
\put(8336,-1381){\rotatebox{300.0}{\makebox(0,0)[lb]{\smash{{\SetFigFont{20}{24.0}{\rmdefault}{\mddefault}{\itdefault}{$y+a_1+a_3$}%
}}}}}
\put(10106,-4381){\makebox(0,0)[lb]{\smash{{\SetFigFont{20}{24.0}{\rmdefault}{\mddefault}{\itdefault}{$t$}%
}}}}
\put(9596,-7126){\rotatebox{60.0}{\makebox(0,0)[lb]{\smash{{\SetFigFont{20}{24.0}{\rmdefault}{\mddefault}{\itdefault}{$y+a_2+b_2$}%
}}}}}
\put(6356,-8101){\makebox(0,0)[lb]{\smash{{\SetFigFont{20}{24.0}{\rmdefault}{\mddefault}{\itdefault}{$z+a_3+b_1$}%
}}}}
\put(4916,-8146){\makebox(0,0)[lb]{\smash{{\SetFigFont{20}{24.0}{\rmdefault}{\mddefault}{\itdefault}{$a_1$}%
}}}}
\put(2489,-5551){\rotatebox{300.0}{\makebox(0,0)[lb]{\smash{{\SetFigFont{20}{24.0}{\rmdefault}{\mddefault}{\itdefault}{$y+t+a_2+b_2$}%
}}}}}
\put(11261,-4441){\makebox(0,0)[lb]{\smash{{\SetFigFont{20}{24.0}{\rmdefault}{\mddefault}{\itdefault}{$x$}%
}}}}
\put(12566,-2161){\rotatebox{60.0}{\makebox(0,0)[lb]{\smash{{\SetFigFont{20}{24.0}{\rmdefault}{\mddefault}{\itdefault}{$y+b_1+b_3$}%
}}}}}
\put(14861,-571){\makebox(0,0)[lb]{\smash{{\SetFigFont{20}{24.0}{\rmdefault}{\mddefault}{\itdefault}{$z+a_2+b_2$}%
}}}}
\put(17714,-1321){\rotatebox{300.0}{\makebox(0,0)[lb]{\smash{{\SetFigFont{20}{24.0}{\rmdefault}{\mddefault}{\itdefault}{$y+a_1+b_3$}%
}}}}}
\put(16256,-8206){\makebox(0,0)[lb]{\smash{{\SetFigFont{20}{24.0}{\rmdefault}{\mddefault}{\itdefault}{$b_1$}%
}}}}
\put(13811,-8236){\makebox(0,0)[lb]{\smash{{\SetFigFont{20}{24.0}{\rmdefault}{\mddefault}{\itdefault}{$z+a_1+b_3$}%
}}}}
\put(11201,-6121){\rotatebox{300.0}{\makebox(0,0)[lb]{\smash{{\SetFigFont{20}{24.0}{\rmdefault}{\mddefault}{\itdefault}{$y+a_2+b_2$}%
}}}}}
\put(18056,-6871){\rotatebox{60.0}{\makebox(0,0)[lb]{\smash{{\SetFigFont{20}{24.0}{\rmdefault}{\mddefault}{\itdefault}{$x+y+a_2+b_2$}%
}}}}}
\put(2726,-3999){\makebox(0,0)[lb]{\smash{{\SetFigFont{25}{30.0}{\rmdefault}{\mddefault}{\itdefault}{$a_1$}%
}}}}
\put(3986,-4554){\makebox(0,0)[lb]{\smash{{\SetFigFont{20}{24.0}{\rmdefault}{\mddefault}{\itdefault}{$a_2$}%
}}}}
\put(4931,-4074){\makebox(0,0)[lb]{\smash{{\SetFigFont{20}{24.0}{\rmdefault}{\mddefault}{\itdefault}{$a_3$}%
}}}}
\put(7609,-4539){\makebox(0,0)[lb]{\smash{{\SetFigFont{20}{24.0}{\rmdefault}{\mddefault}{\itdefault}{$b_2$}%
}}}}
\put(8674,-4029){\makebox(0,0)[lb]{\smash{{\SetFigFont{20}{24.0}{\rmdefault}{\mddefault}{\itdefault}{$b_1$}%
}}}}
\put(18109,-4029){\makebox(0,0)[lb]{\smash{{\SetFigFont{20}{24.0}{\rmdefault}{\mddefault}{\itdefault}{$b_1$}%
}}}}
\put(16999,-4569){\makebox(0,0)[lb]{\smash{{\SetFigFont{20}{24.0}{\rmdefault}{\mddefault}{\itdefault}{$b_2$}%
}}}}
\put(15919,-4074){\makebox(0,0)[lb]{\smash{{\SetFigFont{20}{24.0}{\rmdefault}{\mddefault}{\itdefault}{$b_3$}%
}}}}
\put(13346,-4524){\makebox(0,0)[lb]{\smash{{\SetFigFont{20}{24.0}{\rmdefault}{\mddefault}{\itdefault}{$a_2$}%
}}}}
\put(12116,-4029){\makebox(0,0)[lb]{\smash{{\SetFigFont{20}{24.0}{\rmdefault}{\mddefault}{\itdefault}{$a_3$}%
}}}}
\end{picture}}
\caption{The $P$-region in the cases when (a) $x=0$ and (b) $t=0$.}
\label{arraybaseQc}
\end{figure}
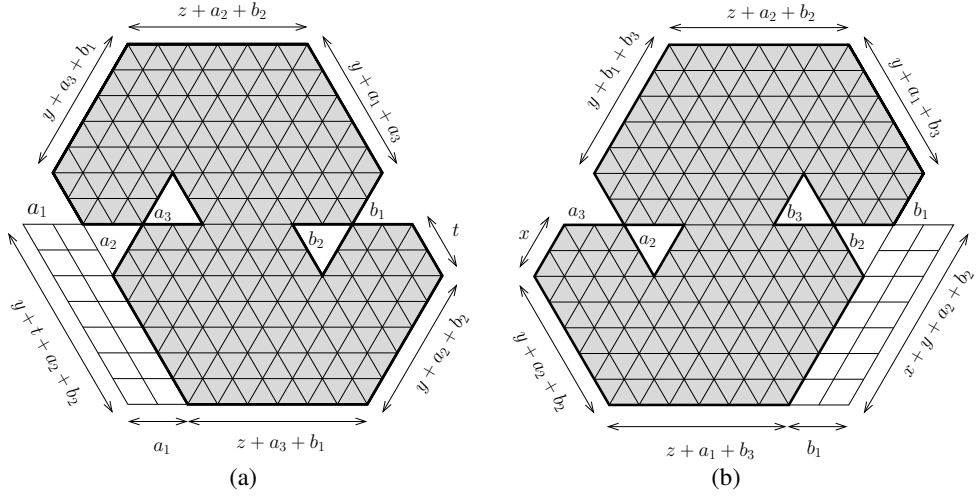

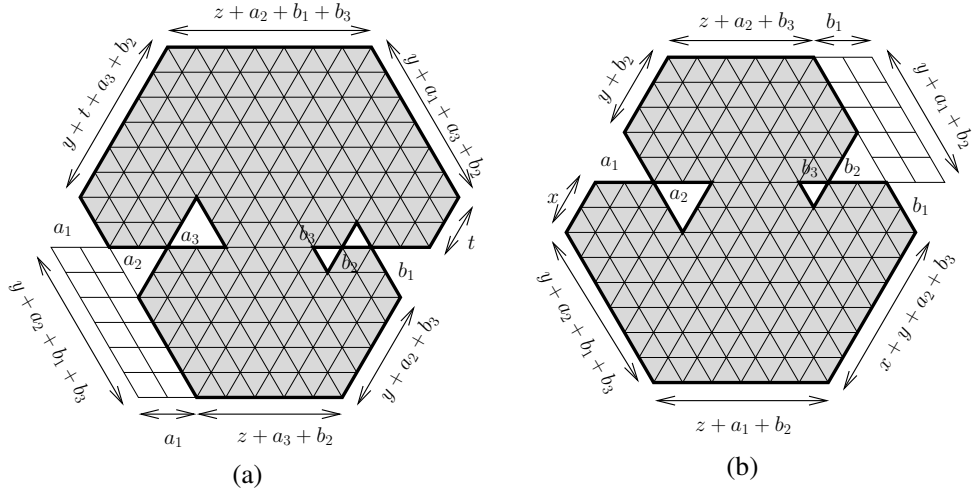
\begin{figure}
  \centering
  \setlength{\unitlength}{3947sp}%
\begingroup\makeatletter\ifx\SetFigFont\undefined%
\gdef\SetFigFont#1#2#3#4#5{%
  \reset@font\fontsize{#1}{#2pt}%
  \fontfamily{#3}\fontseries{#4}\fontshape{#5}%
  \selectfont}%
  \fi\endgroup
  \resizebox{13cm}{!}{
\begin{picture}(0,0)%
\includegraphics{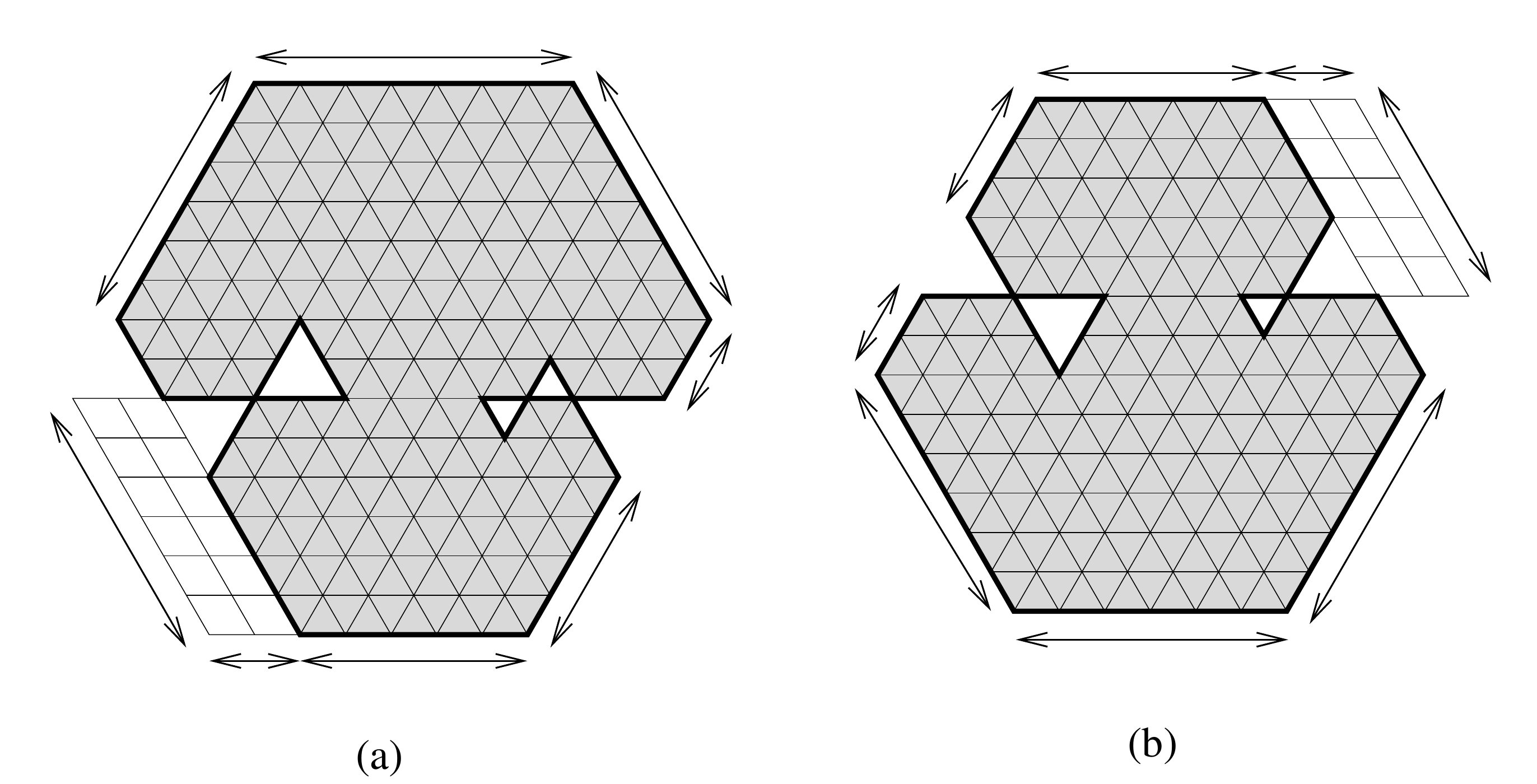}%
\end{picture}%
\begin{picture}(13164,6713)(2247,-7198)
\put(9512,-3294){\makebox(0,0)[lb]{\smash{{\SetFigFont{17}{20.4}{\rmdefault}{\mddefault}{\itdefault}{$x$}%
}}}}
\put(10292,-1944){\rotatebox{60.0}{\makebox(0,0)[lb]{\smash{{\SetFigFont{17}{20.4}{\rmdefault}{\mddefault}{\itdefault}{$y+b_2$}%
}}}}}
\put(11537,-924){\makebox(0,0)[lb]{\smash{{\SetFigFont{17}{20.4}{\rmdefault}{\mddefault}{\itdefault}{$z+a_2+b_3$}%
}}}}
\put(13247,-939){\makebox(0,0)[lb]{\smash{{\SetFigFont{17}{20.4}{\rmdefault}{\mddefault}{\itdefault}{$b_1$}%
}}}}
\put(14447,-1449){\rotatebox{300.0}{\makebox(0,0)[lb]{\smash{{\SetFigFont{17}{20.4}{\rmdefault}{\mddefault}{\itdefault}{$y+a_1+b_2$}%
}}}}}
\put(14057,-5574){\rotatebox{60.0}{\makebox(0,0)[lb]{\smash{{\SetFigFont{17}{20.4}{\rmdefault}{\mddefault}{\itdefault}{$x+y+a_2+b_3$}%
}}}}}
\put(11492,-6369){\makebox(0,0)[lb]{\smash{{\SetFigFont{17}{20.4}{\rmdefault}{\mddefault}{\itdefault}{$z+a_1+b_2$}%
}}}}
\put(9407,-4269){\rotatebox{300.0}{\makebox(0,0)[lb]{\smash{{\SetFigFont{17}{20.4}{\rmdefault}{\mddefault}{\itdefault}{$y+a_2+b_1+b_3$}%
}}}}}
\put(2341,-4396){\rotatebox{300.0}{\makebox(0,0)[lb]{\smash{{\SetFigFont{17}{20.4}{\rmdefault}{\mddefault}{\itdefault}{$y+a_2+b_1+b_3$}%
}}}}}
\put(4382,-6549){\makebox(0,0)[lb]{\smash{{\SetFigFont{17}{20.4}{\rmdefault}{\mddefault}{\itdefault}{$a_1$}%
}}}}
\put(5357,-6519){\makebox(0,0)[lb]{\smash{{\SetFigFont{17}{20.4}{\rmdefault}{\mddefault}{\itdefault}{$z+a_3+b_2$}%
}}}}
\put(7397,-5979){\rotatebox{60.0}{\makebox(0,0)[lb]{\smash{{\SetFigFont{17}{20.4}{\rmdefault}{\mddefault}{\itdefault}{$y+a_2+b_3$}%
}}}}}
\put(8447,-3924){\makebox(0,0)[lb]{\smash{{\SetFigFont{17}{20.4}{\rmdefault}{\mddefault}{\itdefault}{$t$}%
}}}}
\put(7622,-1329){\rotatebox{300.0}{\makebox(0,0)[lb]{\smash{{\SetFigFont{17}{20.4}{\rmdefault}{\mddefault}{\itdefault}{$y+a_1+a_3+b_2$}%
}}}}}
\put(3122,-2589){\rotatebox{60.0}{\makebox(0,0)[lb]{\smash{{\SetFigFont{17}{20.4}{\rmdefault}{\mddefault}{\itdefault}{$y+t+a_3+b_2$}%
}}}}}
\put(5042,-804){\makebox(0,0)[lb]{\smash{{\SetFigFont{17}{20.4}{\rmdefault}{\mddefault}{\itdefault}{$z+a_2+b_1+b_3$}%
}}}}
\put(2904,-3706){\makebox(0,0)[lb]{\smash{{\SetFigFont{17}{20.4}{\rmdefault}{\mddefault}{\itdefault}{$a_1$}%
}}}}
\put(10224,-2806){\makebox(0,0)[lb]{\smash{{\SetFigFont{17}{20.4}{\rmdefault}{\mddefault}{\itdefault}{$a_1$}%
}}}}
\put(7516,-4291){\makebox(0,0)[lb]{\smash{{\SetFigFont{17}{20.4}{\rmdefault}{\mddefault}{\itdefault}{$b_1$}%
}}}}
\put(14401,-3421){\makebox(0,0)[lb]{\smash{{\SetFigFont{17}{20.4}{\rmdefault}{\mddefault}{\itdefault}{$b_1$}%
}}}}
\put(13486,-2896){\makebox(0,0)[lb]{\smash{{\SetFigFont{17}{20.4}{\rmdefault}{\mddefault}{\itdefault}{$b_2$}%
}}}}
\put(12916,-2911){\makebox(0,0)[lb]{\smash{{\SetFigFont{17}{20.4}{\rmdefault}{\mddefault}{\itdefault}{$b_3$}%
}}}}
\put(3841,-4141){\makebox(0,0)[lb]{\smash{{\SetFigFont{17}{20.4}{\rmdefault}{\mddefault}{\itdefault}{$a_2$}%
}}}}
\put(11146,-3241){\makebox(0,0)[lb]{\smash{{\SetFigFont{17}{20.4}{\rmdefault}{\mddefault}{\itdefault}{$a_2$}%
}}}}
\put(4606,-3811){\makebox(0,0)[lb]{\smash{{\SetFigFont{17}{20.4}{\rmdefault}{\mddefault}{\itdefault}{$a_3$}%
}}}}
\put(6189,-3796){\makebox(0,0)[lb]{\smash{{\SetFigFont{17}{20.4}{\rmdefault}{\mddefault}{\itdefault}{$b_3$}%
}}}}
\put(6759,-4186){\makebox(0,0)[lb]{\smash{{\SetFigFont{17}{20.4}{\rmdefault}{\mddefault}{\itdefault}{$b_2$}%
}}}}
\end{picture}}
  \caption{The $Q$-region when (a) $x=0$ and  (b) $t=0$.}\label{arrayPbaseb}
\end{figure}

\begin{proof}[Proof of Theorem \ref{main2}]
By Lemma \ref{ratio}, we only need to show that
\begin{equation}\label{main2eq1b}
\M_1(P_{x,y,z,t}(\textbf{a};\textbf{b}))=q^{A^{(1)}_{x,y,z,t}(\textbf{a};\textbf{b})}\Phi^{q}_{x,y,z,t}(\textbf{a};\textbf{b})
\end{equation}
and
\begin{equation}\label{main2eq2b}
\M_1(Q_{x,y,z,t}(\textbf{a};\textbf{b}))=q^{B^{(1)}_{x,y,z,t}(\textbf{a};\textbf{b})}\Psi^{q}_{x,y,z,t}(\textbf{a};\textbf{b}).
 \end{equation}

The second weight $\wt_2$ (which was not needed in proving the results so far involving only $\wt_1$), does now come up in our proof of these equations (this is the reason $\wt_2$ had to be considered in the first place). The proof is by induction on $x+y+z+t+a+b+\overline{m}+\overline{n}$, where $a=a_1+a_2+\cdots$ and $b=b_1+b_2+\cdots$, and where $\overline{m}$ and $\overline{n}$ are the numbers of positive terms in the sequences $\textbf{a}$ and $\textbf{b}$, respectively. The base cases are the situations when at least one of the parameters $y$, $z$, $\overline{m}$ and $\overline{n}$ is equal to $0$.


We first verify that \eqref{main2eq1b} holds in all these base cases.

If $\overline{n}=0$, we only have the left fern removed from the hexagon, and (\ref{main2eq1b}) follows directly from Theorem \ref{qmain1}. If $\overline{m}=0$, reflecting the region $P_{x,y,z,t}(\emptyset;\textbf{b})$ across a vertical line, we get the region $R_{t,y,z,y+x}(\textbf{b})$  weighted by $\wt_2$. Then (\ref{main2eq1b}) follows from the $M_2$-enumeration formula of equation \eqref{qmain1eq1a}.


Suppose now that $y=0$. One readily checks that the portion of $P_{x,0,z,t}(\textbf{a};\textbf{b})$ above the common axis of the ferns is balanced. Therefore, by the region-splitting lemma \ref{GS}, the weighted tiling count of $P_{x,0,z,t}(\textbf{a};\textbf{b})$ is equal to the product of the weighted tiling counts of the portions that are above and below the fern axis. The top part is a dented semihexagon and the lower part is a dented semihexagon reflected across its base. More precisely, the upper and lower dented semihexagons are: (1) $S(a_1,\dots,a_{m-1},a_m+z+b_{n},b_{n-1},\dots,b_1)$ and  $S(x,a_1,\dots,a_m,z,b_n,\dots,b_1,t)$ if $m$ and $n$ are both even (see Figure \ref{arraybasenew1}(a)), (2) $S(a_1,\dots,a_m,z,b_{n},\dots,b_1)$ and $S(x,a_1,\dots,a_{m-1},a_m+z+b_n,b_{n-1},\dots,b_1,t)$ if $m$ and $n$ are both odd  (see Figure \ref{arraybasenew1}(b)), (3) $S(a_1,\dots,a_{m},z+b_{n},b_{n-1},\dots,b_1)$ and $S(x,a_1,\dots,a_{m-1},a_m+z,b_n,\dots,b_1,t)$ if $m$ is odd and $n$ is even (illustrated in Figure \ref{arraybasenew1}(c)), and (4) $S(a_1,\dots,a_{m-1},a_m+z,b_{n},\dots,b_1)$ and  $S(x,a_1,\dots,a_{m},z+b_n,b_{n-1}\dots,b_1,t)$ if $m$ is even and $n$ is odd (pictured in Figure \ref{arraybasenew1}(d)).

Note that for all parities of $m$ and $n$, if we divide the weight of each right lozenge in the upper part by $q^{x+y+e_a+e_b}$, we get the weight assignment $\wt_1$ on it. Furthermore, if we divide the weight of each right lozenge in the lower part by  $q^{x+y+e_a+e_b+1}$ and reflect the region across its base, we get the weight assignment $\wt_2$ on it, with $q$ eplaced by $q^{-1}$. Then \eqref{main2eq1b} follows from Lemma \ref{semi}.

If $z=0$, there are $y$ unit segments between the rightmost point of the left fern and the leftmost point of the right fern (as the distance between these points is $y+z$ in general). It is not hard to see that, in any lozenge tiling of $P_{x,y,0,t}(\textbf{a};\textbf{b})$, each of these $y$ unit segments must have a vertical lozenge across it (the reason is that the portion of $P_{x,y,0,t}(\textbf{a};\textbf{b})$ that is above the fern axis has an excess of $y$ up-pointing unit triangles compared to down-pointing triangles). After removing these $y$ contiguous forced lozenges, and all the other lozenges they force in their turn, the remaining region disconnects into an upper and a lower part, both of which are dented semihexagons. To be precise, the upper and lower dented semihexagons are: (1) $S(a_1,\dots,a_m,y,b_{n},\dots,b_1)$ and $S(x,a_1,\dots,a_{m-1},a_m+y+b_n,b_{n-1},\dots,b_1,t)$ if $m$ and $n$ are both even (see Figure \ref{arrayQbase}(a)), (2)  $S(a_1,\dots,a_{m-1},a_m+y+b_n,b_{n-1},\dots,b_1)$ and $S(x,a_1,\dots,a_{m},y,b_{n},\dots,b_1,t)$ if $m$ and $n$ are both odd (illustrated in Figure \ref{arrayQbase}(b)), (3)   $S(a_1,\dots,a_{m-1},a_m+y,b_n,\dots,b_1)$ and $S(x,a_1,\dots,a_{m},y+b_{n},b_{n-1},\dots,b_1,t)$ if $m$ is odd and $n$ is even (shown in Figure \ref{arrayQbase}(c)), and (4) $S(a_1,\dots,a_{m},y+b_n,b_{n-1},\dots,b_1)$ and $S(x,a_1,\dots,a_{m-1},a_m+y,b_{n}\dots,b_1,t)$ if $m$ is even and $n$ is odd (see Figure \ref{arrayQbase}(d)). Similarly to the case $y=0$, we obtain (\ref{main2eq1b}) from Lemma \ref{semi}.

We can verify (\ref{main2eq2b}) in a similar fashion. When $\overline{n}=0$, (\ref{main2eq2b}) follows directly from Theorem \ref{qmain1}. To verify \eqref{main2eq2b} for $\overline{m}=0$, we rotate the region  $Q_{x,y,z,t}(\emptyset;\textbf{b})$ by $180^{\circ}$ to get the region $R_{t,x+y,z,y}(\textbf{b})$. By dividing the resulting weights of all right lozenges in this region by $q^{x+2y+t+a+1}$, we obtain the weight assignment $\wt_1$, but with $q$ replaced by $q^{-1}$ (recall that $a=a_1+a_2+\cdots$). Then (\ref{main2eq2b}) follows from equation \eqref{qmain1eq2a}.

When $y=0$, we can also split up our $Q$-region into two semihexagons as in Figure \ref{arrayPbase3}, using Lemma \ref{GS}; when $z=0$, we can partition our region into two semihexagons and several forced vertical lozenges in the middle as in Figure \ref{arrayPbase2}. Then (\ref{main2eq2b}) follows again from Lemma \ref{semi}.

\medskip

For the induction step, we assume therefore that $y,z,\overline{m},\overline{n}\geq1$, and that (\ref{main2eq1b}) and (\ref{main2eq1b}) hold for any $P$-and $Q$-regions with the sum of eight parameters strictly less than $x+y+z+t+a+b+\overline{m}+\overline{n}$. Before considering the recurrences for our regions by applying Kuo condensation, we have two notices below.

\medskip

We can assume that all the terms $a_i$'s and $b_j$'s are positive, for $i=1,2,\dotsc,m$ and $j=1,2,\dotsc,n$. Indeed, if an even number of initial terms in a sequence are equal to $0$, say $a_1=a_2=\dotsc=a_{2l}=0$, we can simply replace $\textbf{a}$ by $\textbf{a}'=(a_{2l+1},\dotsc,a_m)$ to reduce the sum of the seven parameters above. If we have $a_1=a_2=\dotsc=a_{2l-1}=0$, and $a_{2l}>0$, we can remove several forced lozenges along the northwest side of the region to have a new region of the same type with the sum of the seven parameters smaller (see Figure \ref{positivehole}(a) for a $P$-region, similarly for a $Q$-region). Next, if the initial term of a sequence is positive, but some middle term is equal to $0$, we can remove forced lozenges and combine two neighbor holes as in Figure \ref{positivehole}(b) (for a $Q$-region, similarly for a $P$-region), this way we again get  a new region with the sum of the eight parameters strictly less than $x+y+z+t+a+b+\overline{m}+\overline{n}$. Therefore, without loss of generality, we can assume in the rest of this proof that all terms in $\textbf{a}$ and $\textbf{b}$ are positive (i.e. $m=\overline{m}$ and $n=\overline{n}$).

We can   assume further that $x,t>0$. Indeed, if $x=0$, we remove the forced lozenges along the southwest side of the region $P_{0,y,z,t}(\textbf{a};\textbf{b})$, reflect the resulting region about a vertical line, and  obtain the region $Q_{t,y,z,a_1}(\textbf{b};a_2,a_2,\dots,a_{n})$ weighted by $\wt_2$ (see Figure \ref{arraybaseQc}(a)). This new $Q$-region has the sum of the eight parameters 1 less than $x+y+z+t+a+b+m+n$, and the statement follows by the induction hypothesis. For the region $Q_{0,y,z,t}(\textbf{a};\textbf{b})$, removing the forced lozenges along the southwest side and reflecting the resulting region about a horizontal line, we get a weighted version of the region $P_{t,y,z,a_1}(a_2,a_2,\dots,a_{n};\textbf{b})$ (see Figure \ref{arrayPbaseb}(a)). To relate this weight to $\wt_1$,
we divide the weight of each left lozenge by $q^{2y+t+a+1}$,  and reflecting the thus weighted region across a vertical line. We obtain this way the weight assignment $\wt_1$ where $q$ is replaced by $q^{-1}$. Then the theorem follows again by the induction hypothesis. We can verify (\ref{main2eq2b}) for $t=0$ similarly, having now forced lozenges along the southeast side (see Figures  \ref{arraybaseQc}(b) and \ref{arrayPbaseb}(b)).

\medskip

We now apply Kuo's Theorem \ref{kuothm3} to the dual graph $G$ of the region $\overline{P}$ obtained from $P_{x,y,z,t}(\textbf{a};\textbf{b})$ (for $x,t>0$) by adding a strip of unit triangles on top (see Figure \ref{KuoarrayQ}).

\begin{figure}\centering
\setlength{\unitlength}{3947sp}%
\begingroup\makeatletter\ifx\SetFigFont\undefined%
\gdef\SetFigFont#1#2#3#4#5{%
  \reset@font\fontsize{#1}{#2pt}%
  \fontfamily{#3}\fontseries{#4}\fontshape{#5}%
  \selectfont}%
\fi\endgroup%
\resizebox{7cm}{!}{
\begin{picture}(0,0)%
\includegraphics{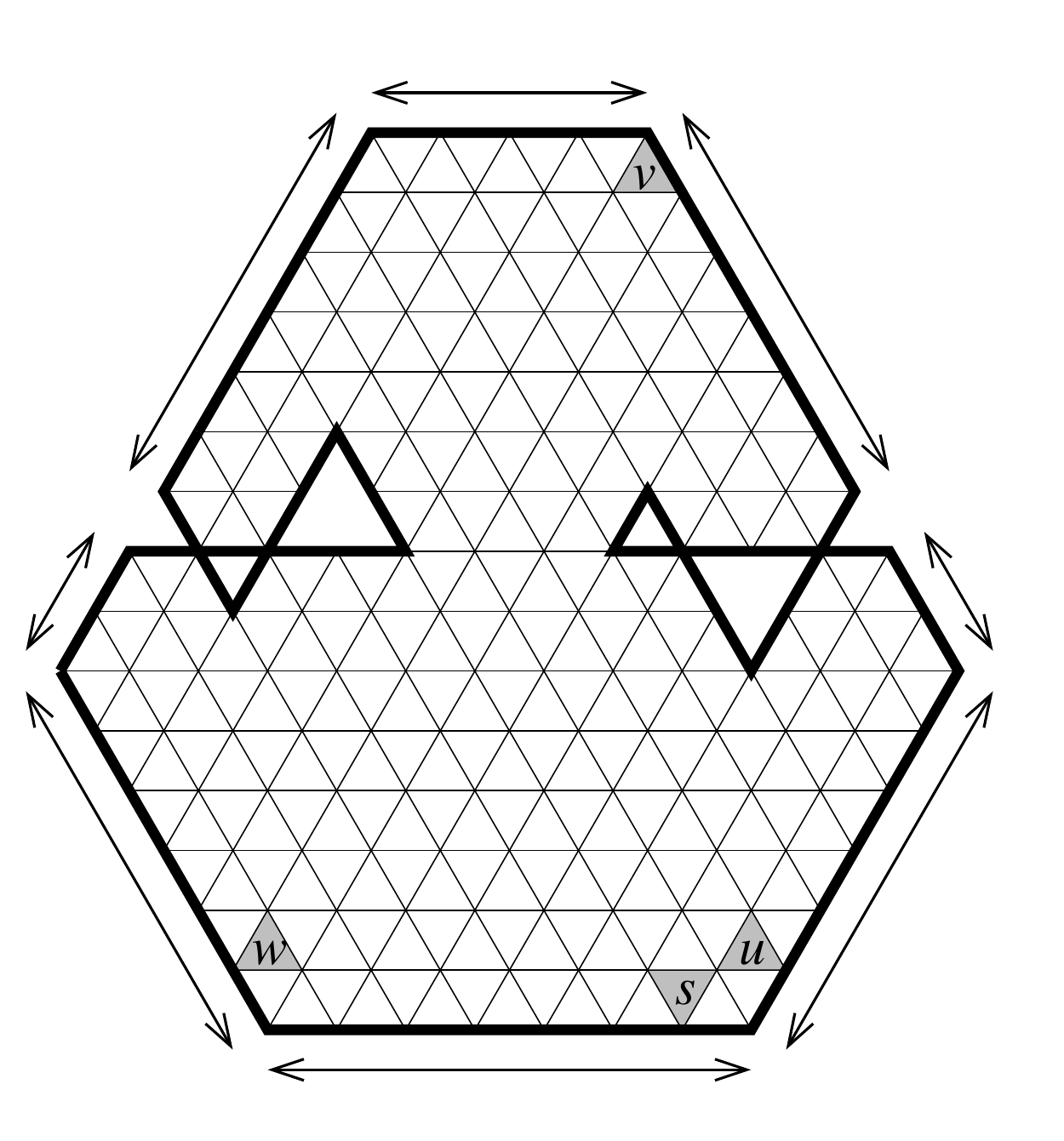}%
\end{picture}%
\begin{picture}(5993,6420)(2000,-7884)
\put(2709,-3826){\rotatebox{60.0}{\makebox(0,0)[lb]{\smash{{\SetFigFont{17}{20.4}{\rmdefault}{\mddefault}{\itdefault}{$y+a_3+b_1+b_3+1$}%
}}}}}
\put(4261,-1816){\makebox(0,0)[lb]{\smash{{\SetFigFont{17}{20.4}{\rmdefault}{\mddefault}{\itdefault}{$z+a_2+b_2-1$}%
}}}}
\put(6127,-2251){\rotatebox{300.0}{\makebox(0,0)[lb]{\smash{{\SetFigFont{17}{20.4}{\rmdefault}{\mddefault}{\itdefault}{$y+a_1+a_3+b_3+1$}%
}}}}}
\put(7621,-4719){\makebox(0,0)[lb]{\smash{{\SetFigFont{17}{20.4}{\rmdefault}{\mddefault}{\itdefault}{$t$}%
}}}}
\put(6946,-7081){\rotatebox{60.0}{\makebox(0,0)[lb]{\smash{{\SetFigFont{17}{20.4}{\rmdefault}{\mddefault}{\itdefault}{$x+y+a_2+b_2$}%
}}}}}
\put(3894,-7869){\makebox(0,0)[lb]{\smash{{\SetFigFont{17}{20.4}{\rmdefault}{\mddefault}{\itdefault}{$z+a_1+a_3+b_1+b_3$}%
}}}}
\put(2094,-6016){\rotatebox{300.0}{\makebox(0,0)[lb]{\smash{{\SetFigFont{17}{20.4}{\rmdefault}{\mddefault}{\itdefault}{$y+t+a_2+b_2$}%
}}}}}
\put(2124,-4711){\makebox(0,0)[lb]{\smash{{\SetFigFont{17}{20.4}{\rmdefault}{\mddefault}{\itdefault}{$x$}%
}}}}
\end{picture}}
\caption{Applying Kuo condensation to the region $\overline{P}$.}
\label{KuoarrayQ}
\end{figure}

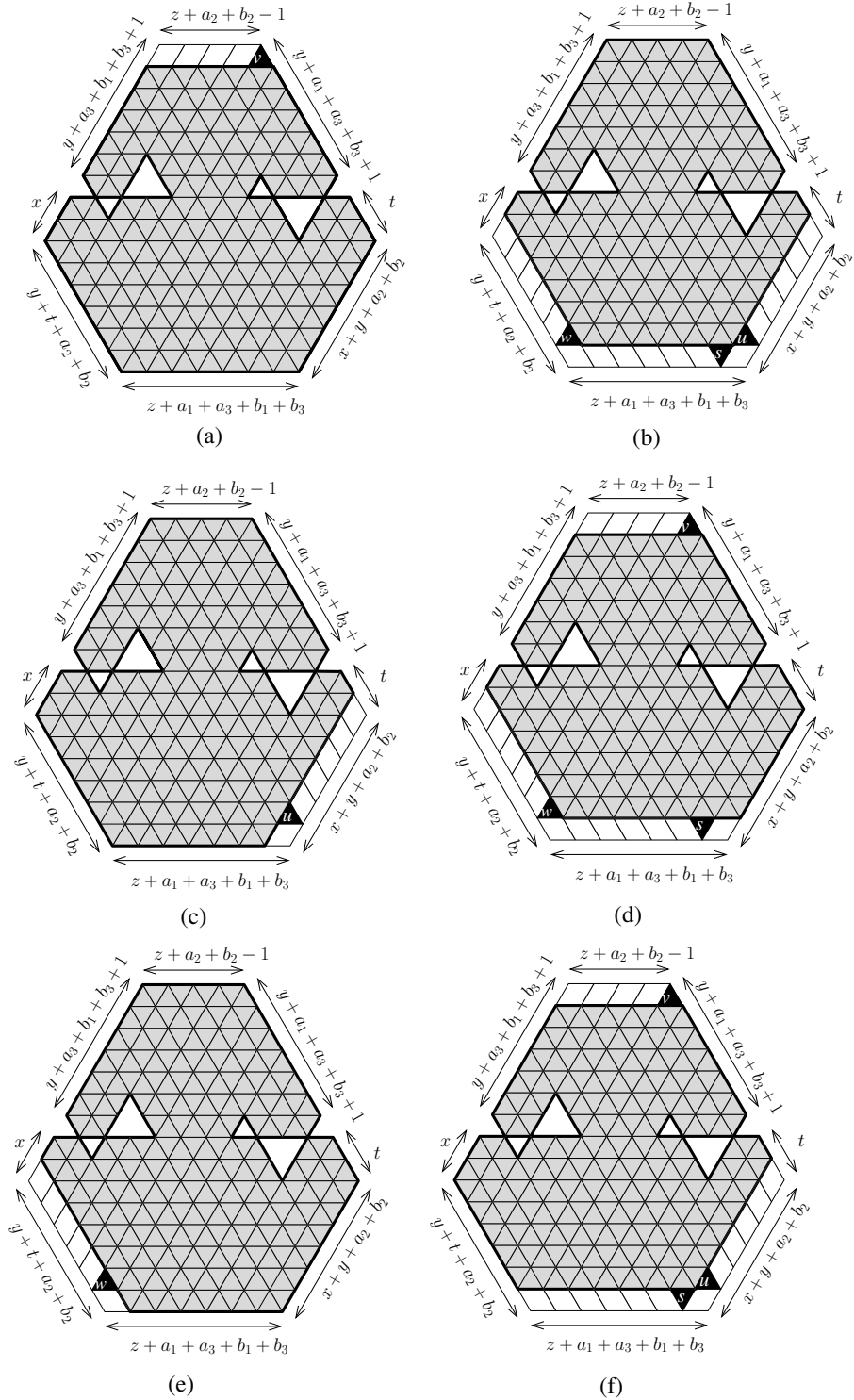
\begin{figure}\centering
\setlength{\unitlength}{3947sp}%
\begingroup\makeatletter\ifx\SetFigFont\undefined%
\gdef\SetFigFont#1#2#3#4#5{%
  \reset@font\fontsize{#1}{#2pt}%
  \fontfamily{#3}\fontseries{#4}\fontshape{#5}%
  \selectfont}%
\fi\endgroup%
\resizebox{12cm}{!}{
\begin{picture}(0,0)%
\includegraphics{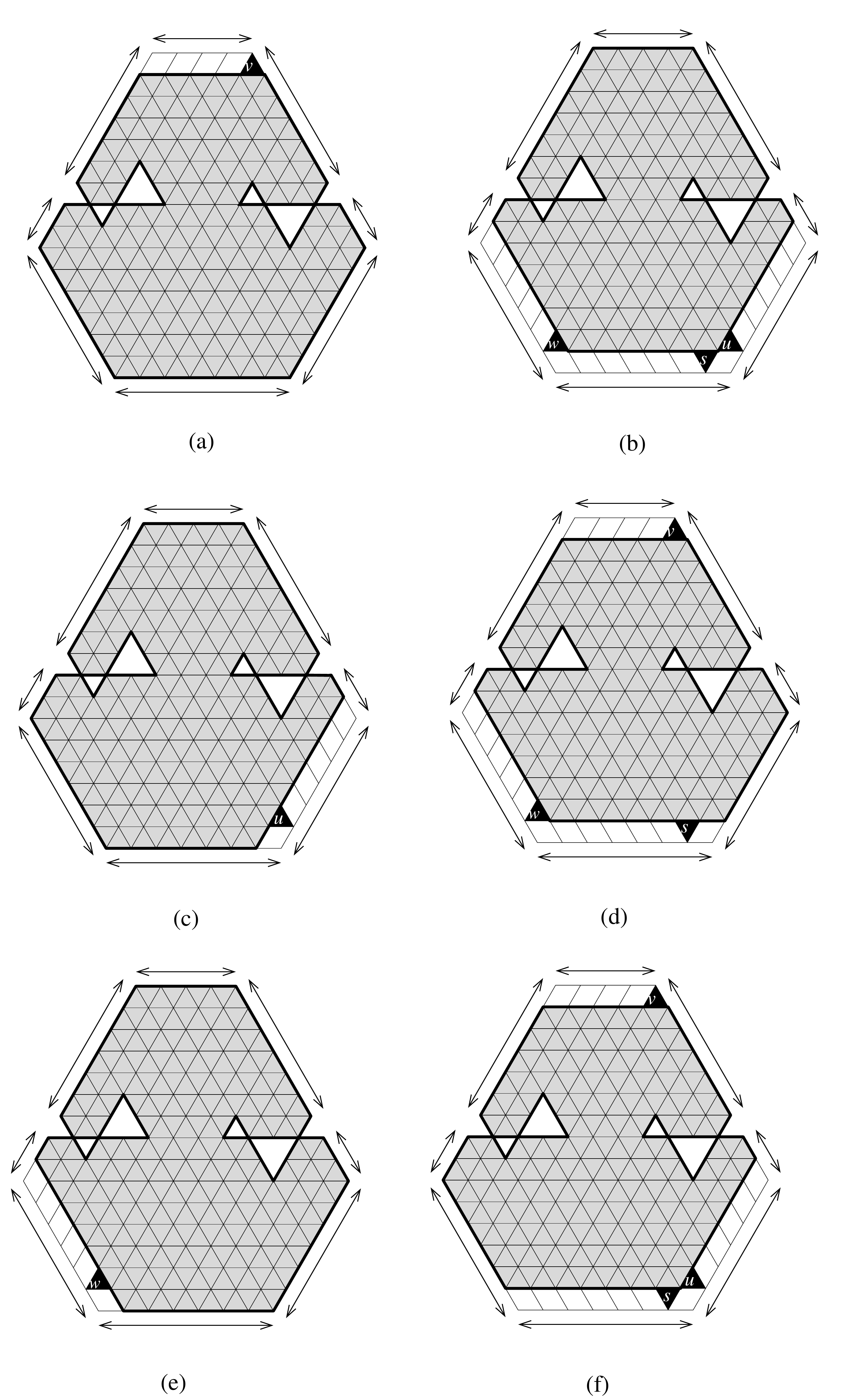}%
\end{picture}%

\begin{picture}(13143,21802)(753,-21515)
\put(1742,-2150){\rotatebox{60.0}{\makebox(0,0)[lb]{\smash{{\SetFigFont{17}{20.4}{\rmdefault}{\mddefault}{\itdefault}{$y+a_3+b_1+b_3+1$}%
}}}}}
\put(3294,-140){\makebox(0,0)[lb]{\smash{{\SetFigFont{17}{20.4}{\rmdefault}{\mddefault}{\itdefault}{$z+a_2+b_2-1$}%
}}}}
\put(5213,-643){\rotatebox{300.0}{\makebox(0,0)[lb]{\smash{{\SetFigFont{17}{20.4}{\rmdefault}{\mddefault}{\itdefault}{$y+a_1+a_3+b_3+1$}%
}}}}}
\put(6654,-3043){\makebox(0,0)[lb]{\smash{{\SetFigFont{17}{20.4}{\rmdefault}{\mddefault}{\itdefault}{$t$}%
}}}}
\put(5979,-5405){\rotatebox{60.0}{\makebox(0,0)[lb]{\smash{{\SetFigFont{17}{20.4}{\rmdefault}{\mddefault}{\itdefault}{$x+y+a_2+b_2$}%
}}}}}
\put(2927,-6193){\makebox(0,0)[lb]{\smash{{\SetFigFont{17}{20.4}{\rmdefault}{\mddefault}{\itdefault}{$z+a_1+a_3+b_1+b_3$}%
}}}}
\put(1127,-4340){\rotatebox{300.0}{\makebox(0,0)[lb]{\smash{{\SetFigFont{17}{20.4}{\rmdefault}{\mddefault}{\itdefault}{$y+t+a_2+b_2$}%
}}}}}
\put(1157,-3035){\makebox(0,0)[lb]{\smash{{\SetFigFont{17}{20.4}{\rmdefault}{\mddefault}{\itdefault}{$x$}%
}}}}
\put(8612,-2075){\rotatebox{60.0}{\makebox(0,0)[lb]{\smash{{\SetFigFont{17}{20.4}{\rmdefault}{\mddefault}{\itdefault}{$y+a_3+b_1+b_3+1$}%
}}}}}
\put(10164,-65){\makebox(0,0)[lb]{\smash{{\SetFigFont{17}{20.4}{\rmdefault}{\mddefault}{\itdefault}{$z+a_2+b_2-1$}%
}}}}
\put(12083,-568){\rotatebox{300.0}{\makebox(0,0)[lb]{\smash{{\SetFigFont{17}{20.4}{\rmdefault}{\mddefault}{\itdefault}{$y+a_1+a_3+b_3+1$}%
}}}}}
\put(13524,-2968){\makebox(0,0)[lb]{\smash{{\SetFigFont{17}{20.4}{\rmdefault}{\mddefault}{\itdefault}{$t$}%
}}}}
\put(12849,-5330){\rotatebox{60.0}{\makebox(0,0)[lb]{\smash{{\SetFigFont{17}{20.4}{\rmdefault}{\mddefault}{\itdefault}{$x+y+a_2+b_2$}%
}}}}}
\put(9744,-6114){\makebox(0,0)[lb]{\smash{{\SetFigFont{17}{20.4}{\rmdefault}{\mddefault}{\itdefault}{$z+a_1+a_3+b_1+b_3$}%
}}}}
\put(7942,-4126){\rotatebox{300.0}{\makebox(0,0)[lb]{\smash{{\SetFigFont{17}{20.4}{\rmdefault}{\mddefault}{\itdefault}{$y+t+a_2+b_2$}%
}}}}}
\put(8027,-2960){\makebox(0,0)[lb]{\smash{{\SetFigFont{17}{20.4}{\rmdefault}{\mddefault}{\itdefault}{$x$}%
}}}}
\put(1607,-9485){\rotatebox{60.0}{\makebox(0,0)[lb]{\smash{{\SetFigFont{17}{20.4}{\rmdefault}{\mddefault}{\itdefault}{$y+a_3+b_1+b_3+1$}%
}}}}}
\put(3159,-7475){\makebox(0,0)[lb]{\smash{{\SetFigFont{17}{20.4}{\rmdefault}{\mddefault}{\itdefault}{$z+a_2+b_2-1$}%
}}}}
\put(5002,-7921){\rotatebox{300.0}{\makebox(0,0)[lb]{\smash{{\SetFigFont{17}{20.4}{\rmdefault}{\mddefault}{\itdefault}{$y+a_1+a_3+b_3+1$}%
}}}}}
\put(6519,-10378){\makebox(0,0)[lb]{\smash{{\SetFigFont{17}{20.4}{\rmdefault}{\mddefault}{\itdefault}{$t$}%
}}}}
\put(5844,-12740){\rotatebox{60.0}{\makebox(0,0)[lb]{\smash{{\SetFigFont{17}{20.4}{\rmdefault}{\mddefault}{\itdefault}{$x+y+a_2+b_2$}%
}}}}}
\put(2686,-13539){\makebox(0,0)[lb]{\smash{{\SetFigFont{17}{20.4}{\rmdefault}{\mddefault}{\itdefault}{$z+a_1+a_3+b_1+b_3$}%
}}}}
\put(937,-11574){\rotatebox{300.0}{\makebox(0,0)[lb]{\smash{{\SetFigFont{17}{20.4}{\rmdefault}{\mddefault}{\itdefault}{$y+t+a_2+b_2$}%
}}}}}
\put(1022,-10370){\makebox(0,0)[lb]{\smash{{\SetFigFont{17}{20.4}{\rmdefault}{\mddefault}{\itdefault}{$x$}%
}}}}
\put(8327,-9395){\rotatebox{60.0}{\makebox(0,0)[lb]{\smash{{\SetFigFont{17}{20.4}{\rmdefault}{\mddefault}{\itdefault}{$y+a_3+b_1+b_3+1$}%
}}}}}
\put(9879,-7385){\makebox(0,0)[lb]{\smash{{\SetFigFont{17}{20.4}{\rmdefault}{\mddefault}{\itdefault}{$z+a_2+b_2-1$}%
}}}}
\put(11798,-7888){\rotatebox{300.0}{\makebox(0,0)[lb]{\smash{{\SetFigFont{17}{20.4}{\rmdefault}{\mddefault}{\itdefault}{$y+a_1+a_3+b_3+1$}%
}}}}}
\put(13239,-10288){\makebox(0,0)[lb]{\smash{{\SetFigFont{17}{20.4}{\rmdefault}{\mddefault}{\itdefault}{$t$}%
}}}}
\put(12564,-12650){\rotatebox{60.0}{\makebox(0,0)[lb]{\smash{{\SetFigFont{17}{20.4}{\rmdefault}{\mddefault}{\itdefault}{$x+y+a_2+b_2$}%
}}}}}
\put(9512,-13438){\makebox(0,0)[lb]{\smash{{\SetFigFont{17}{20.4}{\rmdefault}{\mddefault}{\itdefault}{$z+a_1+a_3+b_1+b_3$}%
}}}}
\put(7712,-11585){\rotatebox{300.0}{\makebox(0,0)[lb]{\smash{{\SetFigFont{17}{20.4}{\rmdefault}{\mddefault}{\itdefault}{$y+t+a_2+b_2$}%
}}}}}
\put(7742,-10280){\makebox(0,0)[lb]{\smash{{\SetFigFont{17}{20.4}{\rmdefault}{\mddefault}{\itdefault}{$x$}%
}}}}
\put(1487,-16696){\rotatebox{60.0}{\makebox(0,0)[lb]{\smash{{\SetFigFont{17}{20.4}{\rmdefault}{\mddefault}{\itdefault}{$y+a_3+b_1+b_3+1$}%
}}}}}
\put(3039,-14686){\makebox(0,0)[lb]{\smash{{\SetFigFont{17}{20.4}{\rmdefault}{\mddefault}{\itdefault}{$z+a_2+b_2-1$}%
}}}}
\put(4958,-15189){\rotatebox{300.0}{\makebox(0,0)[lb]{\smash{{\SetFigFont{17}{20.4}{\rmdefault}{\mddefault}{\itdefault}{$y+a_1+a_3+b_3+1$}%
}}}}}
\put(6421,-17761){\makebox(0,0)[lb]{\smash{{\SetFigFont{17}{20.4}{\rmdefault}{\mddefault}{\itdefault}{$t$}%
}}}}
\put(5724,-19951){\rotatebox{60.0}{\makebox(0,0)[lb]{\smash{{\SetFigFont{17}{20.4}{\rmdefault}{\mddefault}{\itdefault}{$x+y+a_2+b_2$}%
}}}}}
\put(2672,-20739){\makebox(0,0)[lb]{\smash{{\SetFigFont{17}{20.4}{\rmdefault}{\mddefault}{\itdefault}{$z+a_1+a_3+b_1+b_3$}%
}}}}
\put(847,-18811){\rotatebox{300.0}{\makebox(0,0)[lb]{\smash{{\SetFigFont{17}{20.4}{\rmdefault}{\mddefault}{\itdefault}{$y+t+a_2+b_2$}%
}}}}}
\put(902,-17581){\makebox(0,0)[lb]{\smash{{\SetFigFont{17}{20.4}{\rmdefault}{\mddefault}{\itdefault}{$x$}%
}}}}
\put(8027,-16681){\rotatebox{60.0}{\makebox(0,0)[lb]{\smash{{\SetFigFont{17}{20.4}{\rmdefault}{\mddefault}{\itdefault}{$y+a_3+b_1+b_3+1$}%
}}}}}
\put(9579,-14671){\makebox(0,0)[lb]{\smash{{\SetFigFont{17}{20.4}{\rmdefault}{\mddefault}{\itdefault}{$z+a_2+b_2-1$}%
}}}}
\put(11392,-15061){\rotatebox{300.0}{\makebox(0,0)[lb]{\smash{{\SetFigFont{17}{20.4}{\rmdefault}{\mddefault}{\itdefault}{$y+a_1+a_3+b_3+1$}%
}}}}}
\put(12939,-17574){\makebox(0,0)[lb]{\smash{{\SetFigFont{17}{20.4}{\rmdefault}{\mddefault}{\itdefault}{$t$}%
}}}}
\put(12196,-20026){\rotatebox{60.0}{\makebox(0,0)[lb]{\smash{{\SetFigFont{17}{20.4}{\rmdefault}{\mddefault}{\itdefault}{$x+y+a_2+b_2$}%
}}}}}
\put(9061,-20716){\makebox(0,0)[lb]{\smash{{\SetFigFont{17}{20.4}{\rmdefault}{\mddefault}{\itdefault}{$z+a_1+a_3+b_1+b_3$}%
}}}}
\put(7321,-18721){\rotatebox{300.0}{\makebox(0,0)[lb]{\smash{{\SetFigFont{17}{20.4}{\rmdefault}{\mddefault}{\itdefault}{$y+t+a_2+b_2$}%
}}}}}
\put(7442,-17566){\makebox(0,0)[lb]{\smash{{\SetFigFont{17}{20.4}{\rmdefault}{\mddefault}{\itdefault}{$x$}%
}}}}
\end{picture}}
\caption{Obtaining the recurrence for the numbers of tilings of $P$-regions.}
\label{KuoarrayQ2}
\end{figure}

By considering the product of the weights of the forced lozenges shown in Figures \ref{KuoarrayQ2}(a)--(f), and normalizing the weight assignment of the resulting region if needed, we obtain
\begin{equation}
\M(G-\{v\})=q^{(x+2y+t+a+b)(z+e_a+e_b-1)}\M_1(P_{x,y,z,t}(\textbf{a};\textbf{b})),
\end{equation}
\begin{equation}
\M(G-\{u,w,s\})=q\cdot q^{V_{x-1,y+1,z-1,t-1}}\M_1(P_{x-1,y+1,z-1,t-1}(\textbf{a};\textbf{b})),
\end{equation}
\begin{equation}
\M(G-\{u\})=q\M_1(P_{x,y+1,z-1,t-1}(\textbf{a};\textbf{b})),
\end{equation}
\begin{equation}
\M(G-\{v,w,s\})=q^{(x+2y+t+a+b)(z+e_a+e_b-1)+1}\cdot q^{V_{x-1,y,z,t}}\M_1(P_{x-1,y,z,t}(\textbf{a};\textbf{b})),
\end{equation}
\begin{equation}
\M(G-\{w\})=\M_1(P_{x-1,y+1,z-1,t}(\textbf{a};\textbf{b})),
\end{equation}
and
\begin{equation}
\M(G-\{u,v,s\})=q^{(x+2y+t+a+b)(z+e_a+e_b-1)+1}\cdot q^{V_{x,y,z,t-1}}\M_1(P_{x,y,z,t-1}(\textbf{a};\textbf{b})),
\end{equation}
respectively, where $V_{x,y,z,t}=V_{x,y,z,t}(\textbf{a};\textbf{b})$ is the number of right lozenges in each tiling of the region $P_{x,y,z,t}(\textbf{a};\textbf{b})$. A straightforward calculation shows that
\begin{align}\label{rightlozenge}
V_{x,y,z,t}&(\textbf{a};\textbf{b})=\sum_{i=1}^{\lfloor\frac{m+1}{2}\rfloor}a_{2i-1}\left(x+\sum_{j=1}^{i-1}a_{2j}\right)+z(x+y+e_a+o_b)
+\sum_{i=1}^{\lfloor\frac{n}{2}\rfloor}b_{2i}\sum_{j=1}^{i}b_{2j-1}\notag\\
&+\sum_{i=1}^{\lfloor\frac{n+1}{2}\rfloor}b_{2i-1}\left(x+y+e_a+e_b-\sum_{j=1}^{i-1}b_{2j}\right)+
\sum_{i=1}^{\lfloor\frac{m}{2}\rfloor}a_{2i}\left(y+o_a+o_b-\sum_{j=1}^{i}a_{2j-1}\right)\\
&=x(o_a+o_b)+z(x+y+e_a+o_b)+y(e_a+o_b)\notag\\
&+e_a(o_a+o_b)+o_b(e_a+e_b)\notag\\
&+\sum_{i=2}^{\lfloor\frac{m+1}{2}\rfloor}a_{2i-1}\sum_{j=1}^{i-1}a_{2j}
-\sum_{i=1}^{\lfloor\frac{m}{2}\rfloor}a_{2i}\sum_{j=1}^{i}a_{2j-1}+\sum_{i=1}^{\lfloor\frac{n}{2}\rfloor}b_{2i}\sum_{j=1}^{i}b_{2j-1}
-\sum_{i=2}^{\lfloor\frac{n+1}{2}\rfloor}b_{2i-1}\sum_{j=1}^{i-1}b_{2j}.
\end{align}

Plugging the above equality to the recurrence into Kuo's Theorem \ref{kuothm3} and simplifying, we obtain
\begin{align}\label{recur1}
\M_1(P_{x,y,z,t}(\textbf{a};\textbf{b}))&\M_1(P_{x-1,y+1,z-1,t-1}(\textbf{a};\textbf{b}))\notag\\
& =\,  q^{x+y-z+1}\M_1(P_{x,y+1,z-1,t-1}(\textbf{a};\textbf{b}))\M_1(P_{x-1,y,z,t}(\textbf{a};\textbf{b}))\notag\\
&+q^{o_a+o_b+x+y}\M_1(P_{x-1,y+1,z-1,t}(\textbf{a};\textbf{b}))\M_1(P_{x,y,z,t-1}(\textbf{a};\textbf{b})).
\end{align}

In order to obtain an analogous recurrence for our $Q$-regions, apply Kuo's condensation Theorem \ref{kuothm2} to the dual graph $H$ of $Q_{x,y,z,t}(\textbf{a};\textbf{b})$ as shown in Figure \ref{KuoarrayP}(c).

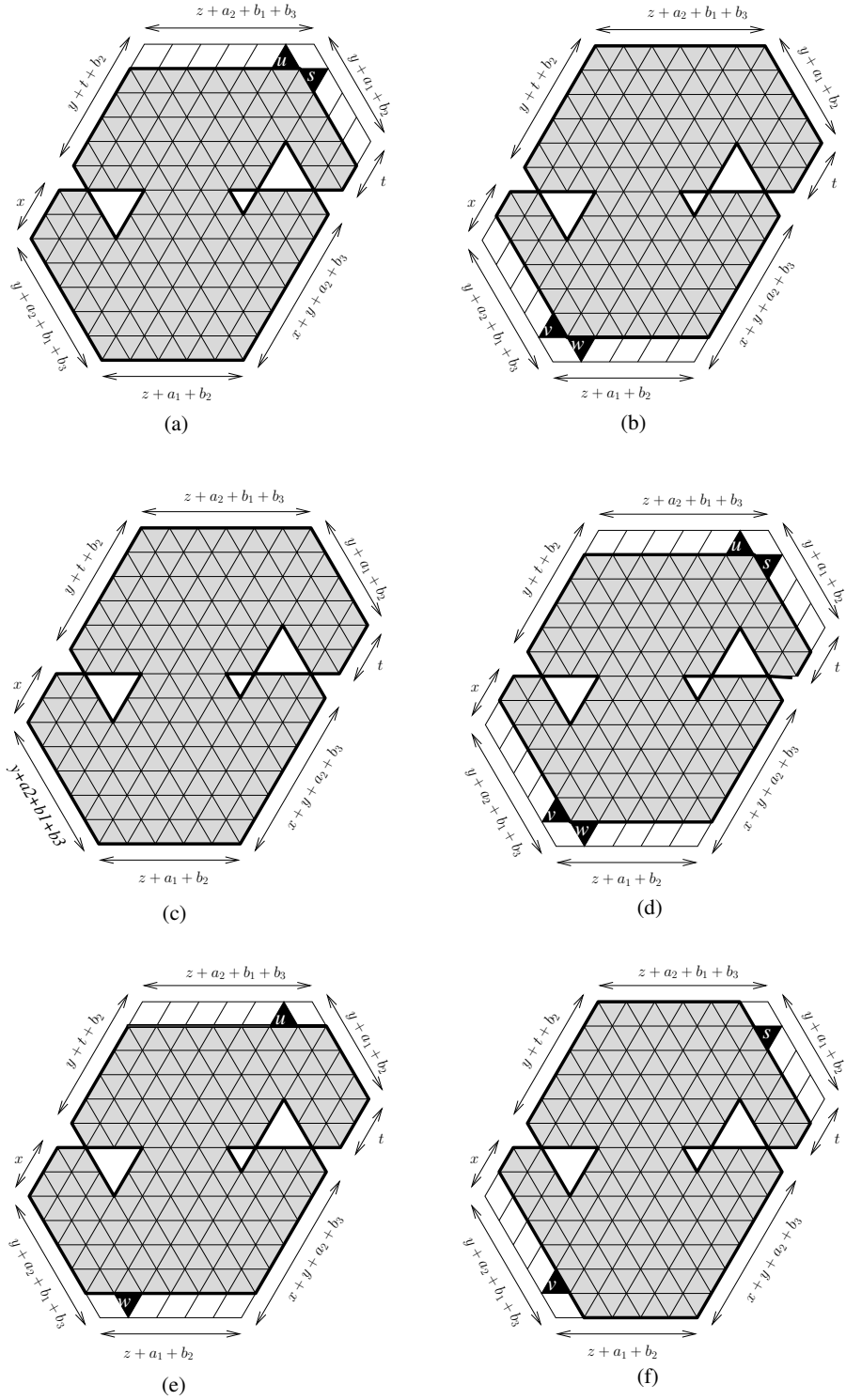
\begin{figure}\centering
\setlength{\unitlength}{3947sp}%
\begingroup\makeatletter\ifx\SetFigFont\undefined%
\gdef\SetFigFont#1#2#3#4#5{%
  \reset@font\fontsize{#1}{#2pt}%
  \fontfamily{#3}\fontseries{#4}\fontshape{#5}%
  \selectfont}%
\fi\endgroup%
\resizebox{12cm}{!}{
\begin{picture}(0,0)%
\includegraphics{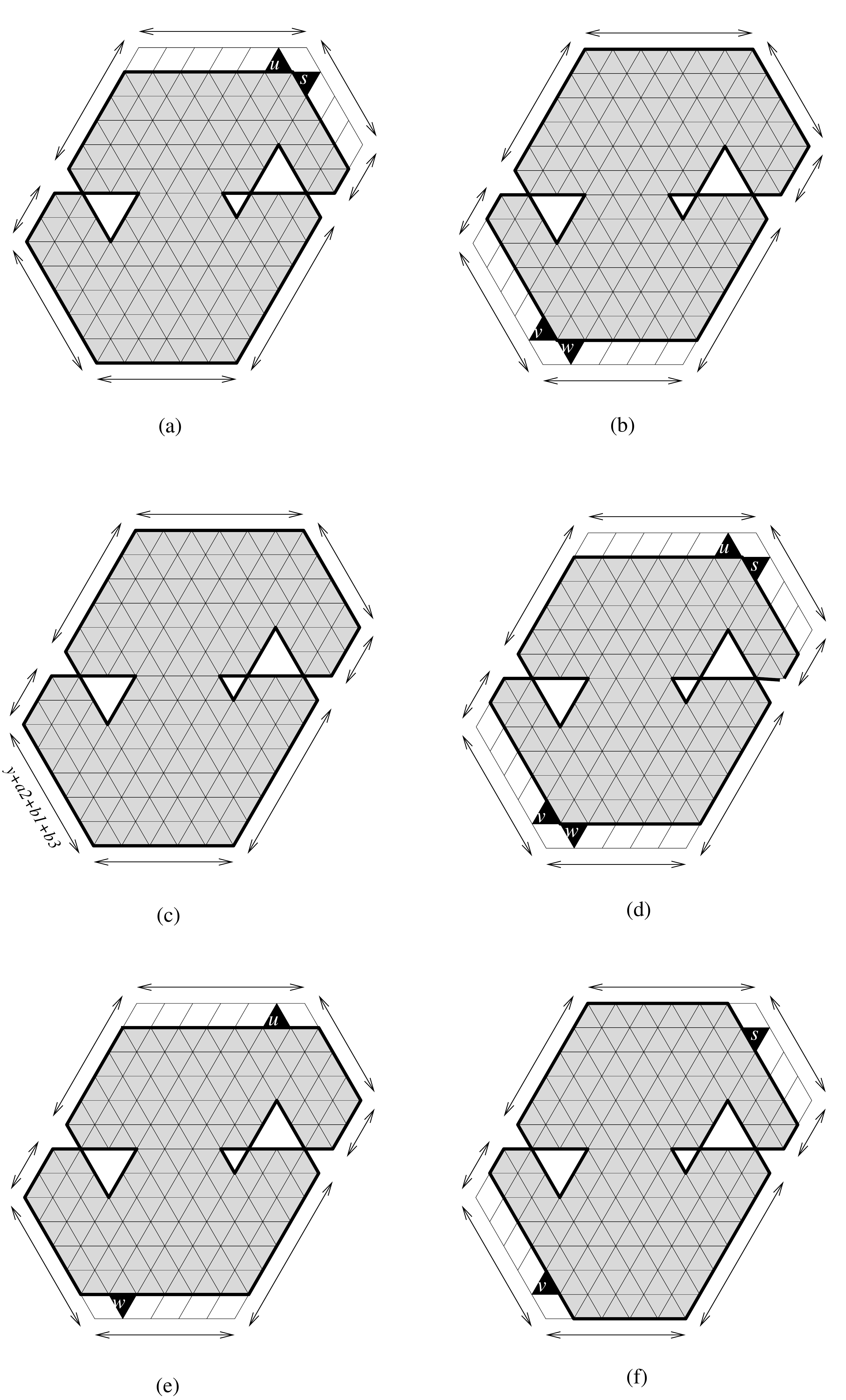}%
\end{picture}%
%
%

\begin{picture}(15684,25957)(77,-25594)
\put(301,-3496){\makebox(0,0)[lb]{\smash{{\SetFigFont{17}{20.4}{\rmdefault}{\mddefault}{\itdefault}{$x$}%
}}}}
\put(1351,-1741){\rotatebox{60.0}{\makebox(0,0)[lb]{\smash{{\SetFigFont{17}{20.4}{\rmdefault}{\mddefault}{\itdefault}{$y+t+b_2$}%
}}}}}
\put(3601, 29){\makebox(0,0)[lb]{\smash{{\SetFigFont{17}{20.4}{\rmdefault}{\mddefault}{\itdefault}{$z+a_2+b_1+b_3$}%
}}}}
\put(6421,-736){\rotatebox{300.0}{\makebox(0,0)[lb]{\smash{{\SetFigFont{17}{20.4}{\rmdefault}{\mddefault}{\itdefault}{$y+a_1+b_2$}%
}}}}}
\put(6991,-3136){\makebox(0,0)[lb]{\smash{{\SetFigFont{17}{20.4}{\rmdefault}{\mddefault}{\itdefault}{$t$}%
}}}}
\put(2596,-7066){\makebox(0,0)[lb]{\smash{{\SetFigFont{17}{20.4}{\rmdefault}{\mddefault}{\itdefault}{$z+a_1+b_2$}%
}}}}
\put(232,-4951){\rotatebox{300.0}{\makebox(0,0)[lb]{\smash{{\SetFigFont{17}{20.4}{\rmdefault}{\mddefault}{\itdefault}{$y+a_2+b_1+b_3$}%
}}}}}
\put(8584,-3526){\makebox(0,0)[lb]{\smash{{\SetFigFont{17}{20.4}{\rmdefault}{\mddefault}{\itdefault}{$x$}%
}}}}
\put(9634,-1771){\rotatebox{60.0}{\makebox(0,0)[lb]{\smash{{\SetFigFont{17}{20.4}{\rmdefault}{\mddefault}{\itdefault}{$y+t+b_2$}%
}}}}}
\put(11884, -1){\makebox(0,0)[lb]{\smash{{\SetFigFont{17}{20.4}{\rmdefault}{\mddefault}{\itdefault}{$z+a_2+b_1+b_3$}%
}}}}
\put(14679,-594){\rotatebox{300.0}{\makebox(0,0)[lb]{\smash{{\SetFigFont{17}{20.4}{\rmdefault}{\mddefault}{\itdefault}{$y+a_1+b_2$}%
}}}}}
\put(15274,-3166){\makebox(0,0)[lb]{\smash{{\SetFigFont{17}{20.4}{\rmdefault}{\mddefault}{\itdefault}{$t$}%
}}}}
\put(10666,-7051){\makebox(0,0)[lb]{\smash{{\SetFigFont{17}{20.4}{\rmdefault}{\mddefault}{\itdefault}{$z+a_1+b_2$}%
}}}}
\put(8515,-4981){\rotatebox{300.0}{\makebox(0,0)[lb]{\smash{{\SetFigFont{17}{20.4}{\rmdefault}{\mddefault}{\itdefault}{$y+a_2+b_1+b_3$}%
}}}}}
\put(244,-12451){\makebox(0,0)[lb]{\smash{{\SetFigFont{17}{20.4}{\rmdefault}{\mddefault}{\itdefault}{$x$}%
}}}}
\put(1294,-10696){\rotatebox{60.0}{\makebox(0,0)[lb]{\smash{{\SetFigFont{17}{20.4}{\rmdefault}{\mddefault}{\itdefault}{$y+t+b_2$}%
}}}}}
\put(3376,-8956){\makebox(0,0)[lb]{\smash{{\SetFigFont{17}{20.4}{\rmdefault}{\mddefault}{\itdefault}{$z+a_2+b_1+b_3$}%
}}}}
\put(6364,-9691){\rotatebox{300.0}{\makebox(0,0)[lb]{\smash{{\SetFigFont{17}{20.4}{\rmdefault}{\mddefault}{\itdefault}{$y+a_1+b_2$}%
}}}}}
\put(6934,-12091){\makebox(0,0)[lb]{\smash{{\SetFigFont{17}{20.4}{\rmdefault}{\mddefault}{\itdefault}{$t$}%
}}}}
\put(2539,-16021){\makebox(0,0)[lb]{\smash{{\SetFigFont{17}{20.4}{\rmdefault}{\mddefault}{\itdefault}{$z+a_1+b_2$}%
}}}}
\put(8644,-12496){\makebox(0,0)[lb]{\smash{{\SetFigFont{17}{20.4}{\rmdefault}{\mddefault}{\itdefault}{$x$}%
}}}}
\put(9694,-10741){\rotatebox{60.0}{\makebox(0,0)[lb]{\smash{{\SetFigFont{17}{20.4}{\rmdefault}{\mddefault}{\itdefault}{$y+t+b_2$}%
}}}}}
\put(11791,-9031){\makebox(0,0)[lb]{\smash{{\SetFigFont{17}{20.4}{\rmdefault}{\mddefault}{\itdefault}{$z+a_2+b_1+b_3$}%
}}}}
\put(14764,-9736){\rotatebox{300.0}{\makebox(0,0)[lb]{\smash{{\SetFigFont{17}{20.4}{\rmdefault}{\mddefault}{\itdefault}{$y+a_1+b_2$}%
}}}}}
\put(15334,-12136){\makebox(0,0)[lb]{\smash{{\SetFigFont{17}{20.4}{\rmdefault}{\mddefault}{\itdefault}{$t$}%
}}}}
\put(10861,-16051){\makebox(0,0)[lb]{\smash{{\SetFigFont{17}{20.4}{\rmdefault}{\mddefault}{\itdefault}{$z+a_1+b_2$}%
}}}}
\put(8575,-13951){\rotatebox{300.0}{\makebox(0,0)[lb]{\smash{{\SetFigFont{17}{20.4}{\rmdefault}{\mddefault}{\itdefault}{$y+a_2+b_1+b_3$}%
}}}}}
\put(266,-21224){\makebox(0,0)[lb]{\smash{{\SetFigFont{17}{20.4}{\rmdefault}{\mddefault}{\itdefault}{$x$}%
}}}}
\put(1316,-19469){\rotatebox{60.0}{\makebox(0,0)[lb]{\smash{{\SetFigFont{17}{20.4}{\rmdefault}{\mddefault}{\itdefault}{$y+t+b_2$}%
}}}}}
\put(3406,-17776){\makebox(0,0)[lb]{\smash{{\SetFigFont{17}{20.4}{\rmdefault}{\mddefault}{\itdefault}{$z+a_2+b_1+b_3$}%
}}}}
\put(6386,-18464){\rotatebox{300.0}{\makebox(0,0)[lb]{\smash{{\SetFigFont{17}{20.4}{\rmdefault}{\mddefault}{\itdefault}{$y+a_1+b_2$}%
}}}}}
\put(6956,-20864){\makebox(0,0)[lb]{\smash{{\SetFigFont{17}{20.4}{\rmdefault}{\mddefault}{\itdefault}{$t$}%
}}}}
\put(2266,-24826){\makebox(0,0)[lb]{\smash{{\SetFigFont{17}{20.4}{\rmdefault}{\mddefault}{\itdefault}{$z+a_1+b_2$}%
}}}}
\put(197,-22679){\rotatebox{300.0}{\makebox(0,0)[lb]{\smash{{\SetFigFont{17}{20.4}{\rmdefault}{\mddefault}{\itdefault}{$y+a_2+b_1+b_3$}%
}}}}}
\put(8636,-21224){\makebox(0,0)[lb]{\smash{{\SetFigFont{17}{20.4}{\rmdefault}{\mddefault}{\itdefault}{$x$}%
}}}}
\put(9686,-19469){\rotatebox{60.0}{\makebox(0,0)[lb]{\smash{{\SetFigFont{17}{20.4}{\rmdefault}{\mddefault}{\itdefault}{$y+t+b_2$}%
}}}}}
\put(11716,-17761){\makebox(0,0)[lb]{\smash{{\SetFigFont{17}{20.4}{\rmdefault}{\mddefault}{\itdefault}{$z+a_2+b_1+b_3$}%
}}}}
\put(14756,-18464){\rotatebox{300.0}{\makebox(0,0)[lb]{\smash{{\SetFigFont{17}{20.4}{\rmdefault}{\mddefault}{\itdefault}{$y+a_1+b_2$}%
}}}}}
\put(15326,-20864){\makebox(0,0)[lb]{\smash{{\SetFigFont{17}{20.4}{\rmdefault}{\mddefault}{\itdefault}{$t$}%
}}}}
\put(10786,-24796){\makebox(0,0)[lb]{\smash{{\SetFigFont{17}{20.4}{\rmdefault}{\mddefault}{\itdefault}{$z+a_1+b_2$}%
}}}}
\put(8544,-22659){\rotatebox{300.0}{\makebox(0,0)[lb]{\smash{{\SetFigFont{17}{20.4}{\rmdefault}{\mddefault}{\itdefault}{$y+a_2+b_1+b_3$}%
}}}}}
\put(5461,-5986){\rotatebox{60.0}{\makebox(0,0)[lb]{\smash{{\SetFigFont{17}{20.4}{\rmdefault}{\mddefault}{\itdefault}{$x+y+a_2+b_3$}%
}}}}}
\put(13676,-6128){\rotatebox{60.0}{\makebox(0,0)[lb]{\smash{{\SetFigFont{17}{20.4}{\rmdefault}{\mddefault}{\itdefault}{$x+y+a_2+b_3$}%
}}}}}
\put(5381,-15068){\rotatebox{60.0}{\makebox(0,0)[lb]{\smash{{\SetFigFont{17}{20.4}{\rmdefault}{\mddefault}{\itdefault}{$x+y+a_2+b_3$}%
}}}}}
\put(13751,-15083){\rotatebox{60.0}{\makebox(0,0)[lb]{\smash{{\SetFigFont{17}{20.4}{\rmdefault}{\mddefault}{\itdefault}{$x+y+a_2+b_3$}%
}}}}}
\put(5396,-23798){\rotatebox{60.0}{\makebox(0,0)[lb]{\smash{{\SetFigFont{17}{20.4}{\rmdefault}{\mddefault}{\itdefault}{$x+y+a_2+b_3$}%
}}}}}
\put(13811,-23813){\rotatebox{60.0}{\makebox(0,0)[lb]{\smash{{\SetFigFont{17}{20.4}{\rmdefault}{\mddefault}{\itdefault}{$x+y+a_2+b_3$}%
}}}}}
\end{picture}}
\caption{Obtaining the recurrence for the numbers of tilings of $Q$-regions.}
\label{KuoarrayP}
\end{figure}

From Figures \ref{KuoarrayP}(a), (b), (d), (e) and (f), by accounting for the weights of the forced lozenges and normalizing the weight assignment of the resulting regions if needed, we obtain
\begin{equation}
\M(G-\{u,s\})=q^{(x+2y+a+b)(z+e_a+o_b-1)}\M_1(Q_{x,y,z,t-1}(\textbf{a};\textbf{b})),
\end{equation}
\begin{equation}
\M(G-\{v,w\})=q^{(z+o_a+e_b-1)}q^{W_{x-1,y,z,t}}\M_1(Q_{x-1,y,z,t}(\textbf{a};\textbf{b})),
\end{equation}
\begin{equation}
\M(G-\{u,v,w,s\})=q^{(x+2y+a+b)(z+e_a+o_b-1)+(z+o_a+e_b-1)}q^{W_{x-1,y,z,t-1}}\M_1(Q_{x-1,y,z,t-1}(\textbf{a};\textbf{b})),
\end{equation}
\begin{equation}
\M(G-\{u,w\})=q^{(x+2y+a+b)(z+e_a+o_b-1)+(z+o_a+e_b-1)}q^{W_{x,y-1,z+1,t}}\M_1(Q_{x,y-1,z+1,t}(\textbf{a};\textbf{b})),
\end{equation}
\begin{equation}
\M(G-\{v,s\})=\M_1(Q_{x-1,y+1,z-1,t-1}(\textbf{a};\textbf{b})),
\end{equation}
respectively, where $W_{x,y,z,t}=W_{x,y,z,t}(\textbf{a};\textbf{b})$ is the number of right lozenges in each tiling of the region $P_{x,y,z,t}(\textbf{a};\textbf{b})$. It is not hard to see that this number is
\begin{align}
W_{x,y,z,t}&(\textbf{a};\textbf{b})=\sum_{i=1}^{\lfloor\frac{m+1}{2}\rfloor}a_{2i-1}\left(x+\sum_{j=1}^{i-1}a_{2j}\right)+z(x+y+t+e_a+e_b)\notag\\
&+\sum_{i=1}^{\lfloor\frac{n+1}{2}\rfloor}b_{2j-1}\left(t+\sum_{j=1}^{i-1}b_{2j}\right)+\sum_{i=1}^{\lfloor\frac{n}{2}\rfloor}b_{2i}\left(x+y+e_a+o_b-\sum_{j=1}^{i}b_{2j-1}\right)
\notag\\
&+\sum_{i=1}^{\lfloor\frac{m}{2}\rfloor}a_{2i}\left(y+t+o_a+e_b-\sum_{j=1}^{i}a_{2j-1}\right)\\
&=x(o_a+e_b)+t(e_a+o_b)+y(e_a+o_b)+z(x+y+t+e_a+e_b)\notag\\
&+e_b(e_a+o_b)+e_a(o_a+e_b)\notag\\
&+\sum_{i=1}^{\lfloor\frac{m+1}{2}\rfloor}a_{2i-1}\sum_{j=2}^{i-1}a_{2j}-\sum_{i=1}^{\lfloor\frac{m}{2}\rfloor}a_{2i}\sum_{j=1}^{i}a_{2j-1}
+\sum_{i=2}^{\lfloor\frac{n+1}{2}\rfloor}b_{2j-1}\sum_{j=1}^{i-1}b_{2j}-\sum_{i=2}^{\lfloor\frac{n}{2}\rfloor}b_{2i}\sum_{j=1}^{i}b_{2j-1}.
\end{align}

Plugging in the above equalities into the recurrence in Kuo's Theorem \ref{kuothm2} and simplifying, we obtain
\begin{align}\label{recur2}
q^{e_a+o_b+z}&\M_1(Q_{x,y,z,t-1}(\textbf{a};\textbf{b}))\M_1(Q_{x-1,y,z,t}(\textbf{a};\textbf{b}))\notag\\
&=\M_1(Q_{x,y,z,t}(\textbf{a};\textbf{b}))\M_1(Q_{x-1,y,z,t-1}(\textbf{a};\textbf{b}))\notag\\
&+q^{x+y+z+t+e_a+b-1}\M_1(Q_{x,y-1,z+1,t}(\textbf{a};\textbf{b}))\M_1(Q_{x-1,y+1,z-1,t-1}(\textbf{a};\textbf{b})).
\end{align}

Note that in both \eqref{recur1} and \eqref{recur2}, the region with indices $x,y,z,t$ is the only one for which the sum of the eight parameters $x+y+z+t+a+b+m+n$ on which we are doing the induction is maximal. Express the weighted count of this region in terms of the weighted counts of the other regions in \eqref{recur1} and \eqref{recur2}. These latter weighted counts are given, by the induction hypothesis, by formulas \eqref{main2eq1b} and \eqref{main2eq2b}. To complete our proof, it suffices to verify that the functions $\Phi^{q}$ and $\Psi^{q}$ satisfy recurrences (\ref{recur1}) and (\ref{recur2}), respectively.

%
%

We show below that $\Phi^{q}$ satisfies the recurrence \eqref{recur1} when $n=2k$ and $m=2l$; the other cases are similar. An analogous calculation verifies that $\Psi^{q}$ satisfies \eqref{recur2}.

We need to verify that
\begin{align}\label{recur3}
&q^{A^{(1)}_{x,y,z,t}(\textbf{a};\textbf{b})+A^{(1)}_{x-1,y+1,z-1,t-1}(\textbf{a};\textbf{b})}\Phi^{q}_{x,y,z,t}(\textbf{a};\textbf{b})\Phi^{q}_{x-1,y+1,z-1,t-1}(\textbf{a};\textbf{b})=\notag\\
&\,\,\,\,\,\,\,\,\,\,\,\,\,\,\,\,\,\,\,\,\,\,q^{(x+y-z+1)+A^{(1)}_{x,y+1,z-1,t-1}(\textbf{a};\textbf{b})+A^{(1)}_{x-1,y,z,t}(\textbf{a};\textbf{b})}\Phi^{q}_{x,y+1,z-1,t-1}(\textbf{a};\textbf{b})\Phi^{q}_{x-1,y,z,t}(\textbf{a};\textbf{b})\notag\\
&\,\,\,\,\,\,\,\,\,\,\,\,\,\,\,\,\,\,\,\,\,\,+q^{(x+y+o_a+o_b)+A^{(1)}_{x-1,y+1,z-1,t}(\textbf{a};\textbf{b})+A^{(1)}_{x,y,z,t-1}(\textbf{a};\textbf{b})}\Phi^{q}_{x-1,y+1,z-1,t}(\textbf{a};\textbf{b})\Phi^{q}_{x,y,z,t-1}(\textbf{a};\textbf{b}).
\end{align}

It is easy to verify from the definition of the function $A^{(1)}_{x,y,z,t}(\textbf{a};\textbf{b})$ in Section 6  that
\begin{align}
\label{Arel1}
A^{(1)}_{x,y,z,t}(\textbf{a};\textbf{b})+A^{(1)}_{x-1,y+1,z-1,t-1}(\textbf{a};\textbf{b})=&A^{(1)}_{x,y+1,z-1,t-1}(\textbf{a};\textbf{b})\notag\\&+A^{(1)}_{x-1,y,z,t}(\textbf{a};\textbf{b})+x+y-z+1.
\end{align}
and
\begin{align}
\label{Arel2}
A^{(1)}_{x,y,z,t}(\textbf{a};\textbf{b})+A^{(1)}_{x-1,y+1,z-1,t-1}(\textbf{a};\textbf{b})=A^{(1)}_{x-1,y+1,z-1,t}(\textbf{a};\textbf{b})+A^{(1)}_{x,y,z,t-1}(\textbf{a};\textbf{b})-e_a+o_b.
\end{align}

By \eqref{Arel1} and \eqref{Arel2}, equation (\ref{recur3}) is seen to be equivalent to
\begin{align}\label{recur4}
\frac{\Phi^{q}_{x-1,y,z,t}(\textbf{a};\textbf{b})}{\Phi^{q}_{x,y,z,t}(\textbf{a};\textbf{b})}
\frac{\Phi^{q}_{x,y+1,z-1,t-1}(\textbf{a};\textbf{b})}{\Phi^{q}_{x-1,y+1,z-1,t-1}(\textbf{a};\textbf{b})}+q^{x+y+a}\frac{\Phi^{q}_{x-1,y+1,z-1,t}(\textbf{a};\textbf{b})}{\Phi^{q}_{x-1,y+1,z-1,t-1}(\textbf{a};\textbf{b})}\frac{\Phi^{q}_{x,y,z,t-1}(\textbf{a};\textbf{b})}{\Phi^{q}_{x,y,z,t}(\textbf{a};\textbf{b})}=1.
\end{align}

Let us simplify the first term on the left-hand side above. Consider the two $\Phi^{q}$-functions in the numerator and denominator of the first fraction in the first term. They differ only in their $x$-parameters.
Cancelling out all the factors not involving the $x$-parameter and using the trivial fact $\Hf_q(n+1)=[n]!_q\Hf_q(n)$, we get
\begin{align}\label{simplify1}
\frac{\Phi^{q}_{x-1,y,z,t}(\textbf{a};\textbf{b})}{\Phi^{q}_{x,y,z,t}(\textbf{a};\textbf{b})}&=\frac{[a+b+x+2y+t-1]!_q}{[a+b+x+2y+z+t-1]!_q}\frac{[a+x-1]!_q}{[a+x+y-1]!_q}\notag\\
&\times \frac{[a+b+x+y+z+t-1]!_q}{[a+b+x+y+t-1]!_q} \frac{[e_a+e_b+x+y+t-1]!_q}{[e_a+e_b+x+t-1]!_q}\notag\\
&\times \frac{s_{q}(x-1,a_1,\dotsc,a_{2k},y+z,b_{2l},\dotsc,b_{1},t)}{s_{q}(x,a_1,\dotsc,a_{2k},y+z,b_{2l},\dotsc,b_{1},t)}.
\end{align}
A similar simplification of the second fraction in the first term on the left hand side of \eqref{recur4} yields
\begin{align}\label{simplify2}
\frac{\Phi^{q}_{x,y+1,z-1,t-1}}{\Phi^{q}_{x-1,y+1,z-1,t-1}}&=\frac{[a+b+x+2y+z+t-1]!_q}{[a+b+x+2y+t]!_q}\frac{[a+x+y]!_q}{[a+x-1]!_q}\notag\\
&\times \frac{[a+b+x+y+t-1]!_q}{[a+b+x+y+z+t-2]!_q} \frac{[e_a+e_b+x+t-2]!_q}{[e_a+e_b+x+y+t-1]!_q}\notag\\
&\times \frac{s_{q}(x,a_1,\dotsc,a_{2k},y+z,b_{2l},\dotsc,b_{1},t-1)}{s_{q}(x-1,a_1,\dotsc,a_{2k},y+z,b_{2l},\dotsc,b_{1},t-1)}.
\end{align}
Furthermore, using definition \eqref{semieqq}, one readily checks that
\small{\begin{align}\label{simplify3}
&\frac{s_q(x,a_1,\dotsc,a_{2k},y+z,b_{2l},\dotsc,b_{1},t-1)}{s_q(x-1,a_1,\dotsc,a_{2k},y+z,b_{2l},\dotsc,b_{1},t-1)}\frac{s_q(x-1,a_1,\dotsc,a_{2k},y+z,b_{2l},\dotsc,b_{1},t)}
{s_q(x,a_1,\dotsc,a_{2k},y+z,b_{2l},\dotsc,b_{1},t)}\notag\\
&=\frac{[e_a+e_b+x+t-1]_q}{[a+b+x+y+z+t-1]_q}.
\end{align}}
\normalsize By (\ref{simplify1}),  (\ref{simplify2}) and (\ref{simplify3}), we get
\begin{equation}
\frac{\Phi^{q}_{x-1,y,z,t}}{\Phi^{q}_{x,y,z,t}}\frac{\Phi^{q}_{x,y+1,z-1,t-1}}{\Phi^{q}_{x-1,y+1,z-1,t-1}}=\frac{[a+x+y]_q}{[a+b+x+2y+t]_q}.
\end{equation}
We can simplify the second term on the left-hand side of \eqref{recur4} in the same way (by comparing now the $t$-parameters of the numerator and denominator in each fraction). This leads to
\begin{equation}
\frac{\Phi^{q}_{x,y,z,t-1}}{\Phi^{q}_{x,y,z,t}}\frac{\Phi^{q}_{x-1,y+1,z-1,t}}{\Phi^{q}_{x-1,y+1,z-1,t-1}}=\frac{[b+y+t]_q}{[a+b+x+2y+t]_q}.
\end{equation}
Therefore, equation (\ref{recur4}) is equivalent to
\begin{equation}
\frac{[a+x+y]_q}{[a+b+x+2y+t]_q}
+q^{x+y+a}\frac{[b+y+t]_q}{[a+b+x+2y+t]_q}=1,
\end{equation}
which is readily checked.
\end{proof}

\section{A Schur function identity}

We recall the Cohn--Larsen--Propp's original tiling formula for the (dented) semihexagon as follows. Let $T_{m,n}(x_1,\dotsc,x_n)$ be the region obtained from the trapezoid of side lengths $m$, $n$, $m+n$, $n$ (clockwise from top) by removing the up-pointing unit triangles from along its bottom that are in positions $x_1,x_2,\dotsc,x_n$ as counted from left to right. Cohn, Larsen, and Propp \cite{CLP} showed that the number of lozenge tilings of  $T_{m,n}(x_1,\dotsc,x_n)$ is given by $\prod_{1\leq i<j\leq n}\frac{x_j-x_i}{j-i}$.

We notice that Theorem \ref{main1} can be restated as following:

\begin{thm}
Consider the hexagon of side-lengths $y+u,z+d,y+u,y+d,z+u,y+d$ (clockwise, starting from the northwestern side). Remove $a+b$ unit triangles from along the horizontal diagonal as follows: Remove $a$ consecutive unit triangles from  the left and $b$ consecutive unit triangles from the right, so that $u$ of the removed unit triangles point up, and $v$ point down (in particular, $u+v=a+b$).
Denote the resulting region by $F_{y,z}(a_1,\dotsc,a_u;b_1,\dotsc,b_{d})$, where the
$a_i$'s and $b_j$'s are the positions of the removed up-pointing and down-pointing unit triangles (from left to right), respectively.

Then
\small{\begin{align}\label{function1a}
\M&\left(F_{y,z}(a_1,\dotsc,a_u;b_1,\dotsc,b_{d})\right)=
 \frac{\Hf(y)\Hf(z)\Hf\left(a+b+2y+z\right)}
 {\Hf(y+z)\Hf\left(a+b+2y\right)}\notag\\
 &\times\frac{\Hf\left(a+b+y\right)}{\Hf\left(a+b+y+z\right)}
 \frac{\Hf\left(a+y\right)\Hf\left(b+y\right)}{\Hf\left(a\right)\Hf\left(b\right)}
 \frac{\Hf\left(d\right)\Hf\left(u\right)}{\Hf\left(d+y\right)\Hf\left(u+y\right)}\notag\\
 &\times \M\Big(T_{u,y+z+d}(a_1,a_2,\dotsc,a_u)\Big)\notag\\
  &\times \M\Big(T_{d,y+z+u}(b_1,b_2,\dotsc,b_d)\Big).
\end{align}}
\end{thm}

Note that, due to forced lozenges, removing an arbitrary fern is equivalent to removing some sequence of contiguous unit triangles, and therefore the above formulation does cover the general case.

By \cite{CLP} and equation (7.105) of \cite{Stanley}, we have\footnote{ The notation $1^n$ in the argument of a Schur function stands for $n$ arguments equal to 1.}
\begin{equation}
\M(T_{u,y+z+d}(a_1,a_2,\dotsc,a_u))=\prod_{1\leq i< j \leq u}\frac{a_j-a_i}{j-i}=\textbf{s}_{\lambda(A)}(1^{u}),
\end{equation}
and
\begin{equation}
\M(T_{d,y+z+u}(b_1,b_2,\dotsc,b_d))=\prod_{1\leq i< j \leq d}\frac{b_j-a_i}{j-i}=\textbf{s}_{\lambda(B)}(1^{d}),
\end{equation}
where $A=\{a_1,\dotsc,a_u\}$, $B=\{b_1,\dotsc,b_d\}$, and for a set $X=\{x_1<\cdots<x_n\}$, $\lambda(X)$ denotes the partition $(x_n-n+1,\dotsc,x_2-1,x_1)$.

On the other hand, it is not hard to see that we also have
\begin{align}
\M&\left(F_{y,z}(a_1,\dotsc,a_u;b_1,\dotsc,b_{d})\right)
=\sum_{|S|=y}\textbf{s}_{\lambda(A\cup S)}(1^{y+u})\textbf{s}_{\lambda(B \cup S)}(1^{y+d}),
\end{align}
where the sum runs over all sets $S=\{s_1,\dotsc,s_y\}$ with $a+1\leq s_1<s_2<\cdots<s_y\leq a+y+z$. Indeed, this follows from the fact that in each tiling, precisely $y$ of the $y+z$ unit segments on the lattice line from along which we removed the $a+b$ unit triangles are straddled by vertical lozenges.

Thus, we obtain the Schur function identity
\begin{align}
\label{id1}
\sum_{|S|=y}&\textbf{s}_{\lambda(A\cup S)}(1^{y+u})\textbf{s}_{\lambda(B \cup S)}(1^{y+d})=\frac{\Hf(y)\Hf(z)\Hf\left(a+b+2y+z\right)}
{\Hf(y+z)\Hf\left(a+b+2y\right)}\notag\\
 &\times\frac{\Hf\left(a+b+y\right)}{\Hf\left(a+b+y+z\right)}
 \frac{\Hf\left(a+y\right)\Hf\left(b+y\right)}{\Hf\left(a\right)\Hf\left(b\right)}
 \frac{\Hf\left(d\right)\Hf\left(u\right)}{\Hf\left(d+y\right)\Hf\left(u+y\right)}
\textbf{s}_{\lambda(A)}(1^{u})\textbf{s}_{\lambda(B)}(1^{d}).
\end{align}


\begin{op}
Let $T=\{t_1,t_2,\dotsc,t_{a+b+y+z}\}$, where $0<t_1<t_2<\dotsc<t_{a+b+y+z}$ are integers, and let $U$ and $D$ be two subsets of $T$ so that $|U|=u$, $|D|=d$, $U\cup D=\{t_1,\dotsc,t_a\}\cup\{t_{a+y+z+1},\dotsc,t_{a+b+y+z}\}$, and $U\cap D=\emptyset$. Is the ratio
\[\frac{\sum_{|S|=y}\textbf{s}_{\lambda(U\cup S)}(x_1,\dotsc,x_{y+n})\textbf{s}_{\lambda(D \cup S)}(x_1,\dotsc,x_{y+n+d-u})}{ \textbf{s}_{\lambda(U)}(x_1,\dotsc,x_n)\textbf{s}_{\lambda(D)}(x_1,\dotsc,x_{n+d-u})},\]
where the sum is taken over all subsets $S=\{s_1,\dotsc,s_y\}$ with $a+1\leq s_1<\cdots<s_y\leq a+y+z$, a simple polynomial in the $x_i$'s, and if so, what is it?
\end{op}

We note that for $x_i=q^{i-1}$, $1\leq i\leq n$, the ratio above is a monomial in $q$ with coefficient equal to the coefficient on the right-hand side of \eqref{id1}, and exponent of $q$ equal to a somewhat complicated expression, which can be worked out using Theorem \ref{main2}.

\section{Some further open problems}

The formulas we found for the number of lozenge tilings of doubly-intruded hexagons allow one to approach some statistical physics questions of a kind that does not seem to have been addressed yet in the literature.

For simplicity, focus on the case when $x+y=y+t=z=N$, all the lobe sizes $a_i$ and $b_i$ of the ferns are equal to 1, both ferns have an even number of lobes, and their combined length $m+n$ is equal to~$M$. Then the region $P_{x,y,z,t}(\textbf{a};\textbf{b})$ is a regular hexagon of side-length $M/2+N$, with two unit-lobed ferns of combined length $M$ intruding from the left and right at height $x$ above the horizontal diagonal.

Given $0<p<2$, in the limit as the parameters grow to infinity so that $M/(M/2+N)\to p$, $x/(M/2+N)\to h$ and $m/n\to r$ for some $0<h<1$ and $r>0$, one may ask for what values $h=h_p$ and $r=r_p$ does the number of lozenge tilings achieve its maximum.

To get some physical intuition about this question, recall that by \cite{CLP} a typical lozenge tiling of a large regular hexagon is such that the corresponding stepped surface is very close to a certain limiting surface (called the limit shape of a boxed plane partition). Due to this, outside the inscribed circle (the ``arctic circle'') the tiling is forced (``frozen'') with probability approaching 1, and near any given point inside the arctic circle it contains lozenges of the three orientations with specific probabilities.

The new ingredient in our set-up is the presence of the two intrusions, which in the fine mesh limit become two slits ``probing'' the random lozenge tiling. The geometry of the slits is not compatible with the pattern in the frozen regions, and it sets great restriction inside the arctic circle. Our question asks for the ratio of the slit lengths and the height where they need to be in order to have a maximum number of tilings.

\begin{conj} For any given $0<p<2$ and any fixed $0<h<1$, the maximum occurs when $r_p=1$ $($i.e., when the slits have the same length$)$.
\end{conj}


\begin{op} Given $0<p<2$, determine the value of $h_p$ for which the overall maximum is achieved. 
\end{op}


These open problems can be phrased also in the case when the hexagon is not asymptotically regular, and when the lobes of the ferns have more general side-lengths. The asymptotics of our explicit formulas should give some insigths into the effect of probing random tilings of hexagons with such intrusions.

\subsection*{Acknowledgements}
The authors would like to thank the anonymous reviewers for their careful reading of the manuscript and for their many helpful comments and suggestions. Special thanks go to the reviewer who provided us with the \texttt{Mathematica} code that produced  Figures \ref{house1} and \ref{house2}.


\end{document}